\documentclass[12pt,titlepage,doublesided]{book}
\pdfoutput=1
\usepackage[T1]{fontenc}

\usepackage{amsfonts,upgreek,euscript}
\usepackage{amsmath,amsthm,amssymb,colonequals}
\usepackage{indentfirst}
\usepackage{amscd}
\usepackage{tikz, tikz-cd}
\usepackage{float}
\usepackage{graphicx, subcaption}
\usepackage{wrapfig}
\usepackage{enumitem}
\usepackage[colorlinks]{hyperref}
\usepackage{listings}
\usepackage{mathtools}
\usepackage{xcolor}
\usepackage{lscape}
\usepackage{longtable}
\usepackage{adjustbox}
\usepackage[all]{xy}
\usetikzlibrary{arrows.meta,arrows,decorations.pathmorphing,decorations.pathreplacing,positioning,shapes.geometric,shapes.misc,decorations.markings,decorations.fractals,calc,patterns}

\tikzset{>=stealth',
	cvertex/.style={circle,draw=black,inner sep=1pt,outer sep=3pt},
	vertex/.style={circle,fill=black,inner sep=1pt,outer sep=3pt},
	star/.style={circle,fill=yellow,inner sep=0.75pt,outer sep=0.75pt},
	tvertex/.style={inner sep=1pt,font=\scriptsize},
	gap/.style={inner sep=0.5pt,fill=white}}

\usepackage{rotating}

\setlist[enumerate]{format=\normalfont}

\definecolor{codegreen}{rgb}{0,0.6,0}
\definecolor{codegray}{rgb}{0.5,0.5,0.5}
\definecolor{codepurple}{rgb}{0.58,0,0.82}
\definecolor{backcolour}{rgb}{0.95,0.95,0.92}
\definecolor{cardinal}{rgb}{0.77, 0.12, 0.23}

\lstdefinelanguage{magma}{
	sensitive=true,
	keywords={adj, and, assert, assert2, assert3, assigned, break, by, case, cat, catch, clear, cmpeq, cmpne, continue, declare, default, delete, diff, div, do, elif, else, end, eq, error, eval, exists, exit, false, for, forall, forward, fprint, fprintf, freeze, ge, gt, if, iload, import, in, intrinsic, is, join, le, load, local, lt, meet, mod, ne, not, notadj, notin, notsubset, or, print, printf, procedure, quit, random, read, readi, rep, repeat, require, requirege, requirerange, restore, return, save, sdiff, selecet, subset, time, then, to, true, try, until, vprint, vprintf, vtime, while, when, where, xor},
	otherkeywords={!, \#, \%, \&, *, +, -, ., /, <, >, ?, \, \^, |, :=, `},
	keywords=[2]{function},
	keywordstyle=\color{magenta},
	keywordstyle=[2]\color{codegreen},
	backgroundcolor=\color{backcolour}, 
	commentstyle=\color{blue},
	numberstyle=\tiny\color{codegray},
	stringstyle=\color{yellow},
	basicstyle=\ttfamily\footnotesize,
	breakatwhitespace=false,         
	breaklines=true,                 
	captionpos=b,                    
	keepspaces=true,                 
	numbers=left,                    
	numbersep=5pt,                  
	showspaces=false,                
	showstringspaces=false,
	showtabs=false,                  
	tabsize=2,
	comment=[l]{//},
	morecomment=[s]{/*}{*/},
	morestring=[b]',
	morestring=[b]"
}
\newtheorem{theorem}{Theorem}[section]
\newtheorem*{theorem*}{Theorem}
\newtheorem{prop}[theorem]{Proposition}
\newtheorem{lemma}[theorem]{Lemma}
\newtheorem{cor}[theorem]{Corollary}

\newtheorem{conj}[theorem]{Conjecture}

\theoremstyle{definition}
\newtheorem{definition}[theorem]{Definition}

\newtheorem{example}[theorem]{Example}
\newtheorem{setup}[theorem]{Setup}
\newtheorem{examples}[theorem]{Examples}
\newtheorem{remark}[theorem]{Remark}

\newtheorem{knitting}[theorem]{Knitting Algorithm}

\numberwithin{equation}{section}

\newcommand\N{\mathbb{N}}
\newcommand\Z{\mathbb{Z}}
\newcommand\Q{\mathbb{Q}}

\newcommand\C{\mathbb{C}}
\newcommand{\mfp}{\mathfrak{p}}
\newcommand{\mfm}{\mathfrak{m}}
\newcommand{\mfq}{\mathfrak{q}}

\def\Cl{\mathop{\rm Cl}\nolimits}
\def\op{\mathop{\rm op}\nolimits}

\def\GL{\mathop{\rm GL}\nolimits}
\def\SL{\mathop{\rm SL}\nolimits}
\def\CM{\mathop{\rm CM}\nolimits}
\def\uCM{\mathop{\underline{\rm CM}}\nolimits}
\newcommand{\Kroth}{\mathop{\rm {K}_{0}}}
\def\Div{\mathop{\rm{Div}}\nolimits}
\def\Prin{\mathop{\rm{Prin}}\nolimits}
\renewcommand{\div}{\mathop{\mathrm{div}}}

\renewcommand{\min}{\mathop{\mathrm{min}}}
\def\height{\mathop{\rm{ht}}\nolimits}

\DeclareMathOperator{\ePie}{\mathrm{e_{0}\Pi^{\uplambda}e_{0}}}
\newcommand{\GammaI}[1]{\mathrm {\Gamma}_{\kern -1pt #1}}

\def\depth{\mathop{\rm depth}\nolimits}
\def\mod{\mathop{\rm mod}\nolimits}
\def\modCat{\mathop{\rm mod}\nolimits}

\def\coh{\mathop{\rm coh}\nolimits}
\def\refl{\mathop{\rm ref}\nolimits}

\def\proj{\mathop{\rm proj}\nolimits}

\def\Hom{\mathop{\rm Hom}\nolimits}
\def\End{\mathop{\rm End}\nolimits}

\def\Ext{\mathop{\rm Ext}\nolimits}
\def\add{\mathop{\rm add}\nolimits}
\def\Cok{\mathop{\rm Cok}\nolimits}
\def\Ker{\mathop{\rm Ker}\nolimits}
\def\ker{\mathop{\rm ker}\nolimits}

\def\rk{\mathop{\sf rk}\nolimits}
\def\Im{\mathop{\rm Im}\nolimits}

\def\Spec{\mathop{\rm Spec}\nolimits}

\def\Groth{\mathop{\rm G_{0}}\nolimits}
\def\oGroth{\mathop{\rm \widetilde{G}_{0}}\nolimits}
\def\dim{\mathop{\rm dim}\nolimits}
\def\perf{\mathop{\rm{perf}}\nolimits}
\def\rad{\mathop{\rm rad}\nolimits}

\def\Dsg{\mathop{\rm{D}_{\sf sg}}\nolimits}
\def\Db{\mathop{\rm{D}^b}\nolimits}

\def\Krothsg{\mathop{\rm{K}_{0}^{\sf sg}}\nolimits}

\newcommand{\con}{\mathrm{con}}

\newcommand{\CL}{\mathrm{\Lambda}_{\con}}

\def\ab{\mathop{\rm ab}\nolimits}

\def\inc{\mathop{\rm inc}\nolimits}
\def\aff{\mathop{\rm aff}\nolimits}
\newcommand{\redzero}{\textcolor{cardinal}{0}}
\newcommand{\greenzero}{\textcolor{codegreen}{0}}

\newcommand{\ADE}{\sf{ADE}}
\newcommand{\tA}{\sf{A}}
\newcommand{\tAn}[1]{\tA_{#1}}
\newcommand{\ctAn}{\mathsf{c}\tAn}

\newcommand{\teA}{\sf{\widetilde{A}}}
\newcommand{\teAn}[1]{\teA_{#1}}

\newcommand{\teD}{\sf{\widetilde{D}}}
\newcommand{\teDn}[1]{\teD_{#1}}

\newcommand{\teE}{\sf{\widetilde{E}}}
\newcommand{\teEn}[1]{\teE_{#1}}

\newcommand{\cA}{\mathcal{A}}
\newcommand{\cB}{\mathcal{B}}
\newcommand{\cC}{\mathcal{C}}

\newcommand{\cE}{\mathcal{E}}

\newcommand{\cO}{\mathcal{O}}

\newcommand{\cS}{\mathcal{S}}

\newcommand{\eB}{\EuScript{B}}
\newcommand{\eC}{\EuScript{C}}

\newcommand{\eR}{\EuScript{R}}
\newcommand{\eS}{\EuScript{S}}

\usepackage{cite}  
\usepackage[top=1.8cm, bottom=1.8cm,left=2.75cm,right=2.75cm]{geometry}
\usepackage[onehalfspacing]{setspace}  
\begin{document}
\begin{titlepage}
	\centering
	\vspace*{3cm}  
	\bfseries\Large
	The K-theory of (compound) Du Val singularities\\
	\vspace{3cm}
	\normalfont\large
	Kellan Steele\\
	\vspace{2cm}
	Submitted in fulfilment of the requirements for the\\
	Degree of Doctor of Philosophy\\
	\vspace{2cm}
	School of Mathematics and Statistics\\
	College of Science and Engineering\\
	University of Glasgow\\
	\vspace{1cm}
	\includegraphics[scale=0.125]{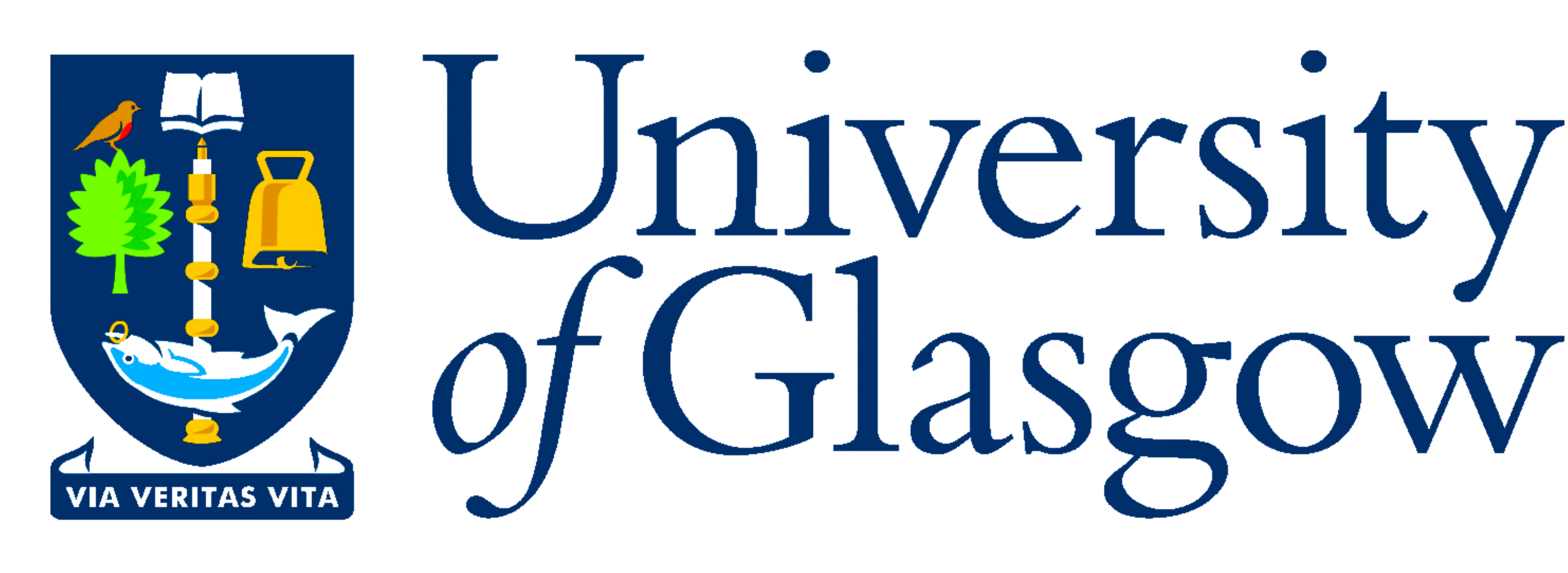}
	\\
	\vspace{1cm}
	September 2020
\end{titlepage}
\frontmatter
\chapter{Abstract}

This thesis gives a complete description of the Grothendieck group and divisor class group for large families of two and three dimensional singularities.
The main results presented throughout, and summarised in Theorem~\ref{thm:conclusion}, give an explicit description of the Grothendieck group  and class group of Kleinian singularities, their deformations, and compound Du Val (cDV) singularities in a variety of settings.
For such rings $R$, the main results assert that there exists an isomorphism $\Groth(R) \cong \Z \oplus \Cl(R)$, and the class group is explicitly presented.

More precisely, we establish these results for 2-dimensional deformations of global type $\tA$ Kleinian singularities, 3-dimensional isolated complete local cDV singularities admitting a noncommutative crepant resolution, any 3-dimensional type $\tA$ complete local cDV singularity, polyhedral quotient singularities (which are non-isolated), and any isolated cDV singularity admitting a minimal model with only type $\ctAn{n}$ singularities.
We also provide an example of a large class of higher dimensional quotient singularities for which this isomorphism does not hold, suggesting that this is a low dimensional phenomenon.

This work requires a range of tools including, but not limited to, Nagata's theorem, knitting techniques, Kn\"{o}rrer periodicity, the singularity category, and the computer-algebra system \texttt{MAGMA}.
Of particular note is the application of knitting techniques which leads to independently interesting results on the symmetry of the quivers underlying the modifying algebras of Kleinian and cDV singularities.
\tableofcontents
\listoftables
\listoffigures
\chapter{Acknowledgements}

This thesis would not have been possible without the support of many people at different times and in different ways.

Many, many thanks to my supervisors, Gwyn Bellamy and Michael Wemyss; thank you for your guidance, patience, positivity, and encouragement, and for teaching me how to squabble about maths properly.
This thesis would not have been possible without your help, and it has been a true pleasure to work with and learn from both of you.

I am grateful to Martin Lorenz for so kindly sharing his code with me.
I am deeply indebted to Julie Beier and Charlie Peck; thank you for your advice and inspiration, and many career-defining chats.

I would like to thank my office mates and friends, Andy, James, Jamie, Jay, Luke, Niall, Okke, Pete, Sarah, and Vitalijs for many mathematical conversations and non-mathematical distractions.
Special thanks go to Ross, for making the office a particularly fun place to be and for carefully proofreading this thesis.
I owe much to Anna and Roxy; your affection and presence in my life has loved, carried, and lifted me throughout my time at Glasgow.

I should especially like to thank my family.
I would not be the person I am today without the courage, enthusiasm, and character of my siblings, Spenser and Bryna.
Carolyn and Riley, thank you for your unfailing support and invaluable humour.
My husband, Flynn; thank you for endlessly encouraging and believing in me and putting up with the unavoidable side effects of doing a PhD in parallel with attempting to have a life.

Finally, my greatest thanks is to my parents.
Your love, encouragement, and support have been a constant source of inspiration throughout my life.
Thank you for never letting me think that anything was beyond my reach; this accomplishment would not have been possible without you.
Thank you.
\chapter{Declaration}

I declare that, except where explicitly stated otherwise in the text, this thesis is the result of my own original work and has not been submitted for any other degree at the University of Glasgow or any other institution.
\mainmatter
\chapter{Introduction}\label{ch:introduction}

Fundamentally, mathematicians are interested in the problem of classification which guides many areas of mathematical research.
Mathematicians have numerous ways of measuring when objects are isomorphic (the same) or not isomorphic (different), and one method commonly used is that of invariants.
An invariant is a property of a mathematical object, or a class of objects, which remains unchanged after the application of operations or transformations.
Often, invariants are not strong enough to determine when two objects are isomorphic; instead, they are useful for determining when two objects are \emph{not} isomorphic.

In algebraic geometry, one uses algebraic techniques to study geometric objects.
In particular, the algebra determines the geometry, and the geometry determines the algebra, allowing for a rich study of both.
The geometric objects of study are called \emph{algebraic varieties} (in our case affine schemes), and traditionally they are viewed as being the set of solutions of a system of polynomial equations over an algebraically closed field.
The crossover between algebra and geometry can be seen through the association to a commutative noetherian ring $R$ the algebraic variety $\Spec R$ given by the set of all prime ideals of $R$.

Suppose that we have two algebraic varieties $X$ and $X'$.
Naturally, one might ask, when are $X$ and $X'$ isomorphic?
Or, more reasonably, when are $X$ and $X'$ not isomorphic?
A good starting point for answering this question is to study the invariants of $X$ and $X'$.

\section{Motivation}

When studying varieties, one quickly learns that singular varieties are much worse behaved than non-singular varieties, but one also learns that this is precisely what makes them interesting.
This thesis is motivated by trying to understand when singular varieties behave like smooth (i.e., non-singular) varieties.
Specifically, we study the invariants \emph{divisor class groups} and \emph{Grothendieck groups} (or K-theory) of certain singularities.
These are both interesting invariants which, as we show, are more intertwined than one might expect.
In general, neither is easy to compute.
Theoretically, one can deduce the class group from the Grothendieck group. The latter is often easier to compute, but we have shifted the difficulty to the equally hard problem of extracting the class group from the Grothendieck group.  

To state our main results, we now swap to using more technical language; the relevant terms are defined in Chapters~\ref{ch:comm_alg_prelims} and~\ref{ch:other_prelims}.
Let $R$ be a 3-dimensional Gorenstein ring with rational singularities.
In favourable settings, $\Spec R$ admits a crepant resolution; however, this does not always happen.
When $\Spec R$ does not admit a crepant resolution, usually one instead studies minimal models $X \to \Spec R$ where $X$ is allowed to be mildly singular.
At its root, this thesis is motivated by understanding how homologically similar minimal models are to crepant resolutions.
One tool for measuring this is the singularity category $\Dsg(X)$ of $X$; see Chapter~\ref{ch:minimal_models}.
The singularity category of $X$ is trivial if and only if $X$ is smooth, so in general for minimal models, $\Dsg(X) \neq 0$.
\begin{conj}\label{conj:motivation}
	Let $R$ be a 3-dimensional Gorenstein ring with rational singularities and $X \to \Spec R$ a minimal model.
	Then $\Kroth(\Dsg(X)) = 0$.
\end{conj}

This turns out to be connected to, and in fact determined by, problems in commutative algebra.
This motivates our focus on two and three dimensional commutative rings.
Now, let $R$ be a normal noetherian integral domain and write $\Cl(R)$ for the class group of $R$ and $\Groth(R)$ for the Grothendieck group of $\modCat R$.
Our goal in this thesis is to gain insight into the following surprising question.
When does
\begin{equation}\label{eq:Groth=Z+Cl_general}
\Groth(R) \cong \Z \oplus \Cl(R)
\end{equation}
hold?
This turns out to be relevant to the smooth versus singular behaviour discussed above, and is surprising because it is not typical behaviour in dimensions three and higher.
In this thesis, we look at when \eqref{eq:Groth=Z+Cl_general} is true, and some situations where it fails.
We prove that this isomorphism holds in a variety of different settings, for example: arbitrary $\ctAn{n}$ singularities, arbitrary isolated cDV singularities with noncommutative crepant resolution(s) (NCCRs), arbitrary cDV singularities with minimal model which has type $\ctAn{n}$ singularities, amongst others.
Our main results are summarised in Theorem~\ref{thm:conclusion}.

\section{Kleinian singularities}\label{sec:intro:ksings}
	A \emph{quotient singularity} has the form $Y/G$, where $Y$ is a smooth variety, and $G$ is a finite\footnote{We do not consider the infinite case in this thesis.} group of automorphisms of $Y$.
	Our main focus is on the quotient singularities $\C^{n}/G$ with $G$ a non-trivial finite subgroup of SL$(n, \C)$.
	These quotient singularities are both Gorenstein \cite{Wat74} and rational \cite{Vie77}.
	We mainly restrict to quotient singularities $\C^{n}/G$ in dimension $2$ and $3$.
	
	Consider the $2$-dimensional case, that is $\C^{2}/G$ where $G$ is a non-trivial finite subgroup of SL$(2, \C)$.
	These groups were classified by Klein: there are two infinite families and three exceptional groups. 
	These quotient singularities are called \emph{Kleinian} or \emph{Du Val singularities}\footnote{In the literature, these have many names such as rational double points or $\ADE$ singularities.}.
	Klein proved that the finite subgroups $G$ of SL$(2, \C)$ are precisely the binary polyhedral groups \cite{klein84}.
	This information is summarised in the three columns on the righthand side of Table~\ref{table:k_sings_dynkin_diagrams}.

	The (coordinate rings of) Kleinian singularities are defined to be the invariant rings $\C[U, V]^{G}$.
	They are generated by three elements, subject to one relation.
	In other words, each ring $R$ is isomorphic to $\C[u, v, x]/(f)$ for some irreducible polynomial $f$.
	As a consequence, each ring $\C[u, v, x]/(f)$ can be seen as functions on a surface in $\C^{3}$ defined by the zero set of $f$.
	Each of these surfaces has a unique singular point at the origin; they are isolated singularities.

\section{The McKay correspondence}\label{sec:intro:mckay_correspondence}
	Consider a finite subgroup $G$ of SL$(2, \C)$ and the natural $2$-dimensional representation $\upomega$ of $G$.
	Let $\upvarrho_{0}, \hdots, \upvarrho_{n}$ be all the irreducible representations of $G$ where $\upvarrho_{0}$ is the trivial representation.
	Furthermore, let $m_{i, j}$ be the multiplicity of $\upvarrho_{i}$ in $\upomega \otimes \upvarrho_{j}$.
	With this setup we give a description, following \cite{mckay80}, of the McKay quiver associated to each finite subgroup $G$ of SL$(2, \C)$.
	\begin{definition}
		The \emph{McKay quiver} of $G$, with respect to $\upomega$, is the quiver (i.e., directed graph) with vertices $0, \hdots, n$ corresponding to the irreducible representations $\upvarrho_{0}, \hdots, \upvarrho_{n},$ and $m_{i, j}$ arrows from vertex $i$ to vertex $j$.
	\end{definition}
	For Kleinian singularities, $\upomega$ is self-dual and so $m_{i, j} = m_{j, i}$ for all $i, j$.
	The McKay correspondence, due to John McKay in \cite{mckay80}, shows that there is a one-to-one correspondence between the McKay quivers of the finite subgroups of SL$(2, \C)$ and the extended Dynkin diagrams.
	This is described in Table~\ref{table:k_sings_dynkin_diagrams}.
	The five (affine) extended Dynkin diagram types, depicted in Figure~\ref{fig:extended_dynkin_diagrams}, are $\teAn{n}$ $(\sf{n} \geq 0)$ and $\teDn{n}$ $(\sf{n} \geq 4)$ with $\sf{n}+1$ vertices, as well as $\teEn{6}$, $\teEn{7}$, and $\teEn{8}$, where each \emph{extended vertex} is shown as an empty circle.
	In each extended Dynkin diagram, the extended vertex corresponds to the trivial representation $\upvarrho_{0}$.
	The number at each vertex is the dimension of the corresponding representation.
	\begin{figure}[H]
		\centering
		$\teAn{n} \quad$
		\begin{tikzpicture}[
		dot/.style = {circle, fill, inner sep=2pt,
			node contents={}},
		every label/.append style = {font=\footnotesize},  
		every edge/.style ={draw, line width=.5pt}
		]
		\path	
		(2,1)	node (v5) [dot, draw=black, fill=white, label=above: $1$]
		(0,0)	node (v0) [dot, label=below: $1$]
		(1,0)	node (v1) [dot, label=below: $1$]
		(2,0)	node (v2) {$\hdots$}
		(3,0)	node (v3) [dot, label=below: $1$]
		(4,0)	node (v4) [dot, label=below: $1$];
		\path   (v0)	edge (v5)
		(v0)	edge (v1)
		(v1)	edge (v2)
		(v2)	edge (v3)
		(v3)	edge (v4)
		(v4)	edge (v5);
		\end{tikzpicture} \qquad \qquad
		$\teDn{n} \quad$
		\begin{tikzpicture}[
		dot/.style = {circle, fill, inner sep=2pt,
			node contents={}},
		every label/.append style = {font=\footnotesize},  
		every edge/.style ={draw, line width=.5pt}
		]
		\path	(0,0)	node (v0) [dot, label=below: $1$]
		(1,0)	node (v1) [dot, label=below: $2$]
		(2,0)	node (v2) {$\hdots$}
		(3,0)	node (v3) [dot, label=below: $2$]
		(4,0)	node (v4) [dot, label=above: $2$]
		(5,0)	node (v5) [dot, label=below: $1$]
		(4,-1)	node (v6) [dot, label=below: $1$]
		(1,1)	node (v7) [dot,draw=black, fill=white, label=above: $1$];
		\path   (v0)	edge (v1)
		(v1)	edge (v2)
		(v1)	edge (v7)
		(v2)	edge (v3)
		(v3)	edge (v4)
		(v4)	edge (v5)
		(v4)	edge (v6);
		\end{tikzpicture}\\
		\qquad \newline
		\qquad \newline
		$\teEn{6} \quad$
		\begin{tikzpicture}[
		dot/.style = {circle, fill, inner sep=2pt,
			node contents={}},
		every label/.append style = {font=\footnotesize},  
		every edge/.style ={draw, line width=.5pt}
		]
		\path	(0,0)	node (v0) [dot, label=below: $1$]
		(1,0)	node (v1) [dot, label=below: $2$]
		(2,0)	node (v2) [dot, label=below: $3$]
		(3,0)	node (v3) [dot, label=below: $2$]
		(4,0)	node (v4) [dot, label=below: $1$]
		(2,1)	node (v5) [dot, label=left: $2$]
		(2,2)	node (v6) [dot,draw=black, fill=white, label=left: $1$];
		\path   (v0)	edge (v1)
		(v1)	edge (v2)
		(v2)	edge (v3)
		(v3)	edge (v4)
		(v2)	edge (v5)
		(v2)	edge (v6);
		\end{tikzpicture} \qquad \qquad
		$\teEn{7} \quad$
		\begin{tikzpicture}[
		dot/.style = {circle, fill, inner sep=2pt,
			node contents={}},
		every label/.append style = {font=\footnotesize},  
		every edge/.style ={draw, line width=.5pt}
		]
		\path	(-1,0)	node (v7) [dot,draw=black, fill=white, label=below: $1$]
		(0,0)	node (v0) [dot, label=below: $2$]
		(1,0)	node (v1) [dot, label=below: $3$]
		(2,0)	node (v2) [dot, label=below: $4$]
		(3,0)	node (v3) [dot, label=below: $3$]
		(4,0)	node (v4) [dot, label=below: $2$]
		(5,0)	node (v5) [dot, label=below: $1$]
		(2,1)	node (v6) [dot, label=above: $2$];
		\path   (v7)	edge (v0)
		(v0)	edge (v1)
		(v1)	edge (v2)
		(v2)	edge (v3)
		(v3)	edge (v4)
		(v4)	edge (v5)
		(v2)	edge (v6);
		\end{tikzpicture}\\
		\qquad \newline
		\qquad \newline
		$\teEn{8} \quad$
		\begin{tikzpicture}[
		dot/.style = {circle, fill, inner sep=2pt,
			node contents={}},
		every label/.append style = {font=\footnotesize},  
		every edge/.style ={draw, line width=.5pt}
		]
		\path	(0,0)	node (v0) [dot, label=below: $2$]
		(1,0)	node (v1) [dot, label=below: $4$]
		(2,0)	node (v2) [dot, label=below: $6$]
		(3,0)	node (v3) [dot, label=below: $5$]
		(4,0)	node (v4) [dot, label=below: $4$]
		(5,0)	node (v5) [dot, label=below: $3$]
		(6,0)	node (v6) [dot, label=below: $2$]
		(7,0)	node (v7) [dot,draw=black, fill=white, label=below: $1$]
		(2,1)	node (v8) [dot, label=above: $3$];
		\path   (v0)	edge (v1)
		(v1)	edge (v2)
		(v2)	edge (v3)
		(v3)	edge (v4)
		(v4)	edge (v5)
		(v5)	edge (v6)
		(v6)	edge (v7)
		(v2)	edge (v8);
		\end{tikzpicture}
		\caption{The extended Dynkin diagrams \cite{eting11}}\label{fig:extended_dynkin_diagrams}
	\end{figure}
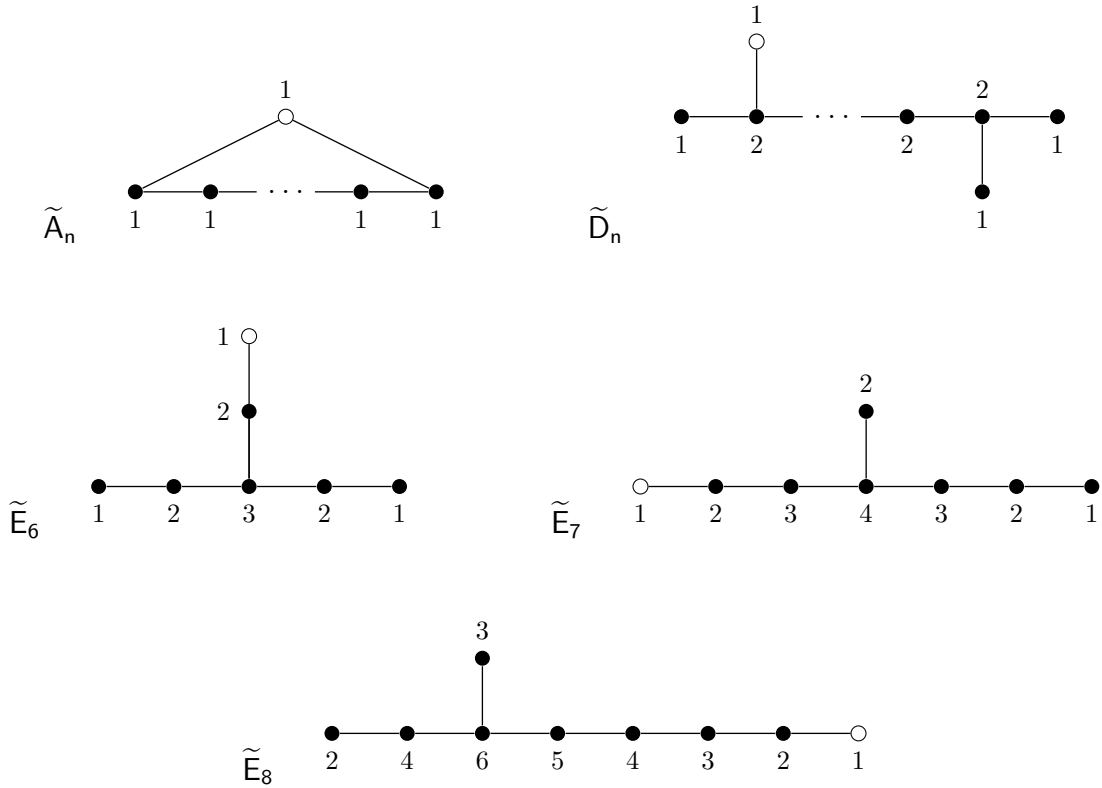
	By removing the extended vertex from an extended Dynkin diagram one obtains a (simply laced) Dynkin diagram.
	A quiver is called \emph{Dynkin} (respectively, \emph{extended Dynkin}) if the underlying graph is of $\ADE$ Dynkin (respectively, extended Dynkin) type.

	From this correspondence it follows that each McKay quiver associated to a Kleinian singularity is, in fact, the double quiver (see \S\ref{sec:deformed_preproj_deformed_ksings}) of an extended $\ADE$ Dynkin diagram \cite{mckay80}.
	The corresponding diagram appears in the first column of Table~\ref{table:k_sings_dynkin_diagrams}.
	\begin{table}[H]
		\centering \doublespacing
		\begin{tabular}{c|c|c|c}
			Type & $G$ & $|G|$  & $f$ \\ \hline
			$\teAn{n}$ & Cyclic $n \geq 1$, $\Z_{n+1}$ & $n+1$ & $x^{n+1} + uv = 0$ \\
			$\teDn{n}$ & Binary dihedral $n \geq 4$, $\mathcal{BD}_{4(n+1)}$ & $4(n-2)$ &  $x(v^{2} - x^{n+1}) + u^{2} = 0$ \\
			$\teEn{6}$ & Binary tetrahedral, $\mathcal{T}$ & $24$ & $x^{4} + v^{3} + u^{2} = 0$ \\
			$\teEn{7}$ & Binary octahedral, $\mathcal{O}$ & $48$ & $x^{3} + xv^{3} + u^{2} = 0$ \\
			$\teEn{8}$ & Binary icosahedral, $\mathcal{I}$ & $120$ & $x^{5} + v^{3} + u^{2} = 0$
		\end{tabular}
		\caption{The Kleinian singularities and extended Dynkin diagrams}
		\label{table:k_sings_dynkin_diagrams}
	\end{table}

	\subsection{The geometric McKay correspondence}\label{subsec:intro:geometric_mckay_correspondence}
		The McKay correspondence can be interpreted from a geometric point of view; an in-depth introduction to this may be found in \cite[Chapter~6]{LeuWie12}.
		As before, let $G$ be a finite non-trivial subgroup of SL$(2, \C)$ and $R \colonequals \C[U, V]^{G}$.
		The variety $\Spec R$ is a Kleinian singularity with an isolated singular point at the origin.
		An important result in algebraic geometry, due to Hironaka \cite{Hi64}, states that every singular variety $\Spec R$ over a field of characteristic $0$ has a \emph{resolution of singularities}.
		In other words, there exists a non-singular variety $X$ and a proper birational map $X \to \Spec R$.
		A \emph{minimal resolution} is a resolution $f \colon X \to \Spec R$ through which any other resolution factors.
		We provide a more in-depth discussion of resolutions of singularities in Chapter~\ref{ch:other_prelims}, with a particular focus on \emph{minimal models} and \emph{noncommutative crepant resolutions (NCCRs)}.

        In general, a minimal resolution of a variety does not exist, but for surfaces, they do, and they are unique.
		As such, it makes sense to talk about \emph{the} minimal resolution of a Kleinian singularity.
		In Figure~\ref{fig:A_1_ksing_illustration}, we illustrate the resolution of an $\tAn{1}$ Kleinian singularity, where the resolution replaces the singular point at the origin by a circle (indicated in red).
		\begin{figure}[H]
			\centering
			\begin{tikzpicture}[bend angle=20, looseness=1,scale=1]
			\node (p1) at (-3,0) {};
			\draw (0,1) ellipse (0.5cm and 0.15cm);
			\draw[gray,densely dotted] (0.5,-1) arc (0:180:0.5cm and 0.15cm);
			\draw (0.5,-1) arc (0:-180:0.5cm and 0.15cm);  
			\draw (-0.5,1) -- (0.5,-1);
			\draw (0.5,1) -- (-0.5,-1);
			\draw (-2.5,1) ellipse (0.5cm and 0.15cm);
			\draw[gray,densely dotted] (-2,-1) arc (0:180:0.5cm and 0.15cm);
			\draw (-2,-1) arc (0:-180:0.5cm and 0.15cm);  
			\draw [bend left] (-3,1) to node (m2a) {} (-3,-1);
			\draw [bend right] (-2,1) to node (m1a){} (-2,-1);
			\coordinate (m1) at (m1a);
			\coordinate (m2) at (m2a);
			\draw[->] (-1.5,0) -- node[above] {} (-0.75,0);
			\draw[densely dotted,red] (m1a) arc (0:180:0.3cm and 0.1cm);
			\draw[red] (m1a) arc (0:-180:0.3cm and 0.1cm); 
			\end{tikzpicture}\caption{Resolution of the cone singularity $\tAn{1}$}\label{fig:A_1_ksing_illustration}
		\end{figure}
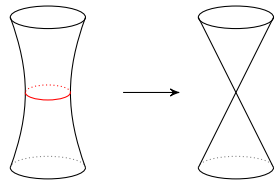
		To the minimal resolution $f \colon X \to \Spec R$ one can associate the \emph{exceptional divisor}.
		The exceptional divisor is a finite chain of curves.
		The arrangement of these curves is encoded in the \emph{dual graph} of the exceptional divisor, in which a vertex replaces each curve, and an edge connects two vertices if the corresponding curves intersect.
		For example, the curve configuration below on the left corresponds to the dual graph on the right.
		\[
		\begin{array}{c}
			\begin{tikzpicture}
			\node at (0,0)
			{\begin{tikzpicture} [bend angle=25, looseness=1,transform shape, rotate=-10]
				\draw[red] (0,0,0) to [bend left=25] (2,0,0);
				\draw[red] (1.8,0,0) to [bend left=25] (3.8,0,0);
				\draw[red] (-1.8,0,0) to [bend left=25] (0.2,0,0);
			\end{tikzpicture} };
			
			\node at (7,0) {
			\begin{tikzpicture} [bend angle=25, looseness=1,transform shape, rotate=-10]
			\filldraw [black!60] (1,0.25,0) circle (2pt);
			\filldraw [black!60] (3,0.25,0) circle (2pt);
			\filldraw [black!60] (-1,0.25,0) circle (2pt);
			\node (1) at (-1,0.25,0) {};
			\node (2) at (1,0.25,0) {};
			\node (3) at (3,0.25,0) {};
			\draw[black!80] (1) -- (2);
			\draw[black!80] (2) -- (3);
			\end{tikzpicture}};
			\draw [->,decorate, 
			decoration={snake,amplitude=.6mm,segment length=3mm,post length=1mm}] 
			(2.75,0.2) -- (3.75,0.2);
			\end{tikzpicture}
		\end{array}
		\]
		It is well known that the resulting dual graph is always an $\ADE$ Dynkin diagram and that any $\ADE$ Dynkin diagram can be realised as the dual graph of the minimal resolution of some Kleinian singularity \cite{duVal34}.
	
\section{Viehweg's setting}\label{sec:intro:viehwegs_setting}
In \cite{Vie77}, Viehweg studies singularities of the form
\[
	A \colonequals \frac{\C[u,v,x_{1}, \hdots , x_{n}]}{(uv - f_{1}^{a_{1}} \cdots  f_{t}^{a_{t}})},
\]
where $f_i\in\C[x_{1}, \hdots , x_{n}]$ are irreducible and pairwise coprime, and each $a_{i} \ge 1$.
Such singularities are a generalisation of type $\tA$ Kleinian singularities and we refer to singularities of this form as being in `Viehweg's setting'.
In Chapter~\ref{ch:divisor_cl_grps}, we focus on singularities in Viehweg's setting, and prove the following results.
\begin{theorem*}
	Let $A$ be the ring defined above, then
	\begin{enumerate}[label=(\arabic*)]
			\item $\Cl(A) \cong \Z^{\oplus t}/ (a_1,\hdots,a_t)$.
			\item If further the Krull dimension of $A$ is two, then \eqref{eq:Groth=Z+Cl_general} holds, namely
			\[ \Groth(A) \cong \Z \oplus \Cl(A). \]
	\end{enumerate}
\end{theorem*}
This has applications to deformations of Kleinian singularities, explained in \S\ref{sec:deformed_preproj_deformed_ksings}, and to $\ctAn{n}$ singularities in Chapter~\ref{ch:cDVs}.

\section{Compound Du Val singularities}\label{sec:intro:cDVs}
Until now, we have mainly discussed the $2$-dimensional setting.
But in this thesis, we do far more than just look at dimension $2$.
A natural question to ask is whether part (2) of the above result generalises to higher dimensions, and a natural way to generalise Kleinian singularities is to look at compound Du Val singularities.
Intuitively, these are $3$-dimensional analogues of Kleinian singularities.

For a variety $X$, write $\cO_{X}$ for the sheaf of functions on the space, and for $p \in X$ write $\widehat{\cO}_{X,p}$ for the completion of $\cO_{X,p}$ at the unique maximal ideal.
Formally, $X$ is said to have \emph{compound Du Val (cDV) singularities} if and only if every $\widehat{\cO}_{X,p}$ is of the form
\[
\frac{\C [[ u, v, x,y ]]}{(f(u, v, x) + yh(u, v, x, y))}
\]
where $\C[u, v, x]/(f)$ is a Kleinian singularity and $h$ is arbitrary.
One source of such singularities is $\C^{3}/G$ where $G$ is one of the finite subgroups of $\SL(3, \C)$ described in \ref{cDV_3dim_A} -- \ref{cDV_3dim_E_8} below (see, e.g. \cite{NollaSekiya11FlopsAM}).
\begin{enumerate}[label=(\alph*)]
	\item\label{cDV_3dim_A}
	The cyclic group of order $n$.
	That is,
	\[ \frac{1}{n}(1,n-1,0) \colonequals \left\langle \left( \begin{matrix}
	\epsilon & 0 & 0 \\
	0 & \epsilon^{-1} & 0 \\
	0 & 0 & 1
	\end{matrix} \right) \right\rangle, \text{ where } \upvarepsilon = e^{2 \pi i /(n+1)}. \]
	
	\item\label{cDV_3dim_D}
	The dihedral group $D_{2n}$ of order $2n$.
	The group $D_{2n}$ is generated by the matrices
	\[ \left( \begin{matrix}
	\epsilon & 0 & 0 \\
	0 & \epsilon^{n-1} & 0 \\
	0 & 0 & 1
	\end{matrix} \right) \text{ and } \left( \begin{matrix}
	0 & 1 & 0 \\
	1 & 0 & 0 \\
	0 & 0 & -1
	\end{matrix} \right), \text{ with } \upvarepsilon = e^{2 \pi i /n}. \]
	
	\item\label{cDV_3dim_E_6}
	The trihedral group (also known as the tetrahedral group) $\mathbb{T}$ of order 12.
	Explicitly,
	\[\mathbb{T} = \left\langle \left( \begin{matrix}
	-1 & 0 & 0 \\
	0 & -1 & 0 \\
	0 & 0 & 1
	\end{matrix} \right),
	\left( \begin{matrix}
	0 & 1 & 0 \\
	0 & 0 & 1 \\
	1 & 0 & 0
	\end{matrix} \right) \right\rangle.\]
	
	\item\label{cDV_3dim_E_7}
	The octahedral group $\mathbb{O}$ of order 24.
	Explicitly,
	\[\mathbb{O} = \left\langle \left( \begin{matrix}
	0 & -1 & 0 \\
	1 & 0 & 0 \\
	0 & 0 & 1
	\end{matrix} \right),
	\left( \begin{matrix}
	0 & 1 & 0 \\
	0 & 0 & 1 \\
	1 & 0 & 0
	\end{matrix} \right) \right\rangle.\]
	
	\item\label{cDV_3dim_E_8}
	The icosahedral group $\mathbb{I}$ of order 60.
	Explicitly,
	\[\mathbb{I} = \left\langle \left( \begin{matrix}
	1 & 0 & 0 \\
	0 & \epsilon & 0 \\
	0 & 0 & \epsilon^{4}
	\end{matrix} \right),
	\frac{1}{\sqrt{5}} \left( \begin{matrix}
	1 & 1 & 1 \\
	2 & s & t \\
	2 & t & s
	\end{matrix} \right) \right\rangle\]
	where $\upvarepsilon = e^{2 \pi i/5}, s = \upvarepsilon^{2} + \upvarepsilon^{3} = \frac{-1-\sqrt{5}}{2}$ and $t= \upvarepsilon + \upvarepsilon^{4} = \frac{-1 + \sqrt{5}}{2}.$
\end{enumerate}
By \cite[Proposition~6.1]{BIKR}, another source of cDV singularities are rings of the form $\C[u, v, x, y] / (uv - f(x,y))$, which further justifies our interest in Viehweg's setting of \S\ref{sec:intro:viehwegs_setting}.

As is a common theme in this thesis, typically one studies cDVs by reducing back to dimension $2$ (i.e., to Kleinian singularities).
This is done via slicing by a generic hyperplane section.
Hence, a cDV is a threefold such that slicing by a generic hyperplane section yields a Kleinian singularity.
As discussed previously, such surface singularities have been studied intensely and are well understood.
Just as with Kleinian singularities, cDVs are classified into $\ADE$ Dynkin type.

In Chapter~\ref{ch:cDVs}, we make the following conjecture and prove that both it and the isomorphism \eqref{eq:Groth=Z+Cl_general} hold for cDVs in various settings.
\begin{conj}[{=\ref{conj:K_0_uCMR=Cl(R)}}]
	Let $R$ be a local cDV singularity.
	Then
	\[ 
	\Kroth(\uCM R) \cong \Cl(R).
	\]
\end{conj}
In addition, in \S\ref{sec:nonisolated_cDV}, we provide evidence which supports the idea that \eqref{eq:Groth=Z+Cl_general} holds for non-isolated cDVs.
Then, in Chapter~\ref{ch:minimal_models}, we prove that \eqref{eq:Groth=Z+Cl_general} holds for any isolated cDV singularity which admits a minimal model with only type $\ctAn{n}$ singularities.

\section{Organisation of the thesis}\label{sec:intro:organisation_of_thesis}
Chapters~\ref{ch:comm_alg_prelims} and~\ref{ch:other_prelims} contain background material, the first covers the more algebraic notions and the second covers the more geometric material.
In Chapter~\ref{ch:divisor_cl_grps}, we calculate the divisor class group of rings in Viehweg's setting and prove that \eqref{eq:Groth=Z+Cl_general} holds in dimension two in this setting.
These results are an extension of results already known in the local setting to the much more difficult global setting.
In addition, using these results we give a description of the Grothendieck group of the centre of the deformed preprojective algebra and prove that it is isomorphic to the description given in \cite{cbh98}.
In Chapter~\ref{ch:knitting}, we introduce a knitting algorithm and prove independently interesting results on the symmetry of the quivers associated to modifying algebras of all Kleinian and cDV singularities.
We use this symmetry in Chapter~\ref{ch:cDVs}, where we prove results similar to those of Chapter~\ref{ch:divisor_cl_grps} but for 3-dimensional local cDV singularities, generalising the thesis of Navkal \cite{navkal13}.
Specifically, we prove that \eqref{eq:Groth=Z+Cl_general} holds for all isolated local cDV singularities admitting an NCCR.
We also prove \eqref{eq:Groth=Z+Cl_general} in various other settings, and provide evidence for similar results in the non-isolated case.
In Chapter~\ref{ch:minimal_models}, we prove that \eqref{eq:Groth=Z+Cl_general} holds for any isolated cDV singularity which admits a minimal model with only type $\ctAn{n}$ singularities, and verify Conjecture~\ref{conj:motivation} in this case.
Finally, Chapter~\ref{ch:conclusion} ends with a summary theorem and some speculations on symplectic reflection groups.

\section{Notation and conventions}\label{sec:intro:notation_conventions}
Throughout, we will work over the algebraically closed field $\C$.
Unless specified otherwise, all rings will be noetherian.
For a ring $\Lambda$, we write:
\begin{itemize}
	\item $\modCat \Lambda$ for the category of finitely generated left $\Lambda$-modules;
	\item $\proj \Lambda$ for the full subcategory of $\modCat \Lambda$ consisting of left projective $\Lambda$-modules;
	\item $\add M$ for the full subcategory consisting of all direct summands of all finite direct sums of $M \in \modCat \Lambda$;
	\item $(-)^{*} \colonequals \Hom_{\Lambda}(-,\Lambda) \colon \modCat \Lambda \to \modCat \Lambda^{\op}$;
	\item $\Db(\Lambda)$ for the bounded derived category of $\modCat \Lambda$;
	\item $\perf \Lambda$ for the full subcategory of $\Db(\Lambda)$ of perfect complexes;
	\item $\Dsg(\Lambda) \coloneqq \Db(\Lambda)/\perf \Lambda$ for the singularity category of $\modCat \Lambda$;
	\item $\Krothsg(\Lambda) \coloneqq \Kroth(\Dsg(\Lambda))$ for the Grothendieck group of $\Dsg(\Lambda)$.
\end{itemize}
When $R$ is a commutative ring, we further write:
\begin{itemize}
	\item $Q(R)$ for the field of fractions of $R$, where $R$ is a domain;
	\item $\refl R$ for the full subcategory of $\modCat R$ consisting of reflexive $R$-modules;
	\item $\CM R$ for the full subcategory of $\modCat R$ consisting of Cohen-Macaulay $R$-modules;
	\item $\underline{\Hom}_{R}(M, N)$ for the quotient of $\Hom_{R}(M, N)$ by the set of all morphism factoring through $\add R$;
	\item $\uCM R$ for the stable category of $\CM R$, where the objects are the same as those of $\CM R$ but the morphisms are defined as $\underline{\Hom}_{R}(M,N)$;
	\item $\Cl(R)$ for the divisor class group of $R$;
	\item $\Groth(R) \coloneqq \Kroth(\modCat R)$ for the Grothendieck group of $\modCat R$;
	\item $\oGroth(R)$ for the reduced Grothendieck group.
\end{itemize}
\chapter{Commutative algebra preliminaries}\label{ch:comm_alg_prelims}
	This chapter contains a collection of technical tools from commutative algebra required for the results in this thesis.
	Specifically, it includes a brief introduction to Cohen-Macaulay modules and Gorenstein rings, followed by a more thorough introduction to the study of divisor class groups of normal noetherian integral domains and algebraic K-theory with an aim towards defining Grothendieck groups.
	
\section{CM modules and Gorenstein rings}\label{sec:prelims:CM_modules_and_Gor_rings}
Let $(R, \mfm)$ be a local ring and $0 \neq M \in \modCat R$.
Then the \emph{depth of} $M$ is defined as
\[ \depth_{R} M \coloneqq \min\{ i \geq 0 \mid \Ext_{R}^{i}(R/\mfm, M) \neq 0 \}. \]
Maintaining that $(R, \mfm)$ is a local ring, the module $M$ is called \emph{(maximal) Cohen-Macaulay} ($\CM$) if $\depth_{R} M = \dim R$.
This definition easily generalises to the non-local (i.e. global) case.
If $R$ is a commutative noetherian ring, then $M \in \modCat R$ is $\CM$ if $M_{\mfp}$ is $\CM$ for all prime ideals $\mfp$ of $R$.
Furthermore, $R$ is a \emph{$\CM$ ring} if $R$ is a $\CM$ $R$-module.

For $M \in \modCat R$, denote by $\add M$ the full subcategory of all direct summands of all finite direct sums of $M$.
For $M, N \in \modCat R$ let $\underline{\Hom}_{R}(M, N)$ be the quotient of $\Hom_{R}(M, N)$ by the set of all morphisms factoring through $\add R$.
We write $\uCM R$ for the stable category of $\CM R$, where the objects are the same as those of $\CM R$, but the set of morphisms between two objects $M$ and $N$ is defined to be $\underline{\Hom}_{R}(M, N)$.

A commutative noetherian ring $R$ is called \emph{Gorenstein} if the localisation $R_{\mfp}$, at each prime ideal $\mfp$, is a local ring which has finite injective dimension as an $R_{\mfp}$-module.
It is well known that a Gorenstein ring is a CM ring.
\begin{example}
	Some well-known Gorenstein rings:
	\[ \C[x_{1}, \hdots , x_{n}], \quad \C[u,v,x]/(uv-x^{n}), \quad \C[u,v,x,y]/(uv-xy). \]
\end{example}

\begin{definition}\label{def:equicodimensional}
	A commutative ring $R$ is called \emph{equi-codimensional} if all its maximal ideals have the same height.
\end{definition}
In this thesis, it is always the case that our rings are either finitely generated over $\C$ or are localisations and completions at a maximal ideal of such rings.
As such, the rings we study are always equi-codimensional \cite[13.4]{eisenCA95}.
Furthermore, when they are CM, they always admit a canonical module $\upomega_{R}$.

\section{Normal rings and the divisor class group}\label{sec:prelims:class_grp}
An integral domain is called \emph{normal} if it is integrally closed in its field of fractions \cite{matsu00}.
In our applications later, with some exceptions (such as $T$ in Lemma~\ref{lem:Anonlyht0} and $\eS$ in Setup~\ref{setup:isolated_arbtypeA_cDV_without_NCCR}), the rings will be both normal and noetherian.

\begin{example}
	The rings $\C[x_{1}, \hdots , x_{n}]$ (see \cite[Example~1]{matsu00}) and $\C[x,y,z]/(xy-z^2)$ (see \cite[Exercise~6.5]{hart77})
	are both normal rings.
\end{example}

Normal rings have many nice properties; we briefly mention some of them now.
A prime ideal $\mfp$ of an integral domain $R$ is said to have \emph{height one} if it is non-zero and there does not exist any prime ideal $\mfq$ such that $(0) \subsetneq \mfq \subsetneq \mfp$.
We give a method for determining the height of an ideal later in this section.

Let $K$ be a field and write $K^{\times} = K \backslash \{ 0 \}$.
A mapping on $K$ is called a \emph{discrete valuation} $\upnu \colon K^{\times} \to \Z$ if $\upnu$ is a homomorphism $\upnu(xy) = \upnu(x) + \upnu(y)$ and $\upnu(x + y) \ge \min ( \upnu(x), \upnu(y) )$.
Furthermore, the \emph{valuation ring} of $\upnu$ can be formed; it consists of $0$ and all $x \in K^{\times}$ such that $\upnu(x) \ge 0$.

Consider a $1$-dimensional noetherian local integral domain $(S, \mfm)$ where $\mfm$ denotes the unique maximal ideal of $S$.
Let $Q(S)$ be the field of fractions of $S$.
Then the ring $S$ is called a \emph{discrete valuation ring} (DVR) if there exists a discrete valuation $\upnu$ of $Q(S)$ such that $S$ is the valuation ring of $\upnu$.
It is known (see, e.g. \cite{atiy69}) that $\mfm$ is the set of all $x \in Q(S)$ for which $\upnu(x) > 0$.
The following proposition (see \cite[Proposition~9.2]{atiy69}) gives an equivalent way of saying when a ring such as $(S, \mfm)$ is a DVR.
\begin{prop}\label{prop:dvr}
	Let $(S, \mfm)$ be a noetherian local integral domain of dimension one.
	Then $S$ is a discrete valuation ring if and only if there exists an element $\uppi \in S$ such that every non-zero ideal of $S$ is of the form $(\uppi^{k})$, for some $k \geq 0$.
\end{prop}

Now, consider a normal noetherian integral domain $R$ with field of fractions $Q(R)$.
Write $X_{1}(R)$ for the set of all height one prime ideals of $R$.
In the terminology of \cite{foss73}, such a ring $R$ satisfies the conditions necessary to be a Krull domain \cite[III, Proposition~7.13]{bass68}.
As such, for each prime ideal $\mfp \in X_{1}(R)$ the localisation of $R$ at $\mfp$, denoted $R_{\mfp}$, is a DVR.
The unique maximal ideal of $R_{\mfp}$ is $\mfp R_{\mfp}$.
By Proposition~\ref{prop:dvr}, $\mfp R_{\mfp}$ is generated by a single element, which we will denote $\uppi_{\mfp}$.

Let $\upnu_{\mfp}$ denote the discrete valuation on $Q(R)$ associated to $R_{\mfp}$, so $\upnu_{\mfp}\colon Q(R)^{\times} \to \Z$ where $Q(R)^{\times}$ is the group of invertible elements of $Q(R)$.
The above implies that, for any $x \in Q(R)^{\times}$, $x$ has a unique factorisation $x = y {\uppi_{\mfp}}^{\upnu_{\mfp}(x)}$ for some $\upnu_{\mfp}(x) \in \Z$ and $y \in R_{\mfp}^{\times}$.

\begin{theorem}\label{thm:valuation}
	Let $R$ be a normal noetherian integral domain with field of fractions $Q(R)$ and let $x$ be any non-zero element of $Q(R)$.
	Then the following hold:
	\begin{enumerate}[label=(\arabic*)]
		\item\label{thm:valuation 1} $\upnu_{\mfp}(x) = 0$ for all but finitely many $\mfp \in X_{1}(R)$,
		\item\label{thm:valuation 2}  $x$ is an element of $R$ if and only if $\upnu_{\mfp}(x) \geq 0$ for all $\mfp \in X_{1}(R)$,
		\item\label{thm:valuation 3} $x$ is a unit of $R$ if and only if $\upnu_{\mfp}(x) = 0$ for all $\mfp \in X_{1}(R)$,
		\item\label{thm:valuation 4} $x$ is in the unique maximal ideal $\mfp R_{\mfp}$ of $R_{\mfp}$ if and only if $\upnu_{\mfp}(x) > 0$,
		\item\label{thm:valuation 5} $R = \cap_{\mfp \in X_{1}(R)}R_{\mfp}$.
	\end{enumerate}
\end{theorem}
\begin{proof}
	For \ref{thm:valuation 1} - \ref{thm:valuation 4} see \cite[Theorem~6.2.1]{ford17} and for \ref{thm:valuation 5} see \cite[Theorem~11.5]{matsu00}.
\end{proof}

We define class groups of a ring $R$ using Weil divisors.
The following definitions of Weil divisors, principal Weil divisors, and the divisor class group can be found in many places and many formats in the literature (see, e.g. \cite{ford17}, \cite{bass68}, \cite{foss73}).
\begin{definition}\label{def:weilDiv}
	Let $R$ be a normal noetherian integral domain with field of fractions $Q(R)$ and $X_{1}(R)$ as before.
	The \emph{group of Weil divisors} $\Div(R)$ of $R$ is defined as the free abelian group with basis $X_{1}(R)$.
\end{definition}
In other words, it is the free $\Z$-module
\[
\Div(R) = \bigoplus\limits_{\mfp\, \in X_{1}(R)} \Z \cdot \mfp
\]
on $X_{1}(R)$.
By Theorem~\ref{thm:valuation}, there is a homomorphism of groups 
\[ \div \colon  Q(R)^{\times} \to \Div(R), \]
defined by
\begin{equation}\label{eq:def_of_elements_of_Prin(R)}
	\div(x) = \sum\limits_{\mfp\, \in X_{1}(R)} \upnu_{\mfp}(x) \cdot \mfp,
\end{equation}
where $\upnu_{\mfp}$ is the valuation associated with the DVR $R_{\mfp}$.
In particular, Theorem~\ref{thm:valuation}\ref{thm:valuation 3} implies that the kernel of $\div$ is equal to the group $R^{\times}$.
It follows that there is an exact sequence
\begin{equation}\label{seq:cl_grp_es_1}
	0 \to \Ker (\div) = R^{\times} \to Q(R)^{\times} \xrightarrow{\div } \Div(R) \to \Cok (\div) \to 0.
\end{equation}

We have the following result for any element of $R$.
\begin{cor}\label{cor:appearance_of_div(r)_r_in_R}
	For $r \in R$, we have $\div(r) = \sum_{\mfp \in X_{1}(R), r \in \mfp} \upnu_{\mfp}(r)\mfp$ with all such $\upnu_{\mfp}(r) > 0$.
\end{cor}
\begin{proof}
	This follows immediately from Theorem~\ref{thm:valuation}\ref{thm:valuation 4} and \eqref{eq:def_of_elements_of_Prin(R)}
\end{proof}

\begin{definition}\label{def:clgrp}
	The \emph{class group} of $R$ is defined to be the cokernel of the homomorphism $\div$, and is denoted $\Cl(R)$.
\end{definition}

The exact sequence \eqref{seq:cl_grp_es_1} becomes
\begin{equation}\label{seq:cl_grp_es_2}
	0 \to R^{\times} \to Q(R)^{\times} \xrightarrow{\div } \Div(R) \to \Cl(R) \to 0.
\end{equation}The image of $\div$ is called the \emph{group of principal Weil divisors} and is denoted $\Prin(R)$.
Thus, the divisor class group of $R$ is the group of Weil divisors modulo the principal Weil divisors.

The divisor class group measures the failure of a ring $R$ to be a unique factorization domain (UFD).
Specifically,
\begin{prop}\label{prop:ufd=>cl=0}{\cite[Proposition~6.2]{hart77}}
	Let $R$ be a noetherian integral domain.
	Then $R$ is a UFD if and only if $X = \Spec R$ is normal and $\Cl(X) = 0$.
\end{prop}
The Picard group $\textnormal{Pic}(X)$ can be viewed as line bundles up to isomorphism, which is naturally a subgroup of $\Cl(X)$.
We do not require the language of line bundles in this thesis, for further details on the topic see \cite[p.~66]{GrifHar}.
The following result may also be found in Hartshorne; see \cite[Corollary~6.16]{hart77}.
\begin{cor}
	Let $X$ be a noetherian, integral, separated locally factorial scheme.
	Then there exists a natural isomorphism $\Cl(X) \cong \textnormal{Pic}(X).$
\end{cor}
The remainder of this section contains other commutative algebra results required for the work in this thesis.
The following is elementary, see, e.g. \cite[Chapter~1]{atiy69}.
\begin{lemma}\label{lem:3isothm}
	Let $R$ be a ring and $I$ an ideal of $R$.
	Then the one-to-one correspondence between ideals of $R$ that contain $I$ and ideals of $R/I$ extends to a one-to-one correspondence of prime ideals.
\end{lemma}

The height one prime ideals are crucial in the definition of the divisor class group.
The following theorem (see \cite[Theorem~1.8A]{hart77}) gives a method for determining the height of a prime ideal of a ring $R$.
\begin{theorem}\label{thm:height+dim}
	Let $k$ be a field, and $R$ be an integral domain which is a finitely generated $k$-algebra.
	Then for any prime ideal $\mfp$ in $R$, we have
	\[
	\height(\mfp) + \dim R/\mfp = \dim R.
	\]
\end{theorem}

Now, we extend Lemma~\ref{lem:3isothm} in such a way that allows us to determine which prime ideals are the height one prime ideals of $R$ containing a non-zero element of $R$.
The proof of this theorem is somewhat more involved than one might expect.
\begin{theorem}\label{thm:htbijection}
	Let $R$ be an integral domain, $s \in R$ a non-zero element and $\upvarphi\colon R \to R/ (s)$ the quotient map.
	Then there exists a bijection
	\[
	\{ \textnormal{height one prime ideals in $R$ containing $s$} \} \xrightarrow{\upvarphi_{*}} 
	\{\textnormal{height zero prime ideals in $R/(s)$} \}.
	\]
\end{theorem}
\begin{proof}
	By Lemma~\ref{lem:3isothm}, there is a bijection
	\[
	\{ \textnormal{prime ideals in $R$ containing $s$} \}
	\xrightarrow{\upvarphi_{*}} 
	\{\textnormal{prime ideals in $R/(s)$} \}.
	\]	
	To begin, assume that $\mfp\in\Spec R$ with $\height(\mfp)=1$ and $s \in \mfp$.
	By definition, $\upvarphi_{*}(\mfp) = \mfp/(s)$.
	We want to show that $\height(\mfp/(s)) = 0$.
	Assume $\height(\mfp/(s)) \geq 1$.
	That is, that there exists a prime ideal $\mathfrak{q}/(s) \subsetneq \mfp/(s)$.
	This implies that $(s) \subseteq \mathfrak{q} \subsetneq \mfp$ in $R$.
	Since $s$ is non-zero we have $(0) \subsetneq \mathfrak{q} \subsetneq \mfp$, which tells us that $\height(\mfp) \geq 2$.
	This is a contradiction since we assumed that $\height(\mfp) = 1$, and so $\height(\mfp/(s)) = 0$.
	
	Now, assume that $\mathfrak{q}\in\Spec R/(s)$ with $\height(\mfq)=0$.
	We want to show that the preimage of $\mfq$ has height one.
	Write $\mfp\colonequals \upvarphi^{-1}(\mfq) = \{r \in R \mid \upvarphi(r) \in \mathfrak{q} \}$.
	Since $R$ is an integral domain, $(0) \subset \mfp$ is the minimal prime ideal of $R$. 
	By Lemma~\ref{lem:3isothm}, $\mfp$ contains $s$; as such $\mfp \neq (0)$.
	It follows that $\height(\mfp) \geq 1$.
	By Krull's Hauptidealsatz \cite[Corollary~11.17]{atiy69}, every prime ideal minimal over $(s)$ has height one.
	If $\height(\mfp) > 1$, then there exists a prime ideal $\mfp'$ such that $(s) \subset \mfp' \subsetneq \mfp$.
	But this would imply that $\mfp'/(s) \subsetneq \mfp/(s) = \mathfrak{q}$ and that $\height(\mfq) \geq 1$, a contradiction.
	Thus $\height(\mfp) = 1$.
\end{proof}

For our purposes a \emph{graded ring} $R = \bigoplus_{i \in \N} R_{i}$ is a commutative ring with identity which decomposes as a direct sum of abelian groups such that $R_{i}R_{j} \subseteq R_{i+j}$ for $i, j \in \N$.
\begin{lemma}\label{lem:R^x = k^x}
	If $R = \bigoplus\limits_{m \ge 0} R_{m}$ is a commutative graded domain with $R_{0} = k$ a field, then $R^{\times} = k^{\times}$.
\end{lemma}
\begin{proof}
	Let $r = r_{0} + \hdots + r_{m}$ be invertible in $R$.
	Then there exists $s = s_{0} + \hdots + s_{n}$ in $R$ such that $rs = 1$.
	Assume that $r_{m}, s_{n} \neq 0$.
	Then $rs = \sum_{j \ge 0} (\sum_{i=0}^{j} r_{i} s_{j-i})$ with leading term $(rs)_{m+n} = r_{m} s_{n}$.
	Since $R$ is a domain and $r_{m}, s_{n} \neq 0$, we have $(rs)_{m+n} \neq 0$.
	But $rs = 1$ implies that $(rs)_{j} = 0$ for all $j > 0$, so $n + m =0$.
	That is, $n=m=0$ and hence $r,s \in k$.
\end{proof}	

In the following (see \cite[Proposition~3.4]{bulj12}), let $\langle t \rangle$ be the cyclic subgroup generated by $t$.
\begin{prop}\label{prop:unitsQ[q-1]=Cx<q>} 
	Let $T$ be a UFD and $t$ an irreducible element in $T$.
	Then the group of units in $T[t^{-1}]$ is $T^{\times} \times \langle t \rangle$.
\end{prop}

\section{Algebraic K-theory}\label{sec:prelims:alg_k_theory}
This section includes the definitions necessary for studying algebraic K-theory and Grothendieck groups.
For a thorough exposition on the history of algebraic K-theory, see \cite{Weib99}.
We begin by introducing additive categories, abelian categories, and exact categories following \cite[Chapter~1]{bass68}.

A category $\cC$ is called an \emph{additive category} if the following properties hold:
\begin{enumerate}
	\item $\cC$ has a zero object,
	\item for all $X,Y \in \cC$, the direct product $X \times Y$ exists, and
	\item for all $X,Y \in \cC$, the set of morphisms $X$ to $Y$ has the structure of an abelian group such that the composition is bilinear.
\end{enumerate}

Let $\cA$ an be additive category.
A functor $F \colon \cA \to \cB$ is called \emph{additive} if $F(A) \oplus F(B) \to F(A \oplus B)$ is an isomorphism for all $A, B \in \cA$.
Equivalently, $F$ is additive if $F(A \oplus B) \to F(A) \oplus F(B)$ is an isomorphism for all $A, B \in \cA$.
For more details on additive functors see \cite[\href{https://stacks.math.columbia.edu/tag/010M}{Tag 010M}]{stacks-project}.
An \emph{abelian category} is defined as an additive category $\cA$ in which all kernels and cokernels exist and the natural map Coim$(a) \to \Im(a)$ is an isomorphism for each morphism $a$ in $\cA$.
An \emph{exact category} is an additive category together with a class of distinguished short sequences $X \to Y \to Z$, satisfying various axioms (see, e.g. \cite{keller96}).
These sequences are called \emph{exact sequences}.

A \emph{full exact subcategory} of an exact category $\cA$ is a full additive subcategory $\cB$, which is closed under extensions.
That is, if $0 \to X \to Y \to Z \to 0$ is exact in $\cA$ and $X$ and $Z$ are objects in $\cB$, then $Y$ is also an object of $\cB$ \cite[\S4]{keller96}.
The importance of full exact subcategories can be understood through the following two facts.
First, kernels of exact functors between abelian categories are full exact subcategories.
Second, the quotient category $\cA/\cB$ can be built.
This quotient category has the same objects as $\cA$, it is abelian, and there is a canonical exact functor $p: \cA \to \cA/\cB$ with kernel equal to $\cB$.
In the literature, full exact subcategories are also referred to as \emph{Serre subcategories} and the above functor $p$ goes by multiple names, most common are \emph{Serre quotient} or \emph{quotient functor}.

In this thesis, $\modCat R$ denotes the category of all finitely generated $R$-modules and $\proj R$ denotes the category of all finitely generated projective $R$-modules.
\begin{remark}
	When $R$ is a ring, $\modCat R$ is an exact category with exact sequences being all short exact sequences.
	Similarly, $\proj R$ is an exact category with exact sequences being those short exact sequences that exist with all terms in $\proj R$.
	In $\proj R$ every exact sequence is split.
\end{remark}
In 1973, Quillen formulated higher algebraic K-theory, defining the higher K-theory of an exact category.
Following \cite{quill73}, we define the zeroth K-group of an exact category as follows.
\begin{definition}\label{def:K_i(modR)}
	The \emph{zeroth} K\emph{-group} $\Kroth(\cE)$ of an exact category $\cE$ is defined as the abelian group generated by the objects $\left[ M \right]$ for each isomorphism class of objects of $\cE$ and one relation $[B] = [A] + [C],$ where $A, B, C \in \cE$, for every exact sequence $A \to B \to C.$
\end{definition}
Let $R$ be a noetherian ring.
In this case, $\modCat R$ is an abelian category.
In this thesis we write $\Groth(R) \colonequals \Kroth(\modCat R)$ for the K-group of the category of finitely generated $R$-modules, and $\Kroth(\proj R)$ for the K-group of the category of finitely generated projective $R$-modules.
Following standard terminology, the group $\Groth(R)$ is called the \emph{Grothendieck group of} $R$.
In other words,
\[ \Groth(R) \coloneqq \Kroth(\modCat R) = \frac{ \oplus \Z [M] }{ \langle [M] = [M'] + [M''] \rangle } \]
for $M \in \modCat R$, as the relations vary over all short exact sequences.

\begin{remark}\label{remark:canonical_iso_modR_projR}
	There is a canonical map $\Kroth(\proj R) \to \Groth(R)$, which, by the resolution theorem of \cite[\S4, Corollary~2]{quill73}, is an isomorphism if $R$ has finite global dimension.
	In general this map is not an isomorphism.
\end{remark}
A general property of Grothendieck groups is the following.
\begin{lemma}\label{lem:filtration_lemma}
	Let $R$ be a ring, $M \in \modCat R$, and consider a filtration by submodules
	$0 = M_{0} \subseteq M_{1} \subseteq \hdots \subseteq M_{n} = M.$
	Then, in $\Groth(R)$ \[ [M] = \sum_{i=0}^{n-1} \left[ \frac{M_{i+1}}{M_{i}} \right]. \]
\end{lemma}
\begin{proof}
	To prove this, we induct on $n$.
	That this holds for $n = 1$ is clear.
	Now let $n > 1$ and consider the sequence
	\[ 0 \to M_{1} \to M \to M/M_{1} \to 0, \]
	where $[M] = [M_{1}] + [M/M_{1}]$.
	This leads to a filtration
	\[ 0 = M_{1}/M_{1} \subseteq M_{2}/M_{1} \subseteq \hdots \subseteq M_{n}/M_{1} = M/M_{1} \]
	of length $n-1$.
	So by induction,
	\begin{equation*}
	[M] = [M_{1}] + [M/M_{1}] = \sum_{i=0}^{n-1} \left[ M_{i+1}/M_{i} \right].\qedhere
	\end{equation*}
\end{proof}

\begin{lemma}\label{lem:every_element_written_in_form_[M]-n[R]}
	Let $R$ be any noetherian ring.
	Then every element of $\Groth(R)$ can be written in the form $[M] - n[R]$ for some $n \ge 0$.
\end{lemma}
\begin{proof}
	Let $L, N \in \modCat R$.
	Note that $n[L] = [L^{\oplus n}]$ where $n \ge 0$, and $n_{1}[L] + n_{2}[N] = [L^{\oplus n_{1}} \oplus N^{\oplus n_{2}}]$ where $n_{1}, n_{2} \ge 0$.
	
	Now, consider a general element $\sum m_{i}[L_{i}] - \sum n_{i} [N_{i}]$ of $\Groth(R)$, where $m_{i}, n_{i} \ge 0$.
	This can be rewritten as $[\oplus L_{i}^{\oplus m_{i}}] - [\oplus N_{i}^{\oplus n_{i}}]$.
	For ease of notation denote $[\oplus L_{i}^{\oplus m_{i}}]$ by $[X]$ and $[\oplus N_{i}^{\oplus n_{i}}]$ by $[Y]$.
	Hence $[X] - [Y]$ is some general element of $\Groth(R)$.
	
	Now, taking the syzygy, there is an exact sequence
	\[ 0 \to \Omega Y \to R^{n} \to Y \to 0, \]
	so that $[\Omega Y] - n[R] = - [Y]$.
	Thus, 
	\begin{align*}
	[X] - [Y] &= [X] + [\Omega Y] - n[R] \\
	&= [X \oplus \Omega Y] - n[R].
	\end{align*}
	Writing $[M] \coloneqq [X \oplus \Omega Y]$ gives the desired form $[M] - n[R]$.
\end{proof}

Introduced below, the localisation and d\'{e}vissage theorems of \cite{quill73} are important tools that only work for abelian categories.
The following result, due to Quillen in \cite[\S5, Theorem~5]{quill73}, gives a long exact sequence of K-groups.
\begin{theorem}[Localisation]\label{thm:quillen_localisation}
	Let $\cB$ be a full exact subcategory of the abelian category $\cA$, $\cA/\cB$ be the quotient category of $\cA$ by $\cB$ and let $i\colon \cB \to \cA$ and $p\colon \cA \to \cA/\cB$ denote the canonical functors.
	Then there is a long exact sequence
	\[ \hdots \xrightarrow{p} \emph{K}_{1}(\cA/\cB) \to \Kroth(\cB) \xrightarrow{i} \Kroth(\cA) \xrightarrow{p} \Kroth(\cA/\cB) \to 0. \]
\end{theorem}
Quillen's d\'{e}vissage theorem \cite[\S5, Theorem~4]{quill73} gives conditions under which two abelian categories have the same K-theory.
\begin{theorem}[D\'{e}vissage]\label{thm:quillen_devissage}
	Let $\cA$ be an abelian category and $\cB$ a non-empty full subcategory which is closed under subobjects, quotients, and finite products in $\cA$.
	Suppose that every object $M$ in $\cA$ has a finite filtration
	\[ 0 = M_{0} \subset M_{1} \subset \hdots \subset M_{n-1} \subset M_{n} = M \]
	with $M_{i}/M_{i-1}$ in $\cB$ for each $i$.
	Then the inclusion functor induces a homotopy equivalence $\emph{K}_{i}(\cB) \cong \emph{K}_{i}(\cA)$ for all $i$.
\end{theorem}

\section{Techniques: Nagata's Theorem}\label{sec:prelims:nagatas_thm}
We introduce a method for computing divisor class groups.
Let $W$ be a multiplicatively closed subset of a normal noetherian integral domain $R$.
Traditionally, Nagata's Theorem, Theorem~\ref{thm:nagata} below, states there is a surjective group homomorphism $\Cl(R) \to \Cl(W^{-1}R)$ with kernel generated by the classes of prime divisors in $X_{1}(R) \backslash X_{1}(W^{-1}R)$.
For more details on this, see \cite[Theorems~7.1 and~7.2]{foss73}.

In this section, we provide a more precise version of Nagata's theorem, following \cite[Theorem~6.2.4]{ford17}.
This version details a method for computing the divisor class group of a normal noetherian integral domain.
Our results in Chapter~\ref{ch:divisor_cl_grps} and Chapter~\ref{ch:cDVs} utilise the techniques of this section.
\begin{theorem}[Nagata's Theorem]\label{thm:nagata}
	Let $R$ be a normal noetherian integral domain with field of fractions $Q(R)$.
	Let $r$ be a non-zero non-invertible element of $R$ with $\div(r) = \sum_{\mfp_{i} \in X_{1}(R)} \upnu_{\mfp_{i}}(r)\,\mfp_{i}$, where $\mfp_{i}$ is a height one prime ideal containing $r$ for $i = 1, \hdots, n$.
	Then the sequence of abelian groups
	\begin{equation}\label{seq:nagata}
	0 \to R^{\times} \longrightarrow R[r^{-1}]^{\times} \xrightarrow{\div} \bigoplus\limits_{i=1}^{n} \Z \cdot \mfp_{i} \longrightarrow \Cl(R) \longrightarrow \Cl(R[r^{-1}]) \to 0
	\end{equation}
	is exact.
\end{theorem}
We include a variation of the proof given in \cite[Theorem~6.2.4]{ford17}, as the method illustrated will be used for results in Chapter~\ref{ch:divisor_cl_grps}.
\begin{proof}
	Clearly $R^{\times} \subseteq R[r^{-1}]^{\times} \subseteq Q(R)^{\times}$, and there is a commutative diagram
	\begin{equation}\label{cd:R*-PrinRf}
	\begin{tikzcd}[arrow style=tikz,>=stealth]
	&0 \arrow{r}
	& R^{\times} \arrow{r}
	\arrow[d, hookrightarrow, "\updelta"]
	& Q(R)^{\times} \arrow{r}{\div}
	\ar[equal]{d}
	& \Prin(R) \arrow{r}
	\arrow{d}{\upalpha}
	& 0\\
	&0 \arrow{r}
	& R[r^{-1}]^{\times} \arrow{r}
	& Q(R[r^{-1}])^{\times} \arrow{r}{\div}
	& \Prin(R[r^{-1}]) \arrow{r}
	& 0
	\end{tikzcd}
	\end{equation}
	with exact rows.
	Notice that this implies that $\upalpha$ is surjective.
	This extends in the obvious way to the following diagram:
	\begin{equation}\label{cd:extendedR*-PrinRf}
	\begin{tikzcd}[arrow style=tikz,>=stealth]
	&
	& 0 \arrow{d}
	& 0 \arrow{d}
	& \Ker(\upalpha) \arrow{d}\\
	&0 \arrow{r}
	& R^{\times} \arrow{r}
	\arrow{d}{\updelta}
	& Q(R)^{\times} \arrow{r}{\div }
	\ar[equal]{d}
	& \Prin(R) \arrow{r}
	\arrow{d}{\upalpha}
	& 0\\
	&0 \arrow{r}
	& R[r^{-1}]^{\times} \arrow{r}
	\arrow{d}
	& Q(R)^{\times} \arrow{r}{\div }
		\arrow{d}
	& \Prin(R[r^{-1}]) \arrow{r}
		\arrow{d}
	& 0\\
	&
	& \Cok(\updelta)
	& 0
	& 0
	&
	\end{tikzcd}
	\end{equation}
	By the Snake Lemma, $\Ker(\upalpha) \cong \Cok (\updelta)$, thus rewriting the first column of (\ref{cd:R*-PrinRf}) gives
	\begin{equation}\label{seq:R*-kera}
	0 \longrightarrow R^{\times} \longrightarrow R[r^{-1}]^{\times} \longrightarrow \Ker (\upalpha) \longrightarrow 0.
	\end{equation}
	Notice that $X_{1}(R[r^{-1}])$ is the subset of $X_{1}(R)$ containing the prime ideals of height one in $R$ that do not contain $r$ \cite[Exercise~2.2.15]{ford17}.
	Therefore $\Div(R[r^{-1}])$ can be viewed as the free $\Z$-submodule of $\Div(R)$ generated by prime ideals $\mfp$ in $X_{1}(R[r^{-1}])$.
	Take $\upbeta$ to be the projection onto $\Div(R[r^{-1}])$ defined by 
	\begin{equation*}
	\mfp \mapsto \left\{ \begin{array}{ll}
	0, & \text{if } r \in \mfp \\
	\mfp, & \textrm{otherwise.}
	\end{array} \right.
	\end{equation*}
	Then the diagram
	\begin{equation}\label{cd:PrinR-ClRf}
	\begin{tikzcd}[arrow style=tikz,>=stealth]
	&0 \arrow{r}
	& \Prin(R) \arrow{r}
	\arrow{d}{\upalpha}
	& \Div(R) \arrow{r}
	\arrow{d}{\upbeta}
	& \Cl(R) \arrow{r}
	\arrow{d}{\exists \upgamma}
	& 0\\
	&0 \arrow{r}
	& \Prin(R[r^{-1}]) \arrow{r}
	& \Div(R[r^{-1}]) \arrow{r}
	& \Cl(R[r^{-1}]) \arrow{r}
	& 0
	\end{tikzcd}	 
	\end{equation}
	commutes, and the rows are exact.
	We already know that both $\upalpha$ and $\upbeta$ are surjective.
	Therefore $\upgamma$ is surjective. 
	This extends in the obvious way to the following diagram:
	\begin{equation}\label{cd:extendedR*-PrinRf2}
	\begin{tikzcd}[arrow style=tikz,>=stealth]
	&0 \arrow{r}
	& \Ker(\upalpha) \arrow{r}
	\arrow{d}
	& \Ker(\upbeta) \arrow{r}
	\arrow{d}
	& \Ker(\upgamma) \arrow{d}\\
	&0 \arrow{r}
	& \Prin(R) \arrow{r}
	\arrow[d, tail, twoheadrightarrow, "\upalpha"]
	& \Div(R) \arrow{r}
	\arrow[d, tail, twoheadrightarrow, "\upbeta"]
	& \Cl(R) \arrow{r}
	\arrow[d, tail, twoheadrightarrow, "\upgamma"]
	& 0\\
	&0 \arrow{r}
	& \Prin(R[r^{-1}]) \arrow{r}
	& \Div(R[r^{-1}]) \arrow{r}{}
	& \Cl(R[r^{-1}]) \arrow{r}
	& 0
	\end{tikzcd}
	\end{equation}
	By the Snake Lemma, we obtain an exact sequence
	\begin{equation}\label{seq:kera-kerg}
	0 \longrightarrow \Ker (\upalpha) \longrightarrow \Ker (\upbeta) \longrightarrow \Ker (\upgamma) \longrightarrow 0.
	\end{equation}
	By definition, the group $\Div(R)$ is free on $X_{1}(R)$.
	Then, since the only height one primes that contain $r$ are $\mfp_{1}, \hdots , \mfp_{n}$, the kernel of $\upbeta$ is the free subgroup $\bigoplus_{i=1}^{n} \Z \cdot \mfp_{i}$.
	From \eqref{cd:PrinR-ClRf}, we also get the following exact sequence
	\begin{equation}\label{seq:kerg-ClRf}
	0 \longrightarrow \Ker (\upgamma) \longrightarrow \Cl(R) \longrightarrow \Cl(R[r^{-1}]) \longrightarrow 0.
	\end{equation}
	Splicing the sequences \eqref{seq:R*-kera}, \eqref{seq:kera-kerg}, and \eqref{seq:kerg-ClRf} gives the sequence \eqref{seq:nagata}, as desired.
\end{proof}
This result provides us with what should be a straightforward method for computing $\Cl(R)$, where $R$ is a normal noetherian integral domain.

\section{Techniques: Algebraic K-theory}
\label{sec:prelims:bass_eagon}
Computing class groups is a difficult task and, in some instances, K-theory is easier to compute.
In this section, we introduce a technique, originally described in \cite{bass68}, for doing so.
We will see several examples in Chapter~\ref{ch:divisor_cl_grps} and Chapter~\ref{ch:cDVs} where this approach works well.

As before, let $R$ be a normal noetherian integral domain with field of fractions $Q(R)$.
Each finitely generated $R$-module $M$ has a well-defined class $[M]$ in $\Groth(R)$.
The map $M \mapsto [M]$ is an additive function (see \S\ref{sec:prelims:alg_k_theory}) such that any other additive function on $\modCat R$ factors through it.
Define the \emph{rank map} $\rk_{R} \colon \modCat R \to \Z$ as
\[ \rk_{R}(M) = \dim_{Q(R)}(M \otimes_{R} Q(R)). \]
Since this is an additive function, it induces a homomorphism $\Groth(R) \to \Groth(Q(R))$ where $\Groth(Q(R)) \cong \Z.$
Under this definition $\rk(R) = 1$, thus $\Groth(R)$ decomposes as
\begin{equation}\label{eq:Groth=Z+oGroth}
	\Groth(R) = \Z \cdot [R] \oplus \oGroth(R),
\end{equation}
where $\oGroth(R)$ is the kernel of the homomorphism induced by the rank map and is called the \emph{reduced Grothendieck group} of $R$.

Let $\eC$ be the full subcategory of all $M$ such that $M \otimes_{R} Q(R) = 0$.
Since localisation is exact, $\eC$ is an abelian category.
\begin{lemma}\label{lem:M_in_C_M_p_finite_length}
	If $M \in \eC$ and $\mfp \in X_{1}(R)$, then $M_{\mfp}$ has finite length.
\end{lemma}
\begin{proof}
	If $\mfp \in X_{1}(R)$ then $R_{\mfp}$ is a 1-dimensional integral domain with unique non-maximal prime ideal $(0)$ and unique maximal ideal $\mfp R_{\mfp}$.
	Since $(M_{\mfp})_{(0)} = M_{(0)} = 0$, $M_{\mfp}$ is only supported at the maximal ideal.
	It follows that $M_{\mfp}$ has finite length \cite[Chapter~10.10]{Cohn03}.
\end{proof}

In the situation of the lemma, we write $l_{\mfp}(M_{\mfp})$ for the length of the $R_{\mfp}$-module $M_{\mfp}$.

\begin{lemma}\label{lem:l(R/(a))=n_p(a)}
	Let $R$ be a normal domain, $\mfp \in X_{1}(R),$ and $a \in R$ nonzero.
	Then we have the equality $\ell_{\mfp}(R_{\mfp}/(a)) = \upnu_{\mfp}(a)$.
\end{lemma}

\begin{proof}
	Since localisation is exact, $R/(a) \in \eC$. 
	The ring $R_{\mfp}$ is a DVR with maximal ideal $\mfp R_{\mfp} = (\uppi_{\mfp})$.
	Thus $a = u \uppi_{\mfp}^{\upnu}$ for some unit $u \in R$, where by definition $\upnu = \upnu_{\mfp}(a)$.
	Hence
	\[
	\ell_{\mfp}(R_{\mfp}/(a)) = \ell_{\mfp}(R_{\mfp}/(\uppi_{\mfp}^{\upnu})) = \upnu_{\mfp}(a).
	\]
\end{proof}

Recall the group $\Div(R)$ from Definition~\ref{def:weilDiv} and the exact sequence \eqref{seq:cl_grp_es_2}.
It follows from Lemma~\ref{lem:M_in_C_M_p_finite_length} that there exists a map $\upchi \colon \eC \to \Div(R)$ given by
\[ M \mapsto \sum_{\mfp \in X_{1}(R)} \ell_{\mfp} (M_{\mfp}) [\mfp]. \]
The map $\upchi$ is an additive function and we denote by the same letter the induced map $\Kroth(\eC) \to \Div(R)$. 
From the inclusion $\eC \subseteq \mod R$ there is an induced map on Grothendieck groups $\Kroth(\eC) \to \Groth(R)$ whose image is $\oGroth(R)$.

We use the following important result (see \cite[Chapter~IX, Proposition~6.6]{bass68}) which shows that there is an induced surjective map $c \colon \Groth(R) \to\Cl(R)$.
\begin{prop}\label{prop:bass_Groth(R)_to_Cl(R)}
	Let $R$ be a normal noetherian integral domain with field of fractions $Q(R)$.
	Then there exists a unique homomorphism $c \colon \Groth(R) \to \Cl(R)$ such that the diagram
	\begin{equation}\label{eq:bass_Groth(R)_to_Cl(R)}
	\begin{tikzcd}[arrow style=tikz,>=stealth]
	&G_{1}(Q(R)) \arrow{r}
	\arrow{d}{\textnormal{det} (\cong)}
	& \Kroth(\eC) \arrow{r}
	\arrow{d}{\upchi}
	& \Groth(R) \arrow{r}{\rk}
	\arrow{d}{\textnormal{c}}
	& \Groth(Q(R)) \arrow{r}
	\arrow{d}
	& 0\\
	& Q(R)^{\times} \arrow{r}{\div}
	& \Div(R) \arrow{r}
	& \Cl(R) \arrow{r}
	& 0
	\end{tikzcd}	 
	\end{equation}
	commutes.
	Moreover, $\upchi$ and $c$ are surjective homomorphisms.
\end{prop}
\begin{proof}
	For full details of the proof, see \cite[Chapter~IX, Proposition~6.6]{bass68}.
	To see that $c$ is surjective let $\mfp \in X_{1}(R)$ and consider the $R$-module $R/\mfp$.
	Then 
	\[ c\left( \left[ \frac{R}{\mfp} \right] \right) = \sum_{\mfq \in X_{1}(R)} \ell_{\mfq}\left( \left( \frac{R}{\mfp} \right)_{\mfq} \right) [\mfq] = [\mfp]. \qedhere \]
\end{proof}

The map $\Kroth(\eC) \to \Div(R) \to \Cl(R)$ factors through the map $\Kroth(\eC) \to \oGroth(R)$ and so induces a map $\upgamma \colon \oGroth(R) \to \Cl(R)$.
We wish to extend this result to give a more precise description of the relationship between $\Groth(R)$ and $\Cl(R)$.
In the following, we first determine the generators of $\Groth(R)$.
\begin{lemma}\label{lem:Groth_generated_by_R/p}
	Let $R$ be a noetherian ring, then $\Groth(R)$ is generated by modules of the form $[R/\mfp]$ where $\mfp \in \Spec R$.
\end{lemma}
\begin{proof}
	This follows directly from \cite[Proposition~3.7]{eisenCA95}.
\end{proof}

The next result is described in \cite{Ea68} and included here for completeness.
\begin{lemma}\label{lem:R/(p+xR)=0}
	Let $R$ be a noetherian integral domain, $\mfp \in \Spec R$, and $x \in R \backslash \mfp$.
	Then $[ R/(\mfp + xR)] = 0$ in $\Groth(R)$.
\end{lemma}
\begin{proof}
	Since $\mfp$ is prime, there is a short exact sequence
	\[ 0 \to R/\mfp \xrightarrow{x} R/\mfp \to R/(\mfp + xR) \to 0. \]
	This implies that $[R/\mfp] = [R/\mfp] + [R/(\mfp + xR)]$.
	Hence $[R/(\mfp + xR)] = 0$ in $\Groth(R)$.
\end{proof}

Since localisation is exact, it follows that $\oGroth(R)$ is generated by all $[R/\mfp]$ where $\mfp \neq (0)$ is a prime ideal of $R$.
That is, $\oGroth(R) = \langle [R/\mfp] \mid \height(\mfp) \ge 1 \rangle$.
Now, denote by $H$ the subgroup of $\oGroth(R)$ generated by all $[R/\mfp]$ where $\height(\mfp) \ge 2$.
In the following proposition we look at the restriction $\upgamma \colon \oGroth(R) \to \Cl(R)$ of the homomorphism $c \colon \Groth(R) \to \Cl(R)$ from Proposition~\ref{prop:bass_Groth(R)_to_Cl(R)}.

\begin{prop}\label{prop:huneke_oGroth(R)_to_Cl(R)}
	Let $R$ be a normal noetherian integral domain.
	Then the map $\upgamma$ is surjective and induces an isomorphism
	\[
		\oGroth(R)/H \cong \Cl(R).
	\]
\end{prop}
\begin{proof}
	Let $X_{1}(R)$ be the set of height one prime ideals of $R$ and $\mfp \in X_{1}(R)$. 
	By Proposition~\ref{prop:bass_Groth(R)_to_Cl(R)}, $c$ is surjective.
	Notice that, for all $\mfp \in X_{1}(R)$, we have $[R/\mfp] \in \oGroth(R)$.
	Hence $\upgamma$ is also surjective.
	
	First, we show that $H$ is in the kernel of $\upgamma$.
	Let $\mathfrak{a} \in \Spec R$ such that $\height(\mathfrak{a}) \ge 2$.
	It follows that
	\[ \upgamma \left( \left[ \frac{R}{\mathfrak{a}} \right] \right) = \sum_{\mathfrak{b} \in X_{1}(R)} \ell_{\mathfrak{b}}\left( \left( \frac{R}{\mathfrak{a}} \right)_{\mathfrak{b}} \right) [\mathfrak{b}] = 0, \]
	where the second equality comes from the proof of Proposition~\ref{prop:bass_Groth(R)_to_Cl(R)}.
	Thus, $H \subseteq \Ker \upgamma$.
	
	Hence $\upgamma$ induces a surjection $\oGroth(R)/H \twoheadrightarrow \Cl(R)$.
	To prove the statement, we next construct a left inverse of this induced map.
	Consider $\upbeta \colon \Div(R) \to \oGroth(R)/H$ defined by $\mfp \mapsto [R/\mfp] + H$.
	We first claim that this factors through $\Cl(R)$.
	This requires us to show that $\Prin(R) \subseteq \Ker \upbeta$.
	
	Let $a \in R$ be non-zero.
	Then $\div(a) = \sum_{\mfp \in X_{1}(R)} \upnu_{\mfp}(a) [\mfp]$.
	Since $R$ is a domain the short exact sequence
	\[
	0 \to R \stackrel{a \cdot}{\longrightarrow} R \to R/(a) \to 0
	\]
	implies that $[R/(a)] = 0$, and also that $[R/(a)] \in \oGroth(R)$.
	Hence in the quotient $\oGroth(R)/H$, 
	\begin{align*}
		0 = [R/(a)] + H &= \sum_{\mfp \in X_{1}(R)} \ell_{\mfp}(R/(a))[R/\mfp] + H \tag{by Lemma~\ref{lem:filtration_lemma}}\\
		&= \sum_{\mfp \in X_{1}(R)} \upnu_{\mfp}(a)[R/\mfp] + H \tag{by Lemma~\ref{lem:l(R/(a))=n_p(a)}}\\
		&= \upbeta(\div(a)).
	\end{align*}
	Thus $\Prin(R) \subseteq \Ker(\upbeta)$ giving an induced map
	 $\tilde{\upbeta} \colon \Cl(R) \to \oGroth(R)/H.$
	Moreover, by inspection $\tilde{\upbeta} \circ \upgamma = \text{ id}_{\oGroth(R)/H}$, so the surjection $\upgamma \colon \oGroth(R)/H \twoheadrightarrow \Cl(R)$ is an isomorphism.
\end{proof}

\chapter{Noncommutative minimal models}\label{ch:other_prelims}
	This chapter provides a brief introduction to resolutions of singularities, Auslander-Reiten (AR) theory, minimal models, and noncommutative crepant resolutions (NCCRs).
	In addition, we introduce a large class of examples of singularities coming from \cite{cbh98}.
	The discussion of minimal models covers material required in Chapter~\ref{ch:minimal_models}, while the discussion of NCCRs and the brief introduction to the Homological Minimal Model Programme (MMP) involves background information for \S\ref{sec:isolated_cdv_with_NCCR}.
	We do not provide many of the technical details of the Homological MMP, but we provide suitable references for those wishing to know more about the technicalities.
	We also note that, throughout this thesis, our approach to minimal models and resolutions of singularities will be noncommutative.
	
\section{Resolving a singularity}\label{sec:resolving_sings}
	Resolving the singularities of a variety $Y$ consists of finding a smooth (i.e., non-singular) variety $X$ such that there is a proper, birational morphism $f: X \to Y$.
	Such a map $f$ is called a \emph{resolution of singularities} of $Y$.
	A major result, due to Hironaka in \cite{Hi64}, states that every singular variety over an algebraically closed field of characteristic $0$ admits a resolution.
	
	With the existence of such resolutions, one might ask if there is a `best' resolution such that $X$ is as `close' to the original space $Y$ as possible.
	This question leads to the notion of a minimal resolution.
	A \emph{minimal resolution} is a resolution $f \colon X \to Y$ such that any other resolution of $Y$ factors through $f$.
	In dimension $2$, which covers the case of Kleinian singularities and their deformations, a minimal resolution always exists and is unique \cite[Theorem~5.9]{Lau71}.
	Unfortunately, as we move to higher dimensions, this is not necessarily the case; minimal resolutions might not exist.
	This led Reid to introduce the notion of `crepancy' in \cite{Re83}.
	If $f \colon X \to Y$ is a resolution of $Y$, then $f$ is called \emph{crepant} if the pullback along $f$ of the canonical divisor of $Y$ is the canonical divisor of $X$.
	
	For resolutions of Kleinian singularities, the notions of crepancy and minimality coincide, so Kleinian singularities admit a unique crepant resolution.
	But, in higher dimensions, crepant resolutions do not always exist, and when they do, usually they are not unique.
	
	\begin{remark}\label{remark:crepant_res_quotient_sings}
	Recall the quotient singularities $\C^{n}/G$ of \S\ref{sec:intro:ksings} where $G$ is a finite subgroup of SL$(n, \C)$.
		\begin{enumerate}
			\item[(a)] When $n=2$, the singularities are precisely the Kleinian singularities.
			As discussed above, a crepant resolution always exists and is unique.
			\item[(b)] When $n=3$, the singularities always admit a crepant resolution but it is usually no longer unique \cite{BKR}.
			\item[(c)] When $n \geq 4$, such quotient singularities do not necessarily have crepant resolutions, e.g., the group $\frac{1}{2}(1,1,1,1)$; see \cite[Example~2.28]{craw01}.
		\end{enumerate}
	\end{remark}
	The fact that crepant resolutions of cDV singularities may not be unique and, in fact, may not exist at all led to the notion of a minimal model, introduced in the next section.

\section{Minimal models}\label{sec:minimal_models_prelims}
The basic idea of a \emph{minimal model} $f \colon X \to \Spec R$, where $\Spec R$ is an isolated complete local cDV singularity (see \S\ref{sec:intro:cDVs}), is that crepancy is more important than $X$ being smooth.
Since $X$ is not necessarily smooth, $f$ may not be a resolution.
As such, one is asking for a crepant morphism $f$ where the singularities of $X$ are not `too bad'.
Formally, minimal models are defined as follows.
\begin{definition}
	Let $f \colon X \to \Spec R$ be a crepant projective, birational map, where $R$ is complete local and $X$ has at worst $\Q$-factorial terminal singularities.
	Then $f$ is called a \emph{minimal model} of $\Spec R$.
\end{definition}
It is well known that, for a threefold $\Spec R$ which is a cDV singularity, minimal models exist, and there are only finitely many of them \cite{kawamataMatsuki87}.
Indeed, Gorenstein terminal singularities are precisely the isolated cDV singularities \cite{Re83}.
For such hypersurfaces, factorial is equivalent to $\Q$-factorial (see, e.g. \cite[Theorem~2.11]{IyWe14QFact}).
In the next section, we study modules that have an intimate relationship with minimal models and allow us to completely understand minimal models of such threefold singularities.

\subsection{Maximal modifying modules}\label{subsec:maximal_modifying_modules}
When $R$ is a commutative ring, we write $(-)^{*}$ for the functor $\Hom_{R}(-,R) \colon \modCat R \to \modCat R$ and call an $R$-module $M$ \emph{reflexive} if the natural map $M \to M^{**}$ is an isomorphism.
Denote by $\refl R$ the full subcategory of $\mod R$ consisting of reflexive $R$-modules.
Assume now that $R$ is complete local.
We call an $R$-module $M$ \emph{basic} if the indecomposable direct summands $M_{i}$ of $M$ are mutually non-isomorphic.
The following definition of (maximal) modifying modules was first given in \cite[Definition~4.1]{IyWe14ARdual}.
\begin{definition}
	Let $R$ be an isolated complete local cDV singularity, then
	\begin{enumerate}[label=(\arabic*)]
		\item a module $M \in \refl R$ is called a \emph{modifying module} if $\End_{R}(M) \in \CM R$;
		\item a module $M \in \refl R$ is called a \emph{maximal modifying (MM) module} if $M$ is modifying and, furthermore, if $M \oplus Y$ is modifying for $Y \in \refl R$, then $Y \in \add M$;
		\item if $M$ is an MM $R$-module, then $\End_{R}(M)$ is a \emph{maximal modification algebra (MMA)}.
	\end{enumerate}
\end{definition}
In other words, $M$ is an MM module if 
\[
\add M = \{ Y \in \refl R \mid \End_{R}(M \oplus Y) \in \CM R \}.
\]

Consider an MMA of the form $\Lambda \coloneqq \End_{R}(M)$, where $M = M_{0} \oplus M_{1} \oplus \cdots \oplus M_{t}$ is a basic MM $R$-module, with $M_{0} \cong R$ and $M_{1}, \hdots, M_{t}$ the non-free indecomposable summands of $M$.
All MMAs are derived equivalent and any algebra derived equivalent to an MMA is also an MMA \cite[Corollary~1.17]{IyWe14ARdual}.

To each minimal model of an isolated complete local cDV singularity, Donovan-Wemyss \cite{DW19twistsandbraids} associate a corresponding finite-dimensional algebra, called a contraction algebra.
Although contraction algebras have deformation-theoretic origins, their role in this thesis comes from the study of MMAs.
With this approach, we define contraction algebras following \cite[Definition~2.11]{DW19twistsandbraids}.
\begin{definition}\label{def:contraction_algebras}
	Given a minimal model $f \colon X \to \Spec R$ of an isolated complete local cDV singularity, let $\Lambda \coloneqq \End_{R}(M)$ be the corresponding MMA\footnote{In the context of this thesis, the correspondence of an endomorphism ring such as $\Lambda$ to an MMA is best understood through the one-to-one correspondence discussed later in \S\ref{subsec:homMMP}.
	For now, just know that this correspondence exists.}.
	Then the \emph{contraction algebra of $f$} is the stable endomorphism algebra
	\begin{equation*}
	\CL \cong \underline{\End}_{R}(M) \cong \underline{\End}_{R} \left( \bigoplus_{j=1}^{t} M_{j} \right) \cong \Lambda/\Lambda e \Lambda,
	\end{equation*}
	where $e$ is the idempotent in $\Lambda$ corresponding to the summand $R$ of $M$.
\end{definition}

It is useful to keep in mind that just as $\Lambda$ can be presented as a quiver with relations (see \S\ref{sec:deformed_preproj_deformed_ksings}) so too can $\CL$.
The quiver corresponding to $\CL$ has the same vertices and relations as those in the quiver of $\Lambda$, except the vertex (and hence the relations) corresponding to $R$ are not included.
As an example, consider the case of $t=3$; then we have the following corresponding quivers.
\begin{figure}[H]
	\centering
	\begin{subfigure}[b]{.3\textwidth}
		\centering
		\begin{tikzcd}
		{1} \arrow[r] \arrow[d, bend right] & {2} \arrow[d] \arrow[l, bend right] \\
		{0} \arrow[u] \arrow[r, bend right] & {3} \arrow[l] \arrow[u, bend right]                                                
		\end{tikzcd}
		\caption{Quiver of $\Lambda$}
		\label{fig:preproj_quiver_An}
	\end{subfigure}
	\begin{subfigure}[b]{.3\textwidth}
		\centering
		\begin{tikzcd}
		{1} \arrow[r] \arrow[d, bend right, color=black!20!white] & {2} \arrow[d] \arrow[l, bend right] \\
		{\textcolor{gray}{0}} \arrow[u, color=black!30!white] \arrow[r, bend right, color=black!20!white] & {3} \arrow[l, color=black!20!white] \arrow[u, bend right]                                                
		\end{tikzcd}
		\caption{Quiver of $\CL$}
		\label{fig:preproj_quiver_An}
	\end{subfigure}
\end{figure}
Note that, for cDV singularities, the contraction algebra corresponding to a minimal model is finite-dimensional and its quiver is known to be symmetric \cite[Proposition~3.2.10]{august19thesis}.
These quivers can arise from Auslander-Reiten theory, introduced next.

\section{Auslander-Reiten theory}\label{subsec:AR_theory}
The AR-theory introduced in this section will be useful later as it leads to a systematic method for computing the quiver of an MMA and allows us to determine Grothendieck groups of complete local rings of finite $\CM$-type.

Let $C$ be an indecomposable CM $R$-module.
Then, an \emph{Auslander-Reitein (AR) sequence ending in $C$} is a short exact sequence
\[ 0 \to A \to B \xrightarrow{f} C \to 0 \]
in $\CM R$ such that every map $D \to C$ in $\CM R$ which is not a split epimorphism factors through $f$.

As a consequence of AR-duality, it is known (see, e.g. \cite[Theorem~3.2]{yosh90}) that AR-sequences exist in $\CM R$ when $R$ is complete local and isolated singularity.
Given this existence, it is possible to attach a quiver to $\CM R$ called the AR-quiver.
It is a graph constructed from the isomorphism classes of the indecomposable $\CM R$-modules and certain maps between them.
To define the AR-quiver, we follow \cite[Definition~5.2]{yosh90}.
\begin{definition}\label{def:AR_quiver}
	The \emph{Auslander-Reiten (AR) quiver} $\Gamma$ of $\CM R$ has vertices that are in one-to-one correspondence with the isomorphism classes of indecomposable $\CM R$-modules $M$. 
	To determine the arrows, for each $M$ consider the AR-sequence which ends at $M$, say 
	\begin{equation}\label{seq:ARsequence}
	0 \to \uptau(M) \to N \xrightarrow{f} M \to 0.
	\end{equation}
	When $R$ is complete local, we can uniquely decompose the module $N$ into indecomposable modules as $N \cong \bigoplus_{i=1}^{t} M_{i}^{\oplus m_{i}}$.
	Then, for each $i$, there are $m_{i}$ arrows from $M_{i}$ to $M$ in the AR-quiver of $\CM R$.
\end{definition}
\begin{remark}
	Consider the ring $S \coloneqq \C[U, V]$ and a finite subgroup $G$ of SL$(2, \C)$.
	It is known (see, e.g. \cite{Aus86}) that the AR-quiver of $S^{G}$ is isomorphic to the McKay quiver of $G$ from \S\ref{sec:intro:mckay_correspondence}.
	In higher dimensions, $\CM R$ usually has infinitely many indecomposable modules, whereas the McKay quiver only has finitely many vertices.
\end{remark}

Now, suppose that $R$ has finite CM type.
Then for each indecomposable module $M_{1},...,M_{t} \in \CM R$, consider the AR-sequence $0 \to \uptau(M_{j}) \to L_{j} \to M_{j} \to 0$.
Expressing $[M_{j}] + [\uptau(M_{j})] - [L_{j}]$ in terms of $[M_{0}], \hdots, [M_{t}]$ gives a tuple $(a_{0j}, \hdots, a_{tj})$ of integers.
These tuples, as $j$ varies between $1$ and $t$, form the \emph{AR-matrix} $\Upupsilon \colon \Z^{t} \to \Z^{t+1}$ (for more details, see \cite[Definition~2.3]{Holm15}).
The next theorem was originally proven in \cite{AR86}, but the form included here can be found in \cite[\S1]{Holm15}.
\begin{theorem}\label{thm:Holm_localG_0=Cok}
	Let $R$ be a complete local ring with a dualising module and suppose that $R$ is of finite $\CM$-type.
	Then there is an isomorphism of abelian groups
	\[ \Groth(R) \cong \Cok \Upupsilon, \]
	where $\Upupsilon \colon \Z^{t} \to \Z^{t+1}$ denotes the AR-matrix.
\end{theorem}

\section{Noncommutative crepant resolutions}\label{sec:NCCRs}
The notion of a smooth noncommutative minimal model, called a noncommutative crepant resolution, was first introduced by Van den Bergh in \cite[Definition~4.1]{VdB09}; for full details and references, see \cite{Le12}.
\begin{definition}\label{def:NCCRs}
	Let $R$ be a normal Gorenstein domain.
	A \emph{noncommutative crepant resolution} (NCCR) of $R$ is a ring of the form $\End_{R}(M)$, where $M$ is a non-zero reflexive $R$-module, such that $\End_{R}(M)$ has finite global dimension, and further is maximal Cohen-Macaulay as an $R$-module.
\end{definition}

In noncommutative algebraic geometry, NCCRs are one of the nicest kinds of resolution.
The following is a standard example (see, e.g. \cite[Example~1.1]{VdB09}) of an NCCR, which covers quotient singularities.
\begin{example}\label{example:End_R(S)=skewgroupring_an_NCCR}
	Let $V$ be a finite-dimensional vector space and $G$ a finite subgroup of $\SL(V)$.
	Write $S \coloneqq \text{Sym}(V)$ and $R \coloneqq S^{G}$.
	Then $\End_{R}(S)$ is an NCCR of $R$.
	By a result of Auslander \cite{AG60BrauerGrpOfCommRing, Aus62purityOfBranchLocus} (see also \cite[Theorem~3.2]{IyRyo13}), it turns out that $\End_{R}(S)$ is isomorphic to the skew group ring $S \# G$ introduced later in Definition~\ref{def:skew_group_ring}.
\end{example}

Another example is that of the suspended pinch point.
\begin{example}(The suspended pinch point)
	Consider the ring
	\[ R \coloneqq \frac{\C[u,v,x,y]}{(uv-x^{2}y)}. \]
	Then $\End_{R}(R \oplus (u,x) \oplus (u,x^{2}))$, $\End_{R}(R \oplus (u,x) \oplus (u,xy))$, and $\End_{R}(R \oplus (u, y) \oplus (u, xy))$ are all NCCRs \cite[Example~3.11]{We12}.
\end{example}

The noncommutative Bondal--Orlov conjecture asks if all NCCRs of a fixed $R$ are derived equivalent.
In \cite[Theorem~1.5]{IyWe13NoncommBondalOr}, the following holds more generally, but we restrict to normal Gorenstein domains as they suffice for our purposes.

\begin{theorem}\label{thm:Y_Lambda_derived_equiv}
	Let $R$ be a Gorenstein normal domain of dimension $d$.
	\begin{enumerate}
		\item\label{thm:Y_Lambda_derived_equiv 1} When $d=2$, all NCCRs of $R$ are Morita equivalent.
		\item\label{thm:Y_Lambda_derived_equiv 2} When $d=3$, all NCCRs of $R$ are derived equivalent.
	\end{enumerate}
\end{theorem}

Recall that there are some situations in which resolutions of singularities exist and are unique.
In the case of NCCRs, they do not necessarily exist, and if they do exist, they are usually not unique up to isomorphism.
The best we can hope for is unique up to Morita equivalence.
\begin{example}
	The ring
	\[ \frac{\C[u,v,x,y]}{(uv - x^{2} + y^{3})} \]
	has no NCCR (see, e.g. \cite[Theorem~1.3]{BIKR}) and, by Theorem~\ref{thm:Y_Lambda_derived_equiv}\eqref{thm:Y_Lambda_derived_equiv 1}, the ring 
	\[ \frac{\C[u,v,x]}{(uv - x^{2})} \]
	has unique NCCR (up to Morita equivalence).
\end{example}

\subsection{Cluster-tilting objects}\label{subsec:cluster_tilting_objects}
Full details on Calabi-Yau algebras and cluster-tilting theory can be found in many places in the literature, see, e.g. \cite{AIR14}, \cite{Le12}, \cite{BuanMarsh}, or \cite{IyRe08}.
For our purposes, we outline the connection to NCCRs here.
Let $k$ be a field.
A $k$-linear $\Hom$-finite exact category $\cC$ is \emph{$d$-Calabi-Yau} ($d$-CY) if there exist functorial isomorphisms
\[ D \Ext_{\cC}^{d-i}(X,Y) \cong \Ext_{\cC}^{d}(Y,X) \]
for all $X,Y$ in $\cC$ where $D \coloneqq \Hom_{k}(-,k)$.
An important class of objects in these categories are cluster-tilting objects, defined in the following way.
\begin{definition}\label{def:cluster_tilt}
	Let $M$ be an object of an exact category $\cC$.
	\begin{enumerate}[label=(\arabic*)]
		\item\label{def:cluster_tilt_1} $M$ is \emph{rigid} if $\Ext_{\cC}^{1}(M,M) = 0$.
		\item\label{def:cluster_tilt_2} $M$ is \emph{maximal rigid} if $M$ is rigid and, moreover, if $M \oplus Y$ is rigid for some $Y \in \cC$ then $Y \in \add M$.
		\item\label{def:cluster_tilt_3} $M$ is a \emph{cluster-tilting} object of $\cC$ (or simply, \emph{cluster-tilting}) if it is rigid and furthermore if $\Ext_{\cC}^{1}(M,Y) = 0$ for some $Y$ in $\cC$ then $Y \in \add M$, and also if $\Ext_{\cC}^{1}(Y,M) = 0$ then $Y \in \add M$.
	\end{enumerate}
\end{definition}
Equivalently, $M$ is a maximal rigid object of $\cC$ if
\[ \add M = \{ Y \in \cC \mid \Ext_{\cC}^{1}(M \oplus Y, M \oplus Y) = 0 \}, \]
and $M$ is cluster-tilting if
\[
	\add M = \{ Y \in \cC \mid \Ext_{\cC}^{1}(M, Y) = 0 \} = \{ X \in \cC \mid \Ext_{\cC}^{1}(X,M) = 0 \}.
\]

The most common example of cluster-tilting objects comes from invariant theory.
\begin{example}\label{example:cluster_tilting}
	As in \S\ref{sec:intro:ksings}, let $G$ be a finite subgroup of SL$(n, \C)$ and $R \coloneqq \C[[x_{1}, \hdots, x_{n}]]^{G}$.
	If $R$ is an isolated singularity, then the $R$-module $\C[[x_{1}, \hdots, x_{n}]]$ is a cluster-tilting object in $\CM R$; see \cite[Theorem~2.5]{Iy07higherAR}.
\end{example}

Just as minimal models and their corresponding MMAs are much more general than, and actually include, the case of NCCRs for cDV singularities, in certain ways modifying and maximal modifying modules are much more general than rigid and maximal rigid modules.
In fact, when $R$ is $3$-dimensional with isolated singularities, modifying modules recover the notion of rigid modules and MM modules recover the notion of maximal rigid modules \cite[Proposition~5.12]{IyWe14ARdual}.

The following result, due to \cite[Theorem~5.2.1]{Iy07},
gives a correspondence between NCCRs and cluster-tilting objects.

\begin{theorem}\label{thm:CM_Rmodule_M_gives_NCCR_M_cluster_tilting}
	Let $R$ be an isolated Gorenstein singularity of dimension $d \ge 2$.
	Then a $\CM R$-module $M$ gives an NCCR if and only if it is cluster-tilting.
\end{theorem}

Recall the precise description of $\Groth(R)$ given in Theorem~\ref{thm:Holm_localG_0=Cok}.
This description was based on the AR-matrix $\Upupsilon$ and requires the very strong assumption that $R$ has finite $\CM$ type.
In \cite[Proposition~7.28]{navkal13}, Navkal extends Theorem~\ref{thm:Holm_localG_0=Cok} by dropping the finiteness assumptions, at the cost of assuming the existence of a cluster-tilting object.
Fix a basic cluster-tilting object $M_{0} \oplus \hdots \oplus M_{t} \in \CM R$ with $M_{0} \cong R$.
For $j>0$ consider
\[ 0 \to L_{j}^{t} \to \cdots \to L_{j}^{1} \to L_{j}^{0} \coloneqq M_{j} \to 0 \]
the higher AR-sequence ending in $M_{j}$ (for definition, see, e.g. \cite[\S6.3]{navkal13}).
Expressing $\sum_{i=0}^{t}(-1)^{i}[L_{j}^{i}]$ in terms of $[M_{0}], \hdots, [M_{t}]$ gives a tuple $(a_{0j}, \hdots, a_{tj})$ of integers.
These tuples, as $j$ varies between $1$ and $t$, form the \emph{higher AR-matrix} $\Upomega \colon \Z^{t} \to \Z^{t+1}$ (for more details, see \cite[\S6.3]{navkal13}).
The next theorem was originally proven in \cite[Theorem~1.3]{navkal13}.

\begin{theorem}\label{thm:NavkalG_0(R)}
	Let $R$ be a Gorenstein complete local ring of dimension $3$ over an algebraically closed field, which has isolated singularities and admits a cluster-tilting object (equivalently, an NCCR).
	Then,
	\[ \Groth(R) \cong \Cok \Upomega \]
	where $\Upomega \colon \Z^{t} \to \Z^{t+1}$ denotes the higher AR matrix.
\end{theorem}

The main result in \cite{navkal13} uses the above to explicitly describe $\Groth(R)$ in the case of type $\tA$ cDVs which admit NCCR(s).
Later, we go further: in \S\ref{sec:isolated_cdv_with_NCCR}, we show that our results in Chapter~\ref{ch:knitting} allow us to describe $\Groth(R)$ for any cDV singularity $R$ that admits an NCCR, and in \S\ref{sec:arbitrary_type_A_cDV} we also describe $\Groth(R)$ for all $\ctAn{n}$ singularities regardless of whether they admit an NCCR, and regardless of whether they are isolated.

\section{The deformed preprojective algebra}\label{sec:deformed_preproj_deformed_ksings}
One way to get larger classes of examples from Kleinian singularities is by studying not just the singularities themselves, but also their deformations.
These deformations are important not only in geometry but also in representation theory and Lie theory.
Deformations of Kleinian singularities were studied by Crawley-Boevey and Holland in \cite{cbh98} using deformed preprojective algebras, which we now recall.
In general, these are noncommutative i.e., not-necessarily-commutative rings.

First, we recall the notions of quivers and path algebras.
Let $Q$ be a (finite) \emph{quiver}, that is, $Q$ is a directed graph consisting of finitely many vertices and arrows.
Say there are $n+1$ vertices, then each vertex is labelled $i=0, \hdots, n$, and each vertex $i$ has the trivial loop $e_{i} \colon i \to i$.
Let $c$ be an arrow, then the \emph{head} of $c$ is the vertex where $c$ points to and is denoted $h(c)$.
The \emph{tail} of $c$ is the vertex where $c$ starts and is denoted $t(c)$.
Informally, we think of $c$ as a map $c \colon t(c) \to h(c)$.
The set of vertices is denoted ${\sf{Q_{0}}}$ and the set of arrows by ${\sf{Q_{1}}}$; both are finite.
Fix a field $k$.
\begin{definition}
	The \emph{path algebra} $kQ$ of the quiver $Q$ is the associative $k$-algebra 
	with $k$-basis given by all trivial loops and non-trivial paths in $Q$.
	Multiplication is given by concatenation of paths\footnote{Throughout this thesis a concatenated path such as $pq$ means ``$p$ then $q$''.}:
	\begin{equation*}
	pq \coloneqq \left\{ \begin{array}{ll}
	p \cdot q, & \text{if } h(p) = t(q) \\
	0, & \textrm{otherwise.}
	\end{array} \right.
	\quad 
	e_{i}p \coloneqq \left\{ \begin{array}{ll}
	p , & \text{if } t(p) = i \\
	0, & \textrm{otherwise.}
	\end{array} \right.
	\quad
	pe_{i} \coloneqq \left\{ \begin{array}{ll}
	p , & \text{if } h(p) = i \\
	0, & \textrm{otherwise.}
	\end{array} \right.
	\end{equation*}
	for all paths $p$ and $q$.
	In addition, for $i, j \in {\sf{Q_{0}}}$, we have
	\[
		e_{i}^{2} = e_{i} \qquad \text{and} \qquad e_{i}e_{j} = 0 \quad (i \neq j).
	\]
\end{definition}
\noindent Notice that the $e_{i}$ are orthogonal idempotents in $kQ$ so the sum of the $e_{i}$'s is the identity element of $kQ$, denoted $1_{kQ}$.
A \emph{relation} in a quiver $Q$ is a $k$-linear combination of paths in $Q$, each with the same head and tail.
Relations can be thought of as telling us that following a specific path $p$ is the same as following a specific path $q$.
The specified relations generate a two-sided ideal of $kQ$.
Denote a quiver $Q$ with a set of relations $R$ by $(Q, R)$ and the path algebra of an associated quiver with relations by $kQ/\langle R \rangle$.

The \emph{double quiver} of the quiver $Q$, denoted $\bar{Q}$, consists of the vertices and arrows $c \colon t(c) \to h(c)$ of $Q$, together with a new arrow $c^{*} \colon h(c) \to t(c)$ for each original arrow $c$.
Let $k\bar{Q}$ be the path algebra of $\bar{Q}$.
The \emph{deformed preprojective algebra} is defined in \cite[\S2]{cbh98} as the associative algebra
\[
	\Pi^{\uplambda}(Q) = k\bar{Q} / \left( \sum_{c \in {\sf{Q_{1}}}} [c, c^{*}] - \sum_{i \in {\sf{Q_{0}}}} \uplambda_{i}e_{i} \right),
\]
where $\uplambda = (\uplambda_{i})_{i \in {\sf{Q_{0}}}} \in k^{\sf{Q_{0}}}$ is a weight, $c^{*}$ is the double arrow of the arrow $c$, and $[c,c^{*}]$ is the commutator $cc^{*} - c^{*}c.$
The \emph{preprojective algebra} $\Pi \coloneqq \Pi^{0}$ is a special case of the deformed preprojective algebra where $\uplambda = 0$.
We write $\Pi^{\uplambda}$ when the choice of $Q$ is clear.

Given an extended Dynkin diagram, choose an orientation to produce a quiver.
The double quivers are depicted in Figure~\ref{fig:preproj_quivers}.
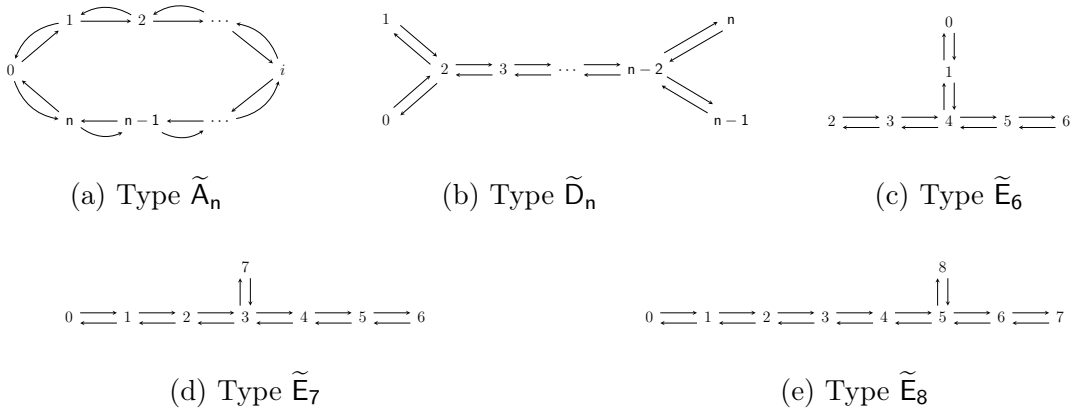
\begin{figure}[H]
	\centering
	\begin{subfigure}[b]{.3\textwidth}
		\centering
		\adjustbox{scale=.5}{\begin{tikzcd}[arrow style=tikz,>=stealth]
		& 1 \arrow[r, swap] \arrow[dl, swap, bend right = 30] & 2 \arrow[r, swap] \arrow[l, swap, bend right = 30] & \arrow[l, swap, bend right = 30] \cdots \arrow[dr, swap] & \\
		0 \arrow[ur, swap] \arrow[dr, swap, bend right = 30] & & & & i \arrow[ul, swap, bend right = 30] \arrow[dl, swap] \\
		& \sf{n} \arrow[ul, swap] \arrow[r, swap, bend right = 30] & \sf{n}-1 \arrow[l, swap] \arrow[r, swap, bend right = 30] & \arrow[l, swap] \cdots \arrow[ur, swap, bend right = 30] &\\
		\end{tikzcd}}
		\caption{Type $\teAn{n}$}
		\label{fig:preproj_quiver_An}
	\end{subfigure}
	\begin{subfigure}[b]{.3\textwidth}
		\centering
		\adjustbox{scale=.5}{\begin{tikzcd}[arrow style=tikz,>=stealth]		
		& 1 \arrow[dr, shift left=.75ex] & & & & & \arrow[dl, shift left=.75ex] \sf{n} & \\
		& & \arrow[dl, shift left=.75ex] \arrow[ul, shift left=.75ex] 2 \arrow[r, shift left=.75ex] & \arrow[l, shift left=.75ex] 3 \arrow[r, shift left=.75ex] & \arrow[l, shift left=.75ex] \cdots \arrow[r, shift left=.75ex] & \arrow[l, shift left=.75ex] \sf{n}-2 \arrow[ur, shift left=.75ex] \arrow[dr, shift left=.75ex] & & \\
		& 0 \arrow[ur, shift left=.75ex] & & & & & \arrow[ul, shift left=.75ex] \sf{n}-1 &\\
		\end{tikzcd}}
		\caption{Type $\teDn{n}$}
		\label{fig:preproj_quiver_Dn}
	\end{subfigure}
	\qquad \begin{subfigure}[b]{.29\textwidth}
		\centering
		\adjustbox{scale=.5}{\begin{tikzcd}[arrow style=tikz,>=stealth]
		&&& 0 \arrow[d, shift left=.75ex] &&&\\
		&&& \arrow[u, shift left=.75ex] 1 \arrow[d, shift left=.75ex] &&&\\
		& 2 \arrow[r, shift left=.75ex] & \arrow[l, shift left=.75ex] 3 \arrow[r, shift left=.75ex] & \arrow[l, shift left=.75ex] \arrow[u, shift left=.75ex] 4 \arrow[r, shift left=.75ex] & \arrow[l, shift left=.75ex] 5 \arrow[r, shift left=.75ex] & \arrow[l, shift left=.75ex] 6 & \\
		\end{tikzcd}}
		\caption{Type $\teEn{6}$}
		\label{fig:preproj_quiver_E6}
	\end{subfigure}
\newline \\
	\begin{subfigure}[b]{.49\textwidth}
		\centering
		\adjustbox{scale=.5}{\begin{tikzcd}[arrow style=tikz,>=stealth]
		&&&& 7 \arrow[d, shift left=.75ex] &&&\\
		& 0 \arrow[r, shift left=.75ex] & \arrow[l, shift left=.75ex] 1 \arrow[r, shift left=.75ex] & \arrow[l, shift left=.75ex] 2 \arrow[r, shift left=.75ex] & \arrow[l, shift left=.75ex] \arrow[u, shift left=.75ex] 3 \arrow[r, shift left=.75ex]& \arrow[l, shift left=.75ex] 4 \arrow[r, shift left=.75ex]  & \arrow[l, shift left=.75ex] 5 \arrow[r, shift left=.75ex] & \arrow[l, shift left=.75ex] 6 & \\
		\end{tikzcd}}
		\caption{Type $\teEn{7}$}
		\label{fig:preproj_quiver_E7}
	\end{subfigure}
	\begin{subfigure}[b]{.5\textwidth}
		\centering
		\adjustbox{scale=.5}{\begin{tikzcd}[arrow style=tikz,>=stealth]
		&&&&&& 8 \arrow[d, shift left=.75ex] &&&\\
		& 0 \arrow[r, shift left=.75ex] & \arrow[l, shift left=.75ex] 1 \arrow[r, shift left=.75ex] & \arrow[l, shift left=.75ex] 2 \arrow[r, shift left=.75ex] & \arrow[l, shift left=.75ex] 3 \arrow[r, shift left=.75ex] & \arrow[l, shift left=.75ex] 4 \arrow[r, shift left=.75ex]& \arrow[l, shift left=.75ex] \arrow[u, shift left=.75ex] 5 \arrow[r, shift left=.75ex]  & \arrow[l, shift left=.75ex] 6 \arrow[r, shift left=.75ex] & \arrow[l, shift left=.75ex] 7 & \\
		\end{tikzcd}}
		\caption{Type $\teEn{8}$}
		\label{fig:preproj_quiver_E8}
	\end{subfigure}\caption{Double quivers of $\ADE$ type}\label{fig:preproj_quivers}
\end{figure}

\subsection{Deformations of Kleinian singularities}\label{subsec:deformations_of_ksings} 
As before, let $S \coloneqq \C[U, V]$ and $R \coloneqq S^{G}$ be a Kleinian singularity.
\begin{definition}\label{def:skew_group_ring} 
	Suppose $G$ is a group acting on a ring $S$.
	The \emph{skew group ring} $S \# G$ is the free $S$-module with elements of $G$ as a basis where multiplication is extended linearly by the rule $(rg)(sh) = r(g \cdot s)gh$ for $r,s \in S$ and $g,h \in G$.
	Here $g \cdot s$ is the image of $s$ under the action of $g$.
\end{definition}
Since the group $G$ acts naturally on the free algebra $\C \langle U, V \rangle$, we can form the skew group algebra $\C \langle U, V \rangle \# G$.
For more details, see \cite[\S3.7]{cbLect}.

Let $\uplambda$ be a weight, that is $\uplambda = (\uplambda_{i})_{i \in {\sf{Q_{0}}}} \in \C^{\sf{Q_{0}}}$, and let $e_{0} \coloneqq \frac{1}{|G|} \sum_{g \in G} g$ be the average of the group elements; $e_{0}$ is an idempotent and an element of $\C \langle U, V \rangle \# G$.
We can identify the centre of the group algebra $Z(\C G)$ with $\C^{\sf{Q_{0}}}$.
Under this identification there exists a unique $z_{\uplambda}$ in $Z(\C G)$ such that the trace of $z_{\uplambda}$ on $\uprho_{i}$ is $(\dim \uprho_{i}) \uplambda_{i}$; for a more in depth explanation see \cite{cbh98}.
As in \cite{cbh98}, define
\[ \cS^{\uplambda}(Q) \coloneqq \frac{\C \langle U, V \rangle \# G}{\langle UV - VU - \uplambda \rangle} \quad \text{ and } \quad \cO^{\uplambda}(Q) \coloneqq e_{0} \cS^{\uplambda}e_{0}. \]

Just as with $\Pi^{\uplambda}$, we frequently write $\cS^{\uplambda}$ and $\cO^{\uplambda}$ when the choice of $Q$ is clear.
We can filter $\cS^{\uplambda}$ with $U$ and $V$ in degree $1$ and elements of $G$ in degree $0$, and this restricts to a filtration of $\cO^{\uplambda}$ \cite{cbh98}.
With these filtrations we can consider the associated graded rings $\text{gr } \cS^{\uplambda}$ and $\text{gr } \cO^{\uplambda}$.
Crawley-Boevey--Holland prove the following result in \cite[Lemma~1.1]{cbh98}.
\begin{lemma}\label{lem:grS^lam_cong_skew_grp_O^lam_cong_coord_ring}
	With the above filtrations, there are isomorphisms:
	\[ \mathrm{gr } \cS^{\uplambda} \cong S \# G \quad \text{and} \quad \mathrm{gr } \cO^{\uplambda} \cong R. \]
\end{lemma}

This means that $\cS^{\uplambda}$ can be viewed as a deformation of $S \# G$, and $\cO^{\uplambda}$ as a deformation of $R$.
Crawley-Boevey--Holland determined many ring theoretic and homological properties of these deformations; see \cite[Lemmas~1.2 and~1.3]{cbh98}.
Provided $G$ is non-trivial, the ring $\cS^{\uplambda}$ is always noncommutative and, in general, $\cO^{\uplambda}$ is also noncommutative.
For us, $\cO^{\uplambda}$ will always be commutative and has Krull dimension 2 since below we will only consider the case when $\uplambda \cdot \updelta = \sum_{i} \uplambda_{i}\updelta_{i} = 0$, see \cite[Theorem~0.4(1) and (5)]{cbh98}.
Here, $\updelta_{i}$ is the dimension of the corresponding representation at vertex $i$ (see \S\ref{sec:intro:mckay_correspondence}).

To determine many of the properties alluded to above, the following result is due to Crawley-Boevey--Holland \cite[Theorem~0.1]{cbh98}, who prove that deformations $\cS^{\uplambda}$ are Morita equivalent to deformed preprojective algebras $\Pi^{\uplambda}$.
\begin{theorem}\label{thm:S^lam_Morita_equiv_Pi^lam}
	There is a Morita equivalence between $\cS^{\uplambda}$ and $\Pi^{\uplambda}$.
	This induces an isomorphism $\cO^{\uplambda} \cong \ePie$.
\end{theorem}

\section{Homological Minimal Model Program}\label{subsec:homMMP}

For this thesis, as we do not require much knowledge of the inner workings of the Homological MMP, we just give a very brief overview.
At a very high level, the goal of the Homological MMP is to find the `best' (i.e., simplest) approximations of singular spaces.
These `best' approximations are precisely the minimal models.

Although geometric by nature, the Homological MMP provides a purely algebraic output, lifting the correspondence of \S\ref{sec:intro:mckay_correspondence} from dimension two to threefold cDV singularities.
In \cite[Corollary~6.9]{We18}, Wemyss shows that for an isolated complete local cDV singularity $R$, there is a one-to-one correspondence
\begin{equation*}
\left\{ \parbox{4cm}{\centering basic MM $R$-module generators} \right\} \longleftrightarrow \left\{ \parbox{4cm}{\centering minimal models $f_{i} \colon X_{i} \to \Spec R$} \right\}.
\end{equation*}
Wemyss then considers the situation where the minimal models of $\Spec R$ are smooth.
In this case, the above correspondence reduces to the following one-to-one correspondence
\begin{equation*}
\left\{ \parbox{4cm}{\centering basic cluster-tilting objects in $\CM R$} \right\} \longleftrightarrow \left\{ \parbox{4cm}{\centering crepant resolutions $f_{i} \colon X_{i} \to \Spec R$} \right\}.
\end{equation*}

Now, suppose there is an MMA of the form $\End_{R}(M)$, where $M$ is a basic MM module with $M_{0} \cong R$ and $M_{1}, \hdots, M_{t}$ the non-free indecomposable summands of $M$.
By the above bijection, $M$ is in one-to-one correspondence with a minimal model $f \colon X \to \Spec R$.
This extends to a bijection between the indecomposable summands of $M$ and the (exceptional) curves in the corresponding minimal model.
The bijection is illustrated in Figure~\ref{fig:curves_indecomp_summands_correspondence}.
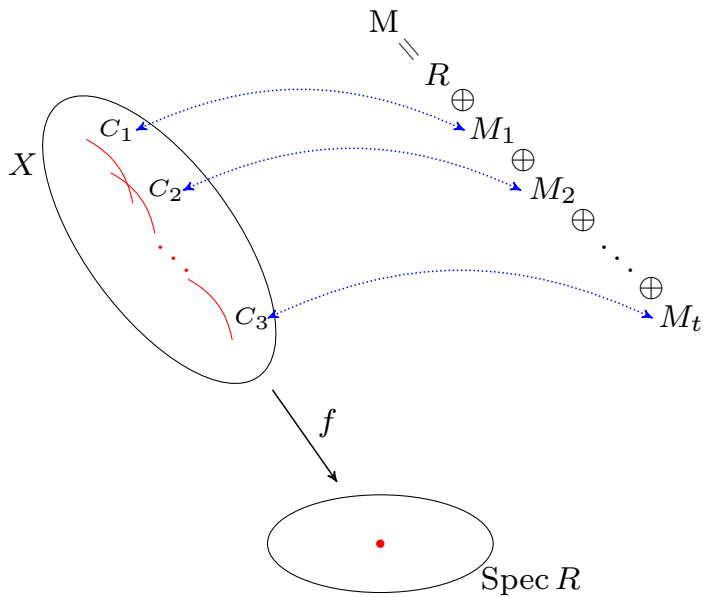
\begin{figure}[H]
	\centering
	\tikzset{Bullet/.style={fill=black,draw,color=#1,circle,minimum size=3pt,scale=0.25}}
	\begin{tikzpicture}[looseness=1, scale=0.5]
	\draw (0,0) ellipse (3cm and 1.3cm);
	\node[Bullet=red] at (0,0) {};
	\draw[rotate=36] (0,10) ellipse (2cm and 4.5cm);
	\draw[rotate=36, red] (0,11.2) to [bend right=25] (0,13.3);
	\draw[rotate=36, red] (0,10.2) to [bend right=25] (0,12.2);
	\node[scale=1.5, red] at (-5.5,7.9) {\scriptsize $\ddots$};
	\draw[rotate=36, red] (0,6.7) to [bend right=25] (0,8.7);
	\node[scale=1.5] at (-7,11) {\tiny $C_{1}$};
	\node[scale=1.5] at (-5.7,9.4) {\tiny $C_{2}$};
	\node[scale=1.5] at (-3.4,6) {\tiny $C_{3}$};
	\node[scale=1.5] at (0.1,13.9) {\scriptsize{M}};
	\node[scale=1.5, rotate=130.5] at (0.8,13.2) {\scriptsize{=}};
	\node[scale=1.5] at (1.5,12.5) {\scriptsize{$R$}};
	\node[scale=1.5] at (2.2,11.8) {\scriptsize{$\oplus$}};
	\node[scale=1.5] at (3,11) {\scriptsize{$M_{1}$}};
	\node[scale=1.5] at (3.8,10.2) {\scriptsize{$\oplus$}};
	\node[scale=1.5] at (4.6,9.4) {\scriptsize{$M_{2}$}};
	\node[scale=1.5] at (5.4,8.6) {\scriptsize{$\oplus$}};
	\node[scale=1.5] at (6.3,7.9) {\scriptsize{$\ddots$}};
	\node[scale=1.5] at (7.2,6.8) {\scriptsize{$\oplus$}};
	\node[scale=1.5] at (8,6) {\scriptsize{$M_{t}$}};
	\draw[<->, line width=0.2mm, bend right=25, densely dotted, blue] (2.25,11) to (-6.5,11);
	\draw[<->, line width=0.2mm, bend right=25, densely dotted, blue] (3.75,9.4) to (-5.25,9.4);
	\draw[<->, line width=0.2mm, bend right=25, densely dotted, blue] (7.25,6) to (-3,6);
	\draw[->, looseness=0,rotate=35, line width=0.2mm] (0,5) to (0,2);
	\node[scale=1.5] at (-9.5,10.1) {\scriptsize{$X$}};
	\node[scale=1.5] at (-1.4,3.2) {\scriptsize{$f$}};
	\node[scale=1.5] at (4,-1) {\scriptsize{$\Spec R$}};
	\end{tikzpicture}
	\caption{One-to-one correspondence of non-free indecomposable summands of a basic MM module $M$ and the (exceptional) curves in the corresponding minimal model.}
	\label{fig:curves_indecomp_summands_correspondence}
\end{figure}
\chapter{Viehweg's setting}\label{ch:divisor_cl_grps}
In this chapter, we study rings of the form
\begin{equation}\label{eq:the_ring_A}
	A \colonequals \frac{\C[u,v,x_{1}, \hdots , x_{n}]}{(uv - f_{1}^{a_{1}} \cdots  f_{t}^{a_{t}})},
\end{equation}
where the $f_{i} \in \C[x_{1}, \hdots , x_{n}]$ are irreducible and pairwise coprime and each $a_{i}$ is at least one.
In particular, each $f_{i}$ is non-constant.
Such rings form part of Viehweg's setting (see \S\ref{sec:intro:viehwegs_setting}).
In this chapter we compute the divisor class group of $A$ in general, and the Grothendieck group of $A$ when $n=1$.
The main result, Theorem~\ref{thm:G_0(R)=Z+Cl(R)_dim2}, asserts that in the \emph{global} setting, when $n=1$, the isomorphism \eqref{eq:Groth=Z+Cl_general} holds for $A$.
In \S\ref{sec:Groth_deformed_preproj}, we apply these results to the centres of type $\tA$ deformed preprojective algebras.

\section{Determining class groups using Nagata's Theorem}\label{sec:cl_grp_nagatas_thm}
In this section, we determine the divisor class groups of $A$, as above, using Nagata's theorem (Theorem~\ref{thm:nagata}).
To begin, we require the following.

\begin{lemma}\label{lem:A_is_normal_ID}
	The scheme $X = \Spec A$ is reduced and irreducible (equivalently, $A$ is an integral domain) and it is normal.
\end{lemma}
\begin{proof}
	Every hypersurface in affine $(n+2)$-space is a complete intersection.
	Complete intersections are CM and they are normal if and only if the singular locus has codimension at least two \cite[II, Proposition~8.23]{hart77}.
	By the Jacobian criterion, the singular locus is included in the locus $(u = v = 0)$.
	This subvariety has dimension $(n+2) - 3 = n-1$.
	Therefore, the singular locus has codimension at least $(n+1) - (n-1) = 2$ in $X$.
	Hence $A$ is a normal ring.

	We next claim $A$ is a domain.
	Since $A$ is a normal ring, it automatically decomposes into a product of normal domains $A \cong A_{1} \oplus \cdots \oplus A_{m}$ \cite[Exercise~9.11]{matsu00}.
	We claim $m =1$.
	Consider $U_{1} = \{ u \neq 0 \} = \Spec \C[u^{\pm 1}, x_{1}, \hdots, x_{n}]$, (which is connected since this ring is a domain).
	Similarly, $U_{2} = \{v \neq 0 \}$ is connected.
	Then these are open sets contained in the smooth locus of $X$.
	Thus we have the following two facts:
	\begin{enumerate}[label=(\alph*)]
		\item The complement to $U = U_1 \cup U_2$ has codimension $2$ in $X$.
		\item $U_1 \cap U_2 \neq \emptyset$.
	\end{enumerate}
	Fact (a) is true since $X \smallsetminus U$ is defined by $u = v = f_{1}^{a_{1}} \cdots f_{t}^{a_{t}} = 0$, and so is a hypersurface in $\C^n$.
	Fact (b) can be seen by taking $u = v = 1$ and finding a point in $Y = \{ x_i \mid f_{1}^{a_{1}} \cdots f_{t}^{a_{t}} = 1 \}$.
	
	Since $U_{1}$ and $U_{2}$ are connected and their intersection is nonempty, $U$ is connected.
	Since $U$ is smooth and connected, it is irreducible.
	Indeed, all stalks of points in $U$ are regular local rings, which are domains.
	For any point lying on the intersection of irreducible components, the local ring at that point is not a domain: it has as many minimal prime ideals as the number of irreducible components (see part (2) of \cite[\href{https://stacks.math.columbia.edu/tag/00ET}{Tag 00ET}]{stacks-project}).
	Hence $U$ is contained in some irreducible component $Y$ of $X = \Spec A$ (see, e.g. part (1) of  \cite[\href{https://stacks.math.columbia.edu/tag/004W}{Tag 004W}]{stacks-project}).
	But the irreducible components of $\Spec A$ are the $\Spec A_{i}$, and so $U \subseteq \Spec A_{i}$ for some $i$.
	Since $A$ is a hypersurface, it is equidimensional; see the remarks following Definition~10.134.1 of  \cite[\href{https://stacks.math.columbia.edu/tag/00S8}{Tag 00S8}]{stacks-project}.
	This implies that $\dim \Spec A_{i} = n+1$ for all $i$.
	Given that we have shown that $\dim X \smallsetminus U \le \dim X - 2$, it follows that $m=1$.
\end{proof}

Note that $u \in A$ is a non-zero element and observe that
\[
A[u^{-1}] = 
\frac{\C[u^{\pm 1}, v, x_{1}, \hdots , x_{n}]}{(v = u^{-1} {f_{1}}^{a_{1}} \hdots {f_{t}}^{a_{t}})}
= 
\C[u^{\pm 1}, x_{1}, \hdots , x_{n}].
\]
Therefore, $A[u^{-1}]$ is a UFD.
As such, every prime ideal of height one in $A[u^{-1}]$ is principal \cite[Theorem~20.1]{matsu00}.

\begin{lemma}\label{lem:A[u]* = C*x<u>}
	The units in $A[u^{-1}]$ are $\C^{\times} \times \langle u \rangle$.
\end{lemma}
\begin{proof}
	From the above $A[u^{-1}]^{\times} = \C[u^{\pm1}, x_{1}, \hdots , x_{n}]^{\times}$.
	Now, take $Q = \C[u, x_{1}, \hdots , x_{n}]$ and $q = u$, so $Q[q^{-1}]^{\times} = A[u^{-1}]^{\times}$.
	By Proposition~\ref{prop:unitsQ[q-1]=Cx<q>}, $Q[q^{-1}]^{\times} = Q^{\times} \times \langle q \rangle$, thus
	\[
	A[u^{-1}]^{\times} = \C[u, x_{1}, \hdots , x_{n}]^{\times} \times \langle u \rangle.
	\]
	By Lemma~\ref{lem:R^x = k^x}, $\C[u, x_{1}, \hdots , x_{n}]^{\times} = \C^{\times}$ and the result follows.
\end{proof}
This result leads to the following description of the units of $A$.
\begin{lemma}\label{lem:A*=C*}
	The units in $A$ are $\C^{\times}$.
\end{lemma}
\begin{proof}
	By Lemma~\ref{lem:A_is_normal_ID}, $A$ is a domain, so the natural ring homomorphism $A \to A[u^{-1}]$ is injective.
	Since homomorphisms preserve units this induces an injective map $A^{\times} \to A[u^{-1}]^{\times}$.
	Using this together with Lemma~\ref{lem:A[u]* = C*x<u>}, if $a \in A^{\times}$, then $a = \uplambda u^{k}$ for some $\uplambda \in \C^{\times}$ and some $k \in \Z$.
	Necessarily, $u^{k}$ is a unit in $A$, since $a$ and $\uplambda$ are both units.
	
	Since $f_{1}$ is not constant it has a root $c$ which is equal to $(c_{1}, \hdots, c_{n})$.
	Now the ring homomorphism $\upvarphi : A \to A/ (u, x_{1} - c_{1}, \hdots , x_{n} - c_{n}) = \C[v]$ sends $u$ to $0$.
	Thus $\upvarphi(u^{k}) = \upvarphi(u)^{k} = 0$, hence $u^{k}$ cannot be a unit unless $k=0$, since homomorphisms preserve units.
	Hence $k=0$, so $a = \uplambda$ and thus $A^{\times} = \C^{\times}$.
\end{proof}

For the following result consider the ring
\[
T\colonequals A/(u) = \frac{\C [v, x_{1}, \hdots , x_{n}]}{(f_{1}^{a_{1}} \hdots f_{t}^{a_{t}})}.
\]

\begin{lemma}\label{lem:Anonlyht0}
	There are finitely many prime ideals of height zero in $T$, and they are all of the form $(f_{i})$ for some $i$.
\end{lemma}	
\begin{proof}
	Suppose that $I$ is a height zero prime ideal in $T$.
	Necessarily, $f_{1}^{a_{1}} \hdots f_{t}^{a_{t}}$ is in $I$.
	Hence there must be some $f_{i} \in I$, which implies that $(f_{i}) \subseteq I$.
	But $T/(f_{i}) \cong \frac{\C [v, x_{1}, \hdots , x_{n}]}{(f_{i})}$, with $f_{i}$ irreducible.
	Therefore $T/(f_{i})$ is an integral domain and so $(f_{i})$ is a prime ideal of $T$.
	Since $I$ has height zero, $(f_{i}) = I$.
	This shows that $(f_{i})$ is a height zero prime ideal, and all height zero prime ideals are of this form.
\end{proof}
Lemma~\ref{lem:Anonlyht0} allows us to give a precise description of the height one prime ideals of $\div(u)$ where $\div$ is defined in \S\ref{sec:prelims:class_grp}.
\begin{lemma}\label{lem:Anonlyprimeideal1}
	The only prime ideals of height one appearing in $\div (u)$ are $\mfp_{i} = (u, f_{i})$.
\end{lemma}
\begin{proof}
	Notice that the ideal $\mfp_{i} = (u, f_{i})$ is a prime ideal in $A$ since
	\begin{align*}
	A/\mfp_i&\cong \left( \frac{\C [u, v, x_{1}, \hdots , x_{n}]}{(uv - f_{1}^{a_{1}} \hdots f_{t}^{a_{t}})} \right) / (u, f_{i})\\
	&= \C[v, x_1, \hdots, x_n]/(f_{i})
	\end{align*}
	which is an integral domain since $f_i$ is irreducible.
	Furthermore, $\C[v, x_{1}, \hdots, x_{n}]/(f_{i})$ is the coordinate ring of a hypersurface in $\mathbb{A}^{n+1}$, thus $\dim A/\mfp_i=n$.
	By Theorem~\ref{thm:height+dim}
	\[
	\height(\mfp_{i}) + \dim A/ \mfp_{i} = \dim A,
	\]
	hence $\height(\mfp_{i}) = 1$.
	Hence the $(u, f_{i})$ are height one primes of $A$, which under $A \to A/(u) = T$ clearly map to $(f_{i})$.
	
	By Lemma~\ref{lem:Anonlyht0} the $(f_{i})$ are all height zero prime ideals of $A/(u)$, hence under the bijection in Theorem~\ref{thm:htbijection}, it follows that the only height one prime ideals of $A$ containing $u$ are those of the form $\mfp_{i} = (u, f_{i})$.
	By Corollary~\ref{cor:appearance_of_div(r)_r_in_R} these are precisely the height one prime ideals appearing in $\div(u)$.
\end{proof}

By Lemma~\ref{lem:Anonlyprimeideal1} it follows that $\div(u) = \sum_{i=1}^{t} b_{i}(u,f_{i})$ for some positive integers $b_{i}$.
\begin{lemma}\label{lem:Anbi=ai}
	Let $a_{i}$ be the exponents in the defining equation of the ring $A$ and $b_{i}$ the coefficients defined above.
	Then $b_{i} = a_{i}$.
\end{lemma}
\begin{proof}
	Without loss of generality, we may assume $i=1$.
	The proof for $i > 1$ is similar.
	Set $\mfp_{1} \colonequals (u,f_{1})$.
	We begin by showing that $v, f_{2}^{a_{2}} \hdots f_{t}^{a_{t}} \notin \mfp_{1}$.
	The methods for proving this for $v$ and $f_{i}^{a_{i}}$ are slightly different.
	
	First we show that $v \notin \mfp_{1}$.
	Suppose that $v \in \mfp_{1}$.
	Then $A/ \mfp_{1}= A/(u, f_{1}, v)$ is isomorphic to $\C[x_{1}, \hdots, x_{n}]/(f_{1})$ which is an integral domain since the $f_{i}$ are irreducible.
	It is a ring of dimension $n-1$.
	Therefore, by Theorem~\ref{thm:height+dim}, $(u,f_1,v) = \mfp_{1}$ is a height two prime ideal of $A$.
	This contradicts the fact that $\height(\mfp_{1})=1$, thus $v \notin \mfp_{1}$.
	
	Next, we show that $f_{2}^{a_{2}} \hdots f_{t}^{a_{t}} \notin \mfp_{1}$.
	If $f_{2}^{a_{2}} \hdots f_{t}^{a_{t}} \in \mfp_{1}$ then there exists some $i$ such that $f_{i} \in \mfp_{1}$.
	Without loss of generality, we may assume $f_{2} \in \mfp_{1}$. 	
	Then $A/ \mfp_{1} = A/(u, f_{1}, f_{2})$ is isomorphic to $\C[v, x_{1}, \hdots, x_{n}]/(f_{1}, f_{2})$.
	As before, factoring by $f_{1}$ gives a domain since $f_{1}$ is irreducible.
	By assumption $f_{2} \notin (f_{1})$.
	By Theorem~\ref{thm:height+dim}, the Krull dimension of $A/ \mfp_{1}$ is $(n-1)$ minus the minimum of the heights of the prime ideals containing $f_{2}$ in $\C[v, x_{1}, \hdots, x_{n}]/(f_{1})$.
	Since the latter ring is a domain and $f_{2}$ is not a unit in this ring, the height of all prime ideals containing $(f_{2})$ is at least one.
	Hence, the Krull dimension of $A/ \mfp_{1}$ is at most $n-2$.
	This implies that $(u, f_{1}, f_{2}) = \mfp_{1}$ is not a height one prime ideal of $A$.
	Again, this is a contradiction, thus $f_{2} \notin \mfp_{1}$.

	By definition, the element $u \in A_{\mfp_{1}}$ can be written as $s \uppi_{\mfp_{1}}^{\upnu_{\mfp_{1}}(u)}$ for some unit $s \in A_{\mfp_{1}}^{\times}$ and $\upnu_{\mfp_{1}}(u)$ as defined in Theorem~\ref{thm:valuation}.
	Thus $v, f_{2}, \hdots , f_{t}$ are all units in $A_{\mfp_{1}}$ and $u = v^{-1} f_{1}^{a_{1}} \hdots f_{t}^{a_{t}}$.
	This has two consequences.
	First, $\mfp_{1} A_{\mfp_{1}} = (u, f_{1}) A_{\mfp_{1}}$ is generated by $f_{1}$.
	So $\uppi_{\mfp_{1}} = f_{1}$.
	Second, 
	\[
	u = f_{1}^{a_{1}}(v^{-1} f_{2}^{a_{2}} \hdots f_{n}^{a_{n}}),
	\]
	and thus, $\upnu_{\mfp_{1}}(u) = a_{1}$.
	The proof that $\upnu_{\mfp_{i}}(u) = a_{i}$ for $i > 1$ is identical.
\end{proof}
By Lemma~\ref{lem:Anbi=ai}, $\div(u) = \sum_{i=1}^{t} a_{i}(u,f_{i})$.
\begin{theorem}\label{thm:Cl(A)_global}
	Let $A$ be as in \eqref{eq:the_ring_A}.
	Then 
	\[ \Cl(A) \cong \Z^{\oplus t}/ (a_1,\hdots,a_t). \]
\end{theorem}
\begin{proof}
	Together, Lemmas~\ref{lem:A[u]* = C*x<u>}, \ref{lem:A*=C*}, and \ref{lem:Anbi=ai} imply that the exact sequence \eqref{seq:nagata} can be rewritten as
	\[
	0 \to A^{\times} = \C^{\times} \longrightarrow \C^{\times} \times \langle u \rangle \longrightarrow \bigoplus_{i=1}^t\Z(u,f_i) \longrightarrow \Cl(A) \longrightarrow \Cl(A[u^{-1}]) \to 0.
	\]
	Since $A[u^{-1}]$ is a UFD, Proposition~\ref{prop:ufd=>cl=0} says that the class group of $A[u^{-1}]$ is trivial.
	Hence the above sequence becomes
	\[
	0 \to \C^{\times} \longrightarrow \C^{\times} \times \langle u \rangle \longrightarrow \bigoplus_{i=1}^t\Z(u,f_i) \longrightarrow \Cl(A) \to 0.
	\]
	Splicing gives the short exact sequence
	\begin{equation*}\label{seq:nagata3split}
	0 \longrightarrow \Z \langle u \rangle 
	\stackrel{\uppsi}{\longrightarrow}
	\bigoplus_{i=1}^t\Z(u,f_i) \longrightarrow \Cl(A) \longrightarrow 0
	\end{equation*}
	where $\uppsi$ is given by $u \mapsto \div(u)$.
	By Lemma~\ref{lem:Anbi=ai}, $\div(u) = \sum_{i=1}^t a_i(u,f_{i})$.
	This simplifies to the short exact sequence
	\begin{equation*}\label{seq:nagata4split}
	0 \longrightarrow \Z \stackrel{\upgamma}{\longrightarrow}
	\Z^{\oplus t} \longrightarrow \Cl(A) \longrightarrow 0
	\end{equation*}
	where $\upgamma$ is given by $1 \mapsto (a_1,\hdots,a_t)$.
	It follows that
	\[
	\Cl(A) = \Z^{\oplus t}/ (a_1,\hdots,a_t).\qedhere
	\]
\end{proof}
	
\section{The K-theory of $A$}\label{sec:cl_grp_k_theory}
In this section, we use algebraic K-theory to compute the divisor class group of type $\tA$ Kleinian singularities.
In the local case, the isomorphism \eqref{eq:Groth=Z+Cl_general} was already known to hold by \cite[Lemma~13.3]{yosh90}.
In fact, it is known to hold for any $2$-dimensional local integral domain.
But here we consider global situations.
As such, the result in Theorem~\ref{thm:G_0(R)=Z+Cl(R)_dim2} is new and covers a wider range of rings.

Recall that $\Groth(R) \colonequals \Kroth(\modCat R)$, where $\modCat R$ is the category of finitely generated left $R$-modules as in \S\ref{sec:prelims:alg_k_theory}.
The ring $A$, as in \eqref{eq:the_ring_A}, is equal to
\[
\frac{\C [u,v,x_{1}, \hdots, x_{n}]}{(uv-f(x_{1}, \hdots, x_{n}))}.
\]
\begin{prop}\label{prop:[A/m]=0}
	Let $\mfm$ be a maximal ideal of $A$. Then $[A/\mfm]=0$ in $\Groth(A)$.
\end{prop}
\begin{proof}
	By Hilbert's Nullstellensatz every maximal ideal $\mfm$ of $A$ is of the form 
	\[ \mfm = (u-a, v-b, x_{1}-c_{1}, \hdots, x_{n}-c_{n}) \]
	for some $a, b, c_{1}, \hdots, c_{n}  \in \C$.
	Suppose $b=0$, and consider the ideal $\mfp = (v, x_{1}-c_{1}, \hdots, x_{n}-c_{n})$.
	Notice that this is a prime ideal of $A$ since
	\[ A/(v, x_{1}-c_{1}, \hdots, x_{n}-c_{n}) \cong \C[u], \]
	which is an integral domain.
	Thus, by Lemma~\ref{lem:R/(p+xR)=0}, $[A/\mfm]= [A/\mfp + (u-a)A] = 0$. 
	The case of $a=0$ follows by symmetry.
	Hence we can suppose $a \neq 0$ and $b \neq 0$.
	In particular, since $\mfm$ is a maximal ideal of $A$, $ab - f(c_{1}, \hdots, c_{n}) = 0$.
	So, $f(c_{1}, \hdots, c_{n}) \neq 0$.
	
	In order to prove the statement it suffices to find a prime ideal $\mfp$ such that $v-b \notin \mfp$ but $\mfp + (v-b)A = \mfm$.
	Consider the ideal $\mfp \colonequals (x_{1}-c_{1}, \hdots , x_{n}-c_{n})$.
	Setting $\lambda \colonequals ab \neq 0$,
	\[
	A/(x_{1}-c_{1}, \hdots , x_{n}-c_{n}) \cong \frac{\C[u,v]}{(uv - \lambda)} \cong \C[u,u^{-1}]
	\]
	where the latter isomorphism is given by $v \mapsto \frac{1}{\lambda} u^{-1}$.
	This quotient is an integral domain, so $\mfp$ is a prime ideal.
	Furthermore,
	\begin{align*}
	\mfp + (v-b)A &= (x_{1}-c_{1}, \hdots , x_{n}-c_{n}) + (v-b)A \\
	&= (v-b, x_{1}-c_{1}, \hdots , x_{n}-c_{n}).
	\end{align*}
	Now we need to show that $\mathfrak{n} \colonequals (v-b, x_{1}-c_{1}, \hdots , x_{n}-c_{n})$ equals $\mfm$.
	Observe that $A/\mathfrak{n} \cong \C$, which is a field.
	Thus $\mathfrak{n}$ is a maximal ideal of $A$ and by Lemma~\ref{lem:R/(p+xR)=0}, $[A/\mathfrak{n}]=0$.
	Since $\mathfrak{n} \subseteq \mfm$ and $\mathfrak{n}$ is maximal, it follows that $\mathfrak{n} = \mfm$.
	Thus $[A/\mfm] = 0$, as desired.
\end{proof}

At the outset, we were not expecting the following result.
\begin{theorem}\label{thm:G_0(R)=Z+Cl(R)_dim2}
	Let $A$ be as in \eqref{eq:the_ring_A}.
	If $A$ has Krull dimension 2 then \eqref{eq:Groth=Z+Cl_general} holds, that is,
	\[ \Groth(A) \cong \Z \oplus \Cl(A). \]
\end{theorem}
\begin{proof}
	By Proposition~\ref{prop:[A/m]=0}, $[ A/\mfm ] = 0$ for all maximal ideals of $A$.
	Since $A$ is $2$-dimensional and equidimensional, all of its maximal ideals are height two prime ideals.
	Thus, by Proposition~\ref{prop:huneke_oGroth(R)_to_Cl(R)}, $\oGroth(A) \cong \Cl(A)$.
	Recall that $\Groth(A)$ decomposes as in \eqref{eq:Groth=Z+oGroth} via the rank map.
	Hence
	\begin{align*}
		 \Groth(A) &\cong \Z \oplus \oGroth(A) \\
		 &\cong \Z \oplus \Cl(A).\qedhere
	\end{align*}
\end{proof}
In future sections we will evidence other situations in which \eqref{eq:Groth=Z+Cl_general} holds.
In particular, we will prove that there exists an isomorphism \eqref{eq:Groth=Z+Cl_general} for certain local cDV singularities.

\section{K-theory of deformed preprojective algebras}\label{sec:Groth_deformed_preproj}

The techniques of \S\ref{sec:cl_grp_nagatas_thm} and \S\ref{sec:cl_grp_k_theory} work especially well for surfaces coming from deformed preprojective algebras.
The main result of this section gives an explicit isomorphism between the description given in \cite[Theorem~10.2]{cbh98} of the Grothendieck group of the centre $\cO^{\uplambda}$ of the deformed preprojective algebra of a type $\teA$ quiver and our computation of the Grothendieck group coming from Theorem~\ref{thm:G_0(R)=Z+Cl(R)_dim2}.

The results in \cite{cbh98} are for general extended Dynkin type.
Here, we restrict to extended Dynkin type $\teA$.
Recall from \S\ref{subsec:deformations_of_ksings} the rings $\cS^{\uplambda}$ and $\cO^{\uplambda}$ where $\uplambda = (\uplambda_{0}, \hdots, \uplambda_{\sf{n}})$.
The ring $\cO^{\uplambda}$ has Krull dimension 2 and, assuming $\uplambda \cdot \updelta = 0$, it is a commutative ring.

The (finite) root system of type $\tA$ is the root system of the Lie algebra $\mathfrak{sl}_{n+1}$.
The Cartan subalgebra of this Lie algebra can be realised as the vector space $\mathfrak{h}$ of diagonal matrices of trace zero.
Therefore, the dual space of $\mathfrak{h}$ is $\mathfrak{h}^* = \mathrm{Span}_{\C} \{ r_{0}, \hdots, r_{\sf{n}} \} / (r_{0} + \cdots + r_{\sf{n}})$.
Inside this complex vector space we have the weight lattice $L = \mathrm{Span}_{\Z} \{ r_{0}, \hdots, r_{\sf{n}} \} / (r_{0} + \cdots + r_{\sf{n}})$ and the root lattice $Q$ spanned by the set of simple roots $\Delta = \{ r_{0} - r_{1}, \hdots, r_{{\sf{n}}-1} - r_{\sf{n}}\}$.
The associated root system is $\Phi = \{ r_{i} - r_{j} \, | \, i \neq j \}$.
Let $\langle - , - \rangle$ denote the natural pairing between $\mathfrak{h}^*$ and $\mathfrak{h}$. 
Now, given a weight $\uplambda \in \mathfrak{h}$, we define $\Phi_{\uplambda} = \{ \upalpha \in \Phi \, | \, \langle \upalpha, \uplambda \rangle = 0 \}$ and $Q_{\uplambda} = \mathrm{Span}_{\Z} \, \Phi_{\uplambda}$ in $Q$.
Write $\upnu_{i} \coloneqq \uplambda_{0} + \cdots + \uplambda_{i}$ where $i = 0, \hdots, {\sf{n}}$.\footnote{This is the same as in Appendix~\ref{app:centre_deform_preproj}, just shifting the index by $-1$.}

In \cite[Lemma~10.1]{cbh98},
Crawley-Boevey--Holland show that for all $i \ge 0$,
\[ \textnormal{K}_{i}(\proj \cS^{\uplambda}) = \textnormal{K}_{i}(\cS^{\uplambda}) \cong \textnormal{K}_{i}(\C)^{\sf{Q_{0}}}. \]
They go on to give a full description of $\Kroth(\cO^{\uplambda})$ and $\Groth(\cO^{\uplambda})$.
Set $\upmu = (\uplambda_{1}, \hdots, \uplambda_{\sf{n}})$, which is the weight vector $\uplambda$ without the weight at vertex 0.
Quillen's localisation (Theorem~\ref{thm:quillen_localisation}) gives a short exact sequence
\begin{equation}\label{seq:quillen_localisation}
\Groth(\Pi^{\upmu}) \xrightarrow{\upvarphi} \Groth(\Pi^{\uplambda}) \to \Groth(\cO^{\uplambda}) \to 0.
\end{equation}
Let $t_{i}$ be the class of the finitely generated projective module $\Pi^{\uplambda} e_{i}$ in $\Groth(\Pi^{\uplambda})$.
By \cite[\S10]{cbh98}, since $\Pi^{\uplambda}$ has finite global dimension, it follows that $\Groth(\Pi^{\uplambda})$ is a free abelian group with basis $t_{i}$ where $i = 0, \hdots, \sf{n}$.
The simple $\Pi^{\upmu}$-modules are in bijection with the vertices $i$ of the quiver of $\Pi^{\upmu}$ such that $\uplambda_{i}=0$ \cite[Lemma~6.1]{cbh98}.
Write $I$ for the vertex set of these vertices and denote the corresponding simple modules by $S_{i}$.
These $S_{i}$ have a resolution over $\Pi^{\uplambda}$ given by
\begin{equation}\label{seq:resolution_of_S_i}
0 \to \Pi^{\uplambda}e_{i} \to \bigoplus_{\substack{a\colon j \to i\\ \text{in } \bar{Q}}} \Pi^{\uplambda}e_{j} \to \Pi^{\uplambda}e_{i} \to S_{i} \to 0.
\end{equation}
From this resolution it follows that the image of $S_{i}$ in $\Groth(\Pi^{\uplambda})$ is $2t_{i} - t_{i-1} - t_{i+1}$.
This leads to the following result.

\begin{theorem}\label{thm:K0(modR)=oplusZP/SNF}
	Let $\uplambda = (\uplambda_{0}, \hdots, \uplambda_{\sf{n}})$.
	Assume that $\uplambda \cdot \updelta = 0$ and that $\uplambda_{i} \ge 0$ for all vertices $i = 1, \hdots, {\sf{n}}$.
	Then
	\[
	\Groth(\cO^{\uplambda}) = \frac{\bigoplus_{i =0}^{\sf{n}} \Z t_{i}}{\left( 2t_{i} - t_{i-1} - t_{i+1} \mid \uplambda_{i} = 0 \right) }
	\]
	and the rank of $\Groth(\cO^{\uplambda})$ is ${\sf{n}}+ 1 - | I|$.
\end{theorem}
\begin{proof}
	This is precisely the result given in \cite[Theorem~10.2]{cbh98} when $\uplambda \cdot \upalpha = 0$ for some $\upalpha \in \Phi$.
	When $\uplambda \cdot \upalpha \neq 0$ for every $\upalpha \in \Phi$, $\cO^{\uplambda}$ is Morita equivalent to $\Pi^{\uplambda}$ by \cite[Corollary~6.9]{cbh98} and $I = \emptyset$.
	In this case, the left most term of \eqref{seq:quillen_localisation} is zero and the map $\Groth(\Pi^{\lambda}) \to \Groth(\cO^{\uplambda})$ is an isomorphism.
\end{proof}

In other words, $\Groth(\cO^{\uplambda}) \cong \Groth(\Pi^{\uplambda})/\Im(\upvarphi)$ where the map $\upvarphi$ is as in \eqref{seq:quillen_localisation}, which describes the image of each $\Pi^{\upmu}$-module $S_{i}$.
Using Theorem~\ref{thm:Cl(A)_global} we give a Lie-theoretic description of the class group of $\cO^{\uplambda}$.
\begin{lemma}\label{lem:Cl(O^lambda)_isomorphic_L/Q_lambda}
	Let $L$ and $Q_{\uplambda}$ be as above.
	The class group of $\cO^{\uplambda}$ is isomorphic to $L / Q_{\uplambda}$. 
\end{lemma}

\begin{proof}
	Combining Theorem~\ref{thm:Cl(A)_global} and Theorem~\ref{thm:app:preproj_type_A_isomorphic_ring_A}, 
	\[
		\Cl(\cO^{\uplambda}) \cong \Z^{\oplus t} / (a_{1}, \hdots, a_{t}) \cong \bigoplus_{i =1}^{t} \Z \cdot s_{i}/ (a_{1}s_{1} + \cdots + a_{t}s_{t}).
	\]
	Decompose $[0,{\sf{n}}] = \bigsqcup_{j = 1}^{t} I_{j}$, where $x,y \in I_{j}$ if and only if $\upnu_{x} = \upnu_{y}$.
	Then $|I_{j}| = a_{j}$.
	So we define the map $L / Q_{\uplambda} \to \Cl(\cO^{\uplambda})$ by $r_{i} \mapsto s_{j}$ if $i \in I_{j}$.
	Then the relation $r_{0} + \cdots + r_{\sf{n}} = 0$ becomes $a_{1}s_{1} + \cdots + a_{t}s_{t} = 0$ and the relations generated by $Q_{\uplambda}$ are just saying that $r_{i}$ and $r_{k}$ are mapped to the same thing if and only if $\upnu_{i} = \upnu_{k}$.
	Thus, it is an isomorphism. 
\end{proof}

Recall that $t_{i}$ denotes the class of $\Pi^{\uplambda}e_{i}$ in $\Groth(\Pi^{\uplambda})$.
If the minimal imaginary root is $\updelta = \sum_{i = 0}^{\sf{n}} \updelta_{i} t_{i}$ then, since $\bar{Q}$ has underlying type $\teA$, we have $\updelta_{i} = 1$ for all $i = 0, \hdots, {\sf{n}}$.
In this case, the rank map on $\Groth(\cO^{\uplambda})$ sends $t_{i}$ to $1$ (more generally, it sends a vector $w$ to $w \cdot \updelta$).
The kernel is spanned by all elements $t_{i} - t_{j}$.
Write $P$ for the $\Z$-lattice spanned by these elements, so that $P$ is the kernel of the rank map.
By Theorem~\ref{thm:K0(modR)=oplusZP/SNF},
\begin{align*}
\Groth(\cO^{\uplambda}) &\cong \frac{\bigoplus_{i = 0}^{\sf{n}} \Z t_{i}}{( 2t_{i} - t_{i-1} - t_{i+1} \mid \uplambda_{i} = 0)} \\
&\cong \Z \oplus \frac{P}{( 2t_{i} - t_{i-1} - t_{i+1} \mid \uplambda_{i} = 0)},
\end{align*}
where the copy of $\Z$ in the second line is the image of the rank map.
We define an isomorphism $L \to P$ by $r_{i} \mapsto t_{i} - t_{i-1}$.

\begin{lemma}\label{lem:Cl_(O^lambda)_extends}
	Assume that $\uplambda_{i} \ge 0$ for all $i = 1, \hdots, {\sf{n}}$.
	The isomorphism $L \to P$ extends to an isomorphism $L / Q_{\uplambda} \cong P / ( 2t_{i} - t_{i-1} - t_{i+1} \, | \, \uplambda_{i} = 0)$. 
\end{lemma}

\begin{proof}
	Note that if $\uplambda_{i} \ge 0$ for all $1 \le i \le {\sf{n}}$, then $Q_{\uplambda}$ is the lattice generated by all $r_{i} - r_{i-1}$ for all $i$ with $\upnu_{i} = \upnu_{i-1}$.
	This is the same as all $i$ such that $\uplambda_{i}= 0$.
	Moreover, the image of $r_{i} - r_{i-1}$ under the map $L \to P$ is $2 t_{i} - t_{i-1} - t_{i+1}$.	
	It follows that the image of $Q_{\lambda}$ under the isomorphism $L \to P$ is precisely the submodule generated by all $2t_{i} - t_{i-1} - t_{i+1}$ for $\uplambda_i = 0. $
\end{proof}

Combining Theorem~\ref{thm:K0(modR)=oplusZP/SNF}, Lemma~\ref{lem:Cl(O^lambda)_isomorphic_L/Q_lambda}, and Lemma~\ref{lem:Cl_(O^lambda)_extends} gives our main result in this section.
\begin{cor}
	Let $\bar{Q}$ be type $\teA$ and assume that $\uplambda_{i} \ge 0$ for all $i =1, \hdots, {\sf{n}}$.
	Then
	\[
	\Groth(\cO^{\uplambda}) \cong \Z \oplus \Cl(\cO^{\uplambda}),
	\]
	with $\Cl(\cO^{\uplambda}) \cong L/Q_{\uplambda} \cong P/( 2t_{i} - t_{i-1} - t_{i+1} \, | \, \uplambda_{i} = 0)$.
\end{cor}

\section{Limitations of techniques}\label{sec:limitations_of_tech}
Initially, the techniques used and described in this chapter appear to be quite promising.
It is possible to use Theorem~\ref{thm:nagata} to compute the divisor class group of the Kleinian singularity of type $\teDn{4}$.
Using this method, the class group of such an (undeformed) Kleinian singularity is $\Z/2\Z \oplus \Z/2\Z$.
Unfortunately, the method does not appear to work for $\teDn{n}$ where ${\sf{n}} > 4$.
The difficulty occurs in the first step of the process, trying to prove that an appropriate localisation of the ring is a UFD.

Not only does the use of Theorem~\ref{thm:nagata} become more complicated when moving past type $\tA$ surfaces, but as we will see in Chapter~\ref{ch:cDVs}, computing class groups and Grothendieck groups becomes far more difficult in higher dimensions, where there is a significant increase in the amount of information one needs to track.
These challenges aside, we still lift the results of this chapter to various rings in dimension $3$.

\chapter{Knitting}\label{ch:knitting}
In this chapter, we compute the quiver of $\End_{R}( \bigoplus_{i \in I} V_{i})$ when $R$ is a Kleinian singularity and $\{V_{i}\}_{i \in I}$ is a subset of the $\CM$ $R$-modules.
The main result, Theorem~\ref{thm:number_arrows_0_to_i_equals_reverse}, asserts that the quiver is always symmetric.
This fact will be a crucial ingredient in computing the K-theory of cDV singularities in \S\ref{sec:isolated_cdv_with_NCCR}.

We use AR-theory and an application of the knitting algorithm to compute these quivers.
In \S\ref{subsec:AR_theory} we introduced AR-sequences and AR-quivers, remarking that, in dimension two, the AR-quiver of $\ADE$ surface singularities coincides with the McKay quiver.
Knitting on this AR-quiver allows us, in Theorem~\ref{thm:number_arrows_0_1_2}, to give an exact description of the possible number of arrows between vertices in the quiver of $\End_{R}( \bigoplus_{i \in I} V_{i})$.

The AR-quivers studied in this chapter are described in Figure~\ref{fig:repetition_quivers}.
\begin{landscape}
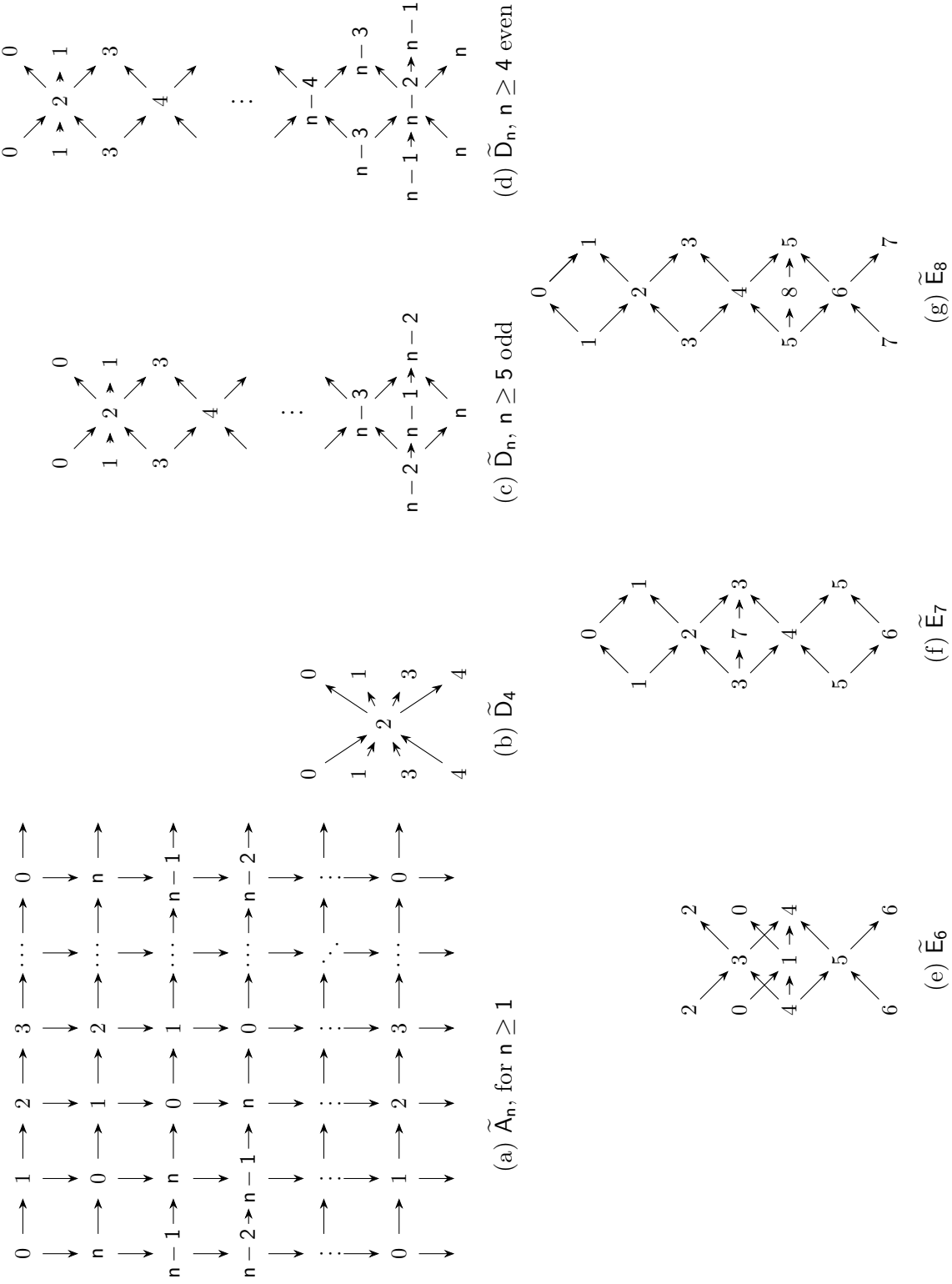
\begin{figure}[b]
	\centering
	\begin{subfigure}[b]{.4\textwidth}
		\centering
		\begin{tikzpicture}[->, shorten >=9pt, shorten <=9pt, scale=0.8]
			\node at ( 0,0) [] {\footnotesize{$0$}};
			\node at ( 0,-1.5) [] {\footnotesize{$\sf{n}$}};
			\node at ( 0,-3) [] {\footnotesize{$\sf{n}-1$}};
			\node at ( 0,-4.5) [] {\footnotesize{$\sf{n}-2$}};
			\node at ( 0,-6) [] {\footnotesize{$\vdots$}};
			\node at ( 0,-7.5) [] {\footnotesize{$0$}};
			
			\node at ( 1.5,0) [] {\footnotesize{$1$}};
			\node at ( 1.5,-1.5) [] {\footnotesize{$0$}};
			\node at ( 1.5,-3) [] {\footnotesize{$\sf{n}$}};
			\node at ( 1.5,-4.5) [] {\footnotesize{$\sf{n}-1$}};
			\node at ( 1.5,-6) [] {\footnotesize{$\vdots$}};
			\node at ( 1.5,-7.5) [] {\footnotesize{$1$}};
			
			\node at ( 3,0) [] {\footnotesize{$2$}};
			\node at ( 3,-1.5) [] {\footnotesize{$1$}};
			\node at ( 3,-3) [] {\footnotesize{$0$}};
			\node at ( 3,-4.5) [] {\footnotesize{$\sf{n}$}};
			\node at ( 3,-6) [] {\footnotesize{$\vdots$}};
			\node at ( 3,-7.5) [] {\footnotesize{$2$}};
			
			\node at ( 4.5,0) [] {\footnotesize{$3$}};
			\node at ( 4.5,-1.5) [] {\footnotesize{$2$}};
			\node at ( 4.5,-3) [] {\footnotesize{$1$}};
			\node at ( 4.5,-4.5) [] {\footnotesize{$0$}};
			\node at ( 4.5,-6) [] {\footnotesize{$\vdots$}};
			\node at ( 4.5,-7.5) [] {\footnotesize{$3$}};
			
			\node at ( 6,0) [] {\footnotesize{$\cdots$}};
			\node at ( 6,-1.5) [] {\footnotesize{$\cdots$}};
			\node at ( 6,-3) [] {\footnotesize{$\cdots$}};
			\node at ( 6,-4.5) [] {\footnotesize{$\cdots$}};
			\node at ( 6,-6) [] {\footnotesize{$\ddots$}};
			\node at ( 6,-7.5) [] {\footnotesize{$\cdots$}};
			
			\node at ( 7.5,0) [] {\footnotesize{$0$}};
			\node at ( 7.5,-1.5) [] {\footnotesize{$\sf{n}$}};
			\node at ( 7.5,-3) [] {\footnotesize{$\sf{n}-1$}};
			\node at ( 7.5,-4.5) [] {\footnotesize{$\sf{n}-2$}};
			\node at ( 7.5,-6) [] {\footnotesize{$\vdots$}};
			\node at ( 7.5,-7.5) [] {\footnotesize{$0$}};
			\draw[->, -{Stealth[]}, color=black] ( 0,0) -- ( 0,-1.5);
			\draw[->, -{Stealth[]}, color=black] ( 0,-1.5) -- ( 0,-3);
			\draw[->, -{Stealth[]}, color=black] ( 0,-3) -- ( 0,-4.5);
			\draw[->, -{Stealth[]}, color=black] ( 0,-4.5) -- ( 0,-6);
			\draw[->, -{Stealth[]}, color=black] ( 0,-6) -- ( 0,-7.5);
			\draw[->, -{Stealth[]}, color=black] ( 0,-7.5) -- ( 0,-9);
			\draw[->, -{Stealth[]}, color=black] ( 1.5,0) -- ( 1.5,-1.5);
			\draw[->, -{Stealth[]}, color=black] ( 1.5,-1.5) -- ( 1.5,-3);
			\draw[->, -{Stealth[]}, color=black] ( 1.5,-3) -- ( 1.5,-4.5);
			\draw[->, -{Stealth[]}, color=black] ( 1.5,-4.5) -- ( 1.5,-6);
			\draw[->, -{Stealth[]}, color=black] ( 1.5,-6) -- ( 1.5,-7.5);
			\draw[->, -{Stealth[]}, color=black] ( 1.5,-7.5) -- ( 1.5,-9);
			\draw[->, -{Stealth[]}, color=black] ( 3,0) -- ( 3,-1.5);
			\draw[->, -{Stealth[]}, color=black] ( 3,-1.5) -- ( 3,-3);
			\draw[->, -{Stealth[]}, color=black] ( 3,-3) -- ( 3,-4.5);
			\draw[->, -{Stealth[]}, color=black] ( 3,-4.5) -- ( 3,-6);
			\draw[->, -{Stealth[]}, color=black] ( 3,-6) -- ( 3,-7.5);
			\draw[->, -{Stealth[]}, color=black] ( 3,-7.5) -- ( 3,-9);
			\draw[->, -{Stealth[]}, color=black] ( 4.5,0) -- ( 4.5,-1.5);
			\draw[->, -{Stealth[]}, color=black] ( 4.5,-1.5) -- ( 4.5,-3);
			\draw[->, -{Stealth[]}, color=black] ( 4.5,-3) -- ( 4.5,-4.5);
			\draw[->, -{Stealth[]}, color=black] ( 4.5,-4.5) -- ( 4.5,-6);
			\draw[->, -{Stealth[]}, color=black] ( 4.5,-6) -- ( 4.5,-7.5);
			\draw[->, -{Stealth[]}, color=black] ( 4.5,-7.5) -- ( 4.5,-9);
			\draw[->, -{Stealth[]}, color=black] ( 6,0) -- ( 6,-1.5);
			\draw[->, -{Stealth[]}, color=black] ( 6,-1.5) -- ( 6,-3);
			\draw[->, -{Stealth[]}, color=black] ( 6,-3) -- ( 6,-4.5);
			\draw[->, -{Stealth[]}, color=black] ( 6,-4.5) -- ( 6,-6);
			\draw[->, -{Stealth[]}, color=black] ( 6,-6) -- ( 6,-7.5);
			\draw[->, -{Stealth[]}, color=black] ( 6,-7.5) -- ( 6,-9);
			\draw[->, -{Stealth[]}, color=black] ( 7.5,0) -- ( 7.5,-1.5);
			\draw[->, -{Stealth[]}, color=black] ( 7.5,-1.5) -- ( 7.5,-3);
			\draw[->, -{Stealth[]}, color=black] ( 7.5,-3) -- ( 7.5,-4.5);
			\draw[->, -{Stealth[]}, color=black] ( 7.5,-4.5) -- ( 7.5,-6);
			\draw[->, -{Stealth[]}, color=black] ( 7.5,-6) -- ( 7.5,-7.5);
			\draw[->, -{Stealth[]}, color=black] ( 7.5,-7.5) -- ( 7.5,-9);
			\draw[->, -{Stealth[]}, color=black] ( 0,0) -- ( 1.5,0);
			\draw[->, -{Stealth[]}, color=black] ( 1.5,0) -- ( 3,0);
			\draw[->, -{Stealth[]}, color=black] ( 3,0) -- ( 4.5,0);
			\draw[->, -{Stealth[]}, color=black] ( 4.5,0) -- ( 6,0);
			\draw[->, -{Stealth[]}, color=black] ( 6,0) -- ( 7.5,0);
			\draw[->, -{Stealth[]}, color=black] ( 7.5,0) -- ( 9,0);
			\draw[->, -{Stealth[]}, color=black] ( 0,-1.5) -- ( 1.5,-1.5);
			\draw[->, -{Stealth[]}, color=black] ( 1.5,-1.5) -- ( 3,-1.5);
			\draw[->, -{Stealth[]}, color=black] ( 3,-1.5) -- ( 4.5,-1.5);
			\draw[->, -{Stealth[]}, color=black] ( 4.5,-1.5) -- ( 6,-1.5);
			\draw[->, -{Stealth[]}, color=black] ( 6,-1.5) -- ( 7.5,-1.5);
			\draw[->, -{Stealth[]}, color=black] ( 7.5,-1.5) -- ( 9,-1.5);
			\draw[->, -{Stealth[]}, color=black] ( 0.2,-3) -- ( 1.5,-3);
			\draw[->, -{Stealth[]}, color=black] ( 1.5,-3) -- ( 3,-3);
			\draw[->, -{Stealth[]}, color=black] ( 3,-3) -- ( 4.5,-3);
			\draw[->, -{Stealth[]}, color=black] ( 4.5,-3) -- ( 6,-3);
			\draw[->, -{Stealth[]}, color=black] ( 6,-3) -- ( 7.3,-3);
			\draw[->, -{Stealth[]}, color=black] ( 7.7,-3) -- ( 9,-3);
			\draw[->, -{Stealth[]}, color=black] ( 0.2,-4.5) -- ( 1.3,-4.5);
			\draw[->, -{Stealth[]}, color=black] ( 1.7,-4.5) -- ( 3,-4.5);
			\draw[->, -{Stealth[]}, color=black] ( 3,-4.5) -- ( 4.5,-4.5);
			\draw[->, -{Stealth[]}, color=black] ( 4.5,-4.5) -- ( 6,-4.5);
			\draw[->, -{Stealth[]}, color=black] ( 6,-4.5) -- ( 7.3,-4.5);
			\draw[->, -{Stealth[]}, color=black] ( 7.7,-4.5) -- ( 9,-4.5);
			\draw[->, -{Stealth[]}, color=black] ( 0,-6) -- ( 1.5,-6);
			\draw[->, -{Stealth[]}, color=black] ( 1.5,-6) -- ( 3,-6);
			\draw[->, -{Stealth[]}, color=black] ( 3,-6) -- ( 4.5,-6);
			\draw[->, -{Stealth[]}, color=black] ( 4.5,-6) -- ( 6,-6);
			\draw[->, -{Stealth[]}, color=black] ( 6,-6) -- ( 7.5,-6);
			\draw[->, -{Stealth[]}, color=black] ( 7.5,-6) -- ( 9,-6);
			\draw[->, -{Stealth[]}, color=black] ( 0,-7.5) -- ( 1.5,-7.5);
			\draw[->, -{Stealth[]}, color=black] ( 1.5,-7.5) -- ( 3,-7.5);
			\draw[->, -{Stealth[]}, color=black] ( 3,-7.5) -- ( 4.5,-7.5);
			\draw[->, -{Stealth[]}, color=black] ( 4.5,-7.5) -- ( 6,-7.5);
			\draw[->, -{Stealth[]}, color=black] ( 6,-7.5) -- ( 7.5,-7.5);
			\draw[->, -{Stealth[]}, color=black] ( 7.5,-7.5) -- ( 9,-7.5);
		\end{tikzpicture}\subcaption{$\teAn{n}$, for $\sf{n} \ge 1$}\label{fig:repetition_quivers_A_n}
	\end{subfigure}
	\begin{subfigure}[b]{.3\textwidth}
		\centering
		\begin{tikzpicture}[->,shorten >=9pt, shorten <=9pt, scale=0.8]
			\node at ( 0,0) [] {\footnotesize{$0$}};
			\node at ( 0,-1) [] {\footnotesize{$1$}};
			\node at ( 0,-2) [] {\footnotesize{$3$}};
			\node at ( 0,-3) [] {\footnotesize{$4$}};
			\node at ( 1,-1.5) [] {\footnotesize{$2$}};
				\draw[->, -{Stealth[]}, color=black] ( 0,0) -- ( 1,-1.5);
				\draw[->, -{Stealth[]}, color=black] ( 0,-1) -- ( 1,-1.5);
				\draw[->, -{Stealth[]}, color=black] ( 0,-2) -- ( 1,-1.5);
				\draw[->, -{Stealth[]}, color=black] ( 0,-3) -- ( 1,-1.5);
			\node at ( 2,0) [] {\footnotesize{$0$}};
			\node at ( 2,-1) [] {\footnotesize{$1$}};
			\node at ( 2,-2) [] {\footnotesize{$3$}};
			\node at ( 2,-3) [] {\footnotesize{$4$}};
				\draw[->, -{Stealth[]}, color=black] ( 1,-1.5) -- ( 2,0);
				\draw[->, -{Stealth[]}, color=black] ( 1,-1.5) -- ( 2,-1);
				\draw[->, -{Stealth[]}, color=black] ( 1,-1.5) -- ( 2,-2);
				\draw[->, -{Stealth[]}, color=black] ( 1,-1.5) -- ( 2,-3);
		\end{tikzpicture}\subcaption{$\teDn{4}$}\label{fig:repetition_quivers_D_4}
	\end{subfigure}
	\centering
	\begin{subfigure}[b]{.3\textwidth}
		\centering
		\begin{tikzpicture}[->,shorten >=9pt, shorten <=9pt, scale=0.8]
			\node at ( 0,0) [] {\footnotesize{$0$}};
			\node at ( 0,-1) [] {\footnotesize{$1$}};
			\node at ( 1,-1) [] {\footnotesize{$2$}};
			\node at ( 0,-2) [] {\footnotesize{$3$}};
			\node at ( 1,-3) [] {\footnotesize{$4$}};
			\node at ( 0,-4) {};
			\node at ( 0,-5) {};
			\node at ( 1,-6) [] {\footnotesize{$\sf{n}-3$}};
			\node at ( -0.5,-7) [] {\footnotesize{$\sf{n}-2$}};
			\node at ( 1,-7) [] {\footnotesize{$\sf{n}-1$}};
			\node at ( 1,-8) [] {\footnotesize{$\sf{n}$}};
				\draw[->, -{Stealth[]}, color=black] ( 0,0) -- ( 1,-1);
				\draw[->, -{Stealth[]}, color=black] ( 0,-1) -- ( 1,-1);
				\draw[->, -{Stealth[]}, color=black] ( 0,-2) -- ( 1,-1);
				\draw[->, -{Stealth[]}, color=black] ( 0,-2) -- ( 1,-3);
				\draw[->, -{Stealth[]}, color=black] ( 0,-4) -- ( 1,-3);
				\draw[->, -{Stealth[]}, color=black] ( 0,-5) -- ( 1,-6);
				\draw[->, -{Stealth[]}, color=black] ( 0,-7) -- ( 1,-6);
				\draw[->, -{Stealth[]}, color=black] ( -0.3,-7) -- ( 0.8,-7);
				\draw[->, -{Stealth[]}, color=black] ( 0,-7) -- ( 1,-8);
			\node at ( 2,0) [] {\footnotesize{$0$}};
			\node at ( 2,-1) [] {\footnotesize{$1$}};
			\node at ( 2,-2) [] {\footnotesize{$3$}};
			\node at ( 1,-4.5) {\vdots};
			\node at ( 2,-4) {};
			\node at ( 2,-5) {};
			\node at ( 2.5,-7) [] {\footnotesize{$\sf{n}-2$}};
				\draw[->, -{Stealth[]}] ( 1,-1) -- ( 2,0);
				\draw[->, -{Stealth[]}] ( 1,-1) -- ( 2,-1);
				\draw[->, -{Stealth[]}] ( 1,-1) -- ( 2,-2);
				\draw[->, -{Stealth[]}] ( 1,-3) -- ( 2,-2);
				\draw[->, -{Stealth[]}] ( 1,-3) -- ( 2,-4);
				\draw[->, -{Stealth[]}] ( 1,-6) -- ( 2,-5);
				\draw[->, -{Stealth[]}] ( 1,-6) -- ( 2,-7);
				\draw[->, -{Stealth[]}] ( 1.2,-7) -- ( 2.3,-7);
				\draw[->, -{Stealth[]}] ( 1,-8) -- ( 2,-7);
		\end{tikzpicture}\subcaption{$\teDn{n}$, $\sf{n} \ge 5$ odd}\label{fig:repetition_quivers_D_n_odd}
	\end{subfigure}
	\begin{subfigure}[b]{.3\textwidth}
		\centering
		\begin{tikzpicture}[->,shorten >=9pt, shorten <=9pt, scale=0.8]
			\node at ( 0,0) [] {\footnotesize{$0$}};
			\node at ( 0,-1) [] {\footnotesize{$1$}};
			\node at ( 1,-1) [] {\footnotesize{$2$}};
			\node at ( 0,-2) [] {\footnotesize{$3$}};
			\node at ( 1,-3) [] {\footnotesize{$4$}};
			\node at ( 0,-4) {};
			\node at ( 0,-5) {};
			\node at ( 1,-6) [] {\footnotesize{$\sf{n}-4$}};
			\node at ( 0,-7) [] {\footnotesize{$\sf{n}-3$}};
			\node at ( -0.5,-8) [] {\footnotesize{$\sf{n}-1$}};
			\node at ( 1,-8) [] {\footnotesize{$\sf{n}-2$}};
			\node at ( 0,-9) [] {\footnotesize{$\sf{n}$}};
			\draw[->, -{Stealth[]}, color=black] ( 0,0) -- ( 1,-1);
			\draw[->, -{Stealth[]}, color=black] ( 0,-1) -- ( 1,-1);
			\draw[->, -{Stealth[]}, color=black] ( 0,-2) -- ( 1,-1);
			\draw[->, -{Stealth[]}, color=black] ( 0,-2) -- ( 1,-3);
			\draw[->, -{Stealth[]}, color=black] ( 0,-4) -- ( 1,-3);
			\draw[->, -{Stealth[]}, color=black] ( 0,-5) -- ( 1,-6);
			\draw[->, -{Stealth[]}, color=black] ( 0,-7) -- ( 1,-6);
			\draw[->, -{Stealth[]}, color=black] ( 0,-7) -- ( 1,-8);
			\draw[->, -{Stealth[]}, color=black] ( -0.3,-8) -- ( 0.8,-8);
			\draw[->, -{Stealth[]}, color=black] ( 0,-9) -- ( 1,-8);
			\draw[->, -{Stealth[]}, color=black] ( 1,-1) -- ( 2,0);
			\draw[->, -{Stealth[]}, color=black] ( 1,-1) -- ( 2,-1);
			\draw[->, -{Stealth[]}, color=black] ( 1,-1) -- ( 2,-2);
			\draw[->, -{Stealth[]}, color=black] ( 1,-3) -- ( 2,-2);
			\draw[->, -{Stealth[]}, color=black] ( 1,-3) -- ( 2,-4);
			\draw[->, -{Stealth[]}, color=black] ( 1,-6) -- ( 2,-5);
			\draw[->, -{Stealth[]}, color=black] ( 1,-6) -- ( 2,-7);
			\draw[->, -{Stealth[]}, color=black] ( 1,-8) -- ( 2,-7);
			\draw[->, -{Stealth[]}, color=black] ( 1.2,-8) -- ( 2.3,-8);
			\draw[->, -{Stealth[]}, color=black] ( 1,-8) -- ( 2,-9);
			\node at ( 2,0) [] {\footnotesize{$0$}};
			\node at ( 2,-1) [] {\footnotesize{$1$}};
			\node at ( 2,-2) [] {\footnotesize{$3$}};
			\node at ( 2,-4) {};
			\node at ( 1,-4.5) {\vdots};
			\node at ( 2,-5) {};
			\node at ( 2,-7) [] {\footnotesize{$\sf{n}-3$}};
			\node at ( 2.5,-8) [] {\footnotesize{$\sf{n}-1$}};
			\node at ( 2,-9) [] {\footnotesize{$\sf{n}$}};
		\end{tikzpicture}\subcaption{$\teDn{n}$, $\sf{n} \ge 4$ even}\label{fig:repetition_quivers_D_n_even}
	\end{subfigure}
	\begin{subfigure}[b]{.3\textwidth}
		\centering
		\begin{tikzpicture}[->,shorten >=7pt, shorten <=7pt, scale=0.8]
			\node at ( 0,1) [] {\footnotesize{$2$}};
			\node at ( 0,0) [] {\footnotesize{$0$}};
			\node at ( 1,0) [] {\footnotesize{$3$}};
			\node at ( 0,-1) [] {\footnotesize{$4$}};
			\node at ( 1,-1) [] {\footnotesize{$1$}};
			\node at ( 1,-2) [] {\footnotesize{$5$}};
			\node at ( 0,-3) [] {\footnotesize{$6$}};
			\draw[->, -{Stealth[]}, color=black] ( 0,1) -- ( 1,0);
			\draw[->, -{Stealth[]}, color=black] ( 0,0) -- ( 1,-1);
			\draw[->, -{Stealth[]}, color=black] ( 0,-1) -- ( 1,0);
			\draw[->, -{Stealth[]}, color=black] ( 0,-1) -- ( 1,-1);
			\draw[->, -{Stealth[]}, color=black] ( 0,-1) -- ( 1,-2);
			\draw[->, -{Stealth[]}, color=black] ( 0,-3) -- ( 1,-2);
			\draw[->, -{Stealth[]}, color=black] ( 1,0) -- ( 2,1);
			\draw[->, -{Stealth[]}, color=black] ( 1,0) -- ( 2,-1);
			\draw[->, -{Stealth[]}, color=black] ( 1,-1) -- ( 2,0);
			\draw[->, -{Stealth[]}, color=black] ( 1,-1) -- ( 2,-1);
			\draw[->, -{Stealth[]}, color=black] ( 1,-2) -- ( 2,-1);
			\draw[->, -{Stealth[]}, color=black] ( 1,-2) -- ( 2,-3);
			\node at ( 2,1) [] {\footnotesize{$2$}};
			\node at ( 2,0) [] {\footnotesize{$0$}};
			\node at ( 2,-1) [] {\footnotesize{$4$}};
			\node at ( 2,-3) [] {\footnotesize{$6$}};
		\end{tikzpicture}\subcaption{$\teEn{6}$}\label{fig:repetition_quivers_E_6}
	\end{subfigure}
	\begin{subfigure}[b]{.33\textwidth}
		\centering
		\begin{tikzpicture}[->,shorten >=7pt, shorten <=7pt, scale=0.8]
			\node at ( 1,0) [] {\footnotesize{$0$}};
			\node at ( 0,-1) [] {\footnotesize{$1$}};
			\node at ( 1,-2) [] {\footnotesize{$2$}};
			\node at ( 1,-3) [] {\footnotesize{$7$}};
			\node at ( 0,-3) [] {\footnotesize{$3$}};
			\node at ( 1,-4) [] {\footnotesize{$4$}};
			\node at ( 0,-5) [] {\footnotesize{$5$}};
			\node at ( 1,-6) [] {\footnotesize{$6$}};
				\draw[->, -{Stealth[]}, color=black] ( 0,-1) -- ( 1,0);
				\draw[->, -{Stealth[]}, color=black] ( 0,-1) -- ( 1,-2);
				\draw[->, -{Stealth[]}, color=black] ( 0,-3) -- ( 1,-2);
				\draw[->, -{Stealth[]}, color=black] ( 0,-3) -- ( 1,-3);
				\draw[->, -{Stealth[]}, color=black] ( 0,-3) -- ( 1,-4);
				\draw[->, -{Stealth[]}, color=black] ( 0,-5) -- ( 1,-4);
				\draw[->, -{Stealth[]}, color=black] ( 0,-5) -- ( 1,-6);
				\draw[->, -{Stealth[]}, color=black] ( 1,0) -- ( 2,-1);
				\draw[->, -{Stealth[]}, color=black] ( 1,-2) -- ( 2,-1);
				\draw[->, -{Stealth[]}, color=black] ( 1,-2) -- ( 2,-3);
				\draw[->, -{Stealth[]}, color=black] ( 1,-3) -- ( 2,-3);
				\draw[->, -{Stealth[]}, color=black] ( 1,-4) -- ( 2,-3);
				\draw[->, -{Stealth[]}, color=black] ( 1,-4) -- ( 2,-5);
				\draw[->, -{Stealth[]}, color=black] ( 1,-6) -- ( 2,-5);
			\node at ( 2,-1) [] {\footnotesize{$1$}};
			\node at ( 2,-3) [] {\footnotesize{$3$}};
			\node at ( 2,-5) [] {\footnotesize{$5$}};
		\end{tikzpicture}\subcaption{$\teEn{7}$}\label{fig:repetition_quivers_E_7}
	\end{subfigure}
	\begin{subfigure}[b]{.33\textwidth}
		\centering
		\begin{tikzpicture}[->,shorten >=7pt, shorten <=7pt, scale=0.8]
			\node at ( 1,0) [] {\footnotesize{$0$}};
			\node at ( 0,-1) [] {\footnotesize{$1$}};
			\node at ( 1,-2) [] {\footnotesize{$2$}};
			\node at ( 0,-3) [] {\footnotesize{$3$}};
			\node at ( 1,-4) [] {\footnotesize{$4$}};
			\node at ( 1,-5) [] {\footnotesize{$8$}};
			\node at ( 0,-5) [] {\footnotesize{$5$}};
			\node at ( 1,-6) [] {\footnotesize{$6$}};
			\node at ( 0,-7) [] {\footnotesize{$7$}};
				\draw[->, -{Stealth[]}, color=black] ( 0,-1) -- ( 1,0);
				\draw[->, -{Stealth[]}, color=black] ( 0,-1) -- ( 1,-2);
				\draw[->, -{Stealth[]}, color=black] ( 0,-3) -- ( 1,-2);
				\draw[->, -{Stealth[]}, color=black] ( 0,-3) -- ( 1,-4);
				\draw[->, -{Stealth[]}, color=black] ( 0,-5) -- ( 1,-4);
				\draw[->, -{Stealth[]}, color=black] ( 0,-5) -- ( 1,-5);
				\draw[->, -{Stealth[]}, color=black] ( 0,-5) -- ( 1,-6);
				\draw[->, -{Stealth[]}, color=black] ( 0,-7) -- ( 1,-6);
				\draw[->, -{Stealth[]}, color=black] ( 1,0) -- ( 2,-1);
				\draw[->, -{Stealth[]}, color=black] ( 1,0) -- ( 2,-1);
				\draw[->, -{Stealth[]}, color=black] ( 1,-2) -- ( 2,-1);
				\draw[->, -{Stealth[]}, color=black] ( 1,-2) -- ( 2,-3);
				\draw[->, -{Stealth[]}, color=black] ( 1,-4) -- ( 2,-3);
				\draw[->, -{Stealth[]}, color=black] ( 1,-4) -- ( 2,-5);
				\draw[->, -{Stealth[]}, color=black] ( 1,-5) -- ( 2,-5);
				\draw[->, -{Stealth[]}, color=black] ( 1,-6) -- ( 2,-5);
				\draw[->, -{Stealth[]}, color=black] ( 1,-6) -- ( 2,-7);
			\node at ( 2,-1) [] {\footnotesize{$1$}};
			\node at ( 2,-3) [] {\footnotesize{$3$}};
			\node at ( 2,-5) [] {\footnotesize{$5$}};
			\node at ( 2,-7) [] {\footnotesize{$7$}};
		\end{tikzpicture}\subcaption{$\teEn{8}$}\label{fig:repetition_quivers_E_8}
	\end{subfigure}
	\caption{AR-quivers with underlying extended $\ADE$ Dynkin diagram where, in each case, the left and right hand sides are identified and in type $\tA$ all vertices labelled $i$ are identified.}\label{fig:repetition_quivers}
\end{figure}
\end{landscape}

\section{Knitting algorithm}\label{sec:knitting_alg}
Let $( \Delta_{\aff} )_{0}$ be the set of vertices of the extended $\ADE$ Dynkin graph $\Delta_{\aff}$, and $I \subseteq ( \Delta_{\aff})_{0}$ a subset which contains the extended vertex $0$.
Set $\GammaI{I} \colonequals e_{I} \Pi e_{I} \cong \End_{R}(\bigoplus_{i \in I} V_{i})$ where $\Pi$ is the preprojective algebra of $\Delta_{\aff}$, $e_{I} = \sum_{i \in I} e_{i}$, and $R$ is a Kleinian singularity.
Then $\GammaI{I}$ can be written as a quiver with relations, where the vertices are in bijection with $I$ and the number of arrows between vertices in the quiver of $\GammaI{I}$ is determined using the following knitting algorithm (see, e.g. \cite[Corollary~3.3]{We11}, \cite{IyWe08SpecialClassificationCM}).
\begin{knitting}\label{knitting_algorithm}
	Let $j \in I$ and consider the AR-quiver with underlying graph $\Delta_{\aff}$.
	\begin{enumerate}
		\item\label{knitting_algorithm step_1} Draw the translation quiver and circle each vertex in $I$.
		For example, for $\teDn{7}$ with $I = \{ 0, 2, 5 \}$ the translation quiver with circled vertices is as follows.
		\vspace{-.3cm}\begin{figure}[H]
			\centering
			\begin{tikzpicture}[->,shorten >=7pt, shorten <=7pt, scale=0.9]
				\draw[-,densely dotted] (0,0.75) -- (0,-5.5);
				\draw[-,densely dotted] (2,0.75) -- (2,-5.5);
				\node at ( 0,0) [circle, draw] {\tiny{$\bullet$}};
				\node at ( 0,-1) [circle, fill= black, inner sep=1pt] {};
				\node at ( 1,-1) [circle, draw] {\tiny{$\bullet$}};
				\node at ( 0,-2) [circle, fill=black, inner sep=1pt] {};
				\node at ( 1,-3) [circle, fill=black, inner sep=1pt] {};
				\node at ( 0,-4) [circle, draw] {\tiny{$\bullet$}};
				\node at ( 1,-4) [circle, fill=black, inner sep=1pt] {};
				\node at ( 1,-5) [circle, fill=black, inner sep=1pt] {};
				\draw[->, -{Stealth[]}, color=black] ( 0,0) -- ( 1,-1);
				\draw[->, -{Stealth[]}, color=black] ( 0,-1) -- ( 1,-1);
				\draw[->, -{Stealth[]}, color=black] ( 0,-2) -- ( 1,-1);
				\draw[->, -{Stealth[]}, color=black] ( 0,-2) -- ( 1,-3);
				\draw[->, -{Stealth[]}, color=black] ( 0,-4) -- ( 1,-3);
				\draw[->, -{Stealth[]}, color=black] ( 0,-4) -- ( 1,-4);
				\draw[->, -{Stealth[]}, color=black] ( 0,-4) -- ( 1,-5);
				\draw[->, -{Stealth[]}, color=black] ( 1,-1) -- ( 2,0);
				\draw[->, -{Stealth[]}, color=black] ( 1,-1) -- ( 2,-1);
				\draw[->, -{Stealth[]}, color=black] ( 1,-1) -- ( 2,-2);
				\draw[->, -{Stealth[]}, color=black] ( 1,-3) -- ( 2,-2);
				\draw[->, -{Stealth[]}, color=black] ( 1,-3) -- ( 2,-4);
				\draw[->, -{Stealth[]}, color=black] ( 1,-4) -- ( 2,-4);
				\draw[->, -{Stealth[]}, color=black] ( 1,-5) -- ( 2,-4);
				\draw[-,densely dotted] (4,0.75) -- (4,-5.5);
				\node at ( 2,0) [circle, draw] {\tiny{$\bullet$}};
				\node at ( 2,-1) [circle, fill= black, inner sep=1pt] {};
				\node at ( 3,-1) [circle, draw] {\tiny{$\bullet$}};
				\node at ( 2,-2) [circle, fill=black, inner sep=1pt] {};
				\node at ( 3,-3) [circle, fill=black, inner sep=1pt] {};
				\node at ( 2,-4) [circle, draw] {\tiny{$\bullet$}};
				\node at ( 3,-4) [circle, fill=black, inner sep=1pt] {};
				\node at ( 3,-5) [circle, fill=black, inner sep=1pt] {};
				\draw[->, -{Stealth[]}, color=black] ( 2,0) -- ( 3,-1);
				\draw[->, -{Stealth[]}, color=black] ( 2,-1) -- ( 3,-1);
				\draw[->, -{Stealth[]}, color=black] ( 2,-2) -- ( 3,-1);
				\draw[->, -{Stealth[]}, color=black] ( 2,-2) -- ( 3,-3);
				\draw[->, -{Stealth[]}, color=black] ( 2,-4) -- ( 3,-3);
				\draw[->, -{Stealth[]}, color=black] ( 2,-4) -- ( 3,-4);
				\draw[->, -{Stealth[]}, color=black] ( 2,-4) -- ( 3,-5);
				\draw[->, -{Stealth[]}, color=black] ( 3,-1) -- ( 4,0);
				\draw[->, -{Stealth[]}, color=black] ( 3,-1) -- ( 4,-1);
				\draw[->, -{Stealth[]}, color=black] ( 3,-1) -- ( 4,-2);
				\draw[->, -{Stealth[]}, color=black] ( 3,-3) -- ( 4,-2);
				\draw[->, -{Stealth[]}, color=black] ( 3,-3) -- ( 4,-4);
				\draw[->, -{Stealth[]}, color=black] ( 3,-4) -- ( 4,-4);
				\draw[->, -{Stealth[]}, color=black] ( 3,-5) -- ( 4,-4);
				\draw[-,densely dotted] (6,0.75) -- (6,-5.5);
				\node at ( 4,0) [circle, draw] {\tiny{$\bullet$}};
				\node at ( 4,-1) [circle, fill= black, inner sep=1pt] {};
				\node at ( 5,-1) [circle, draw] {\tiny{$\bullet$}};
				\node at ( 4,-2) [circle, fill=black, inner sep=1pt] {};
				\node at ( 5,-3) [circle, fill=black, inner sep=1pt] {};
				\node at ( 4,-4) [circle, draw] {\tiny{$\bullet$}};
				\node at ( 5,-4) [circle, fill=black, inner sep=1pt] {};
				\node at ( 5,-5) [circle, fill=black, inner sep=1pt] {};
				\draw[->, -{Stealth[]}, color=black] ( 4,0) -- ( 5,-1);
				\draw[->, -{Stealth[]}, color=black] ( 4,-1) -- ( 5,-1);
				\draw[->, -{Stealth[]}, color=black] ( 4,-2) -- ( 5,-1);
				\draw[->, -{Stealth[]}, color=black] ( 4,-2) -- ( 5,-3);
				\draw[->, -{Stealth[]}, color=black] ( 4,-4) -- ( 5,-3);
				\draw[->, -{Stealth[]}, color=black] ( 4,-4) -- ( 5,-4);
				\draw[->, -{Stealth[]}, color=black] ( 4,-4) -- ( 5,-5);
				\draw[->, -{Stealth[]}, color=black] ( 5,-1) -- ( 6,0);
				\draw[->, -{Stealth[]}, color=black] ( 5,-1) -- ( 6,-1);
				\draw[->, -{Stealth[]}, color=black] ( 5,-1) -- ( 6,-2);
				\draw[->, -{Stealth[]}, color=black] ( 5,-3) -- ( 6,-2);
				\draw[->, -{Stealth[]}, color=black] ( 5,-3) -- ( 6,-4);
				\draw[->, -{Stealth[]}, color=black] ( 5,-4) -- ( 6,-4);
				\draw[->, -{Stealth[]}, color=black] ( 5,-5) -- ( 6,-4);
				\draw[-,densely dotted] (8,0.75) -- (8,-5.5);
				\node at ( 6,0) [circle, draw] {\tiny{$\bullet$}};
				\node at ( 6,-1) [circle, fill= black, inner sep=1pt] {};
				\node at ( 7,-1) [circle, draw] {\tiny{$\bullet$}};
				\node at ( 6,-2) [circle, fill=black, inner sep=1pt] {};
				\node at ( 7,-3) [circle, fill=black, inner sep=1pt] {};
				\node at ( 6,-4) [circle, draw] {\tiny{$\bullet$}};
				\node at ( 7,-4) [circle, fill=black, inner sep=1pt] {};
				\node at ( 7,-5) [circle, fill=black, inner sep=1pt] {};
				\node at ( 8,0) [circle, draw] {\tiny{$\bullet$}};
				\node at ( 8,-1) [circle, fill= black, inner sep=1pt] {};
				\node at ( 8,-2) [circle, fill=black, inner sep=1pt] {};
				\node at ( 8,-4) [circle, draw] {\tiny{$\bullet$}};
				\node at (8.75,-2.5) [] {\textbf{...}};
				\draw[->, -{Stealth[]}, color=black] ( 6,0) -- ( 7,-1);
				\draw[->, -{Stealth[]}, color=black] ( 6,-1) -- ( 7,-1);
				\draw[->, -{Stealth[]}, color=black] ( 6,-2) -- ( 7,-1);
				\draw[->, -{Stealth[]}, color=black] ( 6,-2) -- ( 7,-3);
				\draw[->, -{Stealth[]}, color=black] ( 6,-4) -- ( 7,-3);
				\draw[->, -{Stealth[]}, color=black] ( 6,-4) -- ( 7,-4);
				\draw[->, -{Stealth[]}, color=black] ( 6,-4) -- ( 7,-5);
				\draw[->, -{Stealth[]}, color=black] ( 7,-1) -- ( 8,0);
				\draw[->, -{Stealth[]}, color=black] ( 7,-1) -- ( 8,-1);
				\draw[->, -{Stealth[]}, color=black] ( 7,-1) -- ( 8,-2);
				\draw[->, -{Stealth[]}, color=black] ( 7,-3) -- ( 8,-2);
				\draw[->, -{Stealth[]}, color=black] ( 7,-3) -- ( 8,-4);
				\draw[->, -{Stealth[]}, color=black] ( 7,-4) -- ( 8,-4);
				\draw[->, -{Stealth[]}, color=black] ( 7,-5) -- ( 8,-4);
			\end{tikzpicture}
		\end{figure}
		\vspace{-.8cm}\item\label{knitting_algorithm step_2} Starting in the leftmost occurrence of the vertex $j$, at this vertex $j$ write $1$ with a box around it and write $0$ at all other vertices in the same column.
		Call this the first column.
	
		\item\label{knitting_algorithm step_3} Consider now the second column, which is the next column to the right of the first column.
		Then, for each vertex $k$ in the second column, if there is an arrow from vertex $j$ to vertex $k$ in the translation quiver, write $1$ at this vertex $k$.
		Otherwise, write $0$ at this vertex $k$.
	
		\item\label{knitting_algorithm step_4}Assume that the first $l \ge 2$ columns of the translation quiver have been filled.
		Then, to fill in the $(l+1)^{\text{st}}$-column, proceed as follows: the value at a vertex $r$ (in column $l+1$) is equal to the sum of all the values at vertices $s$ in column $l$ where there is an arrow from vertex $s$ to vertex $r$ in the translation quiver, minus the value at vertex $r$ in column $l-1$.
		If any value is circled, then consider it to be equal to $0$ (the square box at vertex $j$ is not considered circled).
	
		\item\label{knitting_algorithm step_5} Repeat step~\eqref{knitting_algorithm step_4} until a value of $-1$ occurs, then stop.
		
		\item\label{knitting_algorithm step_6} For each vertex $i \in I$, write $r_{j, i}$ for the sum of the non-negative values at all vertices $i$ in the translation quiver.
		The values at these vertices in the translation quiver should be circled since $i \in I$.
	\end{enumerate}
	The vertex $j$ with the boxed 1, the vertices with circled values, and with the $-1$ can all be used to obtain exact sequences which have various universal properties (see, e.g. \cite[4.4]{IyWe08SpecialClassificationCM}).
	For our purposes, it suffices to solely pay attention to the circled values.
	This becomes apparent in the following result; for further details, see \cite[\S4]{We11}.
	\begin{theorem}
		For each vertex $i \in I$, the value of $r_{j, i}$ in the above calculation is precisely the number of arrows from vertex $j$ to vertex $i$ in the quiver of $\GammaI{I}$.
	\end{theorem}
\end{knitting}

\begin{remark}
	This result fixes $j$ and applies the knitting algorithm to produce $r_{j, i}$.
	It is also possible to fix $i$ and knit to obtain $r_{i, j}$, but this requires us to adapt the algorithm slightly.
	We do this as follows: Step~\eqref{knitting_algorithm step_2} begins in the rightmost column of the translation quiver, and we move left.
	Step~\eqref{knitting_algorithm step_3} and Step~\eqref{knitting_algorithm step_4} are similar but are achieved by moving from right to left.
	As in Step~\eqref{knitting_algorithm step_5}, we stop when a $-1$ occurs.
	Step~\eqref{knitting_algorithm step_6} is also similar, that is, for each $i \in I$, write $r_{i, j}$ for the sum of the non-negative values at all vertices $i$ in the translation quiver.
	Just as before, for each vertex $i$, the value of $r_{i, j}$ is precisely the number of arrows from any vertex $i$ to vertex $j$ in the quiver of $\GammaI{I}$.
\end{remark}
We illustrate both algorithms in the following example.	
\begin{example}\label{example:knitting_D_6_(0_1_3)}
	Consider the AR-quiver of $\teDn{6}$ with vertices labelled $0, \hdots, 6$ as in Figure~\ref{fig:repetition_quivers_D_n_even}.
	Let $I = \{ 0, 1, 3 \}$ and suppose we want to determine the number of arrows \emph{from} vertex $0$ to some vertex $i$ in the quiver of $\GammaI{I}$.
	We start the knitting algorithm by drawing the translation quiver from Figure~\ref{fig:repetition_quivers_D_n_even}, omitting any labels.
	\begin{figure}[H]
		\centering
		\begin{tikzpicture}[->,shorten >=5pt, shorten <=5pt, scale=0.8]
			\node at ( 0,0) [circle, fill=black, inner sep=1pt] {};
			\node at ( 0,-1) [circle, fill= black, inner sep=1pt] {};
			\node at ( 1,-1) [circle, fill=black, inner sep=1pt] {};
			\node at ( 0,-2) [circle, fill=black, inner sep=1pt] {};
			\node at ( 1,-3) [circle, fill=black, inner sep=1pt] {};
			\node at ( 0,-3) [circle, fill=black, inner sep=1pt] {};
			\node at ( 0,-4) [circle, fill=black, inner sep=1pt] {};
				\draw[->, -{Stealth[]}, color=black] ( 0,0) -- ( 1,-1);
				\draw[->, -{Stealth[]}, color=black] ( 0,-1) -- ( 1,-1);
				\draw[->, -{Stealth[]}, color=black] ( 0,-2) -- ( 1,-1);
				\draw[->, -{Stealth[]}, color=black] ( 0,-2) -- ( 1,-3);
				\draw[->, -{Stealth[]}, color=black] ( 0,-3) -- ( 1,-3);
				\draw[->, -{Stealth[]}, color=black] ( 0,-4) -- ( 1,-3);
				\draw[->, -{Stealth[]}, color=black] ( 1,-1) -- ( 2,0);
				\draw[->, -{Stealth[]}, color=black] ( 1,-1) -- ( 2,-1);
				\draw[->, -{Stealth[]}, color=black] ( 1,-1) -- ( 2,-2);
				\draw[->, -{Stealth[]}, color=black] ( 1,-3) -- ( 2,-2);
				\draw[->, -{Stealth[]}, color=black] ( 1,-3) -- ( 2,-3);
				\draw[->, -{Stealth[]}, color=black] ( 1,-3) -- ( 2,-4);
			\node at ( 2,0) [circle, fill=black, inner sep=1pt] {};
			\node at ( 2,-1) [circle, fill= black, inner sep=1pt] {};
			\node at ( 3,-1) [circle, fill=black, inner sep=1pt] {};
			\node at ( 2,-2) [circle, fill=black, inner sep=1pt] {};
			\node at ( 3,-3) [circle, fill=black, inner sep=1pt] {};
			\node at ( 2,-3) [circle, fill=black, inner sep=1pt] {};
			\node at ( 2,-4) [circle, fill=black, inner sep=1pt] {};
				\draw[->, -{Stealth[]}, color=black] ( 2,0) -- ( 3,-1);
				\draw[->, -{Stealth[]}, color=black] ( 2,-1) -- ( 3,-1);
				\draw[->, -{Stealth[]}, color=black] ( 2,-2) -- ( 3,-1);
				\draw[->, -{Stealth[]}, color=black] ( 2,-2) -- ( 3,-3);
				\draw[->, -{Stealth[]}, color=black] ( 2,-3) -- ( 3,-3);
				\draw[->, -{Stealth[]}, color=black] ( 2,-4) -- ( 3,-3);
				\draw[->, -{Stealth[]}, color=black] ( 3,-1) -- ( 4,0);
				\draw[->, -{Stealth[]}, color=black] ( 3,-1) -- ( 4,-1);
				\draw[->, -{Stealth[]}, color=black] ( 3,-1) -- ( 4,-2);
				\draw[->, -{Stealth[]}, color=black] ( 3,-3) -- ( 4,-2);
				\draw[->, -{Stealth[]}, color=black] ( 3,-3) -- ( 4,-3);
				\draw[->, -{Stealth[]}, color=black] ( 3,-3) -- ( 4,-4);
			\node at ( 4,0) [circle, fill=black, inner sep=1pt] {};
			\node at ( 4,-1) [circle, fill= black, inner sep=1pt] {};
			\node at ( 5,-1) [circle, fill=black, inner sep=1pt] {};
			\node at ( 4,-2) [circle, fill=black, inner sep=1pt] {};
			\node at ( 5,-3) [circle, fill=black, inner sep=1pt] {};
			\node at ( 4,-3) [circle, fill=black, inner sep=1pt] {};
			\node at ( 4,-4) [circle, fill=black, inner sep=1pt] {};
				\draw[->, -{Stealth[]}, color=black] ( 4,0) -- ( 5,-1);
				\draw[->, -{Stealth[]}, color=black] ( 4,-1) -- ( 5,-1);
				\draw[->, -{Stealth[]}, color=black] ( 4,-2) -- ( 5,-1);
				\draw[->, -{Stealth[]}, color=black] ( 4,-2) -- ( 5,-3);
				\draw[->, -{Stealth[]}, color=black] ( 4,-3) -- ( 5,-3);
				\draw[->, -{Stealth[]}, color=black] ( 4,-4) -- ( 5,-3);
				\draw[->, -{Stealth[]}, color=black] ( 5,-1) -- ( 6,0);
				\draw[->, -{Stealth[]}, color=black] ( 5,-1) -- ( 6,-1);
				\draw[->, -{Stealth[]}, color=black] ( 5,-1) -- ( 6,-2);
				\draw[->, -{Stealth[]}, color=black] ( 5,-3) -- ( 6,-2);
				\draw[->, -{Stealth[]}, color=black] ( 5,-3) -- ( 6,-3);
				\draw[->, -{Stealth[]}, color=black] ( 5,-3) -- ( 6,-4);
			\node at ( 6,0) [circle, fill=black, inner sep=1pt] {};
			\node at ( 6,-1) [circle, fill= black, inner sep=1pt] {};
			\node at ( 7,-1) [circle, fill=black, inner sep=1pt] {};
			\node at ( 6,-2) [circle, fill=black, inner sep=1pt] {};
			\node at ( 7,-3) [circle, fill=black, inner sep=1pt] {};
			\node at ( 6,-3) [circle, fill=black, inner sep=1pt] {};
			\node at ( 6,-4) [circle, fill=black, inner sep=1pt] {};
			\draw[->, -{Stealth[]}, color=black] ( 6,0) -- ( 7,-1);
			\draw[->, -{Stealth[]}, color=black] ( 6,-1) -- ( 7,-1);
			\draw[->, -{Stealth[]}, color=black] ( 6,-2) -- ( 7,-1);
			\draw[->, -{Stealth[]}, color=black] ( 6,-2) -- ( 7,-3);
			\draw[->, -{Stealth[]}, color=black] ( 6,-3) -- ( 7,-3);
			\draw[->, -{Stealth[]}, color=black] ( 6,-4) -- ( 7,-3);
			\draw[->, -{Stealth[]}, color=black] ( 7,-1) -- ( 8,0);
			\draw[->, -{Stealth[]}, color=black] ( 7,-1) -- ( 8,-1);
			\draw[->, -{Stealth[]}, color=black] ( 7,-1) -- ( 8,-2);
			\draw[->, -{Stealth[]}, color=black] ( 7,-3) -- ( 8,-2);
			\draw[->, -{Stealth[]}, color=black] ( 7,-3) -- ( 8,-3);
			\draw[->, -{Stealth[]}, color=black] ( 7,-3) -- ( 8,-4);
			\node at ( 8,0) [circle, fill=black, inner sep=1pt] {};
			\node at ( 8,-1) [circle, fill= black, inner sep=1pt] {};
			\node at ( 8,-2) [circle, fill=black, inner sep=1pt] {};
			\node at ( 8,-3) [circle, fill=black, inner sep=1pt] {};
			\node at ( 8,-4) [circle, fill=black, inner sep=1pt] {};
		\end{tikzpicture}
	\end{figure}
	Since we are calculating the number of arrows out of vertex $0$ in the quiver of $\GammaI{I}$, we will start the knitting calculation in the leftmost column of the translation quiver and add more columns to the right as needed.
	Now, circle all positions corresponding to the vertices in $I$.
	In this example, we circle vertices $0$, $1,$ and $3$ repeatedly.
	In the position of vertex $0$ in the first column place a boxed $1$, and in all other vertices in the first column, write $0$.
	\begin{figure}[H]
		\centering
		\begin{tikzpicture}[->,shorten >=9pt, shorten <=9pt, scale=0.95]
			\node at ( 0,0) [rectangle, draw] {\footnotesize{$1$}};
			\node at ( 0,-1) [circle, draw] {\footnotesize{$0$}};
			\node at ( 1,-1) [circle, fill=black, inner sep=1pt] {};
			\node at ( 0,-2) [circle, draw] {\footnotesize{$0$}};
			\node at ( 1,-3) [circle, fill=black, inner sep=1pt] {};
			\node at ( 0,-3) [] {\footnotesize{$0$}};
			\node at ( 0,-4) [] {\footnotesize{$0$}};
				\draw[->, -{Stealth[]}, color=black] ( 0,0) -- ( 1,-1);
				\draw[->, -{Stealth[]}, color=black] ( 0,-1) -- ( 1,-1);
				\draw[->, -{Stealth[]}, color=black] ( 0,-2) -- ( 1,-1);
				\draw[->, -{Stealth[]}, color=black] ( 0,-2) -- ( 1,-3);
				\draw[->, -{Stealth[]}, color=black] ( 0,-3) -- ( 1,-3);
				\draw[->, -{Stealth[]}, color=black] ( 0,-4) -- ( 1,-3);
				\draw[->, -{Stealth[]}, color=black] ( 1,-1) -- ( 2,0);
				\draw[->, -{Stealth[]}, color=black] ( 1,-1) -- ( 2,-1);
				\draw[->, -{Stealth[]}, color=black] ( 1,-1) -- ( 2,-2);
				\draw[->, -{Stealth[]}, color=black] ( 1,-3) -- ( 2,-2);
				\draw[->, -{Stealth[]}, color=black] ( 1,-3) -- ( 2,-3);
				\draw[->, -{Stealth[]}, color=black] ( 1,-3) -- ( 2,-4);
			\node at ( 2,0) [circle, draw, inner sep=6.2pt] {};
			\node at ( 2,-1) [circle, draw, inner sep=6.2pt] {};
			\node at ( 3,-1) [circle, fill=black, inner sep=1pt] {};
			\node at ( 2,-2) [circle, draw, inner sep=6.2pt] {};
			\node at ( 3,-3) [circle, fill=black, inner sep=1pt] {};
			\node at ( 2,-3) [circle, fill=black, inner sep=1pt] {};
			\node at ( 2,-4) [circle, fill=black, inner sep=1pt] {};
				\draw[->, -{Stealth[]}, color=black] ( 2,0) -- ( 3,-1);
				\draw[->, -{Stealth[]}, color=black] ( 2,-1) -- ( 3,-1);
				\draw[->, -{Stealth[]}, color=black] ( 2,-2) -- ( 3,-1);
				\draw[->, -{Stealth[]}, color=black] ( 2,-2) -- ( 3,-3);
				\draw[->, -{Stealth[]}, color=black] ( 2,-3) -- ( 3,-3);
				\draw[->, -{Stealth[]}, color=black] ( 2,-4) -- ( 3,-3);
		\end{tikzpicture}
	\end{figure}
	The first column is completely filled in, so we move to the second column.
	By Step~\eqref{knitting_algorithm step_3}, since there is an arrow from vertex $0$ to vertex $2$, write $1$ as the value at vertex $2$ in the second column.
	All other vertices in the second column have no arrow from vertex $0$, and so write $0$ at all other vertices in the second column.
	\begin{figure}[H]
		\centering
		\begin{tikzpicture}[->,shorten >=9pt, shorten <=9pt, scale=0.95]
			\node at ( 0,0) [rectangle, draw] {\footnotesize{$1$}};
			\node at ( 0,-1) [circle, draw] {\footnotesize{$0$}};
			\node at ( 1,-1) [] {\footnotesize{$1$}};
			\node at ( 0,-2) [circle, draw] {\footnotesize{$0$}};
			\node at ( 1,-3) [] {\footnotesize{$0$}};
			\node at ( 0,-3) [] {\footnotesize{$0$}};
			\node at ( 0,-4) [] {\footnotesize{$0$}};
				\draw[->, -{Stealth[]}, color=black] ( 0,0) -- ( 1,-1);
				\draw[->, -{Stealth[]}, color=black] ( 0,-1) -- ( 1,-1);
				\draw[->, -{Stealth[]}, color=black] ( 0,-2) -- ( 1,-1);
				\draw[->, -{Stealth[]}, color=black] ( 0,-2) -- ( 1,-3);
				\draw[->, -{Stealth[]}, color=black] ( 0,-3) -- ( 1,-3);
				\draw[->, -{Stealth[]}, color=black] ( 0,-4) -- ( 1,-3);
				\draw[->, -{Stealth[]}, color=black] ( 1,-1) -- ( 2,0);
				\draw[->, -{Stealth[]}, color=black] ( 1,-1) -- ( 2,-1);
				\draw[->, -{Stealth[]}, color=black] ( 1,-1) -- ( 2,-2);
				\draw[->, -{Stealth[]}, color=black] ( 1,-3) -- ( 2,-2);
				\draw[->, -{Stealth[]}, color=black] ( 1,-3) -- ( 2,-3);
				\draw[->, -{Stealth[]}, color=black] ( 1,-3) -- ( 2,-4);
			\node at ( 2,0) [circle, draw, inner sep=6.2pt] {};
			\node at ( 2,-1) [circle, draw, inner sep=6.2pt] {};
			\node at ( 3,-1) [circle, fill=black, inner sep=1pt] {};
			\node at ( 2,-2) [circle, draw, inner sep=6.2pt] {};
			\node at ( 3,-3) [circle, fill=black, inner sep=1pt] {};
			\node at ( 2,-3) [circle, fill=black, inner sep=1pt] {};
			\node at ( 2,-4) [circle, fill=black, inner sep=1pt] {};
				\draw[->, -{Stealth[]}, color=black] ( 2,0) -- ( 3,-1);
				\draw[->, -{Stealth[]}, color=black] ( 2,-1) -- ( 3,-1);
				\draw[->, -{Stealth[]}, color=black] ( 2,-2) -- ( 3,-1);
				\draw[->, -{Stealth[]}, color=black] ( 2,-2) -- ( 3,-3);
				\draw[->, -{Stealth[]}, color=black] ( 2,-3) -- ( 3,-3);
				\draw[->, -{Stealth[]}, color=black] ( 2,-4) -- ( 3,-3);
		\end{tikzpicture}
	\end{figure}
	With the first two columns filled, we proceed to Step~\eqref{knitting_algorithm step_4} to calculate the third column.
	In the following illustration the red vertices and arrows are those under consideration.
	Observe that vertex $0$ in column~3 (the red circled $a$) has an arrow coming from vertex $2$ in column~2 (the red $1$) and the value of vertex $0$ in column~1 (the red boxed $1$) is $1$.
	So, following the instructions in Step~\eqref{knitting_algorithm step_4}, as indicated by the calculation of $a$ below, vertex $0$ in column~3 has a value of $0$.
	\vspace{-.8cm}\begin{figure}[H]
		\centering
		\begin{tikzpicture}[shorten >=9pt, shorten <=9pt, scale=0.95]
			\node[color=red] at ( 0,0) [rectangle, draw, color=red] {\footnotesize{$1$}};
			\node at ( 0,-1) [circle, draw] {\footnotesize{$0$}};
			\node[color=red] at ( 1,-1) [] {\footnotesize{$1$}};
			\node at ( 0,-2) [circle, draw] {\footnotesize{$0$}};
			\node at ( 1,-3) [] {\footnotesize{$0$}};
			\node at ( 0,-3) [] {\footnotesize{$0$}};
			\node at ( 0,-4) [] {\footnotesize{$0$}};
			\draw[->, -{Stealth[]}, color=red] ( 0,0) -- ( 1,-1);
			\draw[->, -{Stealth[]}, color=black] ( 0,-1) -- ( 1,-1);
			\draw[->, -{Stealth[]}, color=black] ( 0,-2) -- ( 1,-1);
			\draw[->, -{Stealth[]}, color=black] ( 0,-2) -- ( 1,-3);
			\draw[->, -{Stealth[]}, color=black] ( 0,-3) -- ( 1,-3);
			\draw[->, -{Stealth[]}, color=black] ( 0,-4) -- ( 1,-3);
			\draw[->, -{Stealth[]}, color=red] ( 1,-1) -- ( 2,0);
			\draw[->, -{Stealth[]}, color=black] ( 1,-1) -- ( 2,-1);
			\draw[->, -{Stealth[]}, color=black] ( 1,-1) -- ( 2,-2);
			\draw[->, -{Stealth[]}, color=black] ( 1,-3) -- ( 2,-2);
			\draw[->, -{Stealth[]}, color=black] ( 1,-3) -- ( 2,-3);
			\draw[->, -{Stealth[]}, color=black] ( 1,-3) -- ( 2,-4);
			\node[color=red] at ( 2,0) [circle, draw,color=red] {\footnotesize{$a$}};
			\node at ( 2,-1) [circle, draw, inner sep=6.2pt] {};
			\node at ( 3,-1) [circle, fill=black, inner sep=1pt] {};
			\node at ( 2,-2) [circle, draw, inner sep=6.2pt] {};
			\node at ( 3,-3) [circle, fill=black, inner sep=1pt] {};
			\node at ( 2,-3) [circle, fill=black, inner sep=1pt] {};
			\node at ( 2,-4) [circle, fill=black, inner sep=1pt] {};
			\draw[->, -{Stealth[]}, color=black] ( 2,0) -- ( 3,-1);
			\draw[->, -{Stealth[]}, color=black] ( 2,-1) -- ( 3,-1);
			\draw[->, -{Stealth[]}, color=black] ( 2,-2) -- ( 3,-1);
			\draw[->, -{Stealth[]}, color=black] ( 2,-2) -- ( 3,-3);
			\draw[->, -{Stealth[]}, color=black] ( 2,-3) -- ( 3,-3);
			\draw[->, -{Stealth[]}, color=black] ( 2,-4) -- ( 3,-3);
			\draw[color=blue, densely dotted] ( 0,0) to [out=90,in=100] ( 6.7,0);
			\draw[color=blue, densely dotted] ( 1,-1) to [out=110,in=100] ( 6,0);
			\node at ( 5.3,0) [] {\footnotesize{$a$}};
			\node at ( 5.65,0) [] {\footnotesize{$=$}};
			\node at ( 6,0) [] {\footnotesize{$1$}};
			\node at ( 6.35,0) [] {\footnotesize{$-$}};
			\node at ( 6.7,0) [] {\footnotesize{$1$}};
			\node at ( 7.05,0) [] {\footnotesize{$=$}};
			\node at ( 7.4,0) [] {\footnotesize{$0$}};
		\end{tikzpicture}
	\end{figure}
	Similarly, in column~3 the values at vertices $1$ and $3$ are both $1$ (notice that they are also circled) and the values at $4$ and $5$ are both $0$ (these vertices are not circled).
	\begin{figure}[H]
		\centering
		\begin{tikzpicture}[->,shorten >=9pt, shorten <=9pt, scale=0.95]
			\node at ( 0,0) [rectangle, draw] {\footnotesize{$1$}};
			\node at ( 0,-1) [circle, draw] {\footnotesize{$0$}};
			\node at ( 1,-1) [] {\footnotesize{$1$}};
			\node at ( 0,-2) [circle, draw] {\footnotesize{$0$}};
			\node at ( 1,-3) [] {\footnotesize{$0$}};
			\node at ( 0,-3) [] {\footnotesize{$0$}};
			\node at ( 0,-4) [] {\footnotesize{$0$}};
				\draw[->, -{Stealth[]}, color=black] ( 0,0) -- ( 1,-1);
				\draw[->, -{Stealth[]}, color=black] ( 0,-1) -- ( 1,-1);
				\draw[->, -{Stealth[]}, color=black] ( 0,-2) -- ( 1,-1);
				\draw[->, -{Stealth[]}, color=black] ( 0,-2) -- ( 1,-3);
				\draw[->, -{Stealth[]}, color=black] ( 0,-3) -- ( 1,-3);
				\draw[->, -{Stealth[]}, color=black] ( 0,-4) -- ( 1,-3);
				\draw[->, -{Stealth[]}, color=black] ( 1,-1) -- ( 2,0);
				\draw[->, -{Stealth[]}, color=black] ( 1,-1) -- ( 2,-1);
				\draw[->, -{Stealth[]}, color=black] ( 1,-1) -- ( 2,-2);
				\draw[->, -{Stealth[]}, color=black] ( 1,-3) -- ( 2,-2);
				\draw[->, -{Stealth[]}, color=black] ( 1,-3) -- ( 2,-3);
				\draw[->, -{Stealth[]}, color=black] ( 1,-3) -- ( 2,-4);
			\node at ( 2,0) [circle, draw] {\footnotesize{$0$}};
			\node at ( 2,-1) [circle, draw] {\footnotesize{$1$}};
			\node at ( 3,-1) [circle, fill=black, inner sep=1pt] {};
			\node at ( 2,-2) [circle, draw] {\footnotesize{$1$}};
			\node at ( 3,-3) [circle, fill=black, inner sep=1pt] {};
			\node at ( 2,-3) [] {\footnotesize{$0$}};
			\node at ( 2,-4) [] {\footnotesize{$0$}};
				\draw[->, -{Stealth[]}, color=black] ( 2,0) -- ( 3,-1);
				\draw[->, -{Stealth[]}, color=black] ( 2,-1) -- ( 3,-1);
				\draw[->, -{Stealth[]}, color=black] ( 2,-2) -- ( 3,-1);
				\draw[->, -{Stealth[]}, color=black] ( 2,-2) -- ( 3,-3);
				\draw[->, -{Stealth[]}, color=black] ( 2,-3) -- ( 3,-3);
				\draw[->, -{Stealth[]}, color=black] ( 2,-4) -- ( 3,-3);
		\end{tikzpicture}
	\end{figure}

	Next, we repeat Step~\eqref{knitting_algorithm step_4}.
	Since the circled $1$'s behave as zero, in this example the $-1$ appears in vertex $2$ in column~4 of the translation quiver.
	\begin{figure}[H]
		\centering
		\begin{tikzpicture}[->,shorten >=9pt, shorten <=9pt, scale=0.95]
			\node at ( 0,0) [rectangle, draw] {\footnotesize{$1$}};
			\node at ( 0,-1) [circle, draw] {\footnotesize{$0$}};
			\node at ( 1,-1) [] {\footnotesize{$1$}};
			\node at ( 0,-2) [circle, draw] {\footnotesize{$0$}};
			\node at ( 1,-3) [] {\footnotesize{$0$}};
			\node at ( 0,-3) [] {\footnotesize{$0$}};
			\node at ( 0,-4) [] {\footnotesize{$0$}};
				\draw[->, -{Stealth[]}, color=black] ( 0,0) -- ( 1,-1);
				\draw[->, -{Stealth[]}, color=black] ( 0,-1) -- ( 1,-1);
				\draw[->, -{Stealth[]}, color=black] ( 0,-2) -- ( 1,-1);
				\draw[->, -{Stealth[]}, color=black] ( 0,-2) -- ( 1,-3);
				\draw[->, -{Stealth[]}, color=black] ( 0,-3) -- ( 1,-3);
				\draw[->, -{Stealth[]}, color=black] ( 0,-4) -- ( 1,-3);
				\draw[->, -{Stealth[]}, color=black] ( 1,-1) -- ( 2,0);
				\draw[->, -{Stealth[]}, color=black] ( 1,-1) -- ( 2,-1);
				\draw[->, -{Stealth[]}, color=black] ( 1,-1) -- ( 2,-2);
				\draw[->, -{Stealth[]}, color=black] ( 1,-3) -- ( 2,-2);
				\draw[->, -{Stealth[]}, color=black] ( 1,-3) -- ( 2,-3);
				\draw[->, -{Stealth[]}, color=black] ( 1,-3) -- ( 2,-4);
			\node at ( 2,0) [circle, draw] {\footnotesize{$0$}};
			\node at ( 2,-1) [circle, draw] {\footnotesize{$1$}};
			\node at ( 3,-1) [] {\footnotesize{$-1$}};
			\node at ( 2,-2) [circle, draw] {\footnotesize{$1$}};
			\node at ( 3,-3) [] {\footnotesize{$0$}};
			\node at ( 2,-3) [] {\footnotesize{$0$}};
			\node at ( 2,-4) [] {\footnotesize{$0$}};
				\draw[->, -{Stealth[]}, color=black] ( 2,0) -- ( 3,-1);
				\draw[->, -{Stealth[]}, color=black] ( 2,-1) -- ( 3,-1);
				\draw[->, -{Stealth[]}, color=black] ( 2,-2) -- ( 3,-1);
				\draw[->, -{Stealth[]}, color=black] ( 2,-2) -- ( 3,-3);
				\draw[->, -{Stealth[]}, color=black] ( 2,-3) -- ( 3,-3);
				\draw[->, -{Stealth[]}, color=black] ( 2,-4) -- ( 3,-3);
		\end{tikzpicture}
	\end{figure}
	For Step~\eqref{knitting_algorithm step_6}, suppose that $i=1$, then sum the values in all occurrences of the circled vertex $1$ in the translation quiver.
	We see that $r_{0, 1} = 1$.
	Similarly, $r_{0, 0} = 0$ and $r_{0, 3} = 1$.
	Hence, there is precisely one arrow from vertex $0$ to vertex $1$ in the quiver of $\GammaI{I}$ and one arrow from vertex $0$ to vertex $3$ in the quiver of $\GammaI{I}$.

	Now, to determine the number of arrows \emph{to} vertex $0$ in the quiver of $\GammaI{I}$ we repeat the knitting algorithm but start at vertex $0$ in the rightmost column of the translation quiver and move left.
	The result of Steps~\eqref{knitting_algorithm step_1}--\eqref{knitting_algorithm step_5} is summarised by the following diagram.
	\begin{figure}[H]
		\centering
		\begin{tikzpicture}[->,shorten >=9pt, shorten <=9pt, scale=0.95]
			\node at ( 1,-1) [] {\footnotesize{$-1$}};
			\node at ( 1,-3) [] {\footnotesize{$0$}};
				\draw[->, -{Stealth[]}, color=black] ( 1,-1) -- ( 2,0);
				\draw[->, -{Stealth[]}, color=black] ( 1,-1) -- ( 2,-1);
				\draw[->, -{Stealth[]}, color=black] ( 1,-1) -- ( 2,-2);
				\draw[->, -{Stealth[]}, color=black] ( 1,-3) -- ( 2,-2);
				\draw[->, -{Stealth[]}, color=black] ( 1,-3) -- ( 2,-3);
				\draw[->, -{Stealth[]}, color=black] ( 1,-3) -- ( 2,-4);
			\node at ( 2,0) [circle, draw] {\footnotesize{$0$}};
			\node at ( 2,-1) [circle, draw] {\footnotesize{$1$}};
			\node at ( 3,-1) [] {\footnotesize{$1$}};
			\node at ( 2,-2) [circle, draw] {\footnotesize{$1$}};
			\node at ( 3,-3) [] {\footnotesize{$0$}};
			\node at ( 2,-3) [] {\footnotesize{$0$}};
			\node at ( 2,-4) [] {\footnotesize{$0$}};
				\draw[->, -{Stealth[]}, color=black] ( 2,0) -- ( 3,-1);
				\draw[->, -{Stealth[]}, color=black] ( 2,-1) -- ( 3,-1);
				\draw[->, -{Stealth[]}, color=black] ( 2,-2) -- ( 3,-1);
				\draw[->, -{Stealth[]}, color=black] ( 2,-2) -- ( 3,-3);
				\draw[->, -{Stealth[]}, color=black] ( 2,-3) -- ( 3,-3);
				\draw[->, -{Stealth[]}, color=black] ( 2,-4) -- ( 3,-3);
				\draw[->, -{Stealth[]}, color=black] ( 3,-1) -- ( 4,0);
				\draw[->, -{Stealth[]}, color=black] ( 3,-1) -- ( 4,-1);
				\draw[->, -{Stealth[]}, color=black] ( 3,-1) -- ( 4,-2);
				\draw[->, -{Stealth[]}, color=black] ( 3,-3) -- ( 4,-2);
				\draw[->, -{Stealth[]}, color=black] ( 3,-3) -- ( 4,-3);
				\draw[->, -{Stealth[]}, color=black] ( 3,-3) -- ( 4,-4);
			\node at ( 4,0) [rectangle, draw] {\footnotesize{$1$}};
			\node at ( 4,-1) [circle, draw] {\footnotesize{$0$}};
			\node at ( 4,-2) [circle, draw] {\footnotesize{$0$}};
			\node at ( 4,-3) [] {\footnotesize{$0$}};
			\node at ( 4,-4) [] {\footnotesize{$0$}};
		\end{tikzpicture}
	\end{figure}
	For Step~\eqref{knitting_algorithm step_6}, sum the circled values at each occurrence of the chosen vertex $i$ to see $r_{0, 0} = 0$, $r_{1, 0} = 1$, and $r_{3, 0} = 1$.
	Thus, in the quiver of $\GammaI{I}$, there is one arrow from vertex $1$ to vertex $0$, and one arrow from vertex $3$ to vertex $0$.
\end{example}

\section{Knitting results}\label{sec:knitting_results}
In Example~\ref{example:knitting_D_6_(0_1_3)} there is a striking symmetry appearing between the knitting calculations which go left to right and those which go right to left.
In fact, combining the final calculations in Example~\ref{example:knitting_D_6_(0_1_3)} gives the following, where the dotted line represents a line of reflection.
\vspace{-.6cm}\begin{figure}[H]
	\centering
	\begin{tikzpicture}[->,shorten >=9pt, shorten <=9pt, scale=0.9]
		\node at ( 1,-1) [] {\footnotesize{$-1$}};
		\node at ( 1,-3) [] {\footnotesize{$0$}};
		\draw[->, -{Stealth[]}, color=black] ( 1,-1) -- ( 2,0);
		\draw[->, -{Stealth[]}, color=black] ( 1,-1) -- ( 2,-1);
		\draw[->, -{Stealth[]}, color=black] ( 1,-1) -- ( 2,-2);
		\draw[->, -{Stealth[]}, color=black] ( 1,-3) -- ( 2,-2);
		\draw[->, -{Stealth[]}, color=black] ( 1,-3) -- ( 2,-3);
		\draw[->, -{Stealth[]}, color=black] ( 1,-3) -- ( 2,-4);
		\node at ( 2,0) [circle, draw] {\footnotesize{$0$}};
		\node at ( 2,-1) [circle, draw] {\footnotesize{$1$}};
		\node at ( 3,-1) [] {\footnotesize{$1$}};
		\node at ( 2,-2) [circle, draw] {\footnotesize{$1$}};
		\node at ( 3,-3) [] {\footnotesize{$0$}};
		\node at ( 2,-3) [] {\footnotesize{$0$}};
		\node at ( 2,-4) [] {\footnotesize{$0$}};
		\draw[->, -{Stealth[]}, color=black] ( 2,0) -- ( 3,-1);
		\draw[->, -{Stealth[]}, color=black] ( 2,-1) -- ( 3,-1);
		\draw[->, -{Stealth[]}, color=black] ( 2,-2) -- ( 3,-1);
		\draw[->, -{Stealth[]}, color=black] ( 2,-2) -- ( 3,-3);
		\draw[->, -{Stealth[]}, color=black] ( 2,-3) -- ( 3,-3);
		\draw[->, -{Stealth[]}, color=black] ( 2,-4) -- ( 3,-3);
		\draw[->, -{Stealth[]}, color=black] ( 3,-1) -- ( 4,0);
		\draw[->, -{Stealth[]}, color=black] ( 3,-1) -- ( 4,-1);
		\draw[->, -{Stealth[]}, color=black] ( 3,-1) -- ( 4,-2);
		\draw[->, -{Stealth[]}, color=black] ( 3,-3) -- ( 4,-2);
		\draw[->, -{Stealth[]}, color=black] ( 3,-3) -- ( 4,-3);
		\draw[->, -{Stealth[]}, color=black] ( 3,-3) -- ( 4,-4);
		\node at ( 4,0) [rectangle, draw] {\footnotesize{$1$}};
		\node at ( 4,-1) [circle, draw] {\footnotesize{$0$}};
		\node at ( 5,-1) [] {\footnotesize{$1$}};
		\node at ( 4,-2) [circle, draw] {\footnotesize{$0$}};
		\node at ( 5,-3) [] {\footnotesize{$0$}};
		\node at ( 4,-3) [] {\footnotesize{$0$}};
		\node at ( 4,-4) [] {\footnotesize{$0$}};
		\draw[-,densely dotted] (4,1.5) -- (4,-5);
		\draw[->, -{Stealth[]}, color=black] ( 3.5,0.75) -- ( 0.5,0.75);
		\draw[->, -{Stealth[]}, color=black] ( 4.5,0.75) -- ( 7.5,0.75);
		\draw[->, -{Stealth[]}, color=black] ( 4,0) -- ( 5,-1);
		\draw[->, -{Stealth[]}, color=black] ( 4,-1) -- ( 5,-1);
		\draw[->, -{Stealth[]}, color=black] ( 4,-2) -- ( 5,-1);
		\draw[->, -{Stealth[]}, color=black] ( 4,-2) -- ( 5,-3);
		\draw[->, -{Stealth[]}, color=black] ( 4,-3) -- ( 5,-3);
		\draw[->, -{Stealth[]}, color=black] ( 4,-4) -- ( 5,-3);
		\draw[->, -{Stealth[]}, color=black] ( 5,-1) -- ( 6,0);
		\draw[->, -{Stealth[]}, color=black] ( 5,-1) -- ( 6,-1);
		\draw[->, -{Stealth[]}, color=black] ( 5,-1) -- ( 6,-2);
		\draw[->, -{Stealth[]}, color=black] ( 5,-3) -- ( 6,-2);
		\draw[->, -{Stealth[]}, color=black] ( 5,-3) -- ( 6,-3);
		\draw[->, -{Stealth[]}, color=black] ( 5,-3) -- ( 6,-4);
		\node at ( 6,0) [circle, draw] {\footnotesize{$0$}};
		\node at ( 6,-1) [circle, draw] {\footnotesize{$1$}};
		\node at ( 7,-1) [] {\footnotesize{$-1$}};
		\node at ( 6,-2) [circle, draw] {\footnotesize{$1$}};
		\node at ( 7,-3) [] {\footnotesize{$0$}};
		\node at ( 6,-3) [] {\footnotesize{$0$}};
		\node at ( 6,-4) [] {\footnotesize{$0$}};
		\draw[->, -{Stealth[]}, color=black] ( 6,0) -- ( 7,-1);
		\draw[->, -{Stealth[]}, color=black] ( 6,-1) -- ( 7,-1);
		\draw[->, -{Stealth[]}, color=black] ( 6,-2) -- ( 7,-1);
		\draw[->, -{Stealth[]}, color=black] ( 6,-2) -- ( 7,-3);
		\draw[->, -{Stealth[]}, color=black] ( 6,-3) -- ( 7,-3);
		\draw[->, -{Stealth[]}, color=black] ( 6,-4) -- ( 7,-3);
	\end{tikzpicture}	
\end{figure}
\vspace{-.65cm}With this depiction it becomes clear that, for $i, j \in I$ where $I = \{ 0, 1, 3 \}$,
\vspace{-.45cm}\begin{align*}
	\# \text{ arrows } j \rightarrow i \text{ in the quiver of } \GammaI{I} &= r_{j, i}\\
	&= \textnormal{ sum of values at circled $i$ from centre to right} \\
	&= \textnormal{ sum of values at circled $i$ from centre to left} \\
	&= r_{i, j} \\
	&= \# \text{ arrows } i \rightarrow j \text{ in the quiver of } \GammaI{I}.
\end{align*}
We use and extend this symmetry in the main result of this chapter, Theorem~\ref{thm:number_arrows_0_to_i_equals_reverse}.
For type $\tA$, the line of reflection looks visually different from the one above.
In particular, it is diagonal.
Despite this visual difference, it is still a line through the original boxed vertex.
\vspace{-.6cm}\begin{figure}[H]
	\centering
	\begin{tikzpicture}[->, shorten >=9pt, shorten <=9pt, scale=0.75]
	\draw[-,densely dotted] (-1,-10) -- ( 10,1);
	\node at ( 0,0) [] {\footnotesize{$0$}};
	\node at ( 0,-1.5) [] {\footnotesize{$\sf{n}$}};
	\node at ( -0.25,-3) [] {\footnotesize{$\sf{n}-1$}};
	\node at ( -0.25,-4.5) [] {\footnotesize{$\sf{n}-2$}};
	\node at ( 0,-6) [] {\footnotesize{$\vdots$}};
	\node at ( 0,-7.5) [] {\footnotesize{$0$}};
	\node at ( 1.5,0) [] {\footnotesize{$1$}};
	\node at ( 1.5,-1.5) [] {\footnotesize{$0$}};
	\node at ( 1.5,-3) [] {\footnotesize{$\sf{n}$}};
	\node at ( 1.5,-4.5) [] {\footnotesize{$\sf{n}-1$}};
	\node at ( 1.5,-6) [] {\footnotesize{$\vdots$}};
	\node at ( 1.5,-7.5) [] {\footnotesize{$1$}};
	\node at ( 3,0) [] {\footnotesize{$2$}};
	\node at ( 3,-1.5) [] {\footnotesize{$1$}};
	\node at ( 3,-3) [] {\footnotesize{$0$}};
	\node at ( 3,-4.5) [] {\footnotesize{$\sf{n}$}};
	\node at ( 3,-6) [] {\footnotesize{$\vdots$}};
	\node at ( 3,-7.5) [] {\footnotesize{$2$}};
	\node at ( 4.5,0) [] {\footnotesize{$3$}};
	\node at ( 4.5,-1.5) [] {\footnotesize{$2$}};
	\node at ( 4.5,-3) [] {\footnotesize{$1$}};
	\node at ( 4.5,-4.5) [rectangle, draw] {\footnotesize{$0$}};
	\node at ( 4.5,-6) [] {\footnotesize{$\vdots$}};
	\node at ( 4.5,-7.5) [] {\footnotesize{$3$}};
	\node at ( 6,0) [] {\footnotesize{$\cdots$}};
	\node at ( 6,-1.5) [] {\footnotesize{$\cdots$}};
	\node at ( 6,-3) [] {\footnotesize{$\cdots$}};
	\node at ( 6,-4.5) [] {\footnotesize{$\cdots$}};
	\node at ( 6,-6) [] {\footnotesize{$\ddots$}};
	\node at ( 6,-7.5) [] {\footnotesize{$\cdots$}};
	\node at ( 7.5,0) [] {\footnotesize{$0$}};
	\node at ( 7.5,-1.5) [] {\footnotesize{$\sf{n}$}};
	\node at ( 7.5,-3) [] {\footnotesize{$\sf{n}-1$}};
	\node at ( 7.5,-4.5) [] {\footnotesize{$\sf{n}-2$}};
	\node at ( 7.5,-6) [] {\footnotesize{$\vdots$}};
	\node at ( 7.5,-7.5) [] {\footnotesize{$0$}};
	\draw[->, -{Stealth[]}, color=black] ( 0,0) -- ( 0,-1.5);
	\draw[->, -{Stealth[]}, color=black] ( 0,-1.5) -- ( 0,-3);
	\draw[->, -{Stealth[]}, color=black] ( 0,-3) -- ( 0,-4.5);
	\draw[->, -{Stealth[]}, color=black] ( 0,-4.5) -- ( 0,-6);
	\draw[->, -{Stealth[]}, color=black] ( 0,-6) -- ( 0,-7.5);
	\draw[->, -{Stealth[]}, color=black] ( 0,-7.5) -- ( 0,-9);
	\draw[->, -{Stealth[]}, color=black] ( 1.5,0) -- ( 1.5,-1.5);
	\draw[->, -{Stealth[]}, color=black] ( 1.5,-1.5) -- ( 1.5,-3);
	\draw[->, -{Stealth[]}, color=black] ( 1.5,-3) -- ( 1.5,-4.5);
	\draw[->, -{Stealth[]}, color=black] ( 1.5,-4.5) -- ( 1.5,-6);
	\draw[->, -{Stealth[]}, color=black] ( 1.5,-6) -- ( 1.5,-7.5);
	\draw[->, -{Stealth[]}, color=black] ( 1.5,-7.5) -- ( 1.5,-9);
	\draw[->, -{Stealth[]}, color=black] ( 3,0) -- ( 3,-1.5);
	\draw[->, -{Stealth[]}, color=black] ( 3,-1.5) -- ( 3,-3);
	\draw[->, -{Stealth[]}, color=black] ( 3,-3) -- ( 3,-4.5);
	\draw[->, -{Stealth[]}, color=black] ( 3,-4.5) -- ( 3,-6);
	\draw[->, -{Stealth[]}, color=black] ( 3,-6) -- ( 3,-7.5);
	\draw[->, -{Stealth[]}, color=black] ( 3,-7.5) -- ( 3,-9);
	\draw[->, -{Stealth[]}, color=black] ( 4.5,0) -- ( 4.5,-1.5);
	\draw[->, -{Stealth[]}, color=black] ( 4.5,-1.5) -- ( 4.5,-3);
	\draw[->, -{Stealth[]}, color=black] ( 4.5,-3) -- ( 4.5,-4.5);
	\draw[->, -{Stealth[]}, color=black] ( 4.5,-4.5) -- ( 4.5,-6);
	\draw[->, -{Stealth[]}, color=black] ( 4.5,-6) -- ( 4.5,-7.5);
	\draw[->, -{Stealth[]}, color=black] ( 4.5,-7.5) -- ( 4.5,-9);
	\draw[->, -{Stealth[]}, color=black] ( 6,0) -- ( 6,-1.5);
	\draw[->, -{Stealth[]}, color=black] ( 6,-1.5) -- ( 6,-3);
	\draw[->, -{Stealth[]}, color=black] ( 6,-3) -- ( 6,-4.5);
	\draw[->, -{Stealth[]}, color=black] ( 6,-4.5) -- ( 6,-6);
	\draw[->, -{Stealth[]}, color=black] ( 6,-6) -- ( 6,-7.5);
	\draw[->, -{Stealth[]}, color=black] ( 6,-7.5) -- ( 6,-9);
	\draw[->, -{Stealth[]}, color=black] ( 7.5,0) -- ( 7.5,-1.5);
	\draw[->, -{Stealth[]}, color=black] ( 7.5,-1.5) -- ( 7.5,-3);
	\draw[->, -{Stealth[]}, color=black] ( 7.5,-3) -- ( 7.5,-4.5);
	\draw[->, -{Stealth[]}, color=black] ( 7.5,-4.5) -- ( 7.5,-6);
	\draw[->, -{Stealth[]}, color=black] ( 7.5,-6) -- ( 7.5,-7.5);
	\draw[->, -{Stealth[]}, color=black] ( 7.5,-7.5) -- ( 7.5,-9);
	\draw[->, -{Stealth[]}, color=black] ( 0,0) -- ( 1.5,0);
	\draw[->, -{Stealth[]}, color=black] ( 1.5,0) -- ( 3,0);
	\draw[->, -{Stealth[]}, color=black] ( 3,0) -- ( 4.5,0);
	\draw[->, -{Stealth[]}, color=black] ( 4.5,0) -- ( 6,0);
	\draw[->, -{Stealth[]}, color=black] ( 6,0) -- ( 7.5,0);
	\draw[->, -{Stealth[]}, color=black] ( 7.5,0) -- ( 9,0);
	\draw[->, -{Stealth[]}, color=black] ( 0,-1.5) -- ( 1.5,-1.5);
	\draw[->, -{Stealth[]}, color=black] ( 1.5,-1.5) -- ( 3,-1.5);
	\draw[->, -{Stealth[]}, color=black] ( 3,-1.5) -- ( 4.5,-1.5);
	\draw[->, -{Stealth[]}, color=black] ( 4.5,-1.5) -- ( 6,-1.5);
	\draw[->, -{Stealth[]}, color=black] ( 6,-1.5) -- ( 7.5,-1.5);
	\draw[->, -{Stealth[]}, color=black] ( 7.5,-1.5) -- ( 9,-1.5);
	\draw[->, -{Stealth[]}, color=black] ( 0,-3) -- ( 1.5,-3);
	\draw[->, -{Stealth[]}, color=black] ( 1.5,-3) -- ( 3,-3);
	\draw[->, -{Stealth[]}, color=black] ( 3,-3) -- ( 4.5,-3);
	\draw[->, -{Stealth[]}, color=black] ( 4.5,-3) -- ( 6,-3);
	\draw[->, -{Stealth[]}, color=black] ( 6,-3) -- ( 7.3,-3);
	\draw[->, -{Stealth[]}, color=black] ( 7.7,-3) -- ( 9,-3);
	\draw[->, -{Stealth[]}, color=black] ( 0,-4.5) -- ( 1.3,-4.5);
	\draw[->, -{Stealth[]}, color=black] ( 1.7,-4.5) -- ( 3,-4.5);
	\draw[->, -{Stealth[]}, color=black] ( 3,-4.5) -- ( 4.5,-4.5);
	\draw[->, -{Stealth[]}, color=black] ( 4.5,-4.5) -- ( 6,-4.5);
	\draw[->, -{Stealth[]}, color=black] ( 6,-4.5) -- ( 7.3,-4.5);
	\draw[->, -{Stealth[]}, color=black] ( 7.7,-4.5) -- ( 9,-4.5);
	\draw[->, -{Stealth[]}, color=black] ( 0,-6) -- ( 1.5,-6);
	\draw[->, -{Stealth[]}, color=black] ( 1.5,-6) -- ( 3,-6);
	\draw[->, -{Stealth[]}, color=black] ( 3,-6) -- ( 4.5,-6);
	\draw[->, -{Stealth[]}, color=black] ( 4.5,-6) -- ( 6,-6);
	\draw[->, -{Stealth[]}, color=black] ( 6,-6) -- ( 7.5,-6);
	\draw[->, -{Stealth[]}, color=black] ( 7.5,-6) -- ( 9,-6);
	\draw[->, -{Stealth[]}, color=black] ( 0,-7.5) -- ( 1.5,-7.5);
	\draw[->, -{Stealth[]}, color=black] ( 1.5,-7.5) -- ( 3,-7.5);
	\draw[->, -{Stealth[]}, color=black] ( 3,-7.5) -- ( 4.5,-7.5);
	\draw[->, -{Stealth[]}, color=black] ( 4.5,-7.5) -- ( 6,-7.5);
	\draw[->, -{Stealth[]}, color=black] ( 6,-7.5) -- ( 7.5,-7.5);
	\draw[->, -{Stealth[]}, color=black] ( 7.5,-7.5) -- ( 9,-7.5);
	\end{tikzpicture}
\end{figure}
\vspace{-.65cm}The main result of this chapter shows that the quiver $\Gamma_I$ is always symmetric.
That is, the number of arrows from $i$ to $j$ equals the number of arrows from $j$ to $i$. 
\begin{theorem}\label{thm:number_arrows_0_to_i_equals_reverse}
	For any $I \subseteq (\Delta_{\aff})_{0}$, $r_{i, j} = r_{j, i}$ in the quiver of $\GammaI{I}$.
\end{theorem}
\begin{proof}
	If $i = j$, then there is nothing to prove as $r_{i,i} = r_{i,i}$.
	So, assume $i \neq j$.
	Then the number of arrows $r_{i,j}$ from vertex $i$ to vertex $j$ in the quiver of $\GammaI{I}$ can be calculated via the knitting algorithm starting in the leftmost column and moving right.
	In exactly the same way, the number of arrows $r_{j,i}$ from vertex $j$ to vertex $i$ in the quiver of $\GammaI{I}$ can be calculated via the knitting algorithm starting in the rightmost column and moving left. 
	As explained above (see, e.g. \cite{Aus86}), the AR-quiver of $\ADE$ surface singularities coincides with the McKay quiver, which, in Figure~\ref{fig:repetition_quivers} is visibly symmetric.
	Hence, $r_{i,j}$ can be obtained from $r_{j,i}$ by reflecting in the line through the original boxed vertex, and vice versa.
	Therefore, $r_{i,j} = r_{j,i}$ for all $i, j \in (\Delta_{\aff})_{0}$ in the quiver of $\GammaI{I}$.
\end{proof}

The final result, Theorem~\ref{thm:number_arrows_0_1_2}, together with intersection theory, is sufficient to obtain the full quiver of $\GammaI{I}$.
However, we refrain from stating the quivers explicitly since in \S\ref{sec:isolated_cdv_with_NCCR} we are only concerned with the number of arrows in to and out of vertex $0$ in the quiver of $\GammaI{I}$.
To state the final result of this chapter, we require the following lemma.
\begin{lemma}\label{lem:typeA_arrows_0_1_2}
	Let $\Delta = \tAn{n}$, and consider $\Delta_{\aff}$.
	Then, for any $I \subseteq (\Delta_{\aff})_{0}$, the quiver of $\Gamma_{I}$ is the double of the extended Dynkin diagram of type $\tA$$_{|I|-1}$ (the extended diagram has $|I|$ vertices), possibly with some loops.
\end{lemma}
\begin{proof}
	Write $I = \{ i_{0}, i_{1}, \hdots, i_{s}, \hdots, i_{m} \}$ where $i_{0} < i_{1} < \cdots < i_{s} < \cdots < i_{m}$.
	In this calculation, we wish to determine the number of arrows in to and out of vertex $i_{s} \in I$.
	Below we will consider subscripts mod $m+1$, so $i_{m+1} = i_{0}$, etc.
	For ease of reading, we omit all arrows in the translation quiver.
	The result of Steps~\eqref{knitting_algorithm step_1}--\eqref{knitting_algorithm step_5} of the knitting algorithm is as follows:
	\vspace{-.85cm}\begin{figure}[H]
		\centering
		\begin{tikzpicture}[scale=0.85]
		\node at ( 7,0) [] {\footnotesize{$i_{s-1}$}};
		\node at ( -1.8,-2) [] {\footnotesize{$i_{s+1}$}};
		
		\draw[color=red, densely dotted] ( 4.3,.3) to [out=80,in=90] ( 6.8,.3);
		\draw[color=red, densely dotted] ( -.3,-4.3) to [out=190,in=190] ( -2.2,-2);
		
		\node at ( 0,0) [rectangle, draw] {\footnotesize{$1$}};
		\node at ( 0,-1) [] {\footnotesize{$1$}};
		\node at ( 0,-2) [] {\footnotesize{$1$}};
		\node at ( 0,-3) [] {\footnotesize{\vdots}};
		\node at ( 0,-4) [circle,draw] {\footnotesize{$1$}};
		\node at ( 0,-5) [] {\footnotesize{$0$}};
		
		\node at ( 1,0) [] {\footnotesize{$1$}};
		\node at ( 2,0) [] {\footnotesize{$1$}};
		\node at ( 3,0) [] {\footnotesize{$\cdots$}};
		\node at ( 4,0) [circle, draw] {\footnotesize{$1$}};
		\node at ( 5,0) [] {\footnotesize{$0$}};
		
		\node at ( 1,-1) [circle, draw] {\footnotesize{$b$}};
		\node at ( 2,-1) [] {\footnotesize{$0$}};
		\node at ( 3,-1) [] {\footnotesize{$\cdots$}};
		\node at ( 4,-1) [] {\footnotesize{$-1$}};
		
		\node at ( 1,-2) [] {\footnotesize{$0$}};
		\node at ( 1,-3) [] {\footnotesize{$\vdots$}};
		\node at ( 1,-4) [] {\footnotesize{$-1$}};
		
		\node at ( 2,-2) [] {\footnotesize{$0$}};
		
		\draw[-,color=blue] ( 4.5,0.5) -- ( 4.5,-0.5);
		\draw[-,color=blue] ( 1.5,-0.5) -- ( 4.5,-0.5);
		\draw[-,color=blue] ( 1.5,-1.5) -- ( 1.5,-0.5);
		\draw[-,color=blue] ( 0.5,-1.5) -- ( 1.5,-1.5);
		\draw[-,color=blue] ( 0.5,-4.5) -- ( 0.5,-1.5);
		\draw[-,color=blue] ( -0.5,-4.5) -- ( 0.5, -4.5);
	\end{tikzpicture}
	\end{figure}
	\vspace{-.45cm}\noindent where $-1 \le b \le 1$.
	Indeed, regardless of which vertices are in $I$, the first time a circled value of $1$ will occur will be at vertex $i_{s-1}$ or vertex $i_{s+1}$ (where the subscripts are taken mod $m+1$).
	The values at all vertices from there on become $0$ or $-1$, as indicated by the blue lines above.
	In fact, if $|I| > 3$ then the only vertices which end up potentially having arrows coming from vertex $i_{s}$ in the quiver of $\GammaI{I}$ are $i_{s-1}$, $i_{s}$, and $i_{s+1}$.
	Hence, for any $i_{q} \in I$ where $i_{q}$ is distinct from $i_{s-1}, i_{s},$ and $i_{s+1}$, there are no arrows from vertex $i_{s}$ to vertex $i_{q}$ in the quiver of $\GammaI{I}$.
	Thus we need only consider the vertices $i_{s-1}, i_{s},$ and $i_{s+1}$.
	
	By inspection, if $i_{s-1} = i_{s+1}$ (which can happen if $s=0$ or $s=m$ since the subscripts are mod $m+1$) then there are two arrows from vertex $i_{s}$ to vertex $i_{s-1} = i_{s+1}$ in the quiver of $\GammaI{I}$ and potentially one loop at vertex $i_{s}$ in the quiver of $\GammaI{I}$.
	While, if $i_{s-1} \neq i_{s+1}$ then, in the quiver of $\GammaI{I}$, there is one arrow from vertex $i_{s}$ to vertex $i_{s-1}$, one arrow from vertex $i_{s}$ to vertex $i_{s+1}$, and potentially one loop at vertex $i_{s}$.
	Hence, the quiver is as stated.
\end{proof}

This result is best understood through an example.
\begin{examples}
	In this example we look at the two situations mentioned in the proof.
	The first is when $i_{s-1} = i_{s+1}$ and the second is when $i_{s-1} \neq i_{s+1}$.
	\begin{enumerate}
		\item To see what occurs when $i_{s-1} = i_{s+1}$ consider the AR-quiver $\teAn{4}$.
		Suppose $I_{1} = \{ 0, 1 \}$ and $I_{2} = \{ 0, 2 \}$, where in both cases $i_{s} = 0$, and $i_{s-1} = i_{s+1}$ is either $1$ or $2$ respectively.
		After running the knitting algorithm, we obtain the following.
		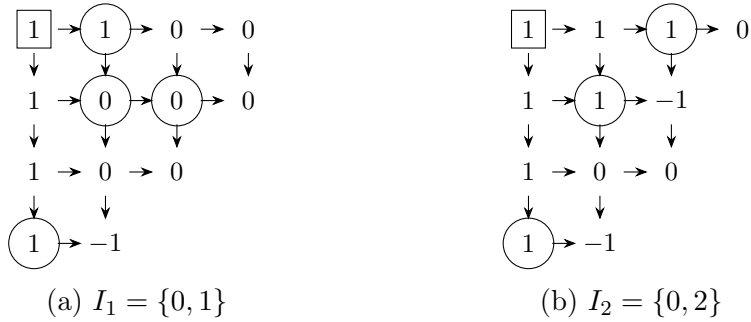
\begin{figure}[H]
			\centering
			\begin{subfigure}[b]{.4\textwidth}
				\centering
				\begin{tikzpicture}[->, shorten >=9pt, shorten <=9pt, scale=.95]
				\node at ( 0,0) [rectangle, draw] {\footnotesize{$1$}};
				\node at ( 0,-1) [] {\footnotesize{$1$}};
				\node at ( 0,-2) [] {\footnotesize{$1$}};
				\node at ( 0,-3) [circle,draw] {\footnotesize{$1$}};
				
				\node at ( 1,0) [circle, draw] {\footnotesize{$1$}};
				\node at ( 2,0) [] {\footnotesize{$0$}};
				\node at ( 3,0) [] {\footnotesize{$0$}};
				
				\node at ( 1,-1) [circle, draw] {\footnotesize{$0$}};
				\node at ( 2,-1) [circle, draw] {\footnotesize{$0$}};
				\node at ( 3,-1) [] {\footnotesize{$0$}};
				
				\node at ( 1,-2) [] {\footnotesize{$0$}};
				\node at ( 1,-3) [] {\footnotesize{$-1$}};
				\node at ( 2,-2) [] {\footnotesize{$0$}};
				
				\draw[->, -{Stealth[]}, color=black] ( 0,0) -- ( 0,-1);
				\draw[->, -{Stealth[]}, color=black] ( 0,-1) -- ( 0,-2);
				\draw[->, -{Stealth[]}, color=black] ( 0,-2) -- ( 0,-3);
				\draw[->, -{Stealth[]}, color=black] ( 1,0) -- ( 1,-1);
				\draw[->, -{Stealth[]}, color=black] ( 1,-1) -- ( 1,-2);
				\draw[->, -{Stealth[]}, color=black] ( 1,-2) -- ( 1,-3);
				\draw[->, -{Stealth[]}, color=black] ( 2,0) -- ( 2,-1);
				\draw[->, -{Stealth[]}, color=black] ( 2,-1) -- ( 2,-2);
				\draw[->, -{Stealth[]}, color=black] ( 3,0) -- ( 3,-1);
				
				\draw[->, -{Stealth[]}, color=black] ( 0,0) -- ( 1,0);
				\draw[->, -{Stealth[]}, color=black] ( 1,0) -- ( 2,0);
				\draw[->, -{Stealth[]}, color=black] ( 2,0) -- ( 3,0);
				\draw[->, -{Stealth[]}, color=black] ( 0,-1) -- ( 1,-1);
				\draw[->, -{Stealth[]}, color=black] ( 1,-1) -- ( 2,-1);
				\draw[->, -{Stealth[]}, color=black] ( 2,-1) -- ( 3,-1);
				\draw[->, -{Stealth[]}, color=black] ( 0,-2) -- ( 1,-2);
				\draw[->, -{Stealth[]}, color=black] ( 1,-2) -- ( 2,-2);
				\draw[->, -{Stealth[]}, color=black] ( 0,-3) -- ( 1,-3);
				\end{tikzpicture}\subcaption{$I_{1} = \{ 0, 1 \}$}\label{fig:circled_values_A_4_01_example}
			\end{subfigure}
			\begin{subfigure}[b]{.4\textwidth}
				\centering
				\begin{tikzpicture}[->, shorten >=9pt, shorten <=9pt, scale=.95]
				\node at ( 0,0) [rectangle, draw] {\footnotesize{$1$}};
				\node at ( 0,-1) [] {\footnotesize{$1$}};
				\node at ( 0,-2) [] {\footnotesize{$1$}};
				\node at ( 0,-3) [circle,draw] {\footnotesize{$1$}};
				
				\node at ( 1,0) [] {\footnotesize{$1$}};
				\node at ( 2,0) [circle, draw] {\footnotesize{$1$}};
				\node at ( 3,0) [] {\footnotesize{$0$}};
				
				\node at ( 1,-1) [circle, draw] {\footnotesize{$1$}};
				\node at ( 2,-1) [] {\footnotesize{$-1$}};
				
				\node at ( 1,-2) [] {\footnotesize{$0$}};
				\node at ( 1,-3) [] {\footnotesize{$-1$}};
				\node at ( 2,-2) [] {\footnotesize{$0$}};
				
				\draw[->, -{Stealth[]}, color=black] ( 0,0) -- ( 0,-1);
				\draw[->, -{Stealth[]}, color=black] ( 0,-1) -- ( 0,-2);
				\draw[->, -{Stealth[]}, color=black] ( 0,-2) -- ( 0,-3);
				\draw[->, -{Stealth[]}, color=black] ( 1,0) -- ( 1,-1);
				\draw[->, -{Stealth[]}, color=black] ( 1,-1) -- ( 1,-2);
				\draw[->, -{Stealth[]}, color=black] ( 1,-2) -- ( 1,-3);
				\draw[->, -{Stealth[]}, color=black] ( 2,0) -- ( 2,-1);
				\draw[->, -{Stealth[]}, color=black] ( 2,-1) -- ( 2,-2);
				
				\draw[->, -{Stealth[]}, color=black] ( 0,0) -- ( 1,0);
				\draw[->, -{Stealth[]}, color=black] ( 1,0) -- ( 2,0);
				\draw[->, -{Stealth[]}, color=black] ( 2,0) -- ( 3,0);
				\draw[->, -{Stealth[]}, color=black] ( 0,-1) -- ( 1,-1);
				\draw[->, -{Stealth[]}, color=black] ( 1,-1) -- ( 2,-1);
				\draw[->, -{Stealth[]}, color=black] ( 0,-2) -- ( 1,-2);
				\draw[->, -{Stealth[]}, color=black] ( 1,-2) -- ( 2,-2);
				\draw[->, -{Stealth[]}, color=black] ( 0,-3) -- ( 1,-3);
				\end{tikzpicture}\subcaption{$I_{2} = \{ 0, 2 \}$}\label{fig:circled_values_A_4_02_example}
			\end{subfigure}\caption{Circled values for varying $I$ in $A_{4}$}\label{fig:circled_values_A4_example}
		\end{figure}
		As demonstrated by Figure~\ref{fig:circled_values_A_4_01_example}, there are two arrows from vertex $0$ to vertex $1$ in the quiver of $\GammaI{I_{1}}$ and no loops.
		As for the quiver of $\GammaI{I_{2}}$, from Figure~\ref{fig:circled_values_A_4_02_example}, there are two arrows from vertex $0$ to vertex $2$ in the quiver of $\GammaI{I_{2}}$ and one loop at vertex $0$.
		
		\item To see what occurs when $i_{s-1} \neq i_{s+1}$ consider the AR-quiver $\teAn{5}$.
		Suppose $I_{3} = \{ 0, 1, 4 \}$ and $I_{4} = \{ 0, 2, 4 \}$.
		In both cases $i_{s} = 0$ and $i_{s+1} = 4$, and $i_{s-1}$ is either $1$ or $2$ respectively.
		Again, the knitting algorithm gives the following.
		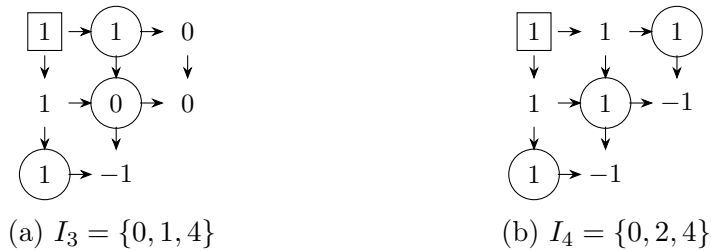
\begin{figure}[H]
			\centering
			\begin{subfigure}[b]{.4\textwidth}
				\centering
				\begin{tikzpicture}[->, shorten >=9pt, shorten <=9pt, scale=.95]
				\node at ( 0,0) [rectangle, draw] {\footnotesize{$1$}};
				\node at ( 0,-1) [] {\footnotesize{$1$}};
				\node at ( 0,-2) [circle,draw] {\footnotesize{$1$}};
				
				\node at ( 1,0) [circle, draw] {\footnotesize{$1$}};
				\node at ( 2,0) [] {\footnotesize{$0$}};
				
				\node at ( 1,-1) [circle, draw] {\footnotesize{$0$}};
				\node at ( 2,-1) [] {\footnotesize{$0$}};
				
				\node at ( 1,-2) [] {\footnotesize{$-1$}};
				
				\draw[->, -{Stealth[]}, color=black] ( 0,0) -- ( 0,-1);
				\draw[->, -{Stealth[]}, color=black] ( 0,-1) -- ( 0,-2);
				\draw[->, -{Stealth[]}, color=black] ( 1,0) -- ( 1,-1);
				\draw[->, -{Stealth[]}, color=black] ( 1,-1) -- ( 1,-2);
				\draw[->, -{Stealth[]}, color=black] ( 2,0) -- ( 2,-1);
				
				\draw[->, -{Stealth[]}, color=black] ( 0,0) -- ( 1,0);
				\draw[->, -{Stealth[]}, color=black] ( 1,0) -- ( 2,0);
				\draw[->, -{Stealth[]}, color=black] ( 0,-1) -- ( 1,-1);
				\draw[->, -{Stealth[]}, color=black] ( 1,-1) -- ( 2,-1);
				\draw[->, -{Stealth[]}, color=black] ( 0,-2) -- ( 1,-2);
				\end{tikzpicture}\subcaption{$I_{3} = \{ 0, 1, 4 \}$}\label{fig:circled_values_A_5_014_example}
			\end{subfigure}
			\begin{subfigure}[b]{.4\textwidth}
				\centering
				\begin{tikzpicture}[->, shorten >=9pt, shorten <=9pt, scale=.95]
				\node at ( 0,0) [rectangle, draw] {\footnotesize{$1$}};
				\node at ( 0,-1) [] {\footnotesize{$1$}};
				\node at ( 0,-2) [circle,draw] {\footnotesize{$1$}};
				
				\node at ( 1,0) [] {\footnotesize{$1$}};
				\node at ( 2,0) [circle, draw] {\footnotesize{$1$}};
				
				\node at ( 1,-1) [circle, draw] {\footnotesize{$1$}};
				\node at ( 2,-1) [] {\footnotesize{$-1$}};
				
				\node at ( 1,-2) [] {\footnotesize{$-1$}};
				
				\draw[->, -{Stealth[]}, color=black] ( 0,0) -- ( 0,-1);
				\draw[->, -{Stealth[]}, color=black] ( 0,-1) -- ( 0,-2);
				\draw[->, -{Stealth[]}, color=black] ( 1,0) -- ( 1,-1);
				\draw[->, -{Stealth[]}, color=black] ( 1,-1) -- ( 1,-2);
				\draw[->, -{Stealth[]}, color=black] ( 2,0) -- ( 2,-1);
				
				\draw[->, -{Stealth[]}, color=black] ( 0,0) -- ( 1,0);
				\draw[->, -{Stealth[]}, color=black] ( 1,0) -- ( 2,0);
				\draw[->, -{Stealth[]}, color=black] ( 0,-1) -- ( 1,-1);
				\draw[->, -{Stealth[]}, color=black] ( 1,-1) -- ( 2,-1);
				\draw[->, -{Stealth[]}, color=black] ( 0,-2) -- ( 1,-2);
				\end{tikzpicture}\subcaption{$I_{4} = \{ 0, 2, 4 \}$}\label{fig:circled_values_A_5_024_example}
			\end{subfigure}\caption{Circled values for varying $I$ in $A_{5}$}\label{fig:circled_values_A5_example}
		\end{figure}
		From Figure~\ref{fig:circled_values_A_5_014_example}, there is one arrow from vertex $0$ to vertex $1$ and one arrow from vertex $0$ to vertex $4$ in the quiver of $\GammaI{I_{3}}$ and there are no loops.
		As for the quiver of $\GammaI{I_{4}}$, from Figure~\ref{fig:circled_values_A_5_024_example}, there is one arrow from vertex $0$ to vertex $2$, one arrow from vertex $0$ to vertex $4$, and one loop at vertex $0$.
	\end{enumerate}
\end{examples}

Now, we state the final result.
Recall that extended Dynkin diagrams are McKay quivers, and are often drawn with the dimension $\updelta_{j}$ of the corresponding irreducible representation (see \S\ref{sec:intro:mckay_correspondence}).
\begin{theorem}\label{thm:number_arrows_0_1_2}
	Let $I \subseteq (\Delta_{\aff})_{0}$ and suppose that $j \in (\Delta_{\aff})_{0}$.
	Then the number of arrows $r_{0, j}$ from vertex $0$ to vertex $j$ in the quiver of $\GammaI{I}$ is either $0$, $1$, or $2$. 
	Furthermore, $r_{0,j} = 2$ if and only if $I = \{ 0, j \}$ and $\updelta_{j} = 1$.
\end{theorem}
\begin{proof}
We prove this via case-by-case analysis, where Theorem~\ref{thm:number_arrows_0_to_i_equals_reverse} halves the amount of work necessary, allowing us to determine only the number of arrows out of vertex $0$ in the quiver of $\GammaI{I}$.
In addition, in each of the following cases several vertices are interchangeable, which, again, lessens the amount of work required.
We will note where this occurs.
In the knitting computations below, formally, we say that vertices below vertex $i$ are `cut off' when the algorithm produces $0$'s in all rows below vertex $i$.

\begin{enumerate}
\vspace{-.5cm}\item 
Consider the AR-quiver $\teAn{n}$ as in Figure~\ref{fig:repetition_quivers_A_n}.
This case is covered by Lemma~\ref{lem:typeA_arrows_0_1_2}.

\item Consider the AR-quiver $\teDn{4}$ as in Figure~\ref{fig:repetition_quivers_D_4}.
We consider all subsets $I \subseteq (\Delta_{\aff})_{0} = \{ 0, 1, 2, 3, 4 \}$ which contain $0$.
Notice that vertices $1$, $3$, and $4$ are effectively interchangeable so that $I = \{ 0, 1 \}$ gives the same circled values in the knitting calculation as $I = \{ 0, 3 \}$ and $I = \{ 0, 4 \}$.
Furthermore, $I = \{ 0, 1, 3 \}$ gives the same circled values in the knitting calculation as $I = \{ 0, 1, 4 \}$.
In addition, if $2 \in I$ then all vertices other than vertex $0$ in the knitting calculation are cut off, and so we get the same circled values for $I = \{ 0, 2, \bullet \}$ where $\bullet$ is any list of numbers $1, 3,$ and $4$.\\
\vspace{-.3cm}
\begin{figure}[H]
	\centering
	\begin{minipage}[b]{.4\textwidth}
		\centering
		\begin{tikzpicture}[->,shorten >=9pt, shorten <=9pt, scale=0.9]
		\node at ( 0,0) [rectangle, draw] {\footnotesize{$1$}};
		\node at ( 0,-1) [circle, draw] {\footnotesize{$0$}};
		\node at ( 0,-2) [] {\footnotesize{$0$}};
		\node at ( 0,-3) [] {\footnotesize{$0$}};
		\node at ( 1,-1.5) [] {\footnotesize{$1$}};
		\draw[->, -{Stealth[]}, color=black] ( 0,0) -- ( 1,-1.5);
		\draw[->, -{Stealth[]}, color=black] ( 0,-1) -- ( 1,-1.5);
		\draw[->, -{Stealth[]}, color=black] ( 0,-2) -- ( 1,-1.5);
		\draw[->, -{Stealth[]}, color=black] ( 0,-3) -- ( 1,-1.5);
		\draw[->, -{Stealth[]}, color=black] ( 1,-1.5) -- ( 2,0);
		\draw[->, -{Stealth[]}, color=black] ( 1,-1.5) -- ( 2,-1);
		\draw[->, -{Stealth[]}, color=black] ( 1,-1.5) -- ( 2,-2);
		\draw[->, -{Stealth[]}, color=black] ( 1,-1.5) -- ( 2,-3);
		\node at ( 2,0) [circle, draw] {\footnotesize{$0$}};
		\node at ( 2,-1) [circle, draw] {\footnotesize{$1$}};
		\node at ( 2,-2) [] {\footnotesize{$1$}};
		\node at ( 2,-3) [] {\footnotesize{$1$}};
		\node at ( 3,-1.5) [] {\footnotesize{$1$}};
		\draw[->, -{Stealth[]}, color=black] ( 2,0) -- ( 3,-1.5);
		\draw[->, -{Stealth[]}, color=black] ( 2,-1) -- ( 3,-1.5);
		\draw[->, -{Stealth[]}, color=black] ( 2,-2) -- ( 3,-1.5);
		\draw[->, -{Stealth[]}, color=black] ( 2,-3) -- ( 3,-1.5); 
		\draw[->, -{Stealth[]}, color=black] ( 3,-1.5) -- ( 4,0);
		\draw[->, -{Stealth[]}, color=black] ( 3,-1.5) -- ( 4,-1);
		\draw[->, -{Stealth[]}, color=black] ( 3,-1.5) -- ( 4,-2);
		\draw[->, -{Stealth[]}, color=black] ( 3,-1.5) -- ( 4,-3);
		\node at ( 4,0) [circle, draw] {\footnotesize{$1$}};
		\node at ( 4,-1) [circle, draw] {\footnotesize{$1$}};
		\node at ( 4,-2) [] {\footnotesize{$0$}};
		\node at ( 4,-3) [] {\footnotesize{$0$}};
		\node at ( 5,-1.5) [] {\footnotesize{$-1$}};
		\draw[->, -{Stealth[]}, color=black] ( 4,0) -- ( 5,-1.5);
		\draw[->, -{Stealth[]}, color=black] ( 4,-1) -- ( 5,-1.5);
		\draw[->, -{Stealth[]}, color=black] ( 4,-2) -- ( 5,-1.5);
		\draw[->, -{Stealth[]}, color=black] ( 4,-3) -- ( 5,-1.5); 
		\end{tikzpicture}\subcaption{$I=\{ 0, 1 \}$}\label{subfig:D_4_0_1}
	\end{minipage}
	\begin{minipage}[b]{.4\textwidth}\vspace{1cm}
		\centering
		\begin{tikzpicture}[->,shorten >=9pt, shorten <=9pt, scale=0.9]
			\node at ( 0,0) [rectangle, draw] {\footnotesize{$1$}};
			\node at ( 0,-1) [] {\footnotesize{$0$}};
			\node at ( 0,-2) [] {\footnotesize{$0$}};
			\node at ( 0,-3) [] {\footnotesize{$0$}};
			\node at ( 1,-1.5) [circle, draw] {\footnotesize{$1$}};
			\draw[->, -{Stealth[]}, color=black] ( 0,0) -- ( 1,-1.5);
			\draw[->, -{Stealth[]}, color=black] ( 0,-1) -- ( 1,-1.5);
			\draw[->, -{Stealth[]}, color=black] ( 0,-2) -- ( 1,-1.5);
			\draw[->, -{Stealth[]}, color=black] ( 0,-3) -- ( 1,-1.5);
			\draw[->, -{Stealth[]}, color=black] ( 1,-1.5) -- ( 2,0);
			\draw[->, -{Stealth[]}, color=black] ( 1,-1.5) -- ( 2,-1);
			\draw[->, -{Stealth[]}, color=black] ( 1,-1.5) -- ( 2,-2);
			\draw[->, -{Stealth[]}, color=black] ( 1,-1.5) -- ( 2,-3);
			\node at ( 2,0) [] {\footnotesize{$-1$}};
			\node at ( 2,-1) [] {\footnotesize{$0$}};
			\node at ( 2,-2) [] {\footnotesize{$0$}};
			\node at ( 2,-3) [] {\footnotesize{$0$}};
		\end{tikzpicture}\subcaption{$I = \{ 0, 2, \bullet \}$}\label{subfig:D_4_0_2}
	\end{minipage}
\end{figure}
\begin{figure}[H]\ContinuedFloat
	\centering
	\begin{minipage}[b]{.4\textwidth}
		\centering
		\begin{tikzpicture}[->,shorten >=9pt, shorten <=9pt, scale=0.9]
			\node at ( 0,0) [rectangle, draw] {\footnotesize{$1$}};
			\node at ( 0,-1) [circle, draw] {\footnotesize{$0$}};
			\node at ( 0,-2) [circle, draw] {\footnotesize{$0$}};
			\node at ( 0,-3) [] {\footnotesize{$0$}};
			\node at ( 1,-1.5) [] {\footnotesize{$1$}};
			\draw[->, -{Stealth[]}, color=black] ( 0,0) -- ( 1,-1.5);
			\draw[->, -{Stealth[]}, color=black] ( 0,-1) -- ( 1,-1.5);
			\draw[->, -{Stealth[]}, color=black] ( 0,-2) -- ( 1,-1.5);
			\draw[->, -{Stealth[]}, color=black] ( 0,-3) -- ( 1,-1.5);
			\draw[->, -{Stealth[]}, color=black] ( 1,-1.5) -- ( 2,0);
			\draw[->, -{Stealth[]}, color=black] ( 1,-1.5) -- ( 2,-1);
			\draw[->, -{Stealth[]}, color=black] ( 1,-1.5) -- ( 2,-2);
			\draw[->, -{Stealth[]}, color=black] ( 1,-1.5) -- ( 2,-3);
			\node at ( 2,0) [circle, draw] {\footnotesize{$0$}};
			\node at ( 2,-1) [circle, draw] {\footnotesize{$1$}};
			\node at ( 2,-2) [circle, draw] {\footnotesize{$1$}};
			\node at ( 2,-3) [] {\footnotesize{$1$}};
			\node at ( 3,-1.5) [] {\footnotesize{$0$}};
			\draw[->, -{Stealth[]}, color=black] ( 2,0) -- ( 3,-1.5);
			\draw[->, -{Stealth[]}, color=black] ( 2,-1) -- ( 3,-1.5);
			\draw[->, -{Stealth[]}, color=black] ( 2,-2) -- ( 3,-1.5);
			\draw[->, -{Stealth[]}, color=black] ( 2,-3) -- ( 3,-1.5); 
			\draw[->, -{Stealth[]}, color=black] ( 3,-1.5) -- ( 4,0);
			\draw[->, -{Stealth[]}, color=black] ( 3,-1.5) -- ( 4,-1);
			\draw[->, -{Stealth[]}, color=black] ( 3,-1.5) -- ( 4,-2);
			\draw[->, -{Stealth[]}, color=black] ( 3,-1.5) -- ( 4,-3);
			\node at ( 4,0) [circle, draw] {\footnotesize{$0$}};
			\node at ( 4,-1) [circle, draw] {\footnotesize{$0$}};
			\node at ( 4,-2) [circle, draw] {\footnotesize{$0$}};
			\node at ( 4,-3) [] {\footnotesize{$-1$}};
		\end{tikzpicture}\subcaption{$I=\{ 0, 1, 3 \}$}\label{subfig:D_4_0_1_3}
	\end{minipage}
	\begin{minipage}[b]{.4\textwidth}
		\centering
		\begin{tikzpicture}[->,shorten >=9pt, shorten <=9pt, scale=0.9]
			\node at ( 0,0) [rectangle, draw] {\footnotesize{$1$}};
			\node at ( 0,-1) [circle, draw] {\footnotesize{$0$}};
			\node at ( 0,-2) [circle, draw] {\footnotesize{$0$}};
			\node at ( 0,-3) [circle, draw] {\footnotesize{$0$}};
			\node at ( 1,-1.5) [] {\footnotesize{$1$}};
			\draw[->, -{Stealth[]}, color=black] ( 0,0) -- ( 1,-1.5);
			\draw[->, -{Stealth[]}, color=black] ( 0,-1) -- ( 1,-1.5);
			\draw[->, -{Stealth[]}, color=black] ( 0,-2) -- ( 1,-1.5);
			\draw[->, -{Stealth[]}, color=black] ( 0,-3) -- ( 1,-1.5);
			\draw[->, -{Stealth[]}, color=black] ( 1,-1.5) -- ( 2,0);
			\draw[->, -{Stealth[]}, color=black] ( 1,-1.5) -- ( 2,-1);
			\draw[->, -{Stealth[]}, color=black] ( 1,-1.5) -- ( 2,-2);
			\draw[->, -{Stealth[]}, color=black] ( 1,-1.5) -- ( 2,-3);
			\node at ( 2,0) [circle, draw] {\footnotesize{$0$}};
			\node at ( 2,-1) [circle, draw] {\footnotesize{$1$}};
			\node at ( 2,-2) [circle, draw] {\footnotesize{$1$}};
			\node at ( 2,-3) [circle, draw] {\footnotesize{$1$}};
			\node at ( 3,-1.5) [] {\footnotesize{$-1$}};
			\draw[->, -{Stealth[]}, color=black] ( 2,0) -- ( 3,-1.5);
			\draw[->, -{Stealth[]}, color=black] ( 2,-1) -- ( 3,-1.5);
			\draw[->, -{Stealth[]}, color=black] ( 2,-2) -- ( 3,-1.5);
			\draw[->, -{Stealth[]}, color=black] ( 2,-3) -- ( 3,-1.5); 
		\end{tikzpicture}\subcaption{$I=\{ 0, 1, 3, 4 \}$}\label{subfig:D_4_0_1_3_4}
	\end{minipage}\caption{Circled values for varying $I$ for $\teDn{4}$}\label{fig:circled_values_D_4}
\end{figure}

From Figure~\ref{subfig:D_4_0_1} it follows that for $I = \{0, 1 \}$ there are two arrows from vertex $0$ to vertex $1$ in the quiver of $\GammaI{I}$ (i.e. $r_{0, 1} = 2$) and one arrow from vertex $0$ to itself.

We collate this information in the table below, where the first column corresponds to the subset $I \subseteq (\Delta_{\aff})_{0}$ and the first row to vertices $j \in (\Delta_{\aff})_{0}$.
The values in the table are the number of arrows $r_{0, j}$ from vertex $0$ to vertex $j$ in the quiver of $\GammaI{I}$.
In particular, a value of zero has three different implications.
The first, denoted $0$, means $j \in I$ and $r_{0, j} = 0$.
The second, denoted $\redzero$, means $j \notin I$ so that $r_{0, j} = 0$.
The third, denoted $\greenzero$, means $j$ is possibly in $I$ and, if $j \in I$, then $r_{0, j} = 0$.
\begin{table}[H]
	\centering
	\begin{tabular}{|c|c|c|c|c|c|}
		\hline
		& $0$ & $1$ & $2$ & $3$ & $4$ \\
		\hline
		$\{ 0,1 \}$ & 1 & 2 & $\redzero$ & $\redzero$ & $\redzero$\\
		$\{ 0,2, \bullet \}$ & 0 & $\greenzero$ & 1 & $\greenzero$ & $\greenzero$ \\
		$\{ 0,1,3 \}$ & 0 & 1 & $\redzero$ & 1 & $\redzero$ \\
		$\{ 0,1,3,4 \}$ & 0 & 1 & $\redzero$ & 1 & 1 \\
		\hline
	\end{tabular}
\end{table}
Hence, the number of arrows from vertex $0$ to any vertex $j \in (\Delta_{\aff})_{0}$ in the quiver of $\GammaI{I}$ is always $0, 1,$ or $2$.

\item Consider the AR-quiver $\teDn{5}$ with vertices labelled as in Figure~\ref{fig:repetition_quivers_D_n_odd}.
We consider all subsets $I \subseteq (\Delta_{\aff})_{0} = \{ 0, 1, 2, 3, 4, 5 \}$ which contain $0$.
If $2 \in I$ then all vertices other than vertex $0$ in the knitting calculation are cut off and so we get the same circled values in the knitting calculation for $I = \{ 0, 2, \bullet \}$ where $\bullet$ is any list of numbers $1, 3, 4,$ and $5$.
If $3 \in I$ then vertices $4$ and $5$ in the knitting calculation are cut off.
Thus, the circled values for $I = \{ 0, 3, \bullet \}$ where $\bullet$ is any list of numbers $4$ and $5$ in the knitting calculation are the same.
Similarly, the circled values for $I = \{ 0, 1, 3, \bullet \}$ where $\bullet$ is any list of 4 and 5 are the same.
It follows that we need only consider the specific cases shown in Figure~\ref{fig:circled_values_D_5}.
\newpage
\begin{figure}[H]
	\centering
	\qquad \begin{subfigure}[b]{.4\textwidth}
		\centering
\subcaption{$I = \{ 0, 1 \}$}\label{subfig:D_5_0_1}
	\end{subfigure}
	\qquad \quad \begin{subfigure}[b]{.4\textwidth}
		\centering
		%
\subcaption{$I = \{ 0, 2, \bullet \}$}\label{subfig:D_5_0_2}
	\end{subfigure} \\ \vspace{.2cm} \quad \\
	\begin{subfigure}[b]{.49\textwidth}
		\centering
		%
\subcaption{$I = \{ 0, 3, \bullet \}$}\label{subfig:D_5_0_3}
	\end{subfigure}
	\begin{subfigure}[b]{.49\textwidth}
		\centering
		%
\subcaption{$I = \{ 0, 4 \}$}\label{subfig:D_5_0_4}
	\end{subfigure} \\ \vspace{.2cm} \quad \\
	\begin{subfigure}[b]{.49\textwidth}
		\centering
		%
\subcaption{$I = \{ 0, 5 \}$}\label{subfig:D_5_0_5}
	\end{subfigure}
	\begin{subfigure}[b]{.49\textwidth}
		\centering
			%
\subcaption{$I = \{ 0, 1, 3, \bullet \}$}\label{subfig:D_5_0_1_3}
	\end{subfigure} \\ \vspace{.2cm} \quad \\
	\begin{subfigure}[b]{.49\textwidth}
		\centering
		%
\subcaption{$I = \{ 0, 1, 4 \}$}\label{subfig:D_5_0_1_4}
	\end{subfigure}
	\begin{subfigure}[b]{.49\textwidth}
		\centering
		%
\subcaption{$I = \{ 0, 1, 5 \}$}\label{subfig:D_5_0_1_5}
	\end{subfigure} \\ \vspace{.2cm} \quad \\
	\begin{subfigure}[b]{.49\textwidth}
		\centering
		%
\subcaption{$I = \{ 0, 4, 5 \}$}\label{subfig:D_5_0_4_5}
	\end{subfigure}
	\begin{subfigure}[b]{.49\textwidth}
		\centering
		%
\subcaption{$I = \{ 0, 1, 4, 5 \}$}\label{subfig:D_5_0_1_4_5}
	\end{subfigure}
\caption{Circled values for varying $I$ for $\teDn{5}$}\label{fig:circled_values_D_5}
\end{figure}
\newpage
As before, in the following table the first column corresponds to the subset $I \subseteq (\Delta_{\aff})_{0}$ and the first row to vertices $j \in (\Delta_{\aff})_{0}$.
The values in the table are the number of arrows $r_{0, j}$ from vertex $0$ to vertex $j$ in the quiver of $\GammaI{I}$ where $0$, $\redzero$, and $\greenzero$ are as above.
\begin{table}[H]
	\centering
	\begin{tabular}{|c|c|c|c|c|c|c|}
		\hline
		& $0$ & $1$ & $2$ & $3$ & $4$ & $5$ \\
		\hline
		$\{ 0,1 \}$ & 1 & 2 & $\redzero$ & $\redzero$ & $\redzero$ & $\redzero$ \\
		$\{ 0,2, \bullet \}$ & 0 & $\greenzero$ & 1 & $\greenzero$ & $\greenzero$ & $\greenzero$ \\
		$\{ 0,3, \bullet \}$ & 0 & $\redzero$& $\redzero$ & 1 & $\greenzero$ & $\greenzero$ \\
		$\{ 0,4 \}$ & 1 & $\redzero$ & $\redzero$ & $\redzero$ & 2 & $\redzero$ \\
		$\{ 0,5 \}$ & 1 & $\redzero$ & $\redzero$ & $\redzero$ & $\redzero$ & 2 \\
		$\{ 0,1,3, \bullet \}$ & 0 & 1 & $\redzero$ & 1 & $\greenzero$ & $\greenzero$ \\
		$\{ 0,1,4 \}$ & 0 & 1 & $\redzero$ & $\redzero$ & 1 & $\redzero$ \\
		$\{ 0,1,5 \}$ & 0 & 1 & $\redzero$ & $\redzero$ & $\redzero$ & 1 \\
		$\{ 0,4,5 \}$ & 1 & $\redzero$ & $\redzero$ & $\redzero$ & 1 & 1 \\
		$\{ 0,1,4,5 \}$ & 0 & 1 & $\redzero$ & $\redzero$ & 1 & 1 \\
		\hline
	\end{tabular}
\end{table}
Hence, the number of arrows from vertex $0$ to any vertex $j \in (\Delta_{\aff})_{0}$ in the quiver of $\GammaI{I}$ is always $0, 1,$ or $2$.

\item Consider the AR-quiver $\teDn{6}$ with vertices labelled as in Figure~\ref{fig:repetition_quivers_D_n_even}.
We consider all subsets $I \subseteq (\Delta_{\aff})_{0} = \{ 0, 1, 2, 3, 4, 5, 6 \}$ which contain $0$.
Notice that vertices $5$ and $6$ in the knitting calculation are interchangeable so that, in the knitting calculation, $I = \{ 0, 5 \}$ gives the same circled values as $I = \{ 0, 6 \}$, and $I = \{ 0, 1, 5 \}$ gives the same circled values as $I = \{ 0, 1, 6 \}$.
In addition, if $2 \in I$ then all vertices other than vertex $0$ in the knitting calculation are cut off and so we get the same circled values in the knitting calculation for $I = \{ 0, 2, \bullet \}$ where $\bullet$ is any list of numbers $1, 3, 4, 5,$ and $6$.
We continue on in this way where, if $3 \in I$, then vertices $4$, $5$ and $6$ in the knitting calculation are cut off.
Hence, the circled values are the same for $I = \{ 0, 3, \bullet \}$ where $\bullet$ is any list of numbers $4, 5$ and $6$.
If $4 \in I$ then vertices $5$ and $6$ are cut off, giving the same circled values for $I = \{ 0, 4, \bullet \}$ where $\bullet$ is any list of numbers $5$ and $6$.
Therefore we need only consider the specific cases illustrated in Figure~\ref{fig:circled_values_D_6}, where we have used similar shortened notation to before.
\begin{figure}[H]	
	\begin{subfigure}[b]{\textwidth}
		\centering
\subcaption{$I = \{ 0, 1 \}$}\label{subfig:D_6_0_1}
	\end{subfigure}
\end{figure}
\begin{figure}[H]\ContinuedFloat
	\vspace{.5cm}\begin{subfigure}[b]{.25\textwidth}
		\centering
		%
\subcaption{$I = \{ 0, 2, \bullet \}$}\label{subfig:D_6_0_2}
	\end{subfigure}
	\begin{subfigure}[b]{.34\textwidth}
		\centering
		%
\subcaption{$I = \{ 0, 3, \bullet \}$}\label{subfig:D_6_0_3}
	\end{subfigure}
	\begin{subfigure}[b]{.4\textwidth}
		\centering
		%
\subcaption{$I = \{ 0, 4, \bullet \}$}\label{subfig:D_6_0_4}
	\end{subfigure} \\ \vspace{.2cm} \quad \\
	\begin{subfigure}[b]{.49\textwidth}
		\centering
		%
\subcaption{$I = \{ 0, 5 \}$}\label{subfig:D_6_0_5}
	\end{subfigure}
	\begin{subfigure}[b]{.49\textwidth}
		\centering
		%
\subcaption{$I = \{ 0, 1, 3, \bullet \}$}\label{subfig:D_6_0_1_3}
	\end{subfigure} \\ \vspace{.2cm} \quad \\
	\begin{subfigure}[b]{.49\textwidth}
		\centering
		%
\subcaption{$I = \{ 0, 1, 4, \bullet \}$}\label{subfig:D_6_0_1_4}
	\end{subfigure}
	\begin{subfigure}[b]{.49\textwidth}
		\centering
		%
\subcaption{$I = \{ 0, 1, 5 \}$}\label{subfig:D_6_0_1_5}
	\end{subfigure} \\ \vspace{.2cm} \quad \\
	\begin{subfigure}[b]{.49\textwidth}
		\centering
		%
\subcaption{$I = \{ 0, 5, 6 \}$}\label{subfig:D_6_0_5_6}
	\end{subfigure}
	\begin{subfigure}[b]{.49\textwidth}
		\centering
		%
\subcaption{$I = \{ 0, 1, 5, 6 \}$}\label{subfig:D_6_0_1_5_6}
	\end{subfigure}
\caption{Circled values for varying $I$ in $\teDn{6}$}\label{fig:circled_values_D_6}
\end{figure}
This leads to the following table where $0$, $\redzero$, and $\greenzero$ are as above.
\begin{table}[H]
	\centering
	\begin{tabular}{|c|c|c|c|c|c|c|c|}
		\hline
		& $0$ & $1$ & $2$ & $3$ & $4$ & $5$ & $6$ \\
		\hline
		$\{ 0,1 \}$ & 1 & 2 & $\redzero$ & $\redzero$ & $\redzero$ & $\redzero$ & $\redzero$ \\
		$\{ 0,2, \bullet \}$ & 0 & $\greenzero$ & 1 & $\greenzero$ & $\greenzero$ & $\greenzero$ & $\greenzero$ \\
		$\{ 0,3, \bullet \}$ & 0 & $\redzero$ & $\redzero$ & 1 & $\greenzero$ & $\greenzero$ & $\greenzero$ \\
		$\{ 0,4, \bullet \}$ & 1 & $\redzero$ & $\redzero$ & $\redzero$ & 1 & $\greenzero$ & $\greenzero$ \\
		$\{ 0,5 \}$ & 1 & $\redzero$ & $\redzero$ & $\redzero$ & $\redzero$ & 2 & $\redzero$ \\
		$\{ 0,1,3, \bullet \}$ & 0 & 1 & $\redzero$ & 1 & $\greenzero$ & $\greenzero$ & $\greenzero$ \\
		$\{ 0,1,4, \bullet \}$ & 0 & 1 & $\redzero$ & $\redzero$ & 1 & $\greenzero$ & $\greenzero$ \\
		$\{ 0,1,5 \}$ & 0 & 1 & $\redzero$ & $\redzero$ & $\redzero$ & 1 & $\redzero$ \\
		$\{ 0,5,6 \}$ & 1 & $\redzero$ & $\redzero$ & $\redzero$ & $\redzero$ & 1 & 1 \\
		$\{ 0,1,5,6 \}$ & 0 & 1 & $\redzero$ & $\redzero$ & $\redzero$ & 1 & 1 \\
		\hline
	\end{tabular}
\end{table}
Hence, the number of arrows from vertex $0$ to any vertex $j \in (\Delta_{\aff})_{0}$ in the quiver of $\GammaI{I}$ is always $0, 1,$ or $2$.

\item Now consider the general case of $\teDn{n}$ for ${\sf{n}} \ge 7$.
As previously, the shape of the AR-quiver splits the analysis into the following cases, where the values in the columns are obtained by extending the knitting calculations vertically.
\begin{table}[H]
	\centering
	\begin{tabular}{|c|c|c|c|c|c|ccc|c|c|}
		\hline
		& $0$ & $1$ & $2$ & $3$ & $4$ & $5$ & $\cdots$ & $\sf{n}-2$ & $\sf{n}-1$ & $n$ \\
		\hline
		$\{ 0,1 \}$ & $1$ & $2$ & $\redzero$ & $\redzero$ & $\redzero$ & $\redzero$ & & $\redzero$ & $\redzero$ & $\redzero$ \\
		$\{ 0,2, \bullet \}$ & 0 & $\greenzero$ & 1 & $\greenzero$ & $\greenzero$ & $\greenzero$ && $\greenzero$ & $\greenzero$ & $\greenzero$ \\
		$\{ 0,3, \bullet \}$ & 0 & $\redzero$ & $\redzero$ & 1 & $\greenzero$ & $\greenzero$ & $\cdots$ & $\greenzero$ & $\greenzero$ & $\greenzero$ \\
		$\{ 0,4, \bullet \}$ & 1 & $\redzero$ & $\redzero$ & $\redzero$ & 1 & $\greenzero$ && $\greenzero$ & $\greenzero$ & $\greenzero$ \\
		\vdots &&& \vdots &&&&&& \vdots &\\
		$\{ 0,\sf{n}-2, \bullet \}$ & 1 & $\redzero$ & $\redzero$ & $\redzero$ & $\redzero$ & $\redzero$ && 1 & $\greenzero$ & $\greenzero$ \\
		$\{ 0,\sf{n}-1 \}$ & 1 & $\redzero$ & $\redzero$ & $\redzero$ & $\redzero$ & $\redzero$ &  & $\redzero$ & 2 & $\redzero$ \\
		$\{ 0,1,3, \bullet \}$ & 0 & 1 & $\redzero$ & 1 & $\greenzero$ & $\greenzero$ & $\cdots$ & $\greenzero$ & $\greenzero$ & $\greenzero$ \\
		$\{ 0,1,4, \bullet \}$ & 0 & 1 & $\redzero$ & $\redzero$ & 1 & $\greenzero$ && $\greenzero$ & $\greenzero$ & $\greenzero$ \\
		\vdots &&& \vdots &&&&&& \vdots &\\
		$\{ 0,1,\sf{n}-2, \bullet \}$ & 0 & 1 & $\redzero$ & $\redzero$ & $\redzero$ & $\redzero$ && 1 & $\greenzero$ & $\greenzero$ \\
		$\{ 0,1,\sf{n}-1 \}$ & 0 & 1 & $\redzero$ & $\redzero$ & $\redzero$ & $\redzero$ && $\redzero$ & 1 & $\redzero$ \\
		$\{ 0,\sf{n}-1,\sf{n} \}$ & 1 & $\redzero$ & $\redzero$ & $\redzero$ & $\redzero$ & $\redzero$ & $\cdots$ & $\redzero$ & 1 & 1 \\
		$\{ 0,1,\sf{n}-1,\sf{n} \}$ & 0 & 1 & $\redzero$ & $\redzero$ & $\redzero$ & $\redzero$ && $\redzero$ & 1 & 1 \\
		\hline
	\end{tabular}
\end{table}

\item Consider the AR-quiver $\teEn{6}$ with vertices labelled as in Figure~\ref{fig:repetition_quivers_E_6}.
We consider all subsets $I \subseteq (\Delta_{\aff})_{0} = \{ 0, 1, 2, 3, 4, 5, 6 \}$ which contain $0$.
If $1 \in I$, then all vertices other than vertex $0$ in the knitting calculation are cut off and so we get the same circled values for $I = \{ 0, 1, \bullet \}$ where $\bullet$ is any list of numbers $2, 3, 4, 5$, and $6$.
If $3 \in I$, then vertex $2$ is cut off so that the circled values of $\{ 0, 2, 3 \}$ are the same as those of $\{ 0, 3 \}$.
We continue on in this way.
Thus, the circled values for $I = \{ 0, 4, \bullet \}$ where $\bullet$ is any list of numbers $2, 3, 5,$ and $6$ are the same.
Finally, if $5 \in I$, then vertex $6$ in the knitting calculation is cut off.
Hence we only need to consider the subsets $\{ 0, 1, \bullet \}, \{ 0, 2 \}, \{ 0, 3, \bullet \}, \{ 0, 4, \bullet \}, \{ 0, 5, \bullet \}, \{ 0, 6 \}, \{ 0, 2, 5, \bullet \}, \{ 0, 2, 6 \}$, $\{ 0, 3, 5, \bullet \},$ and $\{ 0, 3, 6, \bullet \}$.
These are shown in Figure~\ref{fig:circled_values_E_6} where we have used similar shortened notation as above.
\begin{figure}[H]
	\begin{subfigure}[b]{.4\textwidth}
		\centering
\subcaption{$I = \{ 0, 1, \bullet \}$}\label{subfig:E_6_0_1}
	\end{subfigure}
	\begin{subfigure}[b]{.59\textwidth}
		\centering
		%
\subcaption{$I = \{ 0, 2 \}$}\label{subfig:E_6_0_2}
	\end{subfigure} \\ \vspace{.2cm} \quad \\
	\begin{subfigure}[b]{.49\textwidth}
		\centering
		%
\subcaption{$I = \{ 0, 3, \bullet \}$}\label{subfig:E_6_0_3}
	\end{subfigure}
	\begin{subfigure}[b]{.49\textwidth}
		\centering
		%
\subcaption{$I = \{ 0, 4, \bullet \}$}\label{subfig:E_6_0_4}
	\end{subfigure} \\ \vspace{.2cm} \quad \\
	\begin{subfigure}[b]{.44\textwidth}
		\centering
		%
\subcaption{$I = \{ 0, 5, \bullet \}$}\label{subfig:E_6_0_5}
	\end{subfigure}
	\begin{subfigure}[b]{.55\textwidth}
		\centering
		%
\subcaption{$I = \{ 0, 6 \}$}\label{subfig:E_6_0_6}
	\end{subfigure}
\end{figure}
\newpage
\begin{figure}[H]
\centering
	\begin{subfigure}[b]{.49\textwidth}
		\centering
		%
\subcaption{$I = \{ 0, 2, 5, \bullet \}$}\label{subfig:E_6_0_2_5}
	\end{subfigure}
	\begin{subfigure}[b]{.49\textwidth}
		\centering
		%
\subcaption{$I = \{ 0, 2, 6 \}$}\label{subfig:E_6_0_2_6}
	\end{subfigure} \\ \vspace{.2cm} \quad \\
	\begin{subfigure}[b]{.49\textwidth}
		\centering
		%
\subcaption{$I = \{ 0, 3, 5, \bullet \}$}\label{subfig:E_6_0_3_5}
	\end{subfigure}
	\begin{subfigure}[b]{.49\textwidth}
		\centering
		%
\subcaption{$I = \{ 0, 3, 6, \bullet \}$}\label{subfig:E_6_0_3_6}
	\end{subfigure}
\caption{Circled values for varying $I$ for $\teEn{6}$}\label{fig:circled_values_E_6}
\end{figure}
The table is as follows
\begin{table}[H]
	\centering
	\begin{tabular}{|c|c|c|c|c|c|c|c|}
		\hline
		& $0$ & $1$ & $2$ & $3$ & $4$ & $5$ & $6$ \\
		\hline
		$\{ 0,1, \bullet \}$ & 0 & 1 & $\greenzero$ & $\greenzero$ & $\greenzero$ & $\greenzero$ & $\greenzero$ \\
		$\{ 0,2 \}$ & 1 & $\redzero$ & 2 & $\redzero$ & $\redzero$ & $\redzero$ & $\redzero$ \\
		$\{ 0,3, \bullet \}$ & 0 & $\redzero$ &$\greenzero$& 1 & $\redzero$ & $\redzero$ & $\redzero$ \\
		$\{ 0,4, \bullet \}$ & 0 & $\redzero$ & $\greenzero$ & $\greenzero$ & 1 & $\greenzero$ & $\greenzero$ \\
		$\{ 0,5, \bullet \}$ & 0 & $\redzero$ & $\redzero$ & $\redzero$ & $\redzero$ & 1 & $\greenzero$ \\
		$\{ 0,6 \}$ & 1 & $\redzero$ & $\redzero$ & $\redzero$ & $\redzero$ & $\redzero$ & 2 \\
		$\{ 0,2,5, \bullet \}$ & 0 & $\redzero$ & 1 & $\redzero$ & $\redzero$ & 1 & $\greenzero$ \\
		$\{ 0,2,6 \}$ & 1 & $\redzero$ & 1 & $\redzero$ & $\redzero$ & $\redzero$ & 1 \\
		$\{ 0,3,5, \bullet \}$ & 0 & $\redzero$ & $\greenzero$ & 1 & $\redzero$ & 1 & $\greenzero$ \\
		$\{ 0,3,6, \bullet \}$ & 0 & $\redzero$ & $\greenzero$ & 1 & $\redzero$ & $\redzero$ & 1 \\
		\hline
	\end{tabular}
\end{table}
Hence, the number of arrows from vertex $0$ to any vertex $j \in (\Delta_{\aff})_{0}$ in the quiver of $\GammaI{I}$ is always $0, 1,$ or $2$.

\item Consider the AR-quiver $\teEn{7}$ with vertices labelled as in Figure~\ref{fig:repetition_quivers_E_7}.
We consider all subsets $I \subseteq (\Delta_{\aff})_{0} = \{ 0, 1, 2, 3, 4, 5, 6, 7 \}$ which contain $0$.
If $1 \in I$ then all vertices other than vertex $0$ in the knitting calculation are cut off and so we get the same circled values for $I = \{ 0, 1, \bullet \}$ where $\bullet$ is any list of numbers $2, 3, 4, 5, 6$, and $7$.
We continue on in this way.
Hence, the circled values for $I = \{ 0, 2, \bullet \}$ where $\bullet$ is any list of numbers $3, 4, 5, 6,$ and $7$ are the same.
If $3 \in I$, then the circled values for $I = \{ 0, 3, \bullet \}$ where $\bullet$ is any list of numbers $4, 5, 6,$ and $7$ are the same.
If $4 \in I$, then the circled values for $I = \{ 0, 4, \bullet \}$ where $\bullet$ is any list of numbers $5$ and $6$ are the same.
If $5 \in I$, then vertex $6$ in the knitting calculation is cut off.
Hence we only need to consider the subsets $\{ 0, 1, \bullet \}, \{ 0, 2, \bullet \}, \{ 0, 3, \bullet \}, \{ 0, 4, \bullet \}, \{ 0, 5, \bullet \}, \{ 0, 6 \}, \{ 0, 7 \}, \{ 0, 4, 7, \bullet \}, \{ 0, 5, 7, \bullet \},$ and $\{ 0, 6, 7 \}$.
These are shown in Figure~\ref{fig:circled_values_E_7} where we have used similar shortened notation as above.
\begin{figure}[H]
\centering
	\begin{subfigure}[b]{.27\textwidth}
		\centering
\subcaption{$I = \{ 0, 1, \bullet \}$}\label{subfig:E_7_0_1}
	\end{subfigure}
	\begin{subfigure}[b]{.27\textwidth}
		\centering
		%
\subcaption{$I = \{ 0, 2, \bullet \}$}\label{subfig:E_7_0_2}
	\end{subfigure}
	\begin{subfigure}[b]{.4\textwidth}
		\centering
		%
\subcaption{$I = \{ 0, 3, \bullet \}$}\label{subfig:E_7_0_3}
	\end{subfigure} \\ \vspace{.4cm} \quad \\
	\begin{subfigure}[b]{.44\textwidth}
		\centering
		%
\subcaption{$I = \{ 0, 4, \bullet \}$}\label{subfig:E_7_0_4}
	\end{subfigure}
	\begin{subfigure}[b]{.54\textwidth}
		\centering
		%
\subcaption{$I = \{ 0, 5, \bullet \}$}\label{subfig:E_7_0_5}
	\end{subfigure}
\end{figure}
\newpage
\begin{figure}[H]
\centering
	\vspace{.3cm}\begin{subfigure}[b]{\textwidth}
		\centering
		%
\subcaption{$I = \{ 0, 6 \}$}\label{subfig:E_7_0_6}
	\end{subfigure} \\ \vspace{.4cm}\quad \\
	\begin{subfigure}[b]{.49\textwidth}
		\centering
		%
\subcaption{$I = \{ 0, 7 \}$}\label{subfig:E_7_0_7}
	\end{subfigure}
	\begin{subfigure}[b]{.49\textwidth}
		\centering
		%
\subcaption{$I = \{ 0, 4, 7, \bullet \}$}\label{subfig:E_7_0_4_7}
	\end{subfigure} \\ \vspace{.4cm} \quad \\
	\begin{subfigure}[b]{.49\textwidth}
		\centering
		%
\subcaption{$I = \{ 0, 5, 7, \bullet \}$}\label{subfig:E_7_0_5_7}
	\end{subfigure}
	\begin{subfigure}[b]{.49\textwidth}
		\centering
		%
\subcaption{$I = \{ 0, 6, 7 \}$}\label{subfig:E_7_0_6_7}
	\end{subfigure}
\caption{Circled values for varying $I$ for $\teEn{7}$}\label{fig:circled_values_E_7}
\end{figure}
The table is as follows
\begin{table}[H]
	\centering
	\begin{tabular}{|c|c|c|c|c|c|c|c|c|}
		\hline
		& $0$ & $1$ & $2$ & $3$ & $4$ & $5$ & $6$ & $7$ \\
		\hline
		$\{ 0,1, \bullet \}$ & 0 & 1 & $\greenzero$ & $\greenzero$ & $\greenzero$ & $\greenzero$ & $\greenzero$ & $\greenzero$ \\
		$\{ 0,2, \bullet \}$ & 0 & $\redzero$ & 1 & $\greenzero$ & $\greenzero$ & $\greenzero$ & $\greenzero$ & $\greenzero$ \\
		$\{ 0,3, \bullet \}$ & 0 & $\redzero$ & $\redzero$ & 1 & $\greenzero$ & $\greenzero$ & $\greenzero$ & $\greenzero$ \\
		$\{ 0,4, \bullet \}$ & 0 & $\redzero$ & $\redzero$ & $\redzero$ & 1 & $\greenzero$ & $\greenzero$ & $\redzero$ \\
		$\{ 0,5, \bullet \}$ & 1 & $\redzero$ & $\redzero$ & $\redzero$ & $\redzero$ & 1 & $\greenzero$ & $\redzero$ \\
		$\{ 0,6 \}$ & 1 & $\redzero$ & $\redzero$ & $\redzero$ & $\redzero$ & $\redzero$ & 2 & $\redzero$ \\
		$\{ 0,7 \}$ & 0 & $\redzero$ & $\redzero$ & $\redzero$ & $\redzero$ & $\redzero$ & $\redzero$ & 1 \\
		$\{ 0,4,7, \bullet \}$ & 0 & $\redzero$ & $\redzero$ & $\redzero$ & 1 & $\greenzero$ & $\greenzero$ & 1 \\
		$\{ 0,5,7, \bullet \}$ & 0 & $\redzero$ & $\redzero$ & $\redzero$ & $\redzero$ & 1 & $\greenzero$ & 1 \\
		$\{ 0,6,7 \}$ & 0 & $\redzero$ & $\redzero$ & $\redzero$ & $\redzero$ & $\redzero$ & 1 & 1 \\
		\hline
	\end{tabular}
\end{table}
Hence, the number of arrows from vertex $0$ to any vertex $j \in (\Delta_{\aff})_{0}$ in the quiver of $\GammaI{I}$ is always $0$, $1$, or $2$.

\item Consider the AR-quiver $\teEn{8}$ with vertices labelled as in Figure~\ref{fig:repetition_quivers_E_8}.
We consider all subsets $I \subseteq (\Delta_{\aff})_{0} = \{ 0, 1, 2, 3, 4, 5, 6, 7, 8 \}$ which contain $0$.
If $1 \in I$, then all vertices other than vertex $0$ in the knitting calculation are cut off and so we get the same circled values for $I = \{ 0, 1, \bullet \}$ where $\bullet$ is any list of numbers $2, 3, 4, 5, 6, 7,$ and $8$.
We continue on in this way, where if $2 \in I$, then all vertices other than $0$ and $1$ are cut off in the knitting calculation.
If $3 \in I$, then all vertices except $0$, $1$, and $2$ are cut off.
If $4 \in I$, then vertices $5, 6, 7,$ and $8$ are cut off.
If $5 \in I$, then vertices $6, 7,$ and $8$ are cut off and if $6 \in I$ then vertex $7$ is cut off.
Hence, we need only consider the subsets $\{ 0, 1, \bullet \}, \{ 0, 2, \bullet \}, \{ 0, 3, \bullet \}, \{ 0, 4, \bullet \}, \{ 0, 5, \bullet \}, \{ 0, 6, \bullet \}, \{ 0, 7 \}, \{ 0, 8 \}, \{ 0, 6, 8, \bullet \},$ and $\newline \{ 0, 7, 8 \}$.
These are shown in Figure~\ref{fig:circled_values_E_8} where we have used similar shortened notation as above.
\begin{figure}[H]
	\centering
	\quad\begin{subfigure}[b]{.25\textwidth}
		\centering
\subcaption{$I = \{ 0, 1, \bullet \}$}\label{subfig:E_8_0_1}
	\end{subfigure}
	\begin{subfigure}[b]{.32\textwidth}
		\centering
		%
\subcaption{$I = \{ 0, 2, \bullet \}$}\label{subfig:E_8_0_2}
	\end{subfigure}
	\begin{subfigure}[b]{.37\textwidth}
		\centering
		%
\subcaption{$I = \{ 0, 3, \bullet \}$}\label{subfig:E_8_0_3}
	\end{subfigure}
\end{figure}
\newpage
\begin{figure}[H]
\centering
	\vspace{.2cm}\begin{subfigure}[b]{.45\textwidth}
		\centering
		%
\subcaption{$I = \{ 0, 4, \bullet \}$}\label{subfig:E_8_0_4}
	\end{subfigure}
	\begin{subfigure}[b]{.5\textwidth}
		\centering
		%
\subcaption{$I = \{ 0, 5, \bullet \}$}\label{subfig:E_8_0_5}
	\end{subfigure} \\ \vspace{.3cm} \quad \\
	\begin{subfigure}[b]{\textwidth}
		\centering
		%
\subcaption{$I = \{ 0, 6, \bullet \}$}\label{subfig:E_8_0_6}
	\end{subfigure} \\ \vspace{.3cm} \quad \\
	\begin{subfigure}[b]{\textwidth}
		\centering
		%
\subcaption{$I = \{ 0, 7 \}$}\label{subfig:E_8_0_7}
	\end{subfigure}
\end{figure}
\newpage
\begin{figure}[H]
\centering
	\begin{subfigure}[b]{\textwidth}
		\centering
		%
\subcaption{$I = \{ 0, 8 \}$}\label{subfig:E_8_0_8}
	\end{subfigure} \\ \vspace{.3cm} \quad \\
	\begin{subfigure}[b]{\textwidth}
		\centering
		%
\subcaption{$I = \{ 0, 6, 8, \bullet \}$}\label{subfig:E_8_0_6_8}
	\end{subfigure} \\ \vspace{.3cm} \quad \\
	\begin{subfigure}[b]{\textwidth}
		\centering
		%
\subcaption{$I = \{ 0, 7, 8 \}$}\label{subfig:E_8_0_7_8}
	\end{subfigure}
\caption{Circled values for varying $I$ for $\teEn{8}$}\label{fig:circled_values_E_8}
\end{figure}
\newpage
The table is as follows
\begin{table}[H]
	\centering
	\begin{tabular}{|c|c|c|c|c|c|c|c|c|c|}
		\hline
		& $0$ & $1$ & $2$ & $3$ & $4$ & $5$ & $6$ & $7$ & $8$ \\
		\hline
		$\{ 0,1, \bullet \}$ & 0 & 1 & $\greenzero$ & $\greenzero$ & $\greenzero$ & $\greenzero$ & $\greenzero$ & $\greenzero$ & $\greenzero$ \\
		$\{ 0,2, \bullet \}$ & 0 & $\redzero$ & 1 & $\greenzero$ & $\greenzero$ & $\greenzero$ & $\greenzero$ & $\greenzero$ & $\greenzero$ \\
		$\{ 0,3, \bullet \}$ & 0 & $\redzero$ & $\redzero$ & 1 & $\greenzero$ & $\greenzero$ & $\greenzero$ & $\greenzero$ & $\greenzero$ \\
		$\{ 0,4, \bullet \}$ & 0 & $\redzero$ & $\redzero$ & $\redzero$ & 1 & $\greenzero$ & $\greenzero$ & $\greenzero$ & $\greenzero$ \\
		$\{ 0,5, \bullet \}$ & 0 & $\redzero$ & $\redzero$ & $\redzero$ & $\redzero$ & 1 & $\greenzero$ & $\greenzero$ & $\greenzero$ \\
		$\{ 0,6, \bullet \}$ & 0 & $\redzero$ & $\redzero$ & $\redzero$ & $\redzero$ & $\redzero$ & 1 & $\greenzero$ & $\redzero$ \\
		$\{ 0,7 \}$ & 1 & $\redzero$ & $\redzero$ & $\redzero$ & $\redzero$ & $\redzero$ & $\redzero$ & 1 & $\redzero$ \\
		$\{ 0,8 \}$ & 0 & $\redzero$ & $\redzero$ & $\redzero$ & $\redzero$ & $\redzero$ & $\redzero$ & $\redzero$ & 1 \\
		$\{ 0,6,8, \bullet \}$ & 0 & $\redzero$ & $\redzero$ & $\redzero$ & $\redzero$ & $\redzero$ & 1 & $\greenzero$ & 1 \\
		$\{ 0,7,8 \}$ & 0 & $\redzero$ & $\redzero$ & $\redzero$ & $\redzero$ & $\redzero$ & $\redzero$ & 1 & 1 \\
		\hline
	\end{tabular}
\end{table}
Hence, the number of arrows from vertex $0$ to any vertex $j \in (\Delta_{\aff})_{0}$ in the quiver of $\GammaI{I}$ is always $0$ or $1$.
\end{enumerate}
Using case-by-case analysis, the statement follows.
\end{proof}

In the following chapter, we use these results to calculate the K-theory of cDV singularities.

\chapter{Compound Du Val singularities}\label{ch:cDVs}
In this chapter, we study local cDV singularities and investigate their class groups and Grothendieck groups.
Our focus here is on the local; considering the category of all non-local cDV singularities would, in particular, include the category of all smooth affine $3$-folds, and it is not reasonable to expect to be able to compute even the class group in such generality.
We remark that local cDVs are automatically hypersurfaces; see \S\ref{sec:intro:cDVs}.
Therefore a singularity being cDV implies it is Gorenstein and hence CM.

\section{Conjecture and generalities}
The main results of this chapter consist of proofs of both \eqref{eq:Groth=Z+Cl_general} and the following conjecture in various general situations.
\begin{conj}\label{conj:K_0_uCMR=Cl(R)}
	Let $R$ be a local cDV singularity.
	Then
	\[ 
	\Kroth(\uCM R) \cong \Cl(R).
	\]
\end{conj}

The following lemma will be useful.
\begin{lemma}\label{lem:local_ring_Groth_exact_seq_localID_isomorphism}
	For any local ring $R$ there is an exact sequence of groups
	\begin{equation}\label{seq:Groth_sequence_splits_for_local_ID_2}
	0 \to \langle [R] \rangle \to \Kroth(\CM R) \to \Kroth(\uCM R) \to 0.
	\end{equation}
	Furthermore, this sequence splits if $R$ is a local integral domain.
	Hence, in that case, $\Kroth(\CM R) \cong \langle [R] \rangle \oplus \Kroth(\uCM R)$ where $\langle [R] \rangle \cong \Z$.
\end{lemma}
\begin{proof}
	For details on this, see \cite[Chapter~3]{AR86}.
	We have $\langle [R] \rangle \cong \Z$ since integral domains have a rank function sending the equivalence class of $R$ to $1$.
\end{proof}
It is known (see, e.g. \cite[Lemma~13.2]{yosh90}) that for a local ring $R$, 
\begin{equation}\label{eq:Groth(R)=Kroth(CMR)}
	\Groth(R) \colonequals \Kroth(\modCat R) \cong \Kroth(\CM R),
\end{equation}
so we use these interchangeably.
To prove that both \eqref{eq:Groth=Z+Cl_general} and Conjecture~\ref{conj:K_0_uCMR=Cl(R)} hold will also require the following.
\begin{lemma}\label{lem:K_0(uCM R)=Cl(R)iff}
	Let $R$ be a local integral domain.
	Then $\Kroth(\uCM R) \cong \Cl(R)$ if and only if $\Groth(R) \cong \Z \oplus \Cl(R).$
\end{lemma}
\begin{proof}
	Assume $\Kroth(\uCM R) \cong \Cl(R)$.
	By Lemma~\ref{lem:local_ring_Groth_exact_seq_localID_isomorphism}, since $R$ is a local integral domain, $\Kroth(\CM R) \cong \langle [R] \rangle \oplus \Kroth(\uCM R)$.
	Under our assumption, this implies 
	\begin{equation*}
		\Kroth(\CM R) \cong \Z \oplus \Cl(R).
	\end{equation*}
	Therefore, using the equality \eqref{eq:Groth(R)=Kroth(CMR)}, $\Groth(R) \cong \Z \oplus \Cl(R)$.
	
	Now, assume $\Groth(R) \cong \Z \oplus \Cl(R)$.
	Again, by Lemma~\ref{lem:local_ring_Groth_exact_seq_localID_isomorphism}, since $R$ is a local integral domain, we also have $\Groth(R) \cong \Z \oplus \Kroth(\uCM R)$.
	Cancelling $\Z$ factors, $\Kroth(\uCM R) \cong \Cl(R)$ follows.
\end{proof}
This shows that Conjecture~\ref{conj:K_0_uCMR=Cl(R)} is equivalent to the isomorphism \eqref{eq:Groth=Z+Cl_general}.

\begin{remark}\label{remark:conjecture_fails_R_not_ID}
	As stated, Lemma~\ref{lem:K_0(uCM R)=Cl(R)iff} requires $R$ to be an integral domain.
	We will see the importance of this requirement in \S\ref{sec:arbitrary_type_A_cDV}.
\end{remark}

\section{Symmetric quivers for cDV singularities}\label{sec:symmetric_quivers_cDVs}
In this section we apply our results from Chapter~\ref{ch:knitting} to quivers of modifying algebras for cDV singularities.
We require the following technical setup.
\begin{setup}\label{setup:symmetric_quivers_cDVs}
	Let $R$ be a complete local cDV singularity, and $M = M_{0} \oplus M_{1} \oplus \cdots \oplus M_{t}$ any basic modifying module, where $M_{0} \cong R$ and the $M_{i}$ are the non-free indecomposable summands of $M$.
	Set $\Lambda \coloneqq \End_{R}(M)$.
\end{setup}

We can present $\Lambda$ as a quiver with relations, where the vertices are labelled $0, \hdots, t$ corresponding to the modules $M_{0}, \hdots, M_{t}$.
In general, the quiver of a ring $\End_{R}(M)$ where $R$ is a Gorenstein $3$-fold is not symmetric.
However in the cDV Setup~\ref{setup:symmetric_quivers_cDVs}, the main result of this section, Theorem~\ref{thm:complete_local_cdv_quiver_symmetric}, proves that the quiver of $\Lambda$ is symmetric.
To do this requires the `slicing' technique as explained in \cite{We18}.

From \S\ref{subsec:homMMP}, we know $\Lambda$ is derived equivalent to some $X$ in a crepant partial resolution $f \colon X \to \Spec R$.
Briefly, the idea is that we `slice' $X \to \Spec R$ by a generic central element $g \in R$.
By Reid's general elephant principle \cite[Theorems~1.1 and~1.14]{Re83}, this slicing gives
\begin{equation*}
\begin{tikzcd}[arrow style=tikz,>=stealth]
& X_{g} \arrow[r, hookrightarrow] \arrow[d, "\upphi"] & X \arrow[d, "\upvarphi"] \\
&\Spec(R/g) \arrow[r, hookrightarrow] & \Spec R
\end{tikzcd}
\end{equation*}
In particular, $R/g$ is an $\ADE$ singularity and $\upphi$ is a partial crepant resolution.
This slicing technique behaves very nicely; it allows us to consider both $M \in \CM R$ and $M/gM \in \CM R/g$, and for $g$ generic, $\Lambda/g$ is isomorphic to $\End_{R/g}(M/gM)$.
Furthermore, this extends to indecomposable summands of $M$ and $M/gM$; that is, the summand $M_{i}$ slices to $M_{i}/gM_{i}$ for $i = 0, \hdots, t$ \cite[\S5.3]{We18}.
By the McKay correspondence $M_{i}/gM_{i}$ is precisely one of the CM modules corresponding to a vertex in an $\ADE$ Dynkin diagram.

\begin{lemma}\label{lem:quiver_Lambda/g_symmetric}
	Let $\Lambda$ be as in Setup~\ref{setup:symmetric_quivers_cDVs}.
	Then the quiver of $\Lambda/g$ is symmetric, where $g$ is a generic central element.
\end{lemma}
\begin{proof}
	Write $\rad R$ for the radical of $R$.
	As explained in \cite[Lemma~5.19]{IyWe18reductionMMAs}\footnote{Their statement is for type $\tA$, but their proof is general.}, and as above, for generic $g \in \mathrm{rad} R$, $R/g$ is a Kleinian singularity. Furthermore
	\begin{equation*}\label{Lambda/g=EndR/g(M/g)}
	\Lambda/g \cong \End_{R/g}(M/g).
	\end{equation*}
	By Theorem~\ref{thm:number_arrows_0_to_i_equals_reverse}, the quiver of $\Lambda/g$ is symmetric.
\end{proof}

It is known (see, e.g. \cite[\S1]{Segal08}) that the number of arrows between two vertices is given by the dimension of the corresponding $\Ext$ group of the vertex simples.
Let $S_{i}$ denote the simple modules corresponding to vertices $i = 0, \hdots, t$.
\begin{lemma}\label{lem:g_action_trivial_dimExt}
	Let $\Lambda$ be as in Setup~\ref{setup:symmetric_quivers_cDVs} and $g$ a generic central element.
	The action of $g$ on $\Ext^{1}_{\Lambda}(S_{j}, S_{i})$ is trivial for $i \neq j$.
	Furthermore,
	\[ 
	\dim \Ext^{1}_{\Lambda/g}(S_{j}, S_{i}) = \dim \Ext^{1}_{\Lambda}(S_{j}, S_{i})
	\]
\end{lemma}
\begin{proof}
	As explained in \cite[Theorem~7.35]{Rotman09}, elements of $\Ext^{1}_{\Lambda}(S_{j},S_{i})$ can be thought of as equivalence classes of short exact sequences
	\[
	0 \to S_{i} \to B \to S_{j} \to 0.
	\]
	By construction, it suffices to show that $gB = 0$ for any such module $B$.
	
	Now, write $\rad R$ for the radical of $R$, which is equal to the unique maximal ideal since $R$ is local.
	Then necessarily $g \in \rad R$, since if $g \notin \rad R$ then $g$ would be a unit and we would have $R/g = 0$.
	Furthermore, by \cite[Corollary~5.9]{Lam01}, $\rad R \subseteq \rad \Lambda$ so that $g \in \rad \Lambda$.
	It follows that $gS_{i} = 0$ and $gS_{j} = 0$.
	Since $S_{i}$ is simple, either $gB = 0$ or $gB = S_{i}$.
	If $gB = S_{i}$, then multiplication by $g$ gives an isomorphism $B/S_{i} = S_{j} \to S_{i}$.
	However, $i \neq j$, so this is a contradiction.
	Hence, $gB = 0$ and $B$ is a $\Lambda/g\Lambda$-module.
	Hence the first statement holds.

	It follows that the extension
	\[
	0 \to S_{i} \to B \to S_{j} \to 0
	\]
	of $\Lambda$-modules is automatically an extension of $\Lambda/g$-modules, so the second statement holds.
\end{proof}

\begin{remark}
	Lemma~\ref{lem:g_action_trivial_dimExt} is not true if $i=j$.
\end{remark}

Now we prove the main result of this section.
\begin{theorem}\label{thm:complete_local_cdv_quiver_symmetric}
	Let $R$ and $\Lambda$ be as in Setup~\ref{setup:symmetric_quivers_cDVs}.
	Then the quiver of $\Lambda$ is symmetric.
\end{theorem}
\begin{proof}
	By Lemma~\ref{lem:g_action_trivial_dimExt}, provided $i \neq j$,
	\[
		\dim \Ext^{1}_{\Lambda/g}(S_{j}, S_{i}) = \dim \Ext^{1}_{\Lambda}(S_{j}, S_{i}).
	\]
	Hence, the quiver of $\Lambda/g$ is the same as the quiver of $\Lambda$, possibly removing some loops.
	Combining with Lemma~\ref{lem:quiver_Lambda/g_symmetric}, it follows that the quiver of $\Lambda$ is symmetric.
\end{proof}

\section{Isolated cDV singularities with NCCR}\label{sec:isolated_cdv_with_NCCR}
The main result of this section, Theorem~\ref{pf_conj:G_0(R)=Z+Cl(R)_isolated_cDV}, uses our results from \S\ref{sec:symmetric_quivers_cDVs} to prove that both \eqref{eq:Groth=Z+Cl_general} and Conjecture~\ref{conj:K_0_uCMR=Cl(R)} hold for isolated cDV singularities admitting NCCR(s).
This extends results of Navkal \cite{navkal13} to cover any $\ADE$ type, not just type $\tA$.
Crucially, we do not need to rely on manipulating the precise form of the equation defining $R$.
For this section, we maintain the following setup.
\begin{setup}\label{setup:cDV_with_NCCR}
	Let $R$ be an isolated complete local cDV singularity which admits an NCCR $\Lambda \coloneqq \End_{R}(M)$, where $M = M_{0} \oplus M_{1} \oplus \cdots \oplus M_{t}$ is basic, $M_{0} \cong R$, and the $M_{i}$ are non-free indecomposable submodules of $M$.
\end{setup}

Now, we compute the Grothendieck group of a ring $R$ as in Setup~\ref{setup:cDV_with_NCCR}.
We remark that, by Theorem~\ref{thm:CM_Rmodule_M_gives_NCCR_M_cluster_tilting}, $M$ is automatically cluster-tilting.
Hence, the approach of \cite{navkal13} is applicable.
\begin{theorem}\label{thm:Groth(isolated_cDV)}
	Let $R$ and $\Lambda$ be as in Setup~\ref{setup:cDV_with_NCCR}.
	Then
	\[ \Groth(R) \cong \Z^{\oplus (t+1)}. \]
\end{theorem}
\begin{proof}
	Assume $R$ is an isolated $d$-dimensional Gorenstein complete local ring with NCCR $\Lambda$.
	Write $S_{0}, \hdots, S_{t}$ for the simple left $\Lambda$-modules.
	Following \cite[Proposition~7.15]{navkal13}, the rows of the higher AR-matrix $\Upomega$ in Theorem~\ref{thm:NavkalG_0(R)} can be written as\footnote{Where Navkal writes $n$ we write $d-1$, and our indices begin at 0 instead of 1.},
	\[ \Upomega_{lj} = \sum_{i=0}^{d} (-1)^{i} \dim(\Ext_{\Lambda}^{i}(S_{j}, S_{l})). \]
	
	Since $R$ is Gorenstein and $\Lambda$ is an NCCR, it follows that $\Lambda$ is $d$-CY; for further details, see \cite[Proposition~2.4(3) and Theorem~3.2(3)]{IyRe08}).
	Hence, there are natural isomorphisms
	\[ \Ext_{\Lambda}^{i}(S_{j}, S_{l}) \cong D\Ext_{\Lambda}^{d-i}(S_{l}, S_{j}), \]
	where $D \colonequals \Hom_{k}(-, k)$.
	Returning to Setup~\ref{setup:cDV_with_NCCR}, where $d=3$,
	\begin{align*}
		\Upomega_{lj} &= \sum_{i=0}^{3} (-1)^{i} \dim(\Ext_{\Lambda}^{i}(S_{j}, S_{l})) \\
		&= -\dim \Ext_{\Lambda}^{3}(S_{j}, S_{l}) + \dim \Ext_{\Lambda}^{2}(S_{j}, S_{l}) - \dim \Ext_{\Lambda}^{1}(S_{j}, S_{l}) + \dim \Ext_{\Lambda}^{0}(S_{j}, S_{l}) \\
		&= - \dim \Hom_{\Lambda}(S_{l}, S_{j}) + \dim \Ext_{\Lambda}^{1}(S_{l}, S_{j}) - \dim \Ext_{\Lambda}^{1}(S_{j}, S_{l}) + \dim \Hom_{\Lambda}(S_{j}, S_{l}) \tag{by CY duality} \\
		&= \dim \Ext_{\Lambda}^{1}(S_{l}, S_{j}) - \dim \Ext_{\Lambda}^{1}(S_{j}, S_{l}) \tag{by Schur's lemma}
	\end{align*}
	
	Since $R$ is a cDV singularity, by Theorem~\ref{thm:complete_local_cdv_quiver_symmetric}, this is zero, since it is known (see, e.g. \cite[\S1]{Segal08}) that $\Ext^{1}$ between simple modules determines the arrows in the quiver.
	Hence $\Upomega = 0$, so by Theorem~\ref{thm:NavkalG_0(R)},
	\[\Groth(R) \cong \Cok(0) \cong \Z^{\oplus (t+1)}.\qedhere\]
\end{proof}

\begin{remark}
	Note that the proof of Theorem~\ref{thm:Groth(isolated_cDV)} shows, that under the hypothesis of the theorem, $\dim \Ext^{i}_{\Lambda}(S_{j}, S_{l}) = \dim \Ext^{i}_{\Lambda}(S_{l},S_{j})$ for all $i,j,$ and $l$. 
\end{remark}
The use of `$t$' both here and in \S\ref{sec:cl_grp_nagatas_thm} has been done intentionally to highlight the existence of a relationship between $\Groth(R)$ and $\Cl(R)$.
To prove that \eqref{eq:Groth=Z+Cl_general} holds in this setting we need to describe $\Cl(R)$ in a similar fashion to $\Groth(R)$.
This is done in Lemma~\ref{lem:Cl(isolated_cDV)}.
To prove this, we require the following general result, see \cite[Lemma~4.2.1]{VdB04}.
\begin{lemma}\label{lem:pushdown_gives_Cl(X)_isomorphic_Cl(R)}
	Let $f \colon X \to \Spec R$ be a projective birational map between normal noetherian schemes such that, in codimension two, $f$ is an isomorphism.
	The pushdown functor $f_{*}$ restricts to an equivalence between the category of reflexive $X$-modules and the category of reflexive $R$-modules.
\end{lemma}
From \S\ref{subsec:homMMP}, recall that the non-free indecomposable summands of an MM module $N$ (such as the fixed $M$ in Setup~\ref{setup:cDV_with_NCCR}) are in one-to-one correspondence with the exceptional curves in the corresponding minimal model(s) of $R$.
Thus, if $N = N_{0} \oplus \cdots \oplus N_{t}$ is an MM $R$-module, then there are precisely $t$ exceptional curves.
\begin{lemma}\label{lem:Cl(isolated_cDV)}
	Let $R$ be an isolated complete local cDV singularity, then $\Cl(R) \cong \Z^{\oplus t}$.
\end{lemma}
\begin{proof}
	By the Homological MMP (see \S\ref{subsec:homMMP}), there exists a minimal model $f \colon X \to \Spec R$ corresponding to $N$.
	This minimal model may or may not be a crepant resolution.
	Since $R$ is a cDV singularity, $f$ is an isomorphism in codimension $2$ \cite[Lemma~4.2.1]{VdB04}.
	So, Lemma~\ref{lem:pushdown_gives_Cl(X)_isomorphic_Cl(R)} gives a group isomorphism
	\[ \Cl(X) \xrightarrow{\quad \sim \quad} \Cl(R) \]
	via pushdown $f_{*}$.
	It is known (see, e.g. \cite[II.6.16]{hart77}) that there exists an isomorphism $\Cl(X) \cong \text{ Pic}(X)$ since the singularities of $X$ are locally factorial.
	This extends to a third isomorphism $\text{Pic}(X) \cong \Z^{\oplus t}$, by \cite[Lemma~3.4.3]{VdB04}.
	Combining,
	\[ \Cl(R) \cong \Z^{\oplus t}, \]
	where $t$ is the number of curves above the origin.
\end{proof}
The above result holds for a more general setting than Setup~\ref{setup:cDV_with_NCCR}, since it only requires $N$ to be an MM $R$-module.
This includes, as a special case, the situation where $M$ is a cluster-tilting object.
We will need this greater generality in Chapter~\ref{ch:minimal_models}.

Now, we prove that both \eqref{eq:Groth=Z+Cl_general} and Conjecture~\ref{conj:K_0_uCMR=Cl(R)} hold when $R$ is any isolated cDV singularity admitting an NCCR (equivalently, admitting a crepant resolution).
\begin{theorem}\label{pf_conj:G_0(R)=Z+Cl(R)_isolated_cDV}
	Let $R$ be as in Setup~\ref{setup:cDV_with_NCCR}.
	Then Conjecture~\ref{conj:K_0_uCMR=Cl(R)} and \eqref{eq:Groth=Z+Cl_general} both hold, that is,
	\[ \Kroth(\uCM R) \cong \Cl(R) \quad \text{ and } \quad \Groth(R) \cong \Z \oplus \Cl(R). \]
\end{theorem}
\begin{proof}
	We prove the second isomorphism as follows.
	Combining Theorem~\ref{thm:Groth(isolated_cDV)} and Lemma~\ref{lem:Cl(isolated_cDV)}, both $\Groth(R)$ and $\Cl(R)$ can be described by the number of non-free indecomposable summands of the cluster-tilting module $M$, so
	\begin{align*}
		\Groth(R) \cong \Z^{\oplus (t+1)} &\cong \Z \oplus \Z^{\oplus t} \\
		& \cong\Z \oplus \Cl(R).
	\end{align*}
	The first isomorphism then follows from Lemma~\ref{lem:K_0(uCM R)=Cl(R)iff}.
\end{proof}

\section{Arbitrary type $\tA$ cDV singularities}\label{sec:arbitrary_type_A_cDV}
In this section, we focus on arbitrary complete local cDV singularities of type $\tA$.
In particular, this covers cDV singularities with and without NCCRs and also cDV singularities that need not be isolated.
Unlike the previous section, we do not require the existence of an NCCR (equivalently, crepant resolution), as the techniques here are very different.
For example, the results of this chapter heavily depend upon Kn\"{o}rrer periodicity.
As such, this section generalises the results of \cite{navkal13} in two directions, (1) existence of a cluster-tilting object is no longer required, and (2) the singularities are not necessarily isolated.

The fact that we study type $\tA$ cDV singularities means we specifically know the equation in question.
In the main result of this section, Theorem~\ref{thm:conjs_cdv_A_K_0(uCM R)=Cl(R)_and_G_0(R)=Z+Cl(R)_isolated_cDV}, we prove that both \eqref{eq:Groth=Z+Cl_general} and Conjecture~\ref{conj:K_0_uCMR=Cl(R)} hold for 3-dimensional singularities $\eR$ as described in the following setup.
\begin{setup}\label{setup:isolated_arbtypeA_cDV_without_NCCR}
	Define $\eS$ and $\eR$ as
	\begin{equation*}
	\eS \colonequals \frac{\C[[x, y]]}{(f)} \quad \text{ and } \quad
	\eR \colonequals \frac{\C[[u, v, x, y]]}{(uv -  f)}
	\end{equation*}
	where $f = f_{1}^{a_{1}} \cdots  f_{t}^{a_{t}}, f_{i} \in \C[[x, y]]$ and the $f_{i}$ are irreducible and pairwise coprime and each $a_{i}$ is at least one.
\end{setup}
We remark that $\eR$ is a complete local cDV singularity of type $\tA$, and all cDV singularities of type $\tA$ have this form (see, e.g. \cite[\S11.1, Morse Lemma]{arnold88}).
Furthermore, while $\eS$ is not an integral domain, $\eR$ is an integral domain.
This subtle fact affects how we approach proving that \eqref{eq:Groth=Z+Cl_general} and Conjecture~\ref{conj:K_0_uCMR=Cl(R)} hold under Setup~\ref{setup:isolated_arbtypeA_cDV_without_NCCR}.

Let $\mfp_{i} \colonequals (f_{i})/f$ be the minimal prime ideals of $\eS$.
Since $\eS_{\mfp_{i}}$ is a $0$-dimensional Artinian ring with unique maximal ideal $\mfm_{i} \colonequals \mfp_{i}\eS_{\mfp_{i}}$, any finitely generated $\eS_{\mfp_{i}}$-module $M$ is a finite-dimensional $\eS_{\mfp_{i}}/\mfm_{i}$-vector space \cite[Chapter~2]{atiy69}.
Hence, by \cite[Proposition~6.10]{atiy69}, $M$ must have finite length as an $\eS_{\mfp_{i}}$-module.
Moreover, the length of $M$ as an $\eS_{\mfp_{i}}$-module is equal to the dimension of $M$ as an $\eS_{\mfp_{i}}/\mfm_{i}$-vector space.
In other words, $\dim_{\eS_{\mfp_{i}}/\mfm_{i}}(M \otimes_{\eS} \eS_{\mfp_{i}}) = \ell_{\mfp_{i}}(M \otimes_{\eS} \eS_{\mfp_{i}})$.
We abbreviate this notation, using the fact that $\eS \otimes_{\eS} \eS_{\mfp_{i}} \cong \eS_{\mfp_{i}}$ (see, e.g. \cite[Proposition~3.5]{atiy69}).
Hence $\ell_{\mfp_{i}}(\eS \otimes_{\eS} \eS_{\mfp_{i}}) = \ell_{\mfp_{i}}(\eS_{\mfp_{i}})$.

In the following result, we extend the equivalence of dimension and length, asserting an equivalence between $\ell_{\mfp_{i}}(\eS_{\mfp_{i}})$ and $a_{i}$ in the polynomial $f$.
\begin{lemma}\label{lem:a_i=dim=length_of_S}
	$\ell_{\mfp_{i}}(\eS_{\mfp_{i}}) = a_{i}$.
\end{lemma}
\begin{proof}
	Let $\eS$, $\mfp_{i}$, and $\mfm_{i}$ be as defined above.
	Since $\eS_{\mfp_{i}}$ is a local ring it follows that $\ell_{\mfp_{i}}(\eS_{\mfp_{i}})$ is the smallest $r$ such that $\mfm_{i}^{r} = 0$, where $r \in \N$.
	By definition of $\mfm_{i}$, this is the smallest $r$ such that $f_{i}^{r} = 0$ in $\eS_{\mfp_{i}}$.
	This occurs if and only if there exists some $g \notin \mfp_{i}$ such that $gf_{i}^{r} = 0_{\eS}$.
	That is, $gf_{i}^{r} \in \langle f \rangle$ where $g \notin (f_{i})$.
	Since $\C[[u, v]]$ is a UFD, necessarily $r = a_{i}$.
	Therefore $\ell_{\mfp_{i}}(\eS_{\mfp_{i}}) = a_{i}$.
\end{proof}

\begin{lemma}\label{lem:G_0(S)_S_not_ID}
	Let $\eS$ be as in Setup~\ref{setup:isolated_arbtypeA_cDV_without_NCCR}.
	Then $\Groth(\eS) \cong \Z^{\oplus t}$.
\end{lemma}
\begin{proof}
	Let $M \in \modCat \eS$, then $M_{\mfp_{i}} \in \modCat \eS_{\mfp_{i}}$.
	The length function $\ell_{\mfp_{i}} \colon \modCat \eS \to \Z$ is an exact functor on $\modCat \eS$, so $\ell_{\mfp_{i}}$ induces a group homomorphism $\upvarrho \colon \Groth(\eS) \to \Z^{\oplus t}$ given by $[M] \mapsto (\ell_{\mfp_{1}}(M_{\mfp_{1}}), \hdots, \ell_{\mfp_{t}}(M_{\mfp_{t}}))$.
	This homomorphism is surjective since
	\begin{equation*}
	(\eS/\mfp_{i})_{\mfp_{j}} \cong \eS_{\mfp_{j}}/\mfp_{i} \eS_{\mfp_{j}} =\begin{cases}
	\eS_{\mfp_{i}}/\mfp_{i}\eS_{\mfp_{i}}, & \text{ if $i = j$,}\\
	0, & \text{otherwise}.
	\end{cases}
	\end{equation*}
	So that $\eS/\mfp_{i} \mapsto (0, \hdots, 0, 1, 0, \hdots, 0)$ where the $1$ is in the $i^\text{th}$ position.
	
	To show that $\upvarrho$ is an injective homomorphism recall from Lemma~\ref{lem:every_element_written_in_form_[M]-n[R]} that every element of $\Groth(\eS)$ can be written as an element of the form $[M]-n[\eS]$ for some $n \geq 0$.
	Suppose that $[M] - n[\eS] \in \Ker(\upvarrho)$.
	Then under $\upvarrho$
	\[ [M] - n[\eS] \mapsto (\ell_{\mfp_{1}}(M_{\mfp_{1}}) - n \,\ell_{\mfp_{1}}(\eS_{\mfp_{1}}), \hdots, \ell_{\mfp_{t}}(M_{\mfp_{t}})-n \, \ell_{\mfp_{t}}(\eS_{\mfp_{t}})) = 0. \]
	By d\'{e}vissage (see Theorem~\ref{thm:quillen_devissage}), choose a prime filtration $0 = M_{0} \subset M_{1} \subset \cdots \subset M_{t} = M$ of $M$.
	Then $\ell_{\mfp_{i}}(M_{\mfp_{i}})$ is equal to the number of times $\eS/\mfp_{i}$ appears as a factor in the chosen filtration of $M$.
	Furthermore, Lemma~\ref{lem:filtration_lemma} tells us that any element of $\Groth(\eS)$ is equal to the sum of all factors of the chosen filtration.
	By \cite[p.~260]{matsu00}, since $\eS$ is a complete noetherian local ring, it is excellent.
	Hence, by \cite[Lemma 13.4]{yosh90}, automatically $[\eS/\mfm] = 0$.
	Combining these facts gives
	\[
		[M] = \sum_{i=0}^{t} \ell_{\mfp_{i}}(M_{\mfp_{i}}) [\eS/\mfp_{i}] \quad \text{ and } \quad [\eS] = \sum_{i=0}^{t} \ell_{\mfp_{i}}(\eS_{\mfp_{i}}) [\eS/ \mfp_{i}].
	\]
	Hence
	\begin{align*}
	[M] &= \sum_{i=0}^{t} \ell_{\mfp_{i}}(M_{\mfp_{i}}) [\eS/\mfp_{i}] \\
	&= \sum_{i=0}^{t} n \, \ell_{\mfp_{i}}(\eS_{\mfp_{i}})[\eS/\mfp_{i}] \tag{since $[M]-n[\eS] \in \Ker \upvarrho$}\\
	&= n \sum_{i=0}^{t} \ell_{\mfp_{i}}(\eS_{\mfp_{i}}) [\eS/\mfp_{i}] = n[\eS].
	\end{align*}
	Thus $[M] - n[\eS] = 0$ so $\ker(\upvarrho) = 0$, hence $\upvarrho$ is injective.
	Therefore $\upvarrho$ is an isomorphism, giving $\Groth(\eS) \cong \Z^{\oplus t}$.
\end{proof}

\begin{lemma}\label{lem:K_0(uCM(S))_2dim_description}
	Let $\eS$ be as in Setup~\ref{setup:isolated_arbtypeA_cDV_without_NCCR}.  Then
	\[ \Kroth(\uCM \eS) \cong \Z^{\oplus t}/ (a_{1}, \hdots, a_{t}). \]
\end{lemma}
\begin{proof}
	Since $\eS$ is a local ring, by Lemma~\ref{lem:local_ring_Groth_exact_seq_localID_isomorphism}, there is an exact sequence
	\[
	0 \to \langle [ \eS] \rangle \to \Kroth( \CM \eS) \to \Kroth( \uCM \eS) \to 0.
	\]
	Since $\eS$ is not an integral domain this sequence is not necessarily split.
	However, we still have an isomorphism $\Kroth(\uCM \eS) \cong \Kroth(\CM \eS)/\langle [\eS] \rangle$.
	Since $\Kroth(\CM \eS) \cong \Kroth(\modCat \eS) \colonequals \Groth(\eS)$, it follows that
	\begin{equation}\label{eq:K_0(uCM R)_cong_Groth(R)/R}
	\Kroth(\uCM \eS) \cong \Groth(\eS)/ \langle [\eS] \rangle.
	\end{equation}
	By Lemma~\ref{lem:a_i=dim=length_of_S}, $\ell_{\mfp_{i}}(\eS) = a_{i}$, so that under the group homomorphism $\upvarrho$ of Lemma~\ref{lem:G_0(S)_S_not_ID}, $[\eS] \mapsto ( a_{1}, \hdots, a_{t})$.
	Combining this with \eqref{eq:K_0(uCM R)_cong_Groth(R)/R} gives the desired isomorphism.
\end{proof}

The following relies on Kn\"{o}rrer periodicity from \cite{K87} which provides us with an equivalence between the $\uCM$ $\eS$-modules and the $\uCM$ $\eR$-modules.
\begin{cor}\label{cor:K_0(uCM(R))_description}
	Let $\eR$ be as in Setup~\ref{setup:isolated_arbtypeA_cDV_without_NCCR}.  Then
	\[ \Kroth(\uCM \eR) \cong \Z^{\oplus t}/ (a_{1}, \hdots, a_{t}). \]
\end{cor}
\begin{proof}
	By Kn\"{o}rrer periodicity $\uCM \eR \cong \uCM \eS$  \cite[Theorem~3.1]{K87}, so the result follows immediately from Lemma~\ref{lem:K_0(uCM(S))_2dim_description}.
\end{proof}

Notice that under the isomorphism $\upvarrho$ of Lemma~\ref{lem:G_0(S)_S_not_ID}, it is not true that $\Groth(\eS) \cong \Z \oplus \Kroth(\uCM \eS)$.
That is, \eqref{eq:Groth=Z+Cl_general} does not hold for $\eS$, verifying Remark~\ref{remark:conjecture_fails_R_not_ID}.
However, by Lemma~\ref{lem:A_is_normal_ID} $\eR$ is an integral domain, so we can compute $\Groth(\eR)$ using $\Kroth(\uCM \eS) \cong \Kroth(\uCM \eR)$.

\begin{theorem}\label{thm:conjs_cdv_A_K_0(uCM R)=Cl(R)_and_G_0(R)=Z+Cl(R)_isolated_cDV}
	Let $\eR$ be as in Setup~\ref{setup:isolated_arbtypeA_cDV_without_NCCR}.
	Then Conjecture~\ref{conj:K_0_uCMR=Cl(R)} and \eqref{eq:Groth=Z+Cl_general} both hold, that is,
	\[ \Kroth(\uCM \eR) \cong \Cl(\eR) \quad \text{ and } \quad \Groth(\eR) \cong \Z \oplus \Cl(\eR). \]
\end{theorem}
\begin{proof}
	By Corollary~\ref{cor:K_0(uCM(R))_description}
	\[ \Kroth(\uCM \eR) \cong \Z^{\oplus t}/(a_{1}, \hdots, a_{t}). \]
	Then, similar techniques to ours in Theorem~\ref{thm:Cl(A)_global} shows that our class group result also holds in the complete local case (see, e.g. \cite[Proposition~4.3]{DaoHu13}), so that $\Cl(\eR) \cong \Z^{\oplus t}/ (a_{1}, \hdots, a_{t})$.
	Hence there is an isomorphism $\Kroth(\uCM \eR) \cong \Cl(\eR)$.
	Furthermore, by Lemma~\ref{lem:K_0(uCM R)=Cl(R)iff}, since $\eR$ is an integral domain, $\Groth(\eR) \cong \Z \oplus \Cl(\eR)$ follows.
\end{proof}

\subsection{Failure of isomorphism \eqref{eq:Groth=Z+Cl_general}  in higher dimensions}\label{subsec:failure_of_conj_in_higher_dim}
Viehweg's setting is motivated by deformed Kleinian singularities; as such, we would expect them to behave in a similar way to Kleinian singularities.
Our results show that the existence of the isomorphisms \eqref{eq:Groth=Z+Cl_general} is a two and three dimensional phenomena. 
Specifically, it does not hold for the general Viehweg setting, as the following counterexample demonstrates.
\begin{example}
	Let \begin{equation*}
	\eB \colonequals \frac{\C[x,y,z]}{xy - z^{2}} \quad \text{ and } \quad \eC \colonequals \frac{\C[u,v,x,y,z]}{uv = xy - z^{2}}.
	\end{equation*}
	Both $\eB$ and $\eC$ are integral domains (see Lemma~\ref{lem:A_is_normal_ID}) so we do not need to worry about the issues of Remark~\ref{remark:conjecture_fails_R_not_ID}.
	By Theorem~\ref{thm:G_0(R)=Z+Cl(R)_dim2}, $\Groth(\eB) \cong \Z \oplus \Cl(\eB)$ and $\Kroth(\uCM \eB) \cong \Cl(\eB)$ which is isomorphic to $\Z/2\Z$.
	Using Kn\"{o}rrer periodicity \cite[Theorem~3.1]{K87}, we have $\Kroth(\uCM \eB) \cong \Kroth(\uCM \eC)$.
	Hence, 
	\begin{align*}
	\Groth(\eC) &\cong \Z \oplus \Kroth(\uCM \eC) \\
	&\cong \Z \oplus \Kroth(\uCM \eB) \\
	&\cong \Z \oplus \Cl(\eB) \\
	&\cong \Z \oplus \Z / 2\Z.
	\end{align*}
	But, by Theorem~\ref{thm:Cl(A)_global}, $\Cl(\eC) \cong 0$, so clearly $\Groth(\eC) \not\cong \Z \oplus \Cl(\eC)$.
\end{example}

\section{Polyhedral singularities}\label{sec:nonisolated_cDV}
In this section, we apply the results of Brown and Lorenz \cite{BrownLorenz96LinearActions} to show that \eqref{eq:Groth=Z+Cl_general} holds in the setting of polyhedral quotient singularities $\Z_{n}$ for all $n$, $D_{2n}$ for $n \le 100$, $\mathbb{T}$, $\mathbb{O}$, and $\mathbb{I}$.
These are the subgroups of $\SL(3, \C)$ as given in \S\ref{sec:intro:cDVs}.
In particular, this gives examples of non-isolated cDV singularities of a variety of $\ADE$ types where \eqref{eq:Groth=Z+Cl_general} holds.
Although these quotient singularities admit an NCCR, they are all non-isolated, and so the techniques of \S\ref{sec:isolated_cdv_with_NCCR} do not apply.

Let $G$ be a finite subgroup of $\SL(3, \C)$, $S = \C[X_{1}, X_{2}, X_{3}]$, and write $R \coloneqq S^{G}$.
The main result of \cite[\S0]{BrownLorenz96LinearActions} gives a description of the Grothendieck group $\Groth(R)$ of finitely generated $R$-modules.
As noted at the end of \cite{BrownLorenz96LinearActions}, Lorenz implemented a computer programme, using GAP3, which computes the Grothendieck group of invariant rings.
Following his approach, we used \texttt{MAGMA} to implement such a programme (see Appendix~\ref{app:magma_code}).
Through our implementation, we have gathered evidence of the existence of a relationship between $\Groth(R)$ and $\Cl(R)$ in numerous settings, including many not described in the previous chapters.
We parse through the relevant data in the cDV setting now, with some other settings being described in \S\ref{sec:further_results}.

An element $g \in G$ is called a \emph{pseudoreflection} on $V$ if $\dim_{\C}(V^{g}) = n - 1$ \cite[Chapter~2, \S2.6]{ben93}.
Let $J$ be the normal subgroup of $G$ generated by the pseudoreflections on $V$.
Then there exist isomorphisms
\[ \Cl(R) \cong H^{1}(G/J, (S^{J})^{\times}) \cong \Hom(G/J, \C^{\times}). \]
For details, see \cite[\S4.4]{BrownLorenz96LinearActions}.
In our case, $J$ is trivial since $G$ is a subgroup of $\SL(3, \C)$.

These isomorphisms give a straightforward way of computing $\Cl(R)$.
Let $G^{\ab} \coloneqq G/[G, G]$ denote the abelianisation of $G$, where $[G,G]$ denotes the commutator subgroup.
Then it is well known that $\Hom(G/J, \C^{\times})$, the Pontryagin dual of $G/J$, is isomorphic to $G^{\ab}$.
Hence, $\Cl(R) \cong G^{\ab}$.

\begin{prop}\label{prop:Groth(polyhedral_grp)}
	Consider $R \coloneqq S^{G}$, where $G$ is either a polyhedral quotient singularity $\Z_{n}$ for all $n$, $D_{2n}$ for $n \le 100$, $\mathbb{T}$, $\mathbb{O}$, or $\mathbb{I}$ as given in \S\ref{sec:intro:cDVs}.
	Then $R$ is a non-isolated cDV singularity where \eqref{eq:Groth=Z+Cl_general} holds.
	That is, $\Groth(R) \cong \Z \oplus \Cl(R)$.
\end{prop}
\begin{proof}
For the cyclic group given in \ref{cDV_3dim_A}, $R \cong \C[u,v,x,y]/(uv-x^{n})$.
This is a special case of \S\ref{sec:cl_grp_nagatas_thm} and \S\ref{sec:arbitrary_type_A_cDV} so that, by Theorem~\ref{thm:Cl(A)_global}, $\Cl(R) \cong \Z/n\Z$.
This is also clear from the fact that $\Cl(R) \cong G^{\ab}$.
Furthermore, by Theorem~\ref{thm:conjs_cdv_A_K_0(uCM R)=Cl(R)_and_G_0(R)=Z+Cl(R)_isolated_cDV}, $\Groth(R) \cong \Z \oplus \Z/n\Z$.

For the subgroup $D_{2n}$ given in \ref{cDV_3dim_D}, for all even $n$, $\Cl(R) \cong G^{\ab} \cong \Z/2\Z \oplus \Z/2\Z$ and for all odd $n$, $\Cl(R) \cong G^{\ab} \cong \Z/2\Z$.
For $D_{2n}$ where $n \le 100$ we use \texttt{MAGMA} to compute the Grothendieck group of $R$.
In all such cases, for even $n$, $\Groth(R) \cong \Z \oplus \Z/2\Z \oplus \Z/2\Z$ and for odd $n$, $\Groth(R) \cong \Z \oplus \Z/2\Z$.

While the subgroups given in \ref{cDV_3dim_A} and \ref{cDV_3dim_D} correspond to infinite families of subgroups, the subgroups given in \ref{cDV_3dim_E_6}, \ref{cDV_3dim_E_7}, and \ref{cDV_3dim_E_8} ($\mathbb{T}$, $\mathbb{O}$, and $\mathbb{I}$, respectively) are three specific groups.
So, we can use \texttt{MAGMA} to compute the Grothendieck group and class group of $R$ when $G$ is one of $\mathbb{T}$, $\mathbb{O}$, or $\mathbb{I}$.
In fact, the calculation of $\Cl(R)$ for such rings is also clear from the McKay quiver using the number of one-dimensional representations calculated in \cite[\S2]{NollaSekiya11FlopsAM}.
In the McKay quiver of $\mathbb{T}$ there are 3 one-dimensional representations, and the only group of order 3 is $\Z/3\Z$, hence the class group is $\Z/3\Z$.
Similarly, in the McKay quiver of $\mathbb{O}$ there are 2 one-dimensional representations and the only group of order 2 is $\Z/2\Z$.
Finally, in the McKay quiver of $\mathbb{I}$ there is 1 one-dimensional representation and the only group of order 1 is the trivial group.

We use the generators defined for each of the subgroups $\mathbb{T}$, $\mathbb{O}$, and $\mathbb{I}$ as input for the programme of Appendix~\ref{app:magma_code}.
The output data from this code gives $\Groth(R)$.
We collate this data in Table~\ref{table:G_0(R)_3dim_cDV}.
\newpage
\begin{table}[H]
	\centering
	\begin{tabular}{| c | c | c | c |}
		\hline
		Group $G$ & Order & $\Groth(R)$ & $\Cl(R)$ \\
		\hline
		Cyclic & $n$ & $\Z \oplus \Z/n\Z$ & $\Z/n\Z$ \\
		Dihedral $D_{2n}$ & $2n$, $n \le 100$ even & $\Z \oplus \Z/2\Z \oplus \Z/2\Z$ & $\Z/2\Z \oplus \Z/2\Z$  \\
		Dihedral $D_{2n}$ & $2n$, $n < 100$ odd & $\Z \oplus \Z/2\Z $ & $\Z/2\Z$  \\
		Trihedral $\mathbb{T}$ & 12 & $\Z \oplus \Z/3\Z$ & $\Z/3\Z$ \\
		Octahedral $\mathbb{O}$ & 24 & $\Z \oplus \Z/2\Z$ & $\Z/2\Z$ \\
		Icosahedral $\mathbb{I}$ & 60 & $\Z$  & 0 \\
		\hline
	\end{tabular}
	\caption{Data for finite subgroups of SL$(3, \C)$ given in \S\ref{sec:intro:cDVs}\ref{cDV_3dim_A} -- \ref{cDV_3dim_E_8}}
	\label{table:G_0(R)_3dim_cDV}
\end{table}
Comparing the data of the third column against the data of the fourth column shows that, for each $G$ considered, $\Groth(R) \cong \Z \oplus \Cl(R)$.
\end{proof}
\chapter{Minimal models}\label{ch:minimal_models}
In this chapter, we study (noncommutative) minimal models and investigate their class groups and Grothendieck groups.
Our focus here is on proving that \eqref{eq:Groth=Z+Cl_general} holds for a ring $R$ as in the following setup.
\begin{setup}\label{setup:type_A_min_model}
	Let $f \colon X \rightarrow \Spec R$ be a minimal model where $R$ is an isolated complete local cDV singularity and $X$ has only factorial (equivalently, $\Q$-factorial) $\ctAn{n}$ singularities.
\end{setup}
This setup includes all isolated cDV singularities that admit an NCCR (\S\ref{sec:isolated_cdv_with_NCCR}), and also arbitrary isolated $\ctAn{n}$ singularities.
To prove that \eqref{eq:Groth=Z+Cl_general} holds in this setting, we require the language of minimal models and MM modules of \S\ref{sec:minimal_models_prelims} as well as the Homological MMP notions discussed in \S\ref{subsec:homMMP}.
In this language, $f$ is a minimal model with only type $\tA$ singularities.

We set some notation.
Under Setup~\ref{setup:type_A_min_model}, by the bijection given in \S\ref{subsec:homMMP} we automatically have a corresponding basic MM $R$-module $M = M_{0} \oplus M_{1} \oplus \cdots \oplus M_{t}$, with $M_{0} \cong R$ and $M_{1}, \hdots, M_{t}$ the non-free indecomposable summands of $M$, such that $X$ is derived equivalent to $\Lambda \coloneqq \End_{R}(M)$. 
Let $\Db(\coh X)$ and $\perf X$ be as in \S\ref{sec:intro:notation_conventions}.
From this derived equivalence, it follows that $\Db(\coh X)$ is equivalent to $\Db(\modCat \Lambda)$.
Write $\Dsg(X) \colonequals \Db(\coh X) / \perf X$ and define $\Krothsg(X)$ as the Grothendieck group of the idempotent completion of $\Dsg(X)$; for full details, see \cite[Property~0]{pavicEvgen18ktheory}.

The first result of this chapter is an application of Lemma~\ref{lem:Cl(isolated_cDV)} to Setup~\ref{setup:type_A_min_model}, giving a description of the class group of $R$ based on the number of curves above the origin.
\begin{lemma}\label{lem:Cl(type_A_min_model)}
	Let $X$ and $R$ be as in Setup~\ref{setup:type_A_min_model}.
	Then
	\[
	\Cl(R) \cong \Z^{\oplus t}.
	\]
\end{lemma}
\begin{proof}
	This follows immediately from Lemma~\ref{lem:Cl(isolated_cDV)}.
\end{proof}

As $X$ and $\Lambda$ are derived equivalent, it will be particularly useful to have a description of $\Kroth(\proj \Lambda)$ which is similar to that of $\Cl(R)$ given above.
\begin{lemma}\label{lem:Kroth(Lambda)_cong_Z^t+1}
	Let $\Lambda \coloneqq \End_{R}(M)$ as above.
	Then
	\[
		\Kroth(\proj \Lambda) \cong \Z^{\oplus (t+1)}.
	\]
	In particular, $\Kroth(\proj \Lambda)$ is torsion-free.
\end{lemma}
\begin{proof}
	Since $R$ is a local ring, $\Kroth(\proj \Lambda)$ is based by the classes of finitely generated projective $\Lambda$-modules \cite{fullerShutters75projModules}.
	Furthermore, $M$ is a basic MM module with $t+1$ indecomposable summands and these summands are in bijection with the indecomposable projective $\Lambda$-modules.
	Hence $\Kroth(\proj \Lambda) \cong \Z^{\oplus (t+1)}$.
\end{proof}

In the next result, we determine $\Krothsg(X)$.
Later, we will see that this result allows us to prove that there exists an isomorphism $\Kroth(X) \cong \Groth(X)$, even though $X$ is not smooth.
\begin{lemma}\label{lem:Krothsg=oplusKroth(uCM(sheaf))}
	Let $X$ be as in Setup~\ref{setup:type_A_min_model}, then 
	\[
	\Krothsg(X) \cong \bigoplus_{j=1}^{n} \Kroth(\uCM \widehat{\cO}_{X, p_{j}}),
	\]
	where $\{p_{1}, \hdots, p_{n} \}$ are the singular points of $X$.
	Furthermore, $\Krothsg(X) = 0$.
\end{lemma}
\begin{proof}
	Since $R$ is complete local, $\Dsg(X)$ is automatically idempotent complete \cite[Lemma~3.1]{BK12}.
	This and the fact that $X$ has only isolated singularities implies that $\Dsg(X) \cong \bigoplus_{j=1}^{n} \uCM \widehat{\cO}_{X, p_{j}}$ by \cite[Theorem~3.2(2)]{IyWe14QFact}.
	Therefore
	\begin{equation}\label{eq:Ksg=oplusK(uCM O_X,x)}
	\Krothsg(X) \cong \bigoplus_{j=1}^{n} \Kroth(\uCM \widehat{\cO}_{X, p_{j}}).
	\end{equation}
	Recall that $X$ has only factorial $\ctAn{n}$ singularities.
	This implies that all the local rings $\widehat{\cO}_{X, p_{j}}$ are isolated cDV singularities of type $\ctAn{n}$ which are UFDs.
	Furthermore, these are precisely the rings $\eR$ of Setup~\ref{setup:isolated_arbtypeA_cDV_without_NCCR} with $t=1$ and $a_{1}=1$.
	Hence, by Lemma~\ref{lem:K_0(uCM(S))_2dim_description}, for each $j = 1, \hdots, n$ we have $\Kroth(\CM \widehat{\cO}_{X, p_{j}}) \cong \Z$ and so $\Kroth(\uCM \widehat{\cO}_{X, p_{j}}) = 0$ for each $j = 1, \hdots, n$.
	Combining with \eqref{eq:Ksg=oplusK(uCM O_X,x)}, $\Krothsg(X) = 0$.
\end{proof}

\begin{remark}
	This result shows that Setup~\ref{setup:type_A_min_model} provides a suitable setting for $X$ to behave as if it is smooth.
	As mentioned in Conjecture~\ref{conj:motivation}, this is expected to occur more generally.
\end{remark}

Now, recall the language of contraction algebras from \S\ref{subsec:maximal_modifying_modules}.
Let $e$ be the idempotent in $\Lambda$ corresponding to the summand $R$ of $M$.
Then $e\Lambda e \cong \End_{R}(R) \cong R$ and $\Lambda/\Lambda e \Lambda \cong \CL$.
By \cite[Example~2.9]{psarVit14}, there is a recollement of categories
\begin{equation}\label{seq:recollement}
\modCat \CL \xrightarrow{\inc} \modCat \Lambda \xrightarrow{e(-)} \modCat R,
\end{equation}
where $\inc$ is the inclusion functor and $e(-) \coloneqq \Hom_{\Lambda}(e\Lambda, -)$; both are exact.
In fact, by \cite[Proposition~2.7]{psarVit14}, $\modCat \CL$ is a full exact subcategory of $\modCat \Lambda$ and $e(-)$ is naturally equivalent to the quotient functor $\modCat \Lambda \rightarrow \modCat \Lambda/\modCat \CL$.
Hence, $\modCat R \cong \modCat \Lambda/\modCat \CL$.
By Theorem~\ref{thm:quillen_localisation}, it follows that \eqref{seq:recollement} lifts to K-theory giving a long exact sequence ending with
\begin{equation}\label{seq:recollement_ktheory}
\hdots \rightarrow \text{K}_{1}(\modCat R)\rightarrow \Kroth(\modCat \CL) \rightarrow \Kroth(\modCat \Lambda)
\rightarrow \Groth(R) \rightarrow 0.
\end{equation}
In particular, $\Kroth(\modCat \Lambda) \rightarrow \Groth(R)$ is surjective.

\begin{theorem}\label{thm:Groth(type_A_min_model)}
	There is a composition of surjections $\uptheta \colon \Kroth(\proj \Lambda) \rightarrow \Kroth(\modCat \Lambda) \rightarrow \Groth(R)$.
	Furthermore, $\uptheta$ is an isomorphism.
\end{theorem}
\begin{proof}
	By \cite[Lemma~1.10]{pavicEvgen18ktheory}, there is an exact sequence of K-theories ending with
	\begin{equation}\label{seq:exact_seq_singularity_cat}
	\Kroth(X) \xrightarrow{\text{PD}} \Groth(X) \rightarrow \Krothsg(X) \rightarrow 0,
	\end{equation}
	where PD is the canonical morphism induced by $\perf X \subset \Db(X)$.
	Since $X$ is derived equivalent to $\Lambda$, by Lemma~\ref{lem:Kroth(Lambda)_cong_Z^t+1} and Lemma~\ref{lem:Krothsg=oplusKroth(uCM(sheaf))} it follows that there exists a surjection
	\begin{equation*}
	\Z^{\oplus(t+1)} \cong \Kroth(\proj \Lambda) \twoheadrightarrow \Kroth(\mod \Lambda),
	\end{equation*}
	where $\Kroth(\proj \Lambda)$ is torsion-free.
	In addition, from \eqref{seq:recollement_ktheory}, there automatically exists a surjection
	\begin{equation*}
	\Kroth(\mod \Lambda) \twoheadrightarrow \Groth(R).
	\end{equation*}
	Composing these surjections gives a surjection
	\begin{equation*}
	\uptheta \colon \Z^{\oplus(t+1)} \twoheadrightarrow \Groth(R).
	\end{equation*}
	
	Now, suppose $\uptheta$ is not an isomorphism.
	Since $\Ker \uptheta$ is a subgroup of $\Z^{\oplus(t+1)}$, it cannot have torsion.
	Thus, $\Ker \uptheta$ has rank at least one, so that, from the short exact sequence
	\begin{equation*}
	0 \rightarrow \Ker \uptheta \rightarrow \Z^{\oplus(t+1)} \rightarrow \Groth(R) \rightarrow 0,
	\end{equation*}
	the rank of $\Groth(R)$ must be strictly less than $t+1$.

	But, by Proposition~\ref{prop:bass_Groth(R)_to_Cl(R)}, $\Z \oplus \Cl(R)$ is always a quotient of $\Groth(R)$.
	By Lemma~\ref{lem:Cl(type_A_min_model)}, $\Cl(R) \cong \Z^{\oplus t}$, therefore $\Groth(R)$ has $\Z^{\oplus (t+1)}$ as a quotient.
	In particular, the rank of $\Groth(R)$ has to be at least $t+1$, a contradiction.
	Hence, $\uptheta$ must be an isomorphism, giving $\Groth(R) \cong \Z^{\oplus(t+1)}$.
\end{proof}

Combining Lemma~\ref{lem:Cl(type_A_min_model)} and Theorem~\ref{thm:Groth(type_A_min_model)}, the main result of this chapter is an immediate corollary.
\begin{cor}\label{cor:G_0(R)=Z+Cl(R)_minimal_model}
	Let $X$ and $R$ be as in Setup~\ref{setup:type_A_min_model}.
	Then \eqref{eq:Groth=Z+Cl_general} holds, that is,
	\[
	\Groth(R) \cong \Z \oplus \Cl(R).
	\]
\end{cor}

Additionally, the next two corollaries follow immediately from Theorem~\ref{thm:Groth(type_A_min_model)}.
\begin{cor}
	Let $\Lambda$ be as above, then $\Kroth(\proj \Lambda) \cong \Kroth(\modCat \Lambda)$.
	In particular, $\Kroth(\modCat \Lambda)$ is based by projectives and is finitely generated.
\end{cor}

\begin{cor}\label{cor:KrothX=GrothX_minimal_model}
	Let $X$ and $R$ be as in Setup~\ref{setup:type_A_min_model}.
	Then the canonical morphism PD of \eqref{seq:exact_seq_singularity_cat} is an isomorphism.
	That is,
	\[
	\Kroth(X) \cong \Groth(X).
	\]
\end{cor}

\begin{remark}
	If $\Lambda$ has finite global dimension (equivalently, X is smooth) then $\Kroth(X) \cong \Groth(X)$.
	Corollary~\ref{cor:KrothX=GrothX_minimal_model} is interesting precisely because $\Kroth(X) \cong \Groth(X)$ even when $X$ has infinite global dimension.
\end{remark}

\chapter{Conclusion}\label{ch:conclusion}

We conclude this thesis with a summary of the proven results.
Superficially, the results of this thesis can be viewed as computing class groups and Grothendieck groups of Kleinian and compound Du Val (cDV) singularities.
On a deeper level, we do a number of things.

\begin{enumerate}[label=(\arabic*)]
	\item We extend (locally) known results in dimension 2 to a global setting.
	\item We give results in dimension 3 for both isolated and non-isolated cDV singularities.
	\item We prove fundamental and independently interesting results for surface singularities and their quivers.
	\item Finally, in Chapter~\ref{ch:minimal_models}, we show that MMAs behave precisely as they were designed to behave.
	That is, they behave as if they are smooth.
\end{enumerate}

\section{Summary Theorem}
The following theorem simply summarises all the results of this thesis relevant to the isomorphism \eqref{eq:Groth=Z+Cl_general}.
\begin{theorem}\label{thm:conclusion}
	There exists an isomorphism $\Groth(R) \cong \Z \oplus \Cl(R)$ when
	\begin{enumerate}[label=(\arabic*)]
		\item\label{thm:conclusion:1} $\dim R = 2$ and $R$ has only type $\tA$ Kleinian singularities,
		\item\label{thm:conclusion:2} $\dim R = 3$ and $R$ is an isolated complete local cDV singularity admitting NCCR(s),
		\item\label{thm:conclusion:3} $\dim R = 3$ and $R$ is an arbitrary type $\tA$ complete local cDV singularity,
		\item\label{thm:conclusion:4} $\dim R = 3$ and $R$ is an isolated complete local cDV such that there exists a minimal model $f \colon X \rightarrow \Spec R$ where $X$ has only factorial $\ctAn{n}$ singularities.
		\item\label{thm:conclusion:5} $R = S^{G}$, where $G$ is a polyhedral quotient singularity $\Z_{n}$ for all $n$, $D_{2n}$ for $n \le 100$, $\mathbb{T}$, $\mathbb{O}$, or $\mathbb{I}$ as given in \S\ref{sec:intro:cDVs}.
	\end{enumerate}
\end{theorem}
\begin{proof}
	(1) is Theorem~\ref{thm:G_0(R)=Z+Cl(R)_dim2}, (2) is Theorem~\ref{pf_conj:G_0(R)=Z+Cl(R)_isolated_cDV}, (3) is Theorem~\ref{thm:conjs_cdv_A_K_0(uCM R)=Cl(R)_and_G_0(R)=Z+Cl(R)_isolated_cDV}, (4) is Corollary~\ref{cor:G_0(R)=Z+Cl(R)_minimal_model}, and (5) is Proposition~\ref{prop:Groth(polyhedral_grp)}.
\end{proof}

\section{Speculations}\label{sec:further_results}
Another interesting class of quotient singularities comes from a finite group acting symplectically on a symplectic vector space.
Such quotient singularities play an important role in representation theory e.g., in the theory of rational Cherednik algebras introduced by \cite{etinGinz01}.
Specifically, let $G$ be a subgroup of $\GL(\mathfrak{h})$ generated by complex reflections.
Then $G$ acts on $\mathfrak{h}$, so it also acts on the complex vector space $V = \mathfrak{h} \oplus \mathfrak{h}^{*}$.
Such complex reflection groups $G$ were classified by Shephard-Todd in \cite{shephardTodd54}.
A special case is the symmetric group $S_{n}$, which can be thought of as a real reflection group acting on its reflection representation $\mathfrak{h}$.

Using the \texttt{MAGMA} code in Appendix~\ref{app:magma_code} we calculate the Grothendieck group, reduced Grothendieck group, and class group of $R = \C[V]^{S_{n}}$ for suitable values of $n$.
As we have seen throughout this thesis, the Grothendieck group retains interesting representation theoretic and geometric interpretations.
\begin{longtable}{| p{1em} | p{15.5em} | p{14.5em} | p{2.5em} |}
	\hline
	$n$ & $\Groth(R)$ & $\oGroth(R)$ & $\Cl(R)$ \\ \hline
	\endhead
	\hline
	\endfoot
	\endlastfoot
	$2$ & $\Z \oplus \Z/2\Z$ & $\Z/2\Z$ & $\Z/2\Z$ \\
	$3$ & $\Z \oplus \Z/6\Z$ & $\Z/6\Z$ & $\Z/2\Z$ \\
	$4$ & $\Z \oplus \Z/2\Z \oplus \Z/12\Z$ & $\Z/2\Z \oplus \Z/12\Z$ & $\Z/2\Z$ \\
	$5$ & $\Z \oplus \Z/2\Z \oplus \Z/60\Z$ & $\Z/2\Z \oplus \Z/60\Z$ & $\Z/2\Z$ \\
	$6$ & $\Z \oplus \Z/2\Z \oplus \Z/6\Z \oplus \Z/60\Z$ & $\Z/2\Z \oplus \Z/6\Z \oplus \Z/60\Z$ & $\Z/2\Z$ \\
	$7$ & $\Z \oplus \Z/2\Z \oplus \Z/6\Z \oplus \Z/420\Z$ & $\Z/2\Z \oplus \Z/6\Z \oplus \Z/420\Z$ & $\Z/2\Z$ \\
	$8$ & $\Z \oplus \Z/2\Z \oplus \Z/2\Z \oplus \Z/12\Z \oplus \Z/840\Z$ & $\Z/2\Z \oplus \Z/2\Z \oplus \Z/12\Z \oplus \Z/840\Z$ & $\Z/2\Z$ \\
	$9$ & $\Z \oplus \Z/2\Z \oplus \Z/6\Z \oplus \Z/12\Z \oplus \Z/2520\Z$ & $\Z/2\Z \oplus \Z/6\Z \oplus \Z/12\Z \oplus \Z/2520\Z$ & $\Z/2\Z$ \\
	$10$ & $\Z \oplus \Z/2\Z \oplus \Z/2\Z \oplus \Z/6\Z \oplus \Z/60\Z \oplus \Z/2520\Z$ & $\Z/2\Z \oplus \Z/2\Z \oplus \Z/6\Z \oplus \Z/60\Z \oplus \Z/2520\Z$ & $\Z/2\Z$ \\
	\hline
	\caption{Data for $R \coloneqq \C[V]^{S_{n}}$ where $S_{n}$ is the symmetric group on $n$ symbols}\label{table:symmetric_group_output}
\end{longtable}
The data in Table~\ref{table:symmetric_group_output} are examples where \eqref{eq:Groth=Z+Cl_general} clearly is false.
Nonetheless, the data contains intriguing information.
We give the following conjecture, where $V = \mathfrak{h} \oplus \mathfrak{h}^{*}$.
\begin{conj}
	The reduced Grothendieck group of $\C[V]^{S_{n}}$ has order $n!$.
\end{conj}

In Table~\ref{table:shephard_todd_output}, we include data for the exceptional irreducible complex reflection groups (that is, the Shephard-Todd groups) where $m$ is the Shephard-Todd number and $-$ denotes no data; for further details, see, e.g. \cite{cohen76}, \cite{shephardTodd54}.
The data seem to exhibit interesting phenomena as well.
\newpage
\begin{table}[]
	\begin{tabular}{| p{1em} | p{14.5em} | p{14.5em} | p{6.1em} |}
		\hline
		$m$ & \multicolumn{1}{l|}{$\Groth(R)$} & $\oGroth(R)$ & $\Cl(R)$ \\ \hline
		4   & $\Z \oplus \Z/24\Z$ & $\Z/24\Z$ & $\Z/3\Z$ \\
		5   & $\Z \oplus \Z/3\Z \oplus \Z/3\Z \oplus \Z/24\Z$ & $\Z/3\Z \oplus \Z/3\Z \oplus \Z/24\Z$ & $\Z/3\Z \oplus \Z/3\Z$  \\
		6   & $\Z \oplus \Z/2\Z \oplus \Z/24\Z$ & $\Z/2\Z \oplus \Z/24\Z$ & $\Z/6\Z$ \\
		7   & $\Z \oplus \Z/3\Z \oplus \Z/6\Z \oplus \Z/24\Z$ & $\Z/3\Z \oplus \Z/6\Z \oplus \Z/24\Z$ & $\Z/3\Z \oplus \Z/6\Z$ \\
		8   & $\Z \oplus \Z/4\Z \oplus \Z/24\Z$ & $\Z/4\Z \oplus \Z/24\Z$ & $\Z/4\Z$ \\
		9   & $\Z \oplus \Z/2\Z \oplus \Z/4\Z \oplus \Z/48\Z$ & $\Z/2\Z \oplus \Z/4\Z \oplus \Z/48\Z$ & $\Z/2\Z \oplus \Z/4\Z$ \\
		10  & $\Z \oplus \Z/12\Z \oplus \Z/24\Z$ & $\Z/12\Z \oplus \Z/24\Z$ & $\Z/12\Z$ \\
		11  & $\Z \oplus \Z/2\Z \oplus \Z/12\Z \oplus \Z/48\Z$ & $\Z/2\Z \oplus \Z/12\Z \oplus \Z/48\Z$ & $\Z/2\Z \oplus \Z/12\Z$ \\
		12  & $\Z \oplus \Z/2\Z \oplus \Z/24\Z$ & $\Z/2\Z \oplus \Z/24\Z$ & $\Z/2\Z$ \\
		13  & $\Z \oplus \Z/2\Z \oplus \Z/2\Z \oplus \Z/48\Z$& $\Z/2\Z \oplus \Z/2\Z \oplus \Z/48\Z$ & $\Z/2\Z \oplus \Z/2\Z$  \\
		14  & $\Z \oplus \Z/6\Z \oplus \Z/24\Z$ & $\Z/6\Z \oplus \Z/24\Z$ & $\Z/6\Z$ \\
		15  & $\Z \oplus \Z/2\Z \oplus \Z/6\Z \oplus \Z/48\Z$ & $\Z/2\Z \oplus \Z/6\Z \oplus \Z/48\Z$ & $\Z/2\Z \oplus \Z/6\Z$  \\
		16  & $\Z \oplus \Z/5\Z \oplus \Z/120\Z$ & $\Z/5\Z \oplus \Z/120\Z$ & $\Z/5\Z$ \\
		17  & $\Z \oplus \Z/10\Z \oplus \Z/120\Z$ & $\Z/10\Z \oplus \Z/120\Z$ & $\Z/10\Z$ \\
		18  & $\Z \oplus \Z/15\Z \oplus \Z/120\Z$ & $\Z/15\Z \oplus \Z/120\Z$ & $\Z/15\Z$ \\
		19  & $-$ & $-$ & $-$ \\
		20  & $\Z \oplus \Z/3\Z \oplus \Z/120\Z$ & $\Z/3\Z \oplus \Z/120\Z$ & $\Z/3\Z$ \\
		21  & $\Z \oplus \Z/6\Z \oplus \Z/120\Z$ & $\Z/6\Z \oplus \Z/120\Z$ & $\Z/6\Z$ \\
		22  & $\Z \oplus \Z/2\Z \oplus \Z/120\Z$ & $\Z/2\Z \oplus \Z/120\Z$ & $\Z/2\Z$ \\
		23  & $\Z \oplus \Z/2\Z \oplus \Z/2\Z \oplus \Z/30\Z$ & $\Z/2\Z \oplus \Z/2\Z \oplus \Z/30\Z$ & $\Z/2\Z$ \\
		24  & $\Z \oplus \Z/2\Z \oplus \Z/2\Z \oplus \Z/84\Z$ & $\Z/2\Z \oplus \Z/2\Z \oplus \Z/84\Z$ & $\Z/2\Z$ \\
		25  & $\Z \oplus \Z/3\Z \oplus \Z/3\Z \oplus \Z/72\Z$ & $\Z/3\Z \oplus \Z/3\Z \oplus \Z/72\Z$ & $\Z/3\Z$ \\
		26  & $\Z \oplus \Z/6\Z \oplus \Z/6\Z \oplus \Z/72\Z$ & $\Z/6\Z \oplus \Z/6\Z \oplus \Z/72\Z$ & $\Z/6\Z$ \\
		27  & $\Z \oplus \Z/6\Z \oplus \Z/6\Z \oplus \Z/180\Z$ & $\Z/6\Z \oplus \Z/6\Z \oplus \Z/180\Z$ & $\Z/2\Z$ \\
		28  & $\Z \oplus \Z/2\Z \oplus \Z/2\Z \oplus \Z/2\Z \oplus \Z/2\Z \oplus \Z/6\Z \oplus \Z/12\Z \oplus \Z/24\Z$ & $\Z/2\Z \oplus \Z/2\Z \oplus \Z/2\Z \oplus \Z/2\Z \oplus \Z/6\Z \oplus \Z/12\Z \oplus \Z/24\Z$ & $\Z/2\Z \oplus \Z/2\Z$  \\
		29  & $\Z \oplus \Z/2\Z \oplus \Z/2\Z \oplus \Z/2\Z \oplus \Z/2\Z \oplus \Z/4\Z \oplus \Z/16\Z \oplus \Z/240\Z$  & $\Z/2\Z \oplus \Z/2\Z \oplus \Z/2\Z \oplus \Z/2\Z \oplus \Z/4\Z \oplus \Z/16\Z \oplus \Z/240\Z$  & $\Z/2\Z$ \\
		30  & $\Z \oplus \Z/2\Z \oplus \Z/2\Z \oplus \Z/60\Z \oplus \Z/240\Z$ & $\Z/2\Z \oplus \Z/2\Z \oplus \Z/60\Z \oplus \Z/240\Z$ & $\Z/2\Z$ \\
		31  & $-$ & $-$ & $-$ \\
		32  & $-$ & $-$ & $-$ \\
		33  & $\Z \oplus \Z/2\Z \oplus \Z/2\Z \oplus \Z/2\Z \oplus \Z/36\Z \oplus \Z/360\Z$ & $\Z/2\Z \oplus \Z/2\Z \oplus \Z/2\Z \oplus \Z/36\Z \oplus \Z/360\Z$ & $\Z/2\Z$ \\
		34  & $-$ & $-$ & $-$ \\
		35  & $\Z \oplus \Z/2\Z \oplus \Z/6\Z \oplus \Z/12\Z \oplus \Z/360\Z$ & $\Z/2\Z \oplus \Z/6\Z \oplus \Z/12\Z \oplus \Z/360\Z$ & $\Z/2\Z$ \\
		36  & $-$ & $-$ & $-$ \\
		37  & $-$ & $-$ & $-$ \\ \hline
	\end{tabular}\caption{Data for Shephard-Todd groups}\label{table:shephard_todd_output}
\end{table}
\appendix
\chapter[Centre of the deformed preprojective algebra]{The centre of the deformed preprojective algebra}\label{app:centre_deform_preproj}
Let $Q$ be a quiver and $\bar{Q}$ be the double of $Q$.
From \S\ref{sec:deformed_preproj_deformed_ksings}, recall that the deformed preprojective algebra is defined to be 
\[
\Pi^{\uplambda}(Q) = k\bar{Q} / \left( \sum_{c \in {\sf{Q_{1}}}} [d, c] - \sum_{i \in {\sf{Q_{0}}}} \uplambda_{i}e_{i} \right),
\]
where $d = c^{*}$ is the double of the arrow $c$, ${\sf{Q_{0}}}$ is the set of all vertices, and ${\sf{Q_{1}}}$ is the set of all arrows.
We refer to $\uplambda = (\uplambda_{i})_{i \in {\sf{Q_{0}}}}$ as the parameters.
In this appendix, we only consider quivers with underlying extended Dynkin type $\teAn{n}$, which have $\sf{n}+1$ vertices labelled $i= 0, ..., \sf{n}$, and arrows $c_{i}: i \to (i+1)$ and $d_{i}: (i+1) \to i$.
This quiver is shown in Figure \ref{fig:dynDiaExAn}.
\begin{figure}[H]
	\centering
	\begin{tikzcd}
	& 1 \arrow[r, "c_{1}", swap] \arrow[dl, "d_{0}", swap, bend right = 30] & 2 \arrow[r, "c_{2}", swap] \arrow[l, "d_{1}", swap, bend right = 30] & \arrow[l, "d_{2}", swap, bend right = 30] ... \arrow[dr, "c_{i-1}", swap] & \\
	0 \arrow[ur, "c_{0}", swap] \arrow[dr, "d_{\sf{n}}", swap, bend right = 30] & & & & i \arrow[ul, "d_{i-1}", swap, bend right = 30] \arrow[dl, "c_{i}", swap] \\
	& \sf{n} \arrow[ul, "c_{\sf{n}}", swap] \arrow[r, "d_{\sf{n}-1}", swap, bend right = 30] & \sf{n}-1 \arrow[l, "c_{\sf{n}-1}", swap] \arrow[r, "d_{\sf{n}-2}", swap, bend right = 30] & \arrow[l, "c_{\sf{n}-2}", swap] ... \arrow[ur, "d_{i}", swap, bend right = 30] &
	\end{tikzcd}
	\caption{Dynkin diagram $\teAn{n}$}\label{fig:dynDiaExAn}
\end{figure}
In this case the defining relations are $d_{i-1}c_{i-1} - c_{i}d_{i} = \uplambda_{i}e_{i}$ for all $i \in {\sf{Q_{0}}}$.
In addition, throughout this appendix we will assume that $\sum_{i \in {\sf{Q_{0}}}} \uplambda_{i} = 0$.

In this appendix we consider the spherical subalgebra $\ePie$ and give an explicit description of $\ePie$ as a ring of the form $A$ defined in Chapter~\ref{ch:divisor_cl_grps}.
By Theorem~\ref{thm:S^lam_Morita_equiv_Pi^lam}, $\ePie \cong \cO^{\uplambda}$.
\begin{prop}\label{prop:app:central_elements_deformed}
	The elements
	\begin{align*}\label{eq:cntrele}
		u &= c_{0}c_{1} \cdots c_{\sf{n}-1}c_{\sf{n}}, \\
		v &= d_{\sf{n}}d_{\sf{n}-1} \cdots d_{1}d_{0}, \text{ and} \\
		x &= c_{0}d_{0}
	\end{align*}
	generate $\ePie$.
\end{prop}
\begin{proof}
	First, notice that $u$, $v$, and $x$ each start and end at vertex $0$, so they are contained in $\ePie$.
	As explained on page 611 of \cite{cbh98}, there is a filtration $\{F_k \mid \text{degree } k \ge 0 \}$ of the deformed preprojective algebra $\Pi^{\uplambda}$ which is uniquely defined by putting the $e_i$ in degree zero and the arrows in degree one.
	This restricts to a filtration of $\ePie$.
	Furthermore, by Lemma~\ref{lem:grS^lam_cong_skew_grp_O^lam_cong_coord_ring} it follows that the associated graded algebra of $\ePie$ is $e_0 \Pi^{0} e_0$ which is isomorphic to the type $\teAn{n}$ Kleinian singularity $\C[u,v,x]/(uv - x^{{\sf{n}}+1})$ \cite[Theorem~0.1]{cbh98}.
	By inspection, we know that the images of $u$, $v$, and $x$ generate $e_0 \Pi^{0} e_0$.
	
	Let $B$ be the subalgebra of $\ePie$ generated by $u,v,$ and $x$.
	We need to show that $B = \ePie$.
	Define a filtration ${ G_{k} }$ on $B$ by restriction, that is, a filtration $G_{k} = F_{k} \cap B$.
	Then the associated graded map $\mathrm{gr }  B \to \mathrm{gr }  \ePie \cong e_0 \Pi^{0} e_0$ is an embedding.
	It is also surjective by assumption.
	Hence $\mathrm{gr } B \to e_0 \Pi^{0} e_0$ is an isomorphism, and so $B \hookrightarrow \ePie$ is an isomorphism using \cite[Corollary~7.6.14]{mcconRob01}.
\end{proof}

To give a presentation of the ring $\ePie$ requires the following technical result.
\begin{lemma}\label{lem:app:1euve}
	If $j \geq 1$, then
	\[
	c_{0}c_{1} \cdots c_{j}d_{j} \cdots d_{1} d_{0}= (c_{0}d_{0}) \cdot \prod\limits_{k=1}^{j} (c_{0}d_{0} - (\uplambda_{1} + ... + \uplambda_{k})e_{0}).
	\]
\end{lemma}

\begin{proof}
	We begin with the base case.
	Let $j=1$, then we have
	$$\begin{array}{rl}
	c_{0}c_{1}d_{1}d_{0} &= (c_{0}d_{0})(c_{0}d_{0} - \uplambda_{1}e_{0})\\
	&= (c_{0}d_{0})\prod\limits_{k=1}^{j}(c_{0}d_{0} - (\uplambda_{1} + \hdots + \uplambda_{k})e_{0}).
	\end{array}$$
	So, the base case holds.
	Now, assume that the following holds for some $j \geq 1$, namely
	\[
	c_{0}c_{1} \cdots c_{j}d_{j} \cdots d_{1} d_{0}= (c_{0}d_{0}) \cdot \prod\limits_{k=1}^{j} (c_{0}d_{0} - (\uplambda_{1} + ... + \uplambda_{k})e_{0}).
	\]
	We want to show that it holds for $j+1$.
	Observe, that
	\[
	c_{0}c_{1} \cdots c_{j+1} d_{j+1} \cdots d_{1}d_{0} = c_{0} (c_{1} \cdots c_{j+1} d_{j+1} \cdots d_{1}) d_{0}.
	\]
	By the induction hypothesis applied to the vertex 1, we rewrite this as
	\[
	c_{0} \left[ (c_{1}d_{1}) \cdot \prod\limits_{k=1}^{j}(c_{1}d_{1} - (\uplambda_{2} +... +  \uplambda_{k+1})e_{1}) \right] d_{0}.
	\]
Using the relation $d_{0}c_{0} - c_{1}d_{1} = \uplambda_{1}e_{1}$, we substitute each $c_{1}d_{1}$ with $d_{0}c_{0} - \uplambda_{1}e_{1}$.
	Therefore,
	$$\begin{array}{rl}
	&c_{0}\left[ (c_{1}d_{1}) \cdot \prod\limits_{k=1}^{j}(c_{1}d_{1} - (\uplambda_{2} +... +  \uplambda_{k+1})e_{1})\right] d_{0} \\
	&\hspace{2cm} = c_{0} \left[ (d_{0}c_{0} - \uplambda_{1}e_{1}) \cdot \prod\limits_{k=1}^{j}(d_{0}c_{0} - \uplambda_{1}e_{1} - (\uplambda_{2} +... +  \uplambda_{k+1})e_{1}) \right] d_{0}.
	\end{array}$$
	Since $e_{1}d_{0} = d_{0}e_{0}$, the above is
	$$\begin{array}{rl}
	&\hspace{2cm} = (c_{0}d_{0})(c_{0}d_{0} - \uplambda_{1}e_{0}) \cdot \prod\limits_{k=1}^{j}(c_{0}d_{0} - \uplambda_{1}e_{0} - (\uplambda_{2} +... +  \uplambda_{k+1})e_{0})\\
	&\hspace{2cm} = (c_{0}d_{0}) \cdot  \prod\limits_{k=1}^{j+1}(c_{0}d_{0} - (\uplambda_{1} +... +  \uplambda_{k})e_{0}).
	\end{array}$$
	So, our claim holds for $j+1$.
\end{proof}

We are most interested in the $j=n$ special case of Lemma~\ref{lem:app:1euve}, which asserts that
$$\begin{array}{rl}
	uv &= c_{0}c_{1} \cdots c_{\sf{n}}d_{\sf{n}} \cdots d_{1}d_{0} \\
	&= (c_{0}d_{0}) \cdot \prod\limits_{k=1}^{\sf{n}} (c_{0}d_{0} - (\uplambda_{1} + ... + \uplambda_{k}))\\
	&= \prod\limits_{k=1}^{\sf{n}+1}(x - (\uplambda_{1} + \hdots + \uplambda_{k})).
\end{array}$$
It follows that there is an induced map $R \to \ePie$, where
\[
	R \coloneqq \frac{\C[u,v,x]}{\left( uv-\prod\limits_{k=1}^{\sf{n}+1}(x - (\uplambda_{1} + \hdots + \uplambda_{k})) \right)}.
\]
The following is the main result of this appendix.
\begin{theorem}\label{thm:app:preproj_type_A_isomorphic_ring_A}
	The map $R \to \ePie$ is an isomorphism.
\end{theorem}
\begin{proof}
	The ring $R$ is a 2-dimensional ring of the form $A$ defined in Chapter~\ref{ch:divisor_cl_grps}, so by Lemma~\ref{lem:A_is_normal_ID}, it is an integral domain.
	By \cite[\S0]{cbh98}, $\ePie$ is also a 2-dimensional integral domain.
	
	Proposition~\ref{prop:app:central_elements_deformed} shows that the map $R \to \ePie$ is surjective.
	Let $\mathfrak{a}$ be the kernel of this map.
	Since $\ePie$ is a domain, $\mathfrak{a}$ is a prime ideal.
	If $\height(\mathfrak{a}) = m$ then, by Theorem~\ref{thm:height+dim}, $\dim \ePie = \dim R - m$.
	In other words, $2 = 2 - m$, thus $\height(\mathfrak{a}) = 0$.
	However, $R$ is a domain, so the only height zero prime ideal is $0$.
	Hence, $R \cong \ePie$.
\end{proof}
\chapter{Computing the Grothendieck group using \texttt{MAGMA}}\label{app:magma_code}
The \texttt{MAGMA} code below computes the Grothendieck group (and divisor class group) of rings $R=S^{G}$, such as those discussed in this thesis.
Below, the input variable \texttt{G} is always a subgroup of $\GL(n, \C) = \GL(\mathfrak{h})$.
The function \texttt{Groth} can compute the Grothendieck group and class group in two settings:
\begin{enumerate}
	\item (\texttt{G}, \texttt{1}), then \texttt{Groth} computes the Grothendieck group of the invariant ring $\C[\mathfrak{h}]^{G}$.
	\item (\texttt{G}, \texttt{2}), then \texttt{Groth} computes the Grothendieck group of the invariant ring $\C[\mathfrak{h} \oplus \mathfrak{h}^{*}]^{G}$, where $\mathfrak{h}^{*}$ is the dual of $\mathfrak{h}$.
	This code is useful where $G$ is a complex reflection group.
\end{enumerate}
\begin{lstlisting}[language=magma]
NaturalRep:=function(G, ans);
	CR:=CharacterRing(G);
	if ans eq 1 then
		trCl:=[Trace(Classes(G)[i][3]) : i in [1..#Classes(G)]];
		V:= CR ! trCl;
	elif ans eq 2 then
		trCl:=[Trace(Classes(G)[i][3]) : i in [1..#Classes(G)]];
		trClInv:=[Trace(Transpose(Classes(G)[i][3]^(-1))) : i in [1..#Classes(G)]];
		h:= CR ! trCl;
		hdual:= CR ! trClInv;
		V:= h + hdual;
	end if;
return V;
end function;
//----------------------------------------------------------------
Groth:=function(G,ans);
	V:=NaturalRep(G,ans);	
	g:=[**];
	
	for i in [1..Ngens(G)] do
		Append(~g, G.i);
	end for;
	cl:=SubgroupClasses(G);
	table:=CharacterTable(G);
	
	l:=#cl;
	repSubgrp:=[**];
	norm:=[**];
	h:=[**];
	nh:=[**];
	trivH:=[**];
	
	repSubgrp:=[h`subgroup : h in cl | Order(h`subgroup) ne 1];
	l:=#repSubgrp;
	
	for i in [1..l] do
		norm[i]:=Normalizer(G,repSubgrp[i]);
		h[i]:=CharacterTable(repSubgrp[i]);
		nh[i]:=CharacterTable(norm[i]);
	end for;
	
	k:=[];
	sirr:=[];
	vh:=[];
	extpow:=[];
	aHsum:=[];
	aH:=[];
	aHG:=[];
	induced:=[];
	ps:=[];
	
	for i in [1..l] do
		sirr:=[char: char in nh[i] | InnerProduct(PrincipalCharacter( repSubgrp[i]), Restriction(char, repSubgrp[i])) eq 0 ];
	
		W:=Restriction(V, norm[i]);
	
		blh:=[char : char in nh[i] | (InnerProduct(char, W) ne 0) and (InnerProduct(Restriction( char, repSubgrp[i] ), PrincipalCharacter(repSubgrp[i])) eq 0)];
	
		blhchar:=[InnerProduct(char, W)*char : char in blh];
	
		vh:=&+ blhchar;
	
		if vh[1] le 1 then
			t:=[ x : x in Generators(repSubgrp[i]) ];
			Append(~ps, t[1]);
		end if;
	
		p:=[[0]] cat [[1 : k in [1..j]] : j in [1..Degree(vh)]];
	
		aH:=&+[(-1)^(j-1)*Symmetrization(vh, p[j]) : j in [1..Degree(vh)+1]];
	
		aHG:=[aH * char : char in sirr];
	
		for elt in aHG do
			Append(~induced, Induction(elt, G));
		end for;
	
		list:=[InnerProduct(induced[k], table[j]) : j in [1..#table], k in [1..#induced]];
		A:=Matrix(Integers(), #induced, #table, list);
	end for;
	
	groth := [0 : j in [1..(#table - #ElementaryDivisors(A))]] cat ElementaryDivisors(A);
	
	abGroth:=AbelianGroup(groth);
	F:=TorsionSubgroup(abGroth);
	
	sg:=sub< G | [G ! p : p in ps] >;
	pseudorefl:=NormalClosure(G, sg);		
	classGrp:=AbelianQuotient(G/pseudorefl);
	
return groth, F, classGrp;
end function;
\end{lstlisting}

\backmatter
\bibliographystyle{alpha}
\bibliography{phdCitations}
\end{document}